\documentclass[a4paper,12pt]{amsart}
\usepackage{amssymb,hyperref, verbatim,bbm, bm, tikz, color, scalefnt,ifthen}
\usetikzlibrary{calc,arrows}

\newtheorem{thm}{Theorem}[section]
\newtheorem{lemma}[thm]{Lemma}
\newtheorem{cor}[thm]{Corollary}
\newtheorem{prop}[thm]{Proposition}

\theoremstyle{definition}
\newtheorem{defi}[thm]{Definition}
\newtheorem{rem}[thm]{Remark}
\newtheorem{example}[thm]{Example}

\numberwithin{equation}{section}

\begin{document}

\author{Joseph~Chuang}
\address{Department of Mathematics, City University London, Northampton Square, London EC1V 0HB, United Kingdom.}
\email{Joseph.Chuang.1@city.ac.uk}

\author{Hyohe Miyachi}
\address{Department of Mathematics, Osaka City University, 3-3-138 Sugimoto
Sumiyoshi-ku, Osaka 558-8585, Japan.}
\email{miyachi@sci.osaka-cu.ac.jp}

\author{Kai~Meng~Tan}
\address{Department of Mathematics, National University of Singapore, Block S17, 10 Lower Kent Ridge Road, Singapore 119076.} \email{tankm@nus.edu.sg}

\thanks{Supported by EPSRC grants GR/T00917 and GR/T00924, JSPS Grant-in-Aid for Young Scientists (B) No. 24740011 and Singapore Ministry of Education Academic Research Fund R-146-000-172-112.}

\title[Parallelotope tilings and $q$-decomposition numbers]{Parallelotope tilings and \\ $q$-decomposition numbers}
\begin{abstract}
We provide closed formulas for a large subset of the canonical basis vectors of the Fock space representation of $U_q(\widehat{\mathfrak{sl}}_e)$.  These formulas arise from parallelotopes which assemble to form polytopal complexes.  The subgraphs of the $\operatorname{Ext}^1$-quivers of $v$-Schur algebras at complex $e$-th roots of unity generated by simple modules corresponding to these canonical basis vectors are given by the $1$-skeletons of the polytopal complexes.
\end{abstract}
\maketitle
\tableofcontents

\def\Abar{\bar{A}}
\def\Bbar{\bar{B}}
\def\Cbar{\bar{C}}
\def\Dbar{\bar{D}}
\def\Ebar{\bar{E}}
\def\Fbar{\bar{F}}
\def\Gbar{\bar{G}}
\def\Hbar{\bar{H}}
\def\Ibar{\bar{I}}
\def\Jbar{\bar{J}}
\def\Kbar{\bar{K}}
\def\Lbar{\bar{L}}
\def\Mbar{\bar{M}}
\def\Nbar{\bar{N}}
\def\Obar{\bar{O}}
\def\Pbar{\bar{P}}
\def\Qbar{\bar{Q}}
\def\Rbar{\bar{R}}
\def\Sbar{\bar{S}}
\def\Tbar{\bar{T}}
\def\Ubar{\bar{U}}
\def\Vbar{\bar{V}}
\def\Wbar{\bar{W}}
\def\Xbar{\bar{X}}
\def\Ybar{\bar{Y}}
\def\Zbar{\bar{Z}}

\def\abar{\bar{a}}
\def\bbar{\bar{b}}
\def\cbar{\bar{c}}
\def\dbar{\bar{d}}
\def\ebar{\bar{e}}
\def\fbar{\bar{f}}
\def\gbar{\bar{g}}
\def\hbar{\bar{h}}
\def\ibar{\bar{i}}
\def\jbar{\bar{j}}
\def\kbar{\bar{k}}
\def\lbar{\bar{l}}
\def\mbar{\bar{m}}
\def\nbar{\bar{n}}
\def\obar{\bar{o}}
\def\pbar{\bar{p}}
\def\qbar{\bar{q}}
\def\rbar{\bar{r}}
\def\sbar{\bar{s}}
\def\tbar{\bar{t}}
\def\ubar{\bar{u}}
\def\vbar{\bar{v}}
\def\wbar{\bar{w}}
\def\xbar{\bar{x}}
\def\ybar{\bar{y}}
\def\zbar{\bar{z}}

\def\rhobar{\bar{\rho}}
\def\sigmabar{\bar{\sigma}}

\def\lhd{\vartriangleleft}
\def\mapsfrom{\leftarrow\!\shortmid}
\def\Rarr#1{\buildrel #1\over \longrightarrow}
\def\Rarrsquig#1{\buildrel #1\over \rightsquigarrow}
\def\iso{\buildrel \sim\over\to}
\def\osi{\buildrel \sim\over\leftarrow}

\def\GA{{\mathfrak{A}}}
\def\GI{{\mathfrak{I}}}
\def\GM{{\mathfrak{M}}}
\def\GS{{\mathfrak{S}}}
\def\GU{{\mathfrak{U}}}
\def\Sn{{\mathfrak{S}}}

\def\Ga{{\mathfrak{a}}}
\def\Gb{{\mathfrak{b}}}
\def\Gg{{\mathfrak{g}}}
\def\Gh{{\mathfrak{h}}}
\def\Gm{{\mathfrak{m}}}
\def\Gn{{\mathfrak{n}}}
\def\Gp{{\mathfrak{p}}}
\def\Gu{{\mathfrak{u}}}

\def\Ggl{{\mathfrak{gl}}}
\def\Gsl{{\mathfrak{sl}}}

\def\Db{{\mathcal D}^b}
\def\CA{{\mathcal{A}}}
\def\CB{{\mathcal{B}}}
\def\CC{{\mathcal{C}}}
\def\CD{{\mathcal{D}}}
\def\CE{{\mathcal{E}}}
\def\CF{{\mathcal{F}}}
\def\CG{{\mathcal{G}}}
\def\CH{{\mathcal{H}}}
\def\CI{{\mathcal{I}}}
\def\CJ{{\mathcal{J}}}
\def\CK{{\mathcal{K}}}
\def\CL{{\mathcal{L}}}
\def\CM{{\mathcal{M}}}
\def\CN{{\mathcal{N}}}
\def\CO{{\mathcal{O}}}
\def\CP{{\mathcal{P}}}
\def\CQ{{\mathcal{Q}}}
\def\CR{{\mathcal{R}}}
\def\CS{{\mathcal{S}}}
\def\CT{{\mathcal{T}}}
\def\CU{{\mathcal{U}}}
\def\CV{{\mathcal{V}}}
\def\CW{{\mathcal{W}}}
\def\CX{{\mathcal{X}}}
\def\CY{{\mathcal{Y}}}
\def\CZ{{\mathcal{Z}}}
\def\Cb{{\mathcal{b}}}

\def\BA{{\mathbf{A}}}
\def\BB{{\mathbf{B}}}
\def\BC{{\mathbf{C}}}
\def\BD{{\mathbf{D}}}
\def\BE{{\mathbf{E}}}
\def\BF{{\mathbf{F}}}
\def\BG{{\mathbf{G}}}
\def\BH{{\mathbf{H}}}
\def\BI{{\mathbf{I}}}
\def\BJ{{\mathbf{J}}}
\def\BK{{\mathbf{K}}}
\def\BL{{\mathbf{L}}}
\def\BM{{\mathbf{M}}}
\def\BN{{\mathbf{N}}}
\def\BO{{\mathbf{O}}}
\def\BP{{\mathbf{P}}}
\def\BQ{{\mathbf{Q}}}
\def\BR{{\mathbf{R}}}
\def\BS{{\mathbf{S}}}
\def\BT{{\mathbf{T}}}
\def\BU{{\mathbf{U}}}
\def\BV{{\mathbf{V}}}
\def\BW{{\mathbf{W}}}
\def\BX{{\mathbf{X}}}
\def\BY{{\mathbf{Y}}}
\def\BZ{{\mathbf{Z}}}
\def\Bh{{\mathbf{h}}}
\def\Bk{{\mathbf{k}}}
\def\Bs{{\mathbf{s}}}
\def\Bu{{\mathbf{u}}}
\def\Bv{{\mathbf{v}}}
\def\Bw{{\mathbf{w}}}
\def\Bpi{{\mathbf{\pi}}}
\def\Zp{\BZ_p}
\def\eps{\varepsilon}

\renewcommand{\BZ}{\mathbb{Z}}
\renewcommand{\BC}{\mathbb{C}}
\def\Ab{\operatorname{Ab}\nolimits}
\def\ad{\operatorname{ad}\nolimits}
\def\add{\operatorname{add}\nolimits}
\def\Alg{\operatorname{Alg}\nolimits}
\def\Ann{\operatorname{Ann}\nolimits}
\def\smd{\operatorname{smd}\nolimits}
\def\Ass{\operatorname{Ass}\nolimits}
\def\Aut{\operatorname{Aut}\nolimits}
\def\br{\operatorname{br}\nolimits}
\def\Br{\operatorname{Br}\nolimits}
\def\Branch{\operatorname{Branch}\nolimits}
\def\Bun{\operatorname{Bun}\nolimits}
\def\can{{\mathrm{can}}}
\def\Cat{\operatorname{Cat}\nolimits}
\def\Ch{\operatorname{Ch}\nolimits}
\def\ch{\operatorname{ch}\nolimits}
\def\Cl{\operatorname{Cl}\nolimits}
\def\coh{\operatorname{coh}\nolimits}
\def\mcoh{\operatorname{\!-coh}\nolimits}
\def\cohgr{\operatorname{cohgr}\nolimits}
\def\mcohgr{\operatorname{\!-cohgr}\nolimits}
\def\coker{\operatorname{coker}\nolimits}
\def\colim{\operatorname{colim}\nolimits}
\def\Comp{\operatorname{Comp}\nolimits}
\def\cone{\operatorname{cone}\nolimits}
\def\Der{\operatorname{Der}\nolimits}
\def\DG{\operatorname{DG}\nolimits}
\def\DGProj{\operatorname{DGProj}\nolimits}
\def\diag{\operatorname{diag}\nolimits}
\def\diff{\operatorname{diff}\nolimits}
\def\mdiff{\operatorname{\!-diff}\nolimits}
\def\Diff{\operatorname{Diff}\nolimits}
\def\dimgr{\operatorname{dimgr}\nolimits}
\def\DPic{\operatorname{DPic}\nolimits}
\def\exact{\operatorname{exact}\nolimits}
\def\en{{\mathrm{en}}}
\def\End{\operatorname{End}\nolimits}
\def\Endgr{\operatorname{Endgr}\nolimits}
\def\Ens{\operatorname{Ens}\nolimits}
\def\mens{\operatorname{\!-ens}\nolimits}
\def\Ep{\operatorname{Ep}\nolimits}
\def\Ext{\operatorname{Ext}\nolimits}
\def\Fix{\operatorname{Fix}\nolimits}
\def\Fl{\operatorname{Fl}\nolimits}
\def\Fun{\operatorname{Fun}\nolimits}
\def\Gal{\operatorname{Gal}\nolimits}
\def\GL{\operatorname{GL}\nolimits}
\def\gldim{\operatorname{gldim}\nolimits}
\def\good{\operatorname{good}\nolimits}
\def\gr{{\operatorname{gr}\nolimits}}
\def\mgr{{\operatorname{\!-gr}\nolimits}}
\def\grmod{{\operatorname{grmod}\nolimits}}
\def\Gr{{\operatorname{Gr}\nolimits}}
\def\Grp{\operatorname{Grp}\nolimits}
\def\hd{\operatorname{hd}\nolimits}
\def\hdim{\operatorname{hdim}\nolimits}
\def\Hilb{\operatorname{Hilb}\nolimits}
\def\Ho{\operatorname{Ho}\nolimits}
\def\hocolim{\operatorname{hocolim}\nolimits}
\def\holim{\operatorname{holim}\nolimits}
\def\Holim{\operatorname{Holim}\nolimits}
\def\idim{\operatorname{idim}\nolimits}
\def\invlim{\operatorname{invlim}\nolimits}
\def\Hom{\operatorname{Hom}\nolimits}
\def\Homb{\operatorname{\overline{Hom}}\nolimits}
\def\Homgr{\operatorname{Homgr}\nolimits}
\def\IC{\operatorname{IC}\nolimits}
\def\Id{\operatorname{Id}\nolimits}
\def\id{\operatorname{id}\nolimits}
\def\im{\operatorname{im}\nolimits}
\def\ind{\operatorname{ind}\nolimits}
\def\Ind{\operatorname{Ind}\nolimits}
\def\inj{\operatorname{inj}\nolimits}
\def\Inj{\operatorname{Inj}\nolimits}
\def\injgr{\operatorname{injgr}\nolimits}
\def\minjgr{\operatorname{\!-injgr}\nolimits}
\def\mInjgr{\operatorname{\!-Injgr}\nolimits}
\def\Inn{\operatorname{Inn}\nolimits}
\def\Int{\operatorname{Int}\nolimits}
\def\Irr{\operatorname{Irr}\nolimits}
\def\krull{\operatorname{Krulldim}\nolimits}
\def\mlatt{\operatorname{\!-latt}\nolimits}
\def\mLatt{\operatorname{\!-Latt}\nolimits}
\def\Lie{\operatorname{Lie}\nolimits}
\def\mMod{\operatorname{\!-mod}\nolimits}
\def\MOD{\operatorname{Mod}\nolimits}
\def\mMOD{\operatorname{\!-Mod}\nolimits}
\def\MODgr{\operatorname{Modgr}\nolimits}
\def\mMODgr{\operatorname{\!-Modgr}\nolimits}
\def\Modgr{\operatorname{modgr}\nolimits}
\def\mModgr{\operatorname{\!-modgr}\nolimits}
\def\modgr{\operatorname{modgr}\nolimits}
\def\mmodgr{\operatorname{\!-modgr}\nolimits}
\def\Ob{{\operatorname{Ob}\nolimits}}
\def\odd{{\operatorname{odd}\nolimits}}
\def\opp{{\operatorname{opp}\nolimits}}
\def\parf{\operatorname{parf}\nolimits}
\def\mparf{\operatorname{\!-parf}\nolimits}
\def\mperf{\operatorname{\!-perf}\nolimits}
\def\parfgr{\operatorname{parfgr}\nolimits}
\def\mparfgr{\operatorname{\!-parfgr}\nolimits}
\def\pdim{\operatorname{pdim}\nolimits}
\def\perf{\operatorname{perf}\nolimits}
\def\perfgr{\operatorname{perfgr}\nolimits}
\def\PGL{\operatorname{PGL}\nolimits}
\def\plat{\operatorname{plat}\nolimits}
\def\perm{\operatorname{perm}\nolimits}
\def\Perv{\operatorname{Perv}\nolimits}
\def\Pic{\operatorname{Pic}\nolimits}
\def\Picent{\operatorname{Picent}\nolimits}
\def\mproj{\operatorname{\!-proj}\nolimits}
\def\mProj{\operatorname{\!-Proj}\nolimits}
\def\pr{\operatorname{pr}\nolimits}
\def\proj{\operatorname{proj}\nolimits}
\def\Proj{\operatorname{Proj}\nolimits}
\def\Projgr{\operatorname{Projgr}\nolimits}
\def\mProjgr{\operatorname{\!-Projgr}\nolimits}
\def\Projspec{\operatorname{Proj}\nolimits}
\def\projgr{\operatorname{projgr}\nolimits}
\def\mprojgr{\operatorname{\!-projgr}\nolimits}
\def\PSL{\operatorname{PSL}\nolimits}
\def\Out{\operatorname{Out}\nolimits}
\def\qcoh{\operatorname{qcoh}\nolimits}
\def\mqcoh{\operatorname{\!-qcoh}\nolimits}
\def\rad{\operatorname{rad}\nolimits}
\def\ram{\operatorname{ram}\nolimits}
\def\rank{\operatorname{rank}\nolimits}
\def\rang{\operatorname{rang}\nolimits}
\def\rg{\operatorname{rg}\nolimits}
\def\relproj{\operatorname{relproj}\nolimits}
\def\mrelProj{\operatorname{\!-Relproj}\nolimits}
\def\Rep{\operatorname{Rep}\nolimits}
\def\res{\operatorname{res}\nolimits}
\def\Res{\operatorname{Res}\nolimits}
\def\Sc{\operatorname{Sc}\nolimits}
\def\Sch{\operatorname{Sch}\nolimits}
\def\sets{\operatorname{sets}\nolimits}
\def\sgn{\operatorname{sgn}\nolimits}
\def\Sh{\operatorname{Sh}\nolimits}
\def\Shift{\operatorname{Shift}\nolimits}
\def\SL{\operatorname{SL}\nolimits}
\def\SO{\operatorname{SO}\nolimits}
\def\SU{\operatorname{SU}\nolimits}
\def\soc{\operatorname{soc}\nolimits}
\def\Sp{\operatorname{Sp}\nolimits}
\def\Spec{\operatorname{Spec}\nolimits}
\def\Specm{\operatorname{Specm}\nolimits}
\def\St{\operatorname{St}\nolimits}
\def\stab{\operatorname{stab}\nolimits}
\def\mstab{\operatorname{\!-stab}\nolimits}
\def\Stab{\operatorname{Stab}\nolimits}
\def\mStab{\operatorname{\!-Stab}\nolimits}
\def\stabgr{\operatorname{stabgr}\nolimits}
\def\mstabgr{\operatorname{\!-stabgr}\nolimits}
\def\Stmod{\operatorname{Stmod}\nolimits}
\def\StPic{\operatorname{StPic}\nolimits}
\def\Supp{\operatorname{Supp}\nolimits}
\def\Sz{\operatorname{Sz}\nolimits}
\def\Tate{\operatorname{Tate}\nolimits}
\def\mtilt{\operatorname{\!-tilt}\nolimits}
\def\tilt{\operatorname{tilt}\nolimits}
\def\Tor{\operatorname{Tor}\nolimits}
\def\Tr{\operatorname{Tr}\nolimits}
\def\trace{\operatorname{trace}\nolimits}
\def\Trd{\operatorname{Trd}\nolimits}
\def\Tria{\operatorname{Tria}\nolimits}
\def\TrI{\operatorname{TrI}\nolimits}
\def\TrPic{\operatorname{TrPic}\nolimits}
\def\TrPicent{\operatorname{TrPicent}\nolimits}
\def\Uch{\operatorname{Uch}\nolimits}
\def\Vect{\operatorname{Vect}\nolimits}

\def\ie{{\em i.e.}}
\def\eg{{\em e.g.}}
\def\Qlbar{{\bar{\mathbf{Q}}_l}}
\def\Fqbar{{\bar{\mathbf{F}}_q}}

\def\ta{{\tilde{a}}}
\def\tb{{\tilde{b}}}
\def\te{{\tilde{e}}}
\def\tf{{\tilde{f}}}
\def\ti{{\tilde{i}}}
\def\tp{{\tilde{p}}}
\def\ts{{\tilde{s}}}
\def\tA{{\tilde{A}}}
\def\tB{{\tilde{B}}}
\def\tC{{\tilde{C}}}
\def\tD{{\tilde{D}}}
\def\tE{{\tilde{E}}}
\def\tF{{\tilde{F}}}
\def\tG{{\tilde{G}}}
\def\tH{{\tilde{H}}}
\def\tL{{\tilde{L}}}
\def\tM{{\tilde{M}}}
\def\tN{{\tilde{N}}}
\def\tP{{\tilde{P}}}
\def\tQ{{\tilde{Q}}}
\def\tR{{\tilde{R}}}
\def\tS{{\tilde{S}}}
\def\tT{{\tilde{T}}}
\def\tV{{\tilde{V}}}
\def\tW{{\tilde{W}}}
\def\tX{{\tilde{X}}}
\def\tY{{\tilde{Y}}}
\def\tphi{{\tilde{\phi}}}
\def\tHom{{\widetilde{\Hom}}}
\def\tDelta{{\tilde{\Delta}}}
\def\tgamma{{\tilde{\gamma}}}
\def\tGamma{{\tilde{\Gamma}}}
\def\tOmega{{\tilde{\Omega}}}
\def\tzeta{{\tilde{\zeta}}}
\def\tCE{{\tilde{\CE}}}
\def\tCO{{\tilde{\CO}}}
\def\tCP{{\tilde{\CP}}}
\def\tGg{{\tilde{\Gg}}}
\def\CExt{{{\mathcal E}xt}}
\def\CHom{{{\mathcal H}om}}
\def\CEnd{{{\mathcal E}nd}}

\def\uEnd{\operatorname{\underline{End}}\nolimits}
\def\uHom{\operatorname{\underline{Hom}}\nolimits}
\def\uA{\operatorname{\underline{A}}\nolimits}
\def\uH{\operatorname{\underline{H}}\nolimits}
\def\uS{\operatorname{\underline{S}}\nolimits}
\def\uW{\operatorname{\underline{W}}\nolimits}
\def\uc{\operatorname{\underline{c}}\nolimits}
\def\ui{\operatorname{\underline{i}}\nolimits}
\def\uj{\operatorname{\underline{j}}\nolimits}
\def\us{\operatorname{\underline{s}}\nolimits}
\def\ut{\operatorname{\underline{t}}\nolimits}
\def\uw{\operatorname{\underline{w}}\nolimits}

\def\dv{{\dot{v}}}
\def\dw{{\dot{w}}}

\def\la{\langle}
\def\ra{\rangle}




%
\def\mmid{\:\middle|\:}

\def\uleq{\mathcal{U}_\leq}


\def\inc#1#2{{#1}_{\geq #2}}

\def\hull#1{{#1}^{\mathbb{R}}}
\def\hhull#1{{#1}^{h\mathbb{R}}}

\def\hullinc#1#2{\hull{#1}_{\geq #2}}
\def\hhullinc#1#2{\hhull{#1}_{\geq #2}}
\def\plusinc#1#2{{#1}^{+}_{\geq #2}}
\def\reginc#1#2{{#1}^{\operatorname{reg}}_{\geq #2}}
\def\crossing{\kappa}

\def\proj{\operatorname{pr}}

\def\dist#1#2{\mathfrak{d}_{#1}(#2)}

\def\sb{\epsilon}
\def\ssb{\mathbf{e}}

\def\para{\Pi}
\def\rpara#1{\hull{\para(#1)}}
\def\hrpara#1{h\rpara{#1}}
\def\incrpara#1#2{\hullinc{\para(#1)}{#2}}
\def\inchrpara#1#2{\hullinc{h\para(#1)}{#2}}
\def\cube{C}
\def\rcube{\cube^{\mathbb{R}}}

\def\umv{\circ}
\def\mv{\bullet}

\def\smv{\rule[.2ex]{1ex}{1ex}}

\def\tz{\hat{z}}

\def\bm{\mathcal{A}}

\def\good{\mathcal{U}}
\def\goodplus{\mathcal{U}^+}
\def\goodreg{\mathcal{U}^{\operatorname{reg}}}
\def\hrgood#1{\mathcal{U}_{#1}}
\def\rgood#1{\mathcal{V}_{#1}}
\def\wt{\operatorname{wt}}
\def\tilingregion{\mathcal{V}}

\def\bd{\phi}

\def\b{B}
\def\tb{\tilde{B}}
\def\cb{\check{B}}
\def\hb{\hat{B}}

\def\y{\mathfrak{y}} 
\def\w{\bar{w}} 
\def\I{\mathbbm{1}} 
\def\z{\mathfrak{z}} 
\def\M{\mathfrak{m}} 
\def\W{\mathcal{W}} 

\def\Zq{\mathbb{Z}[q,q^{-1}]}

\def\tlambda{{\tilde{\lambda}}}
\def\tmu{{\tilde{\mu}}}
\def\tnu{{\tilde{\nu}}}
\def\talpha{{\tilde{\alpha}}}
\def\tsigma{{\tilde{\sigma}}}
\def\ttau{{\tilde{\tau}}}
\def\teta{{\tilde{\eta}}}
\def\tkappa{{\tilde{\kappa}}}
\def\clambda{{\check{\lambda}}}
\def\cmu{{\check{\mu}}}
\def\cnu{{\check{\nu}}}
\def\calpha{{\check{\alpha}}}
\def\csigma{{\check{\sigma}}}
\def\ctau{{\check{\tau}}}
\def\ceta{{\check{\eta}}}
\def\ckappa{{\check{\kappa}}}
\def\flambda{{\boldsymbol{\lambda}}}
\def\fmu{{\boldsymbol{\mu}}}
\def\fnu{{\boldsymbol{\nu}}}
\def\falpha{{\boldsymbol{\alpha}}}
\def\fsigma{{\boldsymbol{\sigma}}}
\def\ftau{{\boldsymbol{\tau}}}
\def\feta{{\boldsymbol{\eta}}}

\def\ep#1#2{#1^{#2}}
\def\epl#1{\ep{\lambda}{#1}}
\def\epm#1{\ep{\mu}{#1}}
\def\epn#1{\ep{\nu}{#1}}
\def\eps#1{\ep{\sigma}{#1}}
\def\eptl#1{\ep{\tlambda}{#1}}
\def\eptm#1{\ep{\tmu}{#1}}
\def\eptn#1{\ep{\tnu}{#1}}
\def\epts#1{\ep{\tsigma}{#1}}

\def\cheve{E}
\def\chevf{F}

\def\kashe{\tilde{E}}
\def\kashf{\tilde{F}}

\newcommand*{\pb}[1]{\eta^{#1}}
\newcommand*{\hpb}[1]{\hat\eta^{#1}}

\def\proj{p}

\newcommand*{\internal}[1]{\operatorname{Int}(#1)}
\newcommand*{\external}[1]{\operatorname{Ext}(#1)}

\def\rel{\operatorname{rel}}
\def\irr{\operatorname{irr}}

\def\su{\operatorname{su}}

\newcommand*{\beadoperation}[1]{\mathfrak{B}_{#1}}

%
\def\Hook{\operatorname{Hook}}



\def\myscaleA{1}
\def\mynodescaleA{0.6}

\def\mylinewidthA{0.8}

\def\myvertexscaleA{.4}

\def\myarrowlinewidthA{1.5}
\def\myarrowlengthA{.35}
\def\myabsolutearrowlengthA{12}
\def\myarrowcolorA{red}


\def\myscale{1.3}
\def\mysmallscale{1}
\def\mynodescale{0.6}
\def\xshift{2}
\def\yshift{12}
\def\zshift{22}

\def\anglea{-14}
\def\angleb{0}
\def\rotationangle{95}

\def\mylinewidth{0.8}
\def\mysmallerlinewidth{0.3}
\def\hiddenlineopacity{0.2}

\def\myvertexscale{.7}

\def\myarrowlinewidth{1}
\def\myarrowlength{.29}
\def\myabsolutearrowlength{12}
\def\myarrowcolor{red}

\def\xorigin{2}
\def\yorigin{-6.0}
\def\zorigin{2}	

\newcommand*{\mathcolor}{}
\def\mathcolor#1#{\mathcoloraux{#1}}
\newcommand*{\mathcoloraux}[3]{%
	\protect\leavevmode
	\begingroup
	\color#1{#2}#3%
	\endgroup
}

\newcommand{\mysquare}[3]
{
	\coordinate (P00) at #1;
	\coordinate (P10) at ($#1+#2$);
	\coordinate (P01) at ($#1+#3$);
	\coordinate (P11) at ($#1+#2+#3$);
	\draw [line width=\myscaleA*\mylinewidthA, color=black] (P00) -- (P01) -- (P11) -- (P10) -- cycle;
}

\newcommand{\mysquarearrows}[4]
{
	\coordinate (P00) at #1;
	\coordinate (A10) at ($#1+\myarrowlengthA*#2$);
	\coordinate (A01) at ($#1+\myarrowlengthA*#3$);
	\draw [<->, >=latex, line width=\myscaleA*\myarrowlinewidthA, color=\myarrowcolorA] (A01) -- (P00) -- (A10);
}

\newcommand{\mycoloredshiftedcube}[7]
{
	\coordinate (C000) at ($#1+#2$);
	\coordinate (C100) at ($#1+#2+#3$);
	\coordinate (C010) at ($#1+#2+#4$);
	\coordinate (C001) at ($#1+#2+#5$);
	\coordinate (C110) at ($#1+#2+#3+#4$);
	\coordinate (C011) at ($#1+#2+#4+#5$);
	\coordinate (C101) at ($#1+#2+#3+#5$);
	\coordinate (C111) at ($#1+#2+#3+#4+#5$);
	\draw [line width=\myscale*\mysmallerlinewidth, color=#6] (C000) -- (C100);
	\draw [line width=\myscale*\mysmallerlinewidth, color=#6] (C000) -- (C010);
	\draw [line width=\myscale*\mysmallerlinewidth, color=#6] (C000) -- (C001);
	\draw [line width=\myscale*\mysmallerlinewidth, color=#6] (C100) -- (C110);
	\draw [line width=\myscale*\mysmallerlinewidth, color=#6] (C100) -- (C101);
	\draw [line width=\myscale*\mysmallerlinewidth, color=#6] (C010) -- (C110);
	\draw [line width=\myscale*\mysmallerlinewidth, color=#6] (C010) -- (C011);
	\draw [line width=\myscale*\mysmallerlinewidth, color=#6] (C001) -- (C101);
	\draw [line width=\myscale*\mysmallerlinewidth, color=#6] (C001) -- (C011);
	\draw [line width=\myscale*\mysmallerlinewidth, color=#6] (C110) -- (C111);
	\draw [line width=\myscale*\mysmallerlinewidth, color=#6] (C101) -- (C111);
	\draw [line width=\myscale*\mysmallerlinewidth, color=#6] (C011) -- (C111);
}

\newcommand{\myopaqueshiftedcube}[7]
{	\coordinate (C000) at ($#1+#2$);
	\coordinate (C100) at ($#1+#2+#3$);
	\coordinate (C010) at ($#1+#2+#4$);
	\coordinate (C001) at ($#1+#2+#5$);
	\coordinate (C110) at ($#1+#2+#3+#4$);
	\coordinate (C011) at ($#1+#2+#4+#5$);
	\coordinate (C101) at ($#1+#2+#3+#5$);
	\coordinate (C111) at ($#1+#2+#3+#4+#5$);
	\draw [line width = \myscale*\mylinewidth, draw=#6, fill=white, draw opacity=1, fill opacity=#7] (C010) -- (C110) -- (C100) -- (C101) -- (C001) -- (C011) -- cycle;	
\draw [line width = \myscale*\mylinewidth, draw=#6, fill=white, draw opacity=1, fill opacity=#7] (C000) -- (C100);
\draw [line width = \myscale*\mylinewidth, draw=#6, fill=white, draw opacity=1, fill opacity=#7] (C000) -- (C010);
\draw [line width = \myscale*\mylinewidth, draw=#6, fill=white, draw opacity=1, fill opacity=#7] (C000) -- (C001);
}

\newcommand{\mycubearrows}[7]
{
	\coordinate (C000) at ($#1+#2$);
	\coordinate (A100) at ($#1+#2+\myarrowlength*#3$);
	\coordinate (A010) at ($#1+#2+\myarrowlength*#4$);
	\coordinate (A001) at ($#1+#2+\myarrowlength*#5$);
	\draw [line width=\myscale*\myarrowlinewidth, color=\myarrowcolor, ->, >=latex] (C000) -- (A100);
	\draw [line width=\myscale*\myarrowlinewidth, color=\myarrowcolor, ->, >=latex] (C000) -- (A010);
	\draw [line width=\myscale*\myarrowlinewidth, color=\myarrowcolor, ->, >=latex] (C000) -- (A001);
	}

\newcommand{\savedx}{0}
\newcommand{\savedy}{0}
\newcommand{\savedz}{0}

\newcommand{\rotateRPY}[4][0/0/0]
{   \pgfmathsetmacro{\rollangle}{#2}
	\pgfmathsetmacro{\pitchangle}{#3}
	\pgfmathsetmacro{\yawangle}{#4}
	
	\pgfmathsetmacro{\newxx}{cos(\yawangle)*cos(\pitchangle)}
	\pgfmathsetmacro{\newxy}{sin(\yawangle)*cos(\pitchangle)}
	\pgfmathsetmacro{\newxz}{-sin(\pitchangle)}
	\path (\newxx,\newxy,\newxz);
	\pgfgetlastxy{\nxx}{\nxy};
	
	\pgfmathsetmacro{\newyx}{cos(\yawangle)*sin(\pitchangle)*sin(\rollangle)-sin(\yawangle)*cos(\rollangle)}
	\pgfmathsetmacro{\newyy}{sin(\yawangle)*sin(\pitchangle)*sin(\rollangle)+ cos(\yawangle)*cos(\rollangle)}
	\pgfmathsetmacro{\newyz}{cos(\pitchangle)*sin(\rollangle)}
	\path (\newyx,\newyy,\newyz);
	\pgfgetlastxy{\nyx}{\nyy};
	
	\pgfmathsetmacro{\newzx}{cos(\yawangle)*sin(\pitchangle)*cos(\rollangle)+ sin(\yawangle)*sin(\rollangle)}
	\pgfmathsetmacro{\newzy}{sin(\yawangle)*sin(\pitchangle)*cos(\rollangle)-cos(\yawangle)*sin(\rollangle)}
	\pgfmathsetmacro{\newzz}{cos(\pitchangle)*cos(\rollangle)}
	\path (\newzx,\newzy,\newzz);
	\pgfgetlastxy{\nzx}{\nzy};
	
	\foreach \x/\y/\z in {#1}
	{   \pgfmathsetmacro{\transformedx}{\x*\newxx+\y*\newyx+\z*\newzx}
		\pgfmathsetmacro{\transformedy}{\x*\newxy+\y*\newyy+\z*\newzy}
		\pgfmathsetmacro{\transformedz}{\x*\newxz+\y*\newyz+\z*\newzz}
		\xdef\savedx{\transformedx}
		\xdef\savedy{\transformedy}
		\xdef\savedz{\transformedz}
	}
}

\tikzset{RPY/.style={x={(\nxx,\nxy)},y={(\nyx,\nyy)},z={(\nzx,\nzy)}}}

\section{Introduction}\label{S:intro}

Fix an integer $e\geq 2$. Leclerc and Thibon \cite{LT} associate to any pair of partitions $\lambda, \mu$ of the same size a polynomial $d_{\lambda\mu}(q)$, defined as a change of basis coefficient in a Fock space representation of the quantised enveloping algebra $U_q(\widehat{\Gsl}_e)$. The $d_{\lambda\mu}(q)$ were originally called {\em $q$-decomposition numbers} because upon evaluation at $q=1$ they are decomposition numbers of $v$-Schur algebras with the parameter $v$ taken to be a primitive $e$-th root of unity in a field of characteristic zero \cite{A1,VV}.
In fact the $d_{\lambda\mu}(q)$ may be interpreted, before specialisation at $q=1$,
as certain parabolic affine Kazhdan-Lusztig polynomials \cite{VV}, or alternatively as graded decomposition numbers of graded versions
of these $v$-Schur algebras \cite{BK,A2,SW}.

In this paper we show that if the partitions $\lambda$ and $\mu$ are generic, in a sense we make precise below, then either $d_{\lambda\mu}(q)=0$ or
$d_{\lambda\mu}(q)=q^n$ for some $n\geq 0$, and that these generic $q$-decomposition numbers are governed by configurations of parallelotopes.
Such configurations first appeared in the representation theory of symmetric groups and Schur algebras as (two-dimensional) parallelogram tilings in the work of Erdmann and Martin \cite{EM} and of Peach \cite{Peach}. Inspired by \cite{Peach}, Turner considered certain higher dimensional parallelotope tilings as data to define algebras with favourable homological properties, and conjectured a structural relationship between these `Cubist algebras' and blocks of symmetric group algebras over fields of positive characteristic \cite[\S13]{Turner}. His prediction is known to be true for weight $2$ blocks (and $2$-dimensional tilings) \cite[\S11]{CT}; the general case remains open. In this paper we define parallelotope tilings relevant to Turner's conjecture
for blocks of arbitrary weight, in terms of the combinatorics of partitions, and prove a numerical analogue of the conjecture.

James \cite{J} conjectured that the decomposition numbers of a $v$-Schur algebra of degree $n$ with quantum characteristic $e$ over a field of positive characteristic $p$ coincide with those of the corresponding $v$-Schur algebra over a field of characteristic zero, as long as $e>n/p$; a more general version predicts that the same is true for all weight $w$ blocks of $v$-Schur algebras when $e>w$. Williamson \cite{Williamson} found counterexamples to the (original) conjecture of James, but they do not involve the generic partitions we study. That leaves open the question of whether James's conjecture, in its block form, is true for generic decomposition numbers. If one could establish a version of the conjecture of Turner mentioned above, for $v$-Schur algebras in characteristic zero, along the lines of \cite{CT}, making use of the derived equivalences established in \cite{CR} and in \cite{Turner}, a similar argument would then apply to $v$-Schur algebras in positive characteristic as well, and James's conjecture for generic partitions would follow. The proofs of our results involve the analysis of mutations of tilings (see Figure~\ref{fig:weight3mutation}) underlying these derived equivalences.

In order to state our results, we recall some standard combinatorics of partitions. To a partition $\lambda=(\lambda_1,\lambda_2,\dotsc)$, we associate its Young diagram $$[\lambda]=\left\{(i,j)\in\BZ^2\mid 1\leq j \leq \lambda_i\right\}.$$
We will draw Young diagrams with $(i-1,j+1)$ placed northeast of $(i,j)$.
Given two finite subsets $X,Y\subseteq \BZ^2$, we say that $X$ is northeast of $Y$ if
$$\max\{j-i\mid (i,j)\in X\} > \max\{j-i\mid (i,j)\in Y\}.$$
A {\em{rimhook}} of $\lambda$ is a nonempty subset $H\subseteq [\lambda]$
of the form
$$ H=H_{(i,j)}=\left\{(i',j')\in[\lambda]\mid i'\geq i, \,\, j'\geq j\text{ and } (i'+1,j'+1)\notin[\lambda]\right\}.$$
We have $[\mu]=[\lambda]\setminus H$ for a partition $\mu$; we say that $\mu$ is obtained by {\em{unwrapping}} $H$ from $\lambda$, and that
$\lambda$ is obtained by {\em{wrapping}} $H$ onto $\mu$.

Let $\Hook_e(\lambda)$ be the set of rimhooks of $\lambda$ of size divisible by $e$.
Given $H\in\Hook_e(\lambda)$, define a partition $\lambda_H$ as follows. The partition $\nu$ obtained by unwrapping $H$ from $\lambda$ can be produced by
successively unwrapping rimhooks $H^1,\ldots, H^{l}$ of size $e$. Then $\lambda_H$ is obtained from $\nu$ by successively wrapping on rimhooks
$\tilde{H}^l,\ldots, \tilde{H}^1$ of size $e$, where $\tilde{H}^i$ is taken to be minimally northeast of $H^i$.
The definition of $\lambda_H$ is problematic for two reasons: the decomposition of $H$ into a sequence $H^1,\ldots, H^l$ may not be unique, and the
$\tilde{H}^i$ may not exist. Nevertheless we have the following special case of our main result,
Theorem~\ref{thm:main}.
\begin{thm}\label{intro-prop}
		Let $\lambda$ be a generic partition (defined below). Then for all $H\in\operatorname{Hook}_e(\lambda)$,
		$\lambda_H$ is well-defined, and
		$d_{\lambda,\lambda_H}(q)=q$.
		\end{thm}

\begin{example}\label{intro-example1}
	Let $e=4$, and let $\lambda=(5,5,4,2,2,2,1,1)$. There is a unique $H\in\Hook_e(\lambda)$ such that $\left| H\right|=12$.
	Then $\lambda_H=(6,5,5,2,2,2)$, obtained by successively unwrapping
	$H^1=\{(1,5),(2,5),(2,4),(3,4)\}$,
	$H^2=\{(6,2),(6,1),(7,1),(8,1)\}$ and
	$H^3=\{(3,3),(3,2),(4,2),(5,2)\}$, and then wrapping
		$\tilde{H}^3=\{(2,4),(3,4),(3,3),(3,2)\}$,
		$\tilde{H}^2=\{(4,2),(5,2),(6,2),(6,1)\}$ and
		$\tilde{H}^1=\{(1,6),(1,5),(2,5),(3,5)\}$. See Figure~\ref{lambdaHexample}.
				The roles of $H^1$ and $H^2$ can be swapped, but that does not affect the result $\lambda_H$.
		For this example, it is known (via a computer calculation using the Lascoux-Leclerc-Thibon algorithm \cite[\S6.2]{LLT}) that indeed $d_{\lambda,\lambda_H}(q)=q$, even though
		Theorem~\ref{intro-prop} does not apply, as $\lambda$ is not generic.
\begin{figure}
	\caption{Example construction of $\lambda_H$.}
\def\hm{.3}	
\def\shiftx{10}
$$
\begin{tikzpicture}[y={(0,-1)},scale=0.6]
\node at (-0.5,3) {$[\lambda]:$};
\foreach \x/\y in {
		1/1,2/1,3/1,4/1,5/1,
		1/2,2/2,3/2,4/2,5/2,
		1/3,2/3,3/3,4/3,
		1/4,2/4,
		1/5,2/5,
		1/6,2/6,
		1/7,
		1/8
	}{
	\node at (\x,\y) [circle, fill=black, scale=0.3]{};}
\draw [rounded corners, thick, dashed] (5-\hm, 1-\hm)
-- ++(2*\hm,0) -- ++(0,1+2*\hm) -- ++(-1,0) -- ++(0,1) -- ++(-2*\hm, 0) -- ++(0,-1-2*\hm) -- ++(1,0) -- cycle;
\draw [rounded corners, thick, dashed] (1-\hm, 6-\hm)
-- ++(1+2*\hm,0) -- ++(0,2*\hm) -- ++(-1,0) -- ++(0,2) -- ++(-2*\hm, 0) -- cycle;
\draw [rounded corners, thick, dashed] (2-\hm, 3-\hm)
-- ++(1+2*\hm,0) -- ++(0,2*\hm) -- ++(-1,0) -- ++(0,2) -- ++(-2*\hm, 0) -- cycle;
\node at (5,3) {\footnotesize $H^1$};
\node at (2,7) {\footnotesize $H^2$};
\node at (3,4) {\footnotesize $H^3$};

\node at (\shiftx-0.5,3) {$[\lambda_H]:$};
\foreach \x/\y in {
	1/1,2/1,3/1,4/1,5/1,6/1,
	1/2,2/2,3/2,4/2,5/2,
	1/3,2/3,3/3,4/3,5/3,
	1/4,2/4,
	1/5,2/5,
	1/6,2/6
}{
\node at (\x+\shiftx,\y) [circle, fill=black, scale=0.3]{};}
\draw [rounded corners, thick, dashed] (\shiftx+5-\hm, 1-\hm)
-- ++(1+2*\hm,0) -- ++(0,2*\hm) -- ++(-1,0) -- ++(0,2) -- ++(-2*\hm, 0) -- cycle;
\draw [rounded corners, thick, dashed] (\shiftx+4-\hm, 2-\hm)
-- ++(2*\hm,0) -- ++(0,1+2*\hm) -- ++(-2-2*\hm,0) -- ++(0,-2*\hm) -- ++(2,0) -- cycle;
\draw [rounded corners, thick, dashed] (\shiftx+2-\hm, 4-\hm)
-- ++(2*\hm,0) -- ++(0,2+2*\hm) -- ++(-1-2*\hm,0) -- ++(0,-2*\hm) -- ++(1, 0) -- cycle;
\node at (\shiftx+6,2) {\footnotesize $\tilde{H}^1$};
\node at (\shiftx+2,7) {\footnotesize $\tilde{H}^2$};
\node at (\shiftx+4,4) {\footnotesize $\tilde{H}^3$};
\end{tikzpicture}
$$	\label{lambdaHexample}
	\end{figure}
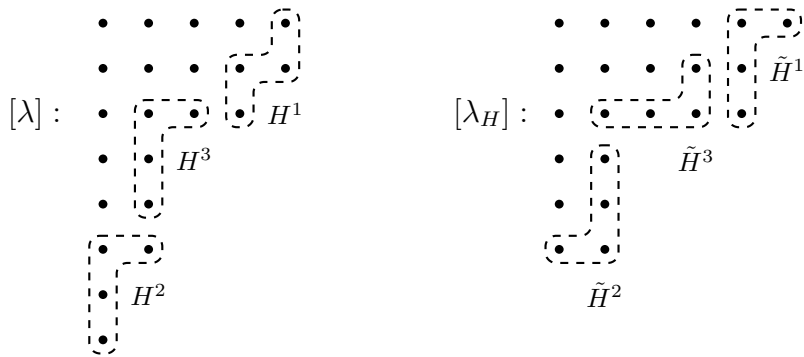
	\end{example}

	The statement of our main result requires the notion of blocks of partitions. The $e$-weight of a partition $\lambda$ is defined as  $w=\left|\Hook_e(\lambda)\right|$. A partition obtained
	by unwrapping $H\in\Hook_e(\lambda)$ from $\lambda$ has $e$-weight $w-(\left| H\right|/e)$. So by succesively unwrapping rimhooks (of size divisible by $e$) we obtain a partition $\kappa$ of $e$-weight $0$, called the $e$-core of $\lambda$, such that $\left|\lambda\right| = \left|\kappa\right|+we$. Moreover $\kappa$ does not depend on which rimhooks are chosen to be unwrapped. A block is defined to be a set of partitions sharing the same $e$-weight and $e$-core. It is known that $d_{\lambda\mu}(q)\neq 0$ only if $\lambda$ and $\mu$ are in the same block \cite[Theorem 6.8]{LLT}.
		Denote by $B_{\operatorname{gen}}$ the subset of generic partitions in a block $B$. The cardinality of the sets $B$ and $B_{\operatorname{gen}}$ depends only on $w$ and $e$, and the proportion $\left| B_{\operatorname{gen}} \right| / \left| B\right|$ of generic partitions in a block, keeping $w$ fixed, tends to $1$ as $e$ tends to infinity (see Remark~\ref{rem:proportion}), so the following theorem describes `almost all' $q$-decomposition numbers when $e\gg w$. 	
	
	\begin{thm} \label{intro-thm} Let $B$ be a block of partitions of $e$-weight $w$.
		There exists a map $\z\colon B_{\operatorname{gen}}\to\mathbb{Z}^w$ such that
		for all $\lambda,\mu \in B_{\operatorname{gen}}$,
		$$d_{\lambda\mu}(q)=
		\begin{cases}
		q^{|\Omega|} & \text{if } \z(\mu)-\z(\lambda) = \sum_{H\in\Omega} (\z(\lambda_H)-\z(\lambda)) \text{ for some }
		\Omega\subseteq \Hook_e(\lambda); \\
		  0 &  \text{otherwise.}
		\end{cases}$$
		
		The $w$-dimensional parallelotopes
		$$\Pi(\lambda)^{\mathbb{R}}=\left\{\z(\lambda)+\sum_{H\in\operatorname{Hook}_e(\lambda)} a_H(\z(\lambda_H)-\z(\lambda))
		\,\bigg|\,
		 0\leq a_H\leq 1 \right\}\subset \mathbb{R}^w$$
		are the $w$-cells of a polytopal complex in $\mathbb{R}^w$, i.e.\ the intersection of any two parallelotopes is either empty or a common face.
		
		Moreover,
		given $\lambda,\mu\in B_{\operatorname{gen}}$,
		there is a non-split extension of $L(\lambda)$ by $L(\mu)$ (simple modules of the $v$-Schur algebra in characteristic zero) if and only if $\z(\lambda)$ and $\z(\mu)$ are adjacent vertices of the polytopal complex, in which case $\Ext^1(L(\lambda),L(\mu))$ is one-dimensional, and
		either $\mu=\lambda_H$ for some $H\in\Hook_e(\lambda)$ or $\lambda=\mu_H$ for some $H\in\Hook_e(\mu)$.
	\end{thm}	

We next explain what it means for a partition to be generic and define the map $\z$ appearing in Theorem~\ref{intro-thm}. In the body of this paper, $\z(\lambda)$ is defined for an arbitrary partition $\lambda$. The alternative description we provide here is valid for any $\lambda$ satisfying the following condition:
 for all $H\in\operatorname{Hook}_e(\lambda)$, any inductive decomposition of $H$ into a sequence $H^1,\ldots, H^l$ of
 rimhooks of size $e$ ends with the same rimhook $H^{\operatorname{last}}$. In fact this condition is equivalent
 to  requiring  that $\lambda$ is a hook-quotient partition, i.e.\ that  every component of the $e$-quotient of $\lambda$ is a hook partition (see Section~\ref{subsection:partitions} and Definition~\ref{def:hookquotient} for the relevant definitions).
	For such a partition $\lambda$, each $H \in \operatorname{Hook}_e(\lambda)$ gives a distinct $H^{\operatorname{last}}$; write
	$\operatorname{Hook}_e(\lambda)=\{H_1,\dotsc,H_w\}$ so that  $H_{i+1}^{\operatorname{last}}$ is northeast of
	$H_i^{\operatorname{last}}$.
		Define $\z_i$ to be the width of $H_i^{\operatorname{last}}$, or one less than the width if
	the northeast-most node of $H_i^{\operatorname{last}}$ is the last node in its row of $[\lambda]$.
	We say that $\lambda$ is $m$-increasing if $\z_{i+1}-\z_i\geq m$ for all $i\in\{0,\ldots,w-1\}$, and that
	$\lambda$ is generic if it is $10$-increasing and $\z_w\leq e-2$. The map $\z:B_{\operatorname{gen}}\to\mathbb{Z}^w$ appearing in Theorem~\ref{intro-thm} is given by
	$\z(\lambda)=(\z_1,\dotsc,\z_w)$. If $\lambda$ is generic and $H\in\Hook_e(\lambda)$, then $\lambda_H$ and $\z(\lambda_H)$ are well-defined, even if $\lambda_H$ is not necessarily generic.

Finally we mention that the vectors $\sb_H=\z(\lambda_H)-\z(\lambda)$ generating $\Pi(\lambda)^{\mathbb{R}}$ in the statement of Theorem~\ref{intro-thm} can be defined alternatively in a simple fashion directly from $\lambda$, rather than through $\lambda_H$ and $\z$, c.f.\ Definition~\ref{def:modified}. In fact, for any hook-quotient partition $\lambda$ and any $H_i\in\Hook_e(\lambda)$,
the set $\CE_i$ of $e$-rimhooks appearing in all inductive decompositions of $H_i$ into rimhooks of size $e$ is fixed, and we can define
$\sb_{H_i}=\ssb_i-\ssb_j$, where $\CE_i\subsetneq \CE_j$ and $j$ is chosen (necessarily uniquely) to minimize $|\CE_j|$, and $\sb_{H_i}=\ssb_i$ if no such $j$ exists. Here $\ssb_i$ means the $i$-th standard basis vector in $\mathbb{Z}^w$.

\begin{example}
Let $e=4$ and $\lambda=(5,5,4,2,2,2,1,1)$, a partition with $4$-core $(2)$ and $4$-weight $5$. Even though $\lambda$ is not generic, $H^{\operatorname{last}}$ is well-defined for all $H\in\Hook_e(\lambda)$,
so we can determine $\z(\lambda)$ following the recipe described above. In Figure~\ref{zlambdaexample}, we depict the five rimhooks $H_i$ of $\lambda$ of size divisible by $e$ in order, and indicate the corresponding rimhooks $H^{\operatorname{last}}_i$, of widths $2,1,2,3,2$, respectively. The northeast-most node of
$H^{\operatorname{last}}_i$ is the last node in its row when $i=1,4,5$. Thus $\z(\lambda)=(1,1,2,2,1)$.
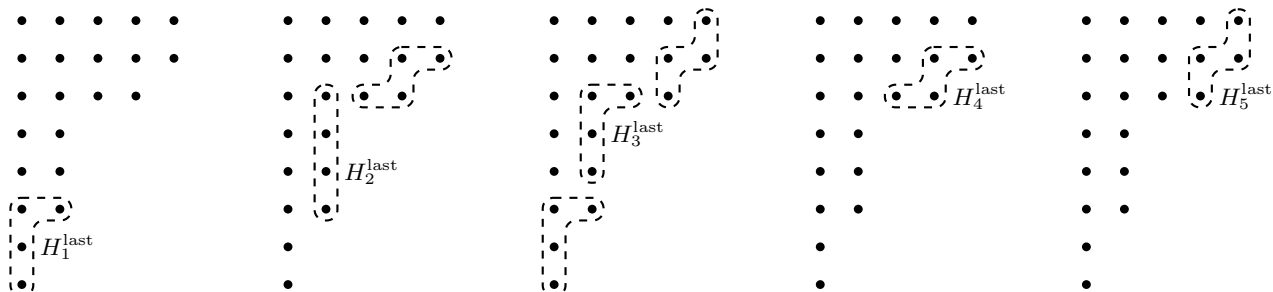
\begin{figure}
	\caption{Example calculation of $\z(\lambda)$.}
	\def\hm{.3}	
	\def\shiftx{7}
	$$
	\begin{tikzpicture}[y={(0,-1)},scale=0.5]
	\foreach \w in {0,1,2,3,4}
	{
	\foreach \x/\y in {
		1/1,2/1,3/1,4/1,5/1,
		1/2,2/2,3/2,4/2,5/2,
		1/3,2/3,3/3,4/3,
		1/4,2/4,
		1/5,2/5,
		1/6,2/6,
		1/7,
		1/8
	}{
	\node at (\x+\w*\shiftx,\y) [circle, fill=black, scale=0.3]{};}
}
\draw [rounded corners, thick, dashed] (1-\hm, 6-\hm)
-- ++(1+2*\hm,0) -- ++(0,2*\hm) -- ++(-1,0) -- ++(0,2) -- ++(-2*\hm, 0) -- cycle;
\node at (2.2,7) {\scriptsize $H_1^{\operatorname{last}}$};

\draw [rounded corners, thick, dashed] (\shiftx + 4-\hm, 2-\hm)
-- ++(1+2*\hm,0) -- ++(0,2*\hm) -- ++(-1,0) -- ++(0,1) -- ++(-1-2*\hm, 0) -- ++(0,-2*\hm) -- ++(1,0) -- cycle;
\draw [rounded corners, thick, dashed] (\shiftx + 2-\hm, 3-\hm)
-- ++(2*\hm,0) -- ++(0,3+2*\hm) -- ++(-2*\hm, 0) -- cycle;
\node at (\shiftx+3.2,5) {\scriptsize $H_2^{\operatorname{last}}$};

\draw [rounded corners, thick, dashed] (2*\shiftx + 5-\hm, 1-\hm)
-- ++(2*\hm,0) -- ++(0,1+2*\hm) -- ++(-1,0) -- ++(0,1) -- ++(-2*\hm, 0) -- ++(0,-1-2*\hm) -- ++(1,0) -- cycle;
\draw [rounded corners, thick, dashed] (2*\shiftx + 1-\hm, 6-\hm)
-- ++(1+2*\hm,0) -- ++(0,2*\hm) -- ++(-1,0) -- ++(0,2) -- ++(-2*\hm, 0) -- cycle;
\draw [rounded corners, thick, dashed] (2*\shiftx + 2-\hm, 3-\hm)
-- ++(1+2*\hm,0) -- ++(0,2*\hm) -- ++(-1,0) -- ++(0,2) -- ++(-2*\hm, 0) -- cycle;
\node at (2*\shiftx+3.2,4) {\scriptsize $H_3^{\operatorname{last}}$};

\draw [rounded corners, thick, dashed] (3*\shiftx + 4-\hm, 2-\hm)
-- ++(1+2*\hm,0) -- ++(0,2*\hm) -- ++(-1,0) -- ++(0,1) -- ++(-1-2*\hm, 0) -- ++(0,-2*\hm) -- ++(1,0) -- cycle;
\node at (3*\shiftx+5.2,3) {\scriptsize $H_4^{\operatorname{last}}$};

\draw [rounded corners, thick, dashed] (4*\shiftx + 5-\hm, 1-\hm)
-- ++(2*\hm,0) -- ++(0,1+2*\hm) -- ++(-1,0) -- ++(0,1) -- ++(-2*\hm, 0) -- ++(0,-1-2*\hm) -- ++(1,0) -- cycle;
\node at (4*\shiftx+5.2,3) {\scriptsize $H_5^{\operatorname{last}}$};

\end{tikzpicture}$$\label{zlambdaexample}
\end{figure}

We have seen in Example~\ref{intro-example1} that $\lambda_{H_3}=(6,5,5,2,2,2)$. It is easy to work out $\lambda_{H_5}=(6,5,3,2,2,2,1,1)$ as well.
 We also leave it to the reader to check that $\z(\lambda_{H_3})=(1,1,3,2,1)$ and $\z(\lambda_{H_5})=(1,1,1,2,2)$. Let $\mu=(6,5,4,2,2,2,1)$, a partition
 in the same block as $\lambda$. Then
 $\z(\mu)=(1,1,2,2,2)$, so that $\z(\mu)-\z(\lambda)= \left(\z(\lambda_{H_3})-\z(\lambda)\right)+ \left(\z(\lambda_{H_5})-\z(\lambda)\right)$. Using
 the Lascoux-Leclerc-Thibon algorithm, one finds that $d_{\lambda\mu}(q)=q^2$, in accord with Theorem~\ref{intro-thm}, even though, strictly speaking, that theorem does not apply, as $\lambda$ is not generic.
\end{example}

\begin{example} Any block of $e$-weight $2$ contains $e(e+3)/2$ partitions, of which $(e-10)(e-11)/2$ are generic.
	In Figure~\ref{weight2figure} the tiling associated to the block $B$ of $17$-weight $2$ with $17$-core $(5,3,1)$ is depicted.
The label of each vertex in the figure is $\z(\lambda)$ of a partition $\lambda\in B_{\operatorname{gen}}$, and the parallelogram $\Pi(\lambda)^{\mathbb{R}}$ is `generated' by the arrows emanating from that vertex.
Note that each edge between labelled vertices carries an arrow, in agreement with the last statement in
Theorem~\ref{intro-thm}.
\begin{figure}\caption{Tiling associated to the block of $17$-weight $2$ with $17$-core $(5,3,1)$.}

\def\h{5} 
\def\g{10} 
\def\core{{3,1,0,0,0,0,0,0,0,0,0,0,0}} 

$$
\begin{tikzpicture}[y={(1,-1)},x={(2,-1)},scale=\myscaleA, every node/.style={font=\scalefont{\mynodescaleA}}]
\foreach \j in {0,...,\h}{
	 \pgfmathsetmacro{\a}{\core[\h-\j]}
	 \ifthenelse{\h>\j}{\pgfmathsetmacro{\b}{\core[\h-\j-1]+1}}{\pgfmathsetmacro{\b}{5000}}
	\foreach \i in {0,...,\j}{
			\ifthenelse{\i=\a}
		{\mysquare{(\i,\j)}{(1,-1)}{(0,1)}}
				{\ifthenelse{\i>\a \AND \i<\b}
		{\mysquare{(\i,\j)}{(1,0)}{(-1,1)}}
		{\mysquare{(\i,\j)}{(1,0)}{(0,1)}}
	}
}
};
\foreach \j in {0,...,\h}{
	\pgfmathsetmacro{\a}{\core[\h-\j]}
	\ifthenelse{\h>\j}{\pgfmathsetmacro{\b}{\core[\h-\j-1]+1}}{\pgfmathsetmacro{\b}{5000}}
	\foreach \i in {0,...,\j}{
		\ifthenelse{\i=\a}
		{\mysquarearrows{(\i,\j)}{(1,-1)}{(0,1)}}
		{\ifthenelse{\i>\a \AND \i<\b}
			{\mysquarearrows{(\i,\j)}{(1,0)}{(-1,1)}}
			{\mysquarearrows{(\i,\j)}{(1,0)}{(0,1)}}
		}
	}
};
\foreach \j  [evaluate=\j as \k using int(\j+\g)]  in {0,...,\h}{
	\pgfmathsetmacro{\a}{\core[\h-\j]}
	\ifthenelse{\h>\j}{\pgfmathsetmacro{\b}{\core[\h-\j-1]+1}}{\pgfmathsetmacro{\b}{5000}}
	\foreach \i in {0,...,\j}{
		\ifthenelse{\i=\a}
		{\node at (\i+.275,\j+.175) {$(\i,\k)$};}
		{\ifthenelse{\i>\a \AND \i<\b}
			{\node at (\i+.175,\j+.275) {$(\i,\k)$};}
			{\node at (\i+.3,\j+.3) {$(\i,\k)$};}
		}
	}
};
\foreach \j [evaluate=\j as \k using int(\j+\g)] in {0,...,\h}{
	\foreach \i in {0,...,\j}{
		\node at (\i,\j)  [circle, fill=black, scale=\myvertexscaleA]{};
	}
};

\end{tikzpicture}
$$

\label{weight2figure}
\end{figure}
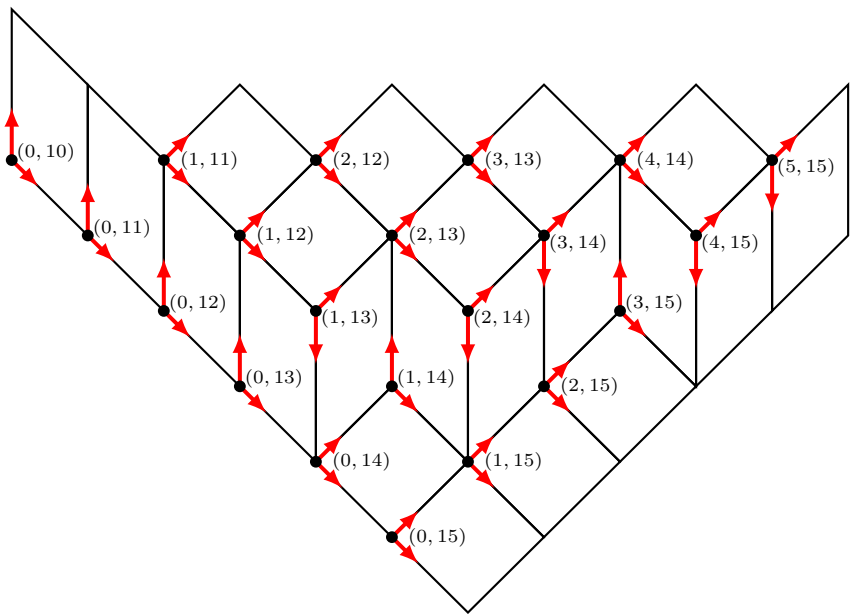
\end{example}

\begin{example} Any block of $e$-weight $3$ contains in total $e(e+1)(e+8)/6$ partitions, of which $(e-21)(e-20)(e-19)/6$ are generic.
	In Figure~\ref{weight3figure} the tiling associated to the block $B$ of $25$-weight $3$ with $25$-core $(15,1^{14})$ is depicted.
	To locate the tiling in $\mathbb{R}^3$, the coordinates of some vertices (not all of the form $\z(\lambda)$ with $\lambda \in B_{\operatorname{gen}}$) are indicated. The $20$ parallelepipeds in the tiling fall into $7$ distinct translation classes.
	
	In Figure~\ref{fig:weight3mutation} the tiling associated to $B$ is shown to the left of that associated to the block $\tilde{B}$ of
	$25$-weight $3$ with $25$-core $(15,1^{13})$, with which $B$ forms a so-called Scopes $[3:1]$-pair, c.f.\ Section \ref{subsection:Fock}.
	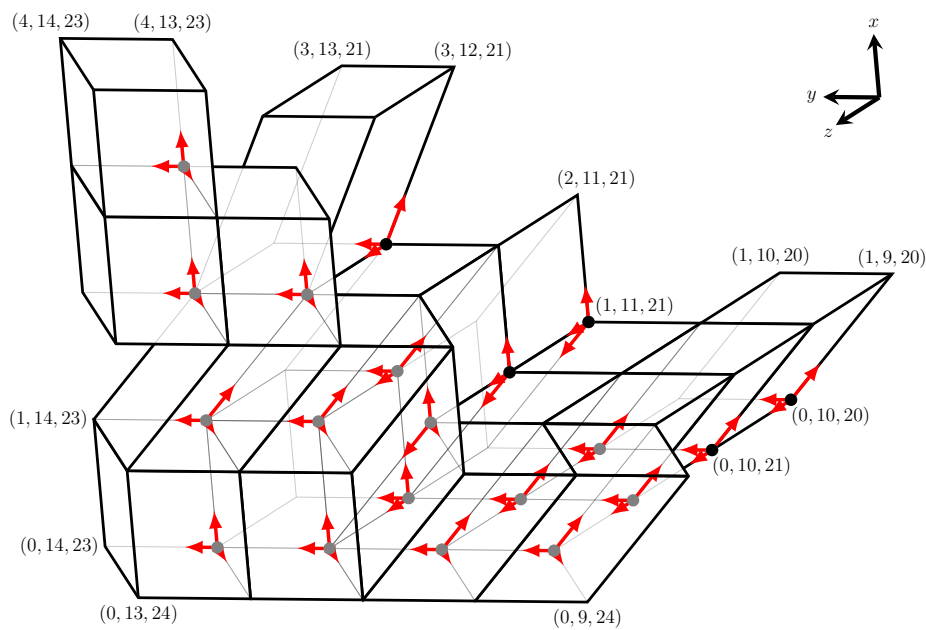
\begin{figure}\caption{Tiling associated to the block of $25$-weight $3$ with $25$-core $(15,1^{14})$.}

$$
\begin{tikzpicture}[
scale=\myscale, every node/.style={scale=\mynodescale}]
\rotateRPY{\anglea}{\angleb}{\rotationangle}
\begin{scope}[RPY]

\foreach  \xx/\yy/\zz in
{0/12/22}
{\myopaqueshiftedcube{(-\xshift,-\yshift,-\zshift)}{(\xx+1,\yy,\zz+1)}{(-1,0,0)}{(0,1,0)}{(0,0,-1)}{black}{1};}
\foreach  \xx/\yy/\zz in
{1/12/22}
{\myopaqueshiftedcube{(-\xshift,-\yshift,-\zshift)}{(\xx+1,\yy-1,\zz+1)}{(-1,1,0)}{(0,1,0)}{(0,0,-1)}{black}{1};}
\foreach  \xx/\yy/\zz in
{1/11/21, 1/11/22}
{\myopaqueshiftedcube{(-\xshift,-\yshift,-\zshift)}{(\xx,\yy,\zz+1)}{(1,0,0)}{(-1,1,0)}{(0,0,-1)}{black}{1};}
\foreach  \xx/\yy/\zz in
{ 0/11/21, 0/11/22, 0/10/20, 0/10/21, 0/10/22}
{\myopaqueshiftedcube{(-\xshift,-\yshift,-\zshift)}{(\xx+1,\yy-1,\zz+1)}{(-1,1,0)}{(0,1,0)}{(0,0,-1)}{black}{1};}
\foreach  \xx/\yy/\zz in
{0/10/23, 0/11/23}
{\myopaqueshiftedcube{(-\xshift,-\yshift,-\zshift)}{(\xx+1,\yy-1,\zz+1)}{(-1,1,0)}{(0,-1,0)}{(0,1,-1)}{black}{1};}	
\foreach  \xx/\yy/\zz in
{1/11/23}
{\myopaqueshiftedcube{(-\xshift,-\yshift,-\zshift)}{(\xx,\yy,\zz+1)}{(-1,0,0)}{(1,-1,0)}{(0,1,-1)}{black}{1};}
\foreach  \xx/\yy/\zz in
{0/12/23, 0/13/23}
{\myopaqueshiftedcube{(-\xshift,-\yshift,-\zshift)}{(\xx+1,\yy,\zz+1)}{(-1,0,0)}{(0,-1,0)}{(0,1,-1)}{black}{1};}
\foreach  \xx/\yy/\zz in
{1/12/23,1/13/23}
{\myopaqueshiftedcube{(-\xshift,-\yshift,-\zshift)}{(\xx+1,\yy-1,\zz+1)}{(-1,1,0)}{(0,-1,0)}{(0,1,-1)}{black}{1};}
\foreach  \xx/\yy/\zz in
{2/12/22}
{\myopaqueshiftedcube{(-\xshift,-\yshift,-\zshift)}{(\xx+1,\yy,\zz)}{(-1,0,1)}{(0,1,0)}{(0,0,-1)}{black}{1};}
\foreach  \xx/\yy/\zz in
{2/12/23, 2/13/23, 3/13/23}
{\myopaqueshiftedcube{(-\xshift,-\yshift,-\zshift)}{(\xx+1,\yy,\zz+1)}{(-1,0,0)}{(0,-1,0)}{(0,1,-1)}{black}{1};}

\begin{scope}[opacity=\hiddenlineopacity, transparency group]
\foreach  \xx/\yy/\zz in
{0/10/20, 0/10/21, 0/10/22, 0/11/21, 0/11/22, 1/12/22}
{\mycoloredshiftedcube{(-\xshift,-\yshift,-\zshift)}{(\xx,\yy,\zz)}{(1,-1,0)}{(0,1,0)}{(0,0,1)}{black};}	
\foreach  \xx/\yy/\zz in
{0/10/23, 0/11/23, 1/12/23, 1/13/23}
{\mycoloredshiftedcube{(-\xshift,-\yshift,-\zshift)}{(\xx,\yy,\zz)}{(1,-1,0)}{(0,1,0)}{(0,-1,1)}{black};}
\foreach  \xx/\yy/\zz in
{0/12/22}
{\mycoloredshiftedcube{(-\xshift,-\yshift,-\zshift)}{(\xx,\yy,\zz)}{(1,0,0)}{(0,1,0)}{(0,0,1)}{black};}
\foreach  \xx/\yy/\zz in
{1/11/21,1/11/22}
{\mycoloredshiftedcube{(-\xshift,-\yshift,-\zshift)}{(\xx,\yy,\zz)}{(1,0,0)}{(-1,1,0)}{(0,0,1)}{black};}
\foreach  \xx/\yy/\zz in
{0/12/23, 0/13/23, 2/12/23, 2/13/23, 3/13/23}
{\mycoloredshiftedcube{(-\xshift,-\yshift,-\zshift)}{(\xx,\yy,\zz)}{(1,0,0)}{(0,1,0)}{(0,-1,1)}{black};}
\foreach  \xx/\yy/\zz in
{2/12/22}
{\mycoloredshiftedcube{(-\xshift,-\yshift,-\zshift)}{(\xx,\yy,\zz)}{(1,0,-1)}{(0,1,0)}{(0,0,1)}{black};}
\foreach  \xx/\yy/\zz in
{1/11/23}
{\mycoloredshiftedcube{(-\xshift,-\yshift,-\zshift)}{(\xx,\yy,\zz)}{(1,0,0)}{(-1,1,0)}{(0,-1,1)}{black};}
\end{scope}

\foreach  \xx/\yy/\zz in
{0/10/20, 0/10/21, 0/10/22, 0/11/21, 0/11/22, 1/12/22}
{\mycubearrows{(-\xshift,-\yshift,-\zshift)}{(\xx,\yy,\zz)}{(1,-1,0)}{(0,1,0)}{(0,0,1)}{black};}	
\foreach  \xx/\yy/\zz in
{0/10/23, 0/11/23, 1/12/23, 1/13/23}
{\mycubearrows{(-\xshift,-\yshift,-\zshift)}{(\xx,\yy,\zz)}{(1,-1,0)}{(0,1,0)}{(0,-1,1)}{black};}
\foreach  \xx/\yy/\zz in
{0/12/22}
{\mycubearrows{(-\xshift,-\yshift,-\zshift)}{(\xx,\yy,\zz)}{(1,0,0)}{(0,1,0)}{(0,0,1)}{black};}
\foreach  \xx/\yy/\zz in
{1/11/21,1/11/22}
{\mycubearrows{(-\xshift,-\yshift,-\zshift)}{(\xx,\yy,\zz)}{(1,0,0)}{(-1,1,0)}{(0,0,1)}{black};}
\foreach  \xx/\yy/\zz in
{0/12/23, 0/13/23, 2/12/23, 2/13/23, 3/13/23}
{\mycubearrows{(-\xshift,-\yshift,-\zshift)}{(\xx,\yy,\zz)}{(1,0,0)}{(0,1,0)}{(0,-1,1)}{black};}
\foreach  \xx/\yy/\zz in
{2/12/22}
{\mycubearrows{(-\xshift,-\yshift,-\zshift)}{(\xx,\yy,\zz)}{(1,0,-1)}{(0,1,0)}{(0,0,1)}{black};}
\foreach  \xx/\yy/\zz in
{1/11/23}
{\mycubearrows{(-\xshift,-\yshift,-\zshift)}{(\xx,\yy,\zz)}{(1,0,0)}{(-1,1,0)}{(0,-1,1)}{black};}

\foreach \xx/\yy/\zz in
{0/10/20, 0/10/21, 0/10/22, 0/11/21, 0/11/22, 0/12/22, 0/10/23, 0/11/23, 0/12/23, 0/13/23,
1/11/21, 1/11/22, 1/12/22, 1/11/23, 1/12/23, 1/13/23,
2/12/23, 2/13/23, 3/13/23, 2/12/22}
{\node at (\xx-\xshift,\yy-\yshift,\zz-\zshift) [circle, fill=gray, scale=\myvertexscale]{};}
\foreach \xx/\yy/\zz in
{0/10/20, 0/10/21, 1/11/21, 1/11/22, 2/12/22}
{\node at (\xx-\xshift,\yy-\yshift,\zz-\zshift) [circle, fill=black, scale=\myvertexscale]{};}

\draw [line width = \myscale*\mylinewidth, draw=black] (1-\xshift,9-\yshift,23-\zshift) -- ++ (0,1,0);
\draw [line width = \myscale*\mylinewidth, draw=black] (1-\xshift,8-\yshift,24-\zshift) -- ++ (0,2,0);
\draw [line width = \myscale*\mylinewidth, draw=black] (2-\xshift,10-\yshift,24-\zshift) -- ++ (0,1,0);	
		
\draw [->, >=stealth, ultra thick] (\xorigin,\yorigin,\zorigin) -- (\xorigin+0.5,\yorigin,\zorigin);
\draw [->, >=stealth, ultra thick] (\xorigin,\yorigin,\zorigin) -- (\xorigin,\yorigin+0.5,\zorigin);
\draw [->,  >=stealth, ultra thick] (\xorigin,\yorigin,\zorigin) -- (\xorigin,\yorigin,\zorigin+0.55);	
	\node [above] at (0.5+\xorigin,\yorigin,\zorigin) {$x$};
	\node [left] at (\xorigin,\yorigin+0.5,\zorigin) {$y$};
	\node [below left] at (\xorigin,\yorigin, \zorigin+0.5) {$z$};

\node [above] at (3-\xshift,12-\yshift,21-\zshift) {$\qquad (3,12,21)$};
\node [above] at (3-\xshift,13-\yshift,21-\zshift) {$(3,13,21) \quad$};
\node [above] at (2-\xshift,11-\yshift,21-\zshift) {$\qquad (2,11,21)$};
\node [above right] at (1-\xshift,11-\yshift,21-\zshift) {$(1,11,21)$};
\node [above] at (1-\xshift,10-\yshift,20-\zshift) {$(1,10,20) \quad$};
\node [above] at (1-\xshift,9-\yshift,20-\zshift) {$(1,9,20)$};
\node [below right] at (0-\xshift,10-\yshift,20-\zshift) {$\!\!(0,10,20)$};
\node [below right] at (0-\xshift,10-\yshift,21-\zshift) {$\!\!(0,10,21)$};
\node [above] at (4-\xshift,14-\yshift,23-\zshift) {$(4,14,23) \quad$};
\node [above] at (4-\xshift,13-\yshift,23-\zshift) {$(4,13,23)$};
\node [left] at (1-\xshift,14-\yshift,23-\zshift) {$(1,14,23)$};
\node [left] at (0-\xshift,14-\yshift,23-\zshift) {$(0,14,23)$};
\node [below] at (0-\xshift,13-\yshift,24-\zshift) {$(0,13,24)$};
\node [below] at (0-\xshift,9-\yshift,24-\zshift) {$(0,9,24)$};
			
\end{scope}

\end{tikzpicture}
$$
\label{weight3figure}
	\end{figure}
	
\begin{figure}\caption{Comparison of tilings associated to two blocks forming a Scopes pair.}
$$
\begin{tikzpicture}[
scale=\mysmallscale, every node/.style={scale=\mynodescale}]
\rotateRPY{\anglea}{\angleb}{\rotationangle}
\begin{scope}[RPY]

\foreach  \xx/\yy/\zz in
{0/12/22}
{\myopaqueshiftedcube{(-\xshift,-\yshift,-\zshift)}{(\xx+1,\yy,\zz+1)}{(-1,0,0)}{(0,1,0)}{(0,0,-1)}{black}{1};}
\foreach  \xx/\yy/\zz in
{1/12/22}
{\myopaqueshiftedcube{(-\xshift,-\yshift,-\zshift)}{(\xx+1,\yy-1,\zz+1)}{(-1,1,0)}{(0,1,0)}{(0,0,-1)}{black}{1};}
\foreach  \xx/\yy/\zz in
{1/11/21, 1/11/22}
{\myopaqueshiftedcube{(-\xshift,-\yshift,-\zshift)}{(\xx,\yy,\zz+1)}{(1,0,0)}{(-1,1,0)}{(0,0,-1)}{black}{1};}
\foreach  \xx/\yy/\zz in
{ 0/11/21, 0/11/22, 0/10/20, 0/10/21, 0/10/22}
{\myopaqueshiftedcube{(-\xshift,-\yshift,-\zshift)}{(\xx+1,\yy-1,\zz+1)}{(-1,1,0)}{(0,1,0)}{(0,0,-1)}{black}{1};}
\foreach  \xx/\yy/\zz in
{0/10/23, 0/11/23}
{\myopaqueshiftedcube{(-\xshift,-\yshift,-\zshift)}{(\xx+1,\yy-1,\zz+1)}{(-1,1,0)}{(0,-1,0)}{(0,1,-1)}{black}{1};}	
\foreach  \xx/\yy/\zz in
{1/11/23}
{\myopaqueshiftedcube{(-\xshift,-\yshift,-\zshift)}{(\xx,\yy,\zz+1)}{(-1,0,0)}{(1,-1,0)}{(0,1,-1)}{black}{1};}
\foreach  \xx/\yy/\zz in
{0/12/23, 0/13/23}
{\myopaqueshiftedcube{(-\xshift,-\yshift,-\zshift)}{(\xx+1,\yy,\zz+1)}{(-1,0,0)}{(0,-1,0)}{(0,1,-1)}{black}{1};}
\foreach  \xx/\yy/\zz in
{1/12/23,1/13/23}
{\myopaqueshiftedcube{(-\xshift,-\yshift,-\zshift)}{(\xx+1,\yy-1,\zz+1)}{(-1,1,0)}{(0,-1,0)}{(0,1,-1)}{black}{1};}
\foreach  \xx/\yy/\zz in
{2/12/22}
{\myopaqueshiftedcube{(-\xshift,-\yshift,-\zshift)}{(\xx+1,\yy,\zz)}{(-1,0,1)}{(0,1,0)}{(0,0,-1)}{black}{1};}
\foreach  \xx/\yy/\zz in
{2/12/23, 2/13/23, 3/13/23}
{\myopaqueshiftedcube{(-\xshift,-\yshift,-\zshift)}{(\xx+1,\yy,\zz+1)}{(-1,0,0)}{(0,-1,0)}{(0,1,-1)}{black}{1};}

\begin{scope}[opacity=\hiddenlineopacity, transparency group]
\foreach  \xx/\yy/\zz in
{0/10/20, 0/10/21, 0/10/22, 0/11/21, 0/11/22, 1/12/22}
{\mycoloredshiftedcube{(-\xshift,-\yshift,-\zshift)}{(\xx,\yy,\zz)}{(1,-1,0)}{(0,1,0)}{(0,0,1)}{black};}	
\foreach  \xx/\yy/\zz in
{0/10/23, 0/11/23, 1/12/23, 1/13/23}
{\mycoloredshiftedcube{(-\xshift,-\yshift,-\zshift)}{(\xx,\yy,\zz)}{(1,-1,0)}{(0,1,0)}{(0,-1,1)}{black};}
\foreach  \xx/\yy/\zz in
{0/12/22}
{\mycoloredshiftedcube{(-\xshift,-\yshift,-\zshift)}{(\xx,\yy,\zz)}{(1,0,0)}{(0,1,0)}{(0,0,1)}{black};}
\foreach  \xx/\yy/\zz in
{1/11/21,1/11/22}
{\mycoloredshiftedcube{(-\xshift,-\yshift,-\zshift)}{(\xx,\yy,\zz)}{(1,0,0)}{(-1,1,0)}{(0,0,1)}{black};}
\foreach  \xx/\yy/\zz in
{0/12/23, 0/13/23, 2/12/23, 2/13/23, 3/13/23}
{\mycoloredshiftedcube{(-\xshift,-\yshift,-\zshift)}{(\xx,\yy,\zz)}{(1,0,0)}{(0,1,0)}{(0,-1,1)}{black};}
\foreach  \xx/\yy/\zz in
{2/12/22}
{\mycoloredshiftedcube{(-\xshift,-\yshift,-\zshift)}{(\xx,\yy,\zz)}{(1,0,-1)}{(0,1,0)}{(0,0,1)}{black};}
\foreach  \xx/\yy/\zz in
{1/11/23}
{\mycoloredshiftedcube{(-\xshift,-\yshift,-\zshift)}{(\xx,\yy,\zz)}{(1,0,0)}{(-1,1,0)}{(0,-1,1)}{black};}
\end{scope}
\end{scope}	
\end{tikzpicture}
\qquad
\begin{tikzpicture}[
scale=\mysmallscale, every node/.style={scale=\mynodescale}]
\rotateRPY{\anglea}{\angleb}{\rotationangle}
\begin{scope}[RPY]
\foreach  \xx/\yy/\zz in
{1/12/22}
{\myopaqueshiftedcube{(-\xshift,-\yshift,-\zshift)}{(\xx,\yy,\zz+1)}{(1,0,0)}{(-1,1,0)}{(0,0,-1)}{black}{1};}
\foreach  \xx/\yy/\zz in
{0/10/20, 0/11/21, 0/12/22, 0/10/21, 0/11/22, 0/10/22}
{\myopaqueshiftedcube{(-\xshift,-\yshift,-\zshift)}{(\xx+1,\yy-1,\zz+1)}{(-1,1,0)}{(0,1,0)}{(0,0,-1)}{black}{1};}
\foreach  \xx/\yy/\zz in
{0/10/23, 0/11/23, 0/12/23}
{\myopaqueshiftedcube{(-\xshift,-\yshift,-\zshift)}{(\xx+1,\yy-1,\zz+1)}{(-1,1,0)}{(0,-1,0)}{(0,1,-1)}{black}{1};}	
\foreach  \xx/\yy/\zz in
{1/11/21,
	1/11/22}
{\myopaqueshiftedcube{(-\xshift,-\yshift,-\zshift)}{(\xx+1,\yy,\zz+1)}{(-1,0,0)}{(0,1,0)}{(0,0,-1)}{black}{1};}
\foreach  \xx/\yy/\zz in
{2/12/22}
{\myopaqueshiftedcube{(-\xshift,-\yshift,-\zshift)}{(\xx+1,\yy,\zz)}{(-1,0,1)}{(0,1,0)}{(0,0,-1)}{black}{1};}
\foreach  \xx/\yy/\zz in
{1/11/23}
{\myopaqueshiftedcube{(-\xshift,-\yshift,-\zshift)}{(\xx+1,\yy,\zz+1)}{(-1,0,0)}{(0,-1,0)}{(0,1,-1)}{black}{1};}
\foreach  \xx/\yy/\zz in
{1/12/23}
{\myopaqueshiftedcube{(-\xshift,-\yshift,-\zshift)}{(\xx,\yy,\zz+1)}{(-1,0,0)}{(1,-1,0)}{(0,1,-1)}{black}{1};}
\foreach  \xx/\yy/\zz in
{0/13/23}
{\myopaqueshiftedcube{(-\xshift,-\yshift,-\zshift)}{(\xx+1,\yy,\zz+1)}{(-1,0,0)}{(0,-1,0)}{(0,1,-1)}{black}{1};}
\foreach  \xx/\yy/\zz in
{1/13/23}
{\myopaqueshiftedcube{(-\xshift,-\yshift,-\zshift)}{(\xx+1,\yy-1,\zz+1)}{(-1,1,0)}{(0,-1,0)}{(0,1,-1)}{black}{1};}
\foreach  \xx/\yy/\zz in
{2/12/23, 2/13/23, 3/13/23}
{\myopaqueshiftedcube{(-\xshift,-\yshift,-\zshift)}{(\xx+1,\yy,\zz+1)}{(-1,0,0)}{(0,-1,0)}{(0,1,-1)}{black}{1};}	
\begin{scope}[opacity=\hiddenlineopacity, transparency group]
\foreach  \xx/\yy/\zz in
{0/10/20, 0/10/21, 0/10/22, 0/11/21, 0/11/22, 0/12/22}
{\mycoloredshiftedcube{(-\xshift,-\yshift,-\zshift)}{(\xx,\yy,\zz)}{(1,-1,0)}{(0,1,0)}{(0,0,1)}{black};}	
\foreach  \xx/\yy/\zz in
{0/10/23, 0/11/23, 0/12/23, 1/13/23}
{\mycoloredshiftedcube{(-\xshift,-\yshift,-\zshift)}{(\xx,\yy,\zz)}{(1,-1,0)}{(0,1,0)}{(0,-1,1)}{black};}
\foreach  \xx/\yy/\zz in
{1/11/21, 1/11/22}
{\mycoloredshiftedcube{(-\xshift,-\yshift,-\zshift)}{(\xx,\yy,\zz)}{(1,0,0)}{(0,1,0)}{(0,0,1)}{black};}
\foreach  \xx/\yy/\zz in
{1/12/22}
{\mycoloredshiftedcube{(-\xshift,-\yshift,-\zshift)}{(\xx,\yy,\zz)}{(1,0,0)}{(-1,1,0)}{(0,0,1)}{black};}
\foreach  \xx/\yy/\zz in
{1/11/23, 0/13/23,
	2/12/23, 2/13/23, 3/13/23}
{\mycoloredshiftedcube{(-\xshift,-\yshift,-\zshift)}{(\xx,\yy,\zz)}{(1,0,0)}{(0,1,0)}{(0,-1,1)}{black};}
\foreach  \xx/\yy/\zz in
{2/12/22}
{\mycoloredshiftedcube{(-\xshift,-\yshift,-\zshift)}{(\xx,\yy,\zz)}{(1,0,-1)}{(0,1,0)}{(0,0,1)}{black};}
\foreach  \xx/\yy/\zz in
{1/12/23}
{\mycoloredshiftedcube{(-\xshift,-\yshift,-\zshift)}{(\xx,\yy,\zz)}{(1,0,0)}{(-1,1,0)}{(0,-1,1)}{black};}
\end{scope}
\end{scope}	
\end{tikzpicture}
$$	
\label{fig:weight3mutation}	
\end{figure}
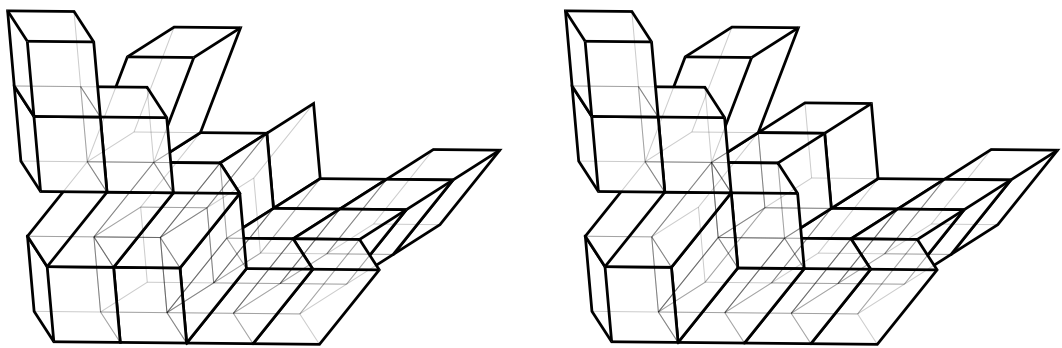		
\end{example}

In fact, the statements that we prove in the main body of this paper, Theorem \ref{thm:main}, Corollary \ref{C:polytopalcomplex} and Theorem \ref{thm:extquiver}, are stronger than those in
Theorem \ref{intro-thm}, in that they hold for larger classes of partitions.    For example Theorem \ref{thm:main} determines $d_{\lambda\mu}(q)$ for arbitrary partitions $\lambda$ and $4$-increasing partitions $\mu$; in other words, it fully determines the canonical basis vectors of the Fock space labelled by $4$-increasing partitions.

  While we have described our results using the combinatorics of Young diagrams in this introduction, it is actually easier to formulate and prove them using the combinatorics of the abacus.  This is the approach that we follow in the body, and the results can then be stated in full strength.

We now indicate the contents and layout of this paper.  In the next section, we introduce the notations that we shall use and remind the reader of some background material, concluding with a description of an action of the Weyl group $\mathcal{W}$ of $\widehat{\mathfrak{sl}}_e$ on the set of partitions, obtained by identifying the latter with the crystal basis of the Fock space representation. It seems to be a difficult problem to determine in a simple manner whether two given partitions are in the same orbit of this action. Nevertheless, in Section \ref{S:label}, we define a new labelling for partitions, extending the map $\z$ briefly introduced in this section, and prove that it indexes certain $\W$-orbits of partitions, including all orbits of generic partitions.  In Section \ref{S:main}, we define the parallelotopes of hook-quotient partitions and state our main theorem (Theorem \ref{thm:main}) on $q$-decomposition numbers.  In Section \ref{S:movealong}, we show how to obtain from $\lambda$ any partition $\mu$ lying in the parallelotope of $\lambda$ by successive applications of
the procedure $\lambda\mapsto\lambda_H$ described in this introduction.
  Armed with this key result, we then proceed to prove Theorem \ref{thm:main} in Section \ref{section:mainproof}, by an explicit inductive construction of the canonical basis vectors that we need.  In Section \ref{S:cubes}, we show that the parallelotopes  that we have described in this introduction can be obtained in a natural way as the projections of hypercubes lying in higher dimensional spaces, and we reformulate Theorem \ref{thm:main} in terms of hypercubes associated to partitions. In the following section we prove that the parallelotopes (and hypercubes) of the partitions concerned assemble to form tilings with nice properties and prove in particular Corollary \ref{C:polytopalcomplex} and Theorem \ref{thm:extquiver}. We conclude in Section \ref{section:Mullineux} with a short investigation into the effect of the Mullineux-Kleshchev involution for $e$-regular partitions on our map $\z$, and provide an alternative algorithm for this involution for the $e$-regular partitions that we have been concerned with, which is likely to be more efficient than Mullineux's or Kleshchev's algorithms when the $e$-weights of the partitions are small compared to their sizes.

\section{Preliminaries}

In this section, we set up the notations that we shall use, and remind the reader of the background that we shall need.
As before, we fix an integer $e \geq 2$.

\subsection{Notations}\label{section:notations}
Throughout, we use the following notations:

\begin{enumerate}
\item For $a,b \in \mathbb{Z}$,
\begin{align*}
[a,\, b] &= \{ x \in \mathbb{Z} \mid a \leq x \leq b \}; \\
(a,\, b] &= \{ x \in \mathbb{Z} \mid a < x \leq b \}; \\
[a,\, b) &= \{ x \in \mathbb{Z} \mid a \leq x < b \}; \\
(a,\, b) &= \{ x \in \mathbb{Z} \mid a < x < b \}.
\end{align*}
\item
 We write $a\equiv_e b$ to mean $e \mid (a-b)$, and $a \geq_e b$ to mean $a \equiv_e b$ and $a\geq b$ (with obvious extensions to $a >_e b$, $a\leq_e b$ and $a <_e b$).
 \item Given a subset $\Omega\in\mathbb{R}^n$, and $m \in \mathbb{Z}$, we let
 $$\inc{\Omega}{m}=\left\{(x_1,\dotsc,x_n) \in \Omega \mid x_{i+1}-x_i \geq m \text{ for all } i\in [1,n)\right\}$$
 be the set of {\em $m$-increasing} elements of $\Omega$.
 \item Given a subset $\Omega\in\mathbb{Z}^n$, we denote by $\hull{\Omega}$ the convex hull of $\Omega$ in $\mathbb{R}^n$.
 \item We denote by $\I_{\mathtt{p}}$ the indicator function taking the value $1$ if $\mathtt{p}$ is true and $0$ otherwise.
\end{enumerate}
\subsection{Partitions, $\beta$-numbers, abacus, blocks}\label{subsection:partitions}
A partition $\lambda = (\lambda_1,\lambda_2,\dotsc)$ is a weakly decreasing sequence of non-negative integers which are eventually zero.  We identify $\lambda$ with its finite subsequence $(\lambda_1,\lambda_2,\dotsc, \lambda_l)$ when $\lambda_{l+1} = 0$. In particular, $(0,0,\dotsc)$ may be denoted as $(0)$ or $\emptyset$.  The size of $\lambda$ is $|\lambda|=\sum_i \lambda_i$. Write $\mathcal{P}$ for the set of all partitions.

To each $\lambda$, recall its associated its Young diagram $[\lambda] = \{ (i,j) \mid j \in [1,\, \lambda_i] \}$, whose elements are called nodes.  The $e$-residue of a node $(i,j)$ is the residue class of $j-i$ modulo $e$.  If the $e$-residue of $(i,j)$ is $r$, then we also call $(i,j)$ an $r$-node.  If removing $(i,j)$ from $[\lambda]$ produces the Young diagram $[\mu]$ of another partition $\mu$, we call $(i,j)$ a removable node of $\lambda$ and an addable node of $\mu$.

Given a partition $\lambda=(\lambda_{1},\lambda_{2},\ldots)$, let%
\[
\beta(\lambda)=\{\lambda_{1}-1,\lambda_{2}-2,\ldots\}\subseteq\mathbb{Z}%
\]
be the infinite set of {\em $\beta$-numbers} associated to $\lambda$.
Following \cite[\S2.7]{JK}, consider an abacus with $e$ infinite vertical runners labelled $0, 1, \dotsc, e-1$ from left to right, with positions labelled by the integers increasing from left to right and top to bottom, so that runner $r$ contains the integers congruent to $r$ modulo $e$. We represent $\lambda$ on the abacus by placing a bead at each position corresponding to a $\beta$-number of $\lambda$, leaving the other positions unoccupied. We note that the $i$-th bead at position $\lambda_i-i$ lies on runner $r$ if and only if the rightmost node $(i,\lambda_i)$ of row $i$ of $[\lambda]$ has $e$-residue $r$.

One can easily read off the partition $\lambda$ from its abacus display:  if the occupied positions on the abacus display are $b_1,b_2,\dotsc$, arranged in descending order, then $\lambda_i = |\{ x \in \mathbb{Z} \setminus \beta(\lambda)  \mid x < b_i \}|$, the number of unoccupied positions before $b_i$.

The partition $\lambda'$ {\em conjugate} to $\lambda$ is defined by
$(i,j)\in[\lambda'] \iff (j,i)\in[\lambda]$, or, equivalently, by
$x\in\beta(\lambda')\iff -x\notin\beta(\lambda)$.

Moving a bead from a position $x$ of the abacus display of $\lambda$ to an unoccupied position $y$ ($x>y$) corresponds to removing a rimhook of size $x-y$ from (the Young diagram of) $\lambda$.  In particular moving a bead on runner $r$ to its unoccupied preceding position (on runner $r-1$) corresponds to removing a removable $r$-node from $\lambda$.  We call such beads {\em removable}.  Similarly, moving a bead on runner $r-1$ to its unoccupied succeeding position (on runner $r$) corresponds to adding an addable $r$-node to $\lambda$.  We call such beads {\em addable}.

A removable bead at position $x$ on runner $a$ is {\em normal} if and only if reading the abacus below position $x$, there are at least as many removable beads on runner $a$ than addable beads on runner $a-1$.  Formally, $x \in \beta(\lambda)$ is normal if and only if
$$
|\{ x <_e t \leq_e y \mid t \in \beta(\lambda),\ t-1 \notin \beta(\lambda)\}| \geq |\{ x <_e t \leq_e y \mid t \notin \beta(\lambda),\ t-1 \in \beta(\lambda)\}|
$$
for all $y >_e x$.

When we slide the beads in the abacus display of $\lambda$ as high up their respective runners as possible, we obtain (the abacus display of) its {\em{$e$-core}}.


For each $b\in \mathbb{Z}$ let
$$
\wt_{\lambda}(b)=|\{ x <_e b \mid  x\notin\beta(\lambda)\}|
$$
be the number of unoccupied positions above $b$ in its runner. 
This is the {\em $e$-weight} of $b$ when $b \in \beta(\lambda)$.
Let
$$
w_\lambda=\sum_{b\in\beta(\lambda)}\wt_{\lambda}(b).
$$
We call $w_\lambda$ the {\em{$e$-weight}} of $\lambda$; this is
the usual notion of the $e$-weight of a partition:
the number of $e$-hooks (i.e.\ rimhooks of size $e$) removed in succession from the Young diagram of $\lambda$ to reach its $e$-core.

We now describe how the set $\mathcal{P}$ of partitions decomposes as a disjoint union of {\em blocks}, depending on $e$. The {\em $e$-quotient} of a partition $\lambda$ is an $e$-tuple $(\lambda^{(0)},\ldots,\lambda^{(e-1)})$ of partitions, where each $\lambda^{(r)}$ is the partition read off from runner $r$ of the abacus display of $\lambda$; thus the parts of $\lambda^{(r)}$ are the various $\wt_{\lambda}(b)$ for $b \in \beta(\lambda)$ with $b \equiv_e r$.
Given an $e$-core partition $\kappa$ and $w\in\mathbb{Z}_{\geq 0}$, the block $B_{\kappa,w}$ is defined to be the set of partitions with $e$-core $\kappa$ and $e$-weight $w$. Taking $e$-quotients gives a bijection between $B_{\kappa,w}$ and the set of $e$-tuples $(\lambda^{(0)},\ldots,\lambda^{(e-1)})$ of partitions such that
$\sum_{i=0}^{e-1} |\lambda^{(i)}|=w$.

We call $B_{\kappa,w}$ a {\em Rouquier block} if
the abacus display of its $e$-core $\kappa$ has at least $w-1$ removable beads (and no addable beads) on runner $a$,
for all $a\in [1,\,e)$.  It is clear 
that there exists at least one Rouquier block (in fact, infinitely many) of any $w\in \BZ_{\geq 0}$.

\subsection{Fock space representation of $U_q(\widehat{\mathfrak{sl}}_e)$}\label{subsection:Fock}

Let $q$ be an indeterminate.  The quantum affine algebra $U_q(\widehat{\mathfrak{sl}}_e)$ is a unital $\mathbb{C}(q)$-algebra generated by $\{ E_i, F_i, K_i, K_i^{-1} \mid i \in [0,e) \}$ subject to certain relations which we do not need here (interested readers may refer to \cite{L} for more details).

Its Fock space representation $\CF :=\bigoplus_{\lambda\in\CP} \mathbb{C}(q) \lambda$ has the set $\CP$ of all partitions as its basis as a $\mathbb{C}(q)$-vector space \cite{H,MM}.  For our purposes, we only require the action of $E_i$ and $F_i$ on $\CF$, which we shall now describe.

Let $\lambda$ be a partition and let $i$ be a residue class modulo $e$.  Let $C \subseteq \beta(\lambda)$ be a set of addable beads of $\lambda$ on runner $i-1$, and write $C^+$ for $\{c+1 \mid c \in C\}$.  Let $\mu$ be the partition such that $\beta(\mu) = \beta(\lambda) \cup C^+ \setminus C$.  We call $C$ an addable $i$-$\beta$-subset of $\lambda$ and $C^+$ a removable $i$-$\beta$-subset of $\mu$.  Define
\begin{align*}
N_E(\lambda,\mu) &= \sum_{\substack{y \in \beta(\lambda): \\ y \text{ removable}}} |\{ b \in C^+ \mid b >_e y \}| - \sum_{\substack{x \in \beta(\mu): \\ x \text{ addable}}} |\{ c \in C \mid c >_e x \}|, \\
N_F(\lambda,\mu) &= \sum_{\substack{x \in \beta(\mu): \\ x \text{ addable}}} |\{ c \in C \mid c <_e x \}| - \sum_{\substack{y \in \beta(\lambda): \\ y \text{ removable}}} |\{ b \in C^+ \mid b <_e y \}|.
\end{align*}
Then, for $k \in \mathbb{Z}^+$, we have
\begin{align*}
E_i^{(k)} (\mu) &= \sum_{\lambda} q^{N_E(\lambda,\mu)} \lambda, \\
F_i^{(k)} (\lambda) &= \sum_{\mu} q^{N_F(\lambda,\mu)} \mu,
\end{align*}
where the first sum runs over all $\lambda$ with $\beta(\lambda) = \beta(\mu) \cup C \setminus C^+$ for some removable $i$-$\beta$-subset $C^+$ of $\mu$ with $|C^+|=k$, and the second sum runs over all $\mu$ with $\beta(\mu) = \beta(\lambda) \cup C^+ \setminus C$ for some addable $i$-$\beta$-subset $C$ of $\lambda$ with $|C|=k$.  Here, and hereafter, $E_i^{(k)} = \frac{1}{[k]_q!} E_i^k$ and $F_i^{(k)} = \frac{1}{[k]_q!}F_i^k$, where $[k]_q! = [k]_q[k-1]_q\dotsm[1]_q$ and $[i]_q = q^{-i+1} + q^{-i+3} + \dotsb + q^{i-3} + q^{i-1} \in \mathbb{Z}[q,q^{-1}]$ for all $i \in \mathbb{Z}^+$.
 The action of $E_i^{(k)}$ and $F_i^{(k)}$ is of course determined by that of $E_i$ and $F_i$; nonetheless we will find it convenient to use directly the formulae for their action given above.

In \cite{LT}, Leclerc and Thibon defined a bar involution $x \mapsto \overline{x}$ on $\mathcal{F}$, which satisfies $\overline{a(q)x + y} = a(q^{-1}) \overline {x} + \overline{y}$, $\overline{E_i(x)} = E_i(\overline{x})$ and $\overline{F_i(x)} = F_i(\overline{x})$ for all $a(q) \in \mathbb{C}(q)$, $x,y \in \mathcal{F}$ and $i \in [0,\,e)$.  They proved the existence of another distinguished basis $\{ G(\lambda) \mid \lambda \in \mathcal{P} \}$ of $\mathcal{F}$, called the canonical basis, that has the following characterization:
$$ G(\lambda) - \lambda \in \sum_{\mu \in \mathcal{P}} q\mathbb{Z}[q] \mu,\quad \overline{G(\lambda)} = G(\lambda).$$

Let $\left\langle - , - \right\rangle$ be the symmetric bilinear form on $\mathcal{F}$ with respect to which $\mathcal{P}$ is orthonormal.
For $\lambda,\mu \in \mathcal{P}$, define $d_{\lambda\mu}(q) \in \mathbb{C}(q)$ by
$$
d_{\lambda\mu}(q) = \langle G(\mu), \lambda \rangle.
$$
The $d_{\lambda\mu}(q)$ are now commonly known as $q$-decomposition numbers, and they enjoy many remarkable properties. We shall only need the following one: if $d_{\lambda\mu}(q)\neq 0$, then $\lambda$ and $\mu$ are in the same block.

Let $A$ be the subring of $\mathbb{C}(q)$ consisting of functions without pole at $q=0$, and let $\CF_A$ be the $A$-lattice in $\CF$ generated by all $\lambda\in\CP$. Then (the image of) $\CP$ in $\CF_A/q\CF_A$ is a crystal basis \cite{MM}, so Kashiwara's operators $\kashe_i$ and $\kashf_i$ act on $\CP\cup\{0\}$. A combinatorial description of the action of $\kashe_i$ and $\kashf_i$ is given in \cite{MM} (see also \cite{LLT}); for our purposes we need only that of $\kashe_i$.
Let $\lambda\in\CP$. If the abacus of $\lambda$ has no normal removable beads on runner $i$, then $\kashe_i(\lambda)=0$. Otherwise $\kashe_i(\lambda)$ is the partition whose abacus is obtained by moving the top normal bead on runner $i$ to its preceding position on runner $i-1$.

According to Kashiwara \cite{Ka}, an induced action of the Weyl group $\W$ of $\widehat{\mathfrak{sl}}_e$, a Coxeter group of type $A_{e-1}^{(1)}$, on the crystal basis, and therefore on $\CP$, is obtained.
 Denote the simple reflections of $\W$ by $s_0,s_1,\dotsc, s_{e-1}$.  The action of $s_i$ on $\lambda\in\CP$, whose $e$-core $\kappa$ has $k$ removable beads on runner $i$ for some $k \in \mathbb{Z}_{\geq 0}$ and has no addable beads on runner $i-1$, is given by $s_i(\lambda)=\kashe_i^k(\lambda)$; in other words $s_i(\lambda)$ is obtained by moving the top $k$ normal removable beads on runner $i$ to their preceding positions on runner $i-1$.
Similarly, $s_i(\lambda)=\kashf_i^k(\lambda)$ if $\kappa$ has $k$ addable beads on runner $i-1$ for some $k \in \mathbb{Z}_{\geq 0}$ and has no removable beads on runner $i$.
Note that the $e$-core of $s_i(\lambda)$ is $s_i(\kappa)$, and the $e$-weight of $s_i(\lambda)$ is equal to the $e$-weight of $\lambda$. Consequently for any $w\geq 0$, we have an induced transitive action of $\W$ on the set of blocks of $e$-weight $w$, given by $s_i(B_{\kappa,w})=B_{s_i(\kappa),w}$. In the terminology of Scopes~\cite[Definition 2.1]{S}, the blocks
$\b=B_{\kappa,w}$ and $\tb=s_i(\b) = B_{s_i(\kappa),w}$ are said to form a $[w:k]$-pair. Every $\W$-orbit of partitions of $e$-weight $w$ contains a single partition in each block of $e$-weight $w$.

The action of $U_q(\widehat{\mathfrak{sl}}_e)$ on $\CF$ induces another action of $\W$ on $\CP$, in which $s_i$ acts by swapping runners $i-1$ and $i$ of the abacus, i.e.\ moving all removable beads on runner $i$ to their preceding positions on runner $i-1$ and
all addable beads on runner $i-1$ to their succeeding positions on runner $i$. We do not make explicit use of this more common, combinatorially simpler action, the orbits of which are determined by the $e$-weight and $e$-quotient (up to an easily specified permutation of its components). Thus when we speak of a $\W$-action on $\CP$, we always mean the crystal action described in the preceding paragraph.

\def\tz{\hat{z}}
\def\bm{\mathcal{A}}
\def\good{\mathcal{U}}

\section{A new labelling for partitions} \label{S:label}

We introduce some combinatorics of partitions based on James's abacus, and show that they can
be used to identify certain $\W$-orbits on the set of partitions.

\subsection{Bead movements}

Let $\lambda$ be a partition of $e$-weight $w_{\lambda}$.
Define
\begin{align*}
\bm(\lambda):&=\left\{(b;q)\mid b\in\beta(\lambda),
\tfrac{b-q}{e}\in [0,\, \wt_{\lambda}(b))\right\} \\
&= \{ (b; b-ie) \mid b \in \beta(\lambda),\ i \in [0,\, \wt_{\lambda}(b)) \}.
\end{align*}
We regard this $w_{\lambda}$-element set as the set of
{\em bead movements} needed to reduce $\lambda$
to its $e$-core. Here $(b;q)$ represents the $(\frac{b-q}{e}+1)$-th movement of the bead that is originally positioned at $b$; the movement {\em starts} at $q$ and {\em ends} at $q-e$.  Each bead movement corresponds to a rimhook of $\lambda$ of size divisible by $e$, and thus $\bm(\lambda)$ corresponds naturally and bijectively with the set $\operatorname{Hook}_e(\lambda)$ introduced in Section \ref{S:intro}.  We call $(b;b)$ the {\em initial} bead movement of $b$, and $(b;b - (\wt_{\lambda}(b) -1)e)$ the {\em final} bead movement of $b$. 
If $j$ is a residue class modulo $e$, write $\bm^{(j)}(\lambda)$ for the set of bead movements in runner $j$, i.e. $\bm^{(j)}(\lambda) = \{ (b;q) \in \bm(\lambda) \mid q \equiv_e j \}$.

Totally order $\mathcal{\bm}(\lambda)$ by the reverse lexicographical order, i.e. $(b;q) < (b';q')$ if and only if $q < q'$ or both $q = q'$ and $b < b'$.  So the bead movements are ordered according to their starting positions, and those with the same starting positions are ordered according to the original positions of the associated beads.

We now investigate the induced effect of a simple reflection $s_a \in \W$ on $\bm(\lambda)$.  The image $\tilde{\lambda}=s_a(\lambda)$ of $\lambda$ has the same $e$-weight, and hence the same number of bead movements.
    It is clear that $\bm^{(j)} (\lambda) = \bm^{(j)}(\tilde{\lambda})$ for all $j \ne a,a-1$.  Interchanging the roles of $\lambda$ and $\tlambda$ if necessary, we may assume that $\lambda$ has more, say $k> 0$ more, removable $a$-nodes than addable $a$-nodes. Then $\lambda$ has $k+l$ normal beads on runner $a$ of its abacus display, for some $l\geq 0$, and $\tlambda$ is obtained by moving the top $k$ normal beads across into runner $a-1$.

Let $B_\lambda= \beta(\lambda)\setminus\beta(\tlambda)$.  Then $\{ b-1 \mid b \in B_{\lambda} \} = \beta(\tlambda) \setminus \beta(\lambda)$, and $ \wt_\lambda(b)-\wt_\tlambda(b-1) = l$ for all $b \in B_{\lambda}$.  We now describe the change in the bead movements when $\lambda$ becomes $\tlambda$.
Let $b \in B_{\lambda}$, $w = \wt_{\lambda}(b)$ and $\tilde{w} = \wt_{\tlambda}(b-1)$.  Define
$$
B_{>b}:=\{c\in\beta(\tlambda)\mid c>_e b\}, \qquad
B_{>b -1} :=\{ c' \in \beta(\lambda) \mid c' >_e b-1\}.
$$
Then $|B_{>b}| - |B_{>b-1}| = l = w - \tilde{w}$.  Let
$$
B_{>b} = \{ c_1<\dotsb < c_r \},\qquad B_{>b-1} = \{ c'_1 < \dotsb < c'_{r-l} \}.$$
Then $r\geq l = w-\tilde{w}$.  Firstly on runner $a$, since the bead at position $b$ is moved (to $b-1$), we lose the bead movements $\{(b; b-ie) \mid i \in [0,\, w) \}$, while the $e$-weight of each bead at $c_j$ ($j \in [1,\, r]$) increases by 1 because of this move, and so we gain the bead movements $\{(c_j; b+(j- w)e) \mid j \in [1,\, r]\}$ (since the bead at $b$, which could be moved to $b - we$ when obtaining the $e$-core, is now missing, the bead at $c_1$ can now move to this position, the one at $c_2$ can move to $b + (1-w)e$, etc).  On the other hand, on runner $a-1$, the new bead at position $b-1$ has $e$-weight $\tilde{w} = w - l$, and so we gain the bead movements $\{(b-1; b-1-ie) \mid i \in [0,\, \tilde{w}) \}$, while the $e$-weight of each bead at $c'_j$ ($j \in [1,\, r-l]$) decreases by 1 because of this new bead, and so we lose the bead movements $\{(c'_j; b-1+(j-\tilde{w})e) \mid j \in [1,\, r-l] \}$.

Summarising, let $\Gamma_b$ and $\Delta_b$ be the sets of bead movements gained and lost when going from $\lambda$ to $\tlambda$, that are associated to the bead $b \in B_{\lambda}$.  Then
\begin{align*}
\Gamma_b &= \{ (c_j; b+(j-w)e) \mid j \in [1,\, r] \} \cup \{ (b-1;b-1-ie) \mid i \in [0,\, \tilde{w}) \}; \\
\Delta_b &= \{ (b;b-ie) \mid i \in [0,\, w) \} \cup \{ (c'_j; b-1 + (j-\tilde{w})e) \mid j \in [1,\, r-l] \} 
\end{align*}
It is straightforward to verify that $\phi_b : \Delta_b \to \Gamma_b$, defined by
\begin{align*}
\phi_b(b; b-ie) &=
\begin{cases}
(c_{w - i}; b-ie), &\text{if } \max\{0,w -r\} \leq i < w,\\
(b-1; b-1-ie), \;\, &\text{if } 0 \leq i < w-r;
\end{cases} \\
\phi_b(c'_j; b-1 + (j-\tilde{w})e) &=
\begin{cases}
(b-1; b-1 + (j-\tilde{w})e), &\text{if } 1\leq j \leq \min\{r-l,\tilde{w}\}, \\
(c_{j+l}; b+(j-\tilde{w})e), &\text{if } \tilde{w} < j \leq r-l;
\end{cases}
\end{align*}
is a bijection.
Gluing $\phi_b$ together for each $b \in B_{\lambda}$ then produces a bijection between $\bm(\lambda) \setminus \bm(\tlambda)$ and $\bm(\tlambda) \setminus \bm(\lambda)$, which we extend to a bijection $\phi$ from $\bm(\lambda)$ to $\bm(\tlambda)$ by taking $\phi$ to be the identity on $\bm(\lambda)\cap\bm(\tlambda)$.

\begin{example}
	In this example, we let $e=2$, $\lambda=(10,9,8,8,8,6,6,5,3,3,1,0,\ldots)$ and $a=1$, and follow the notation above. We have $\beta(\lambda)=(9,7,5,4,3,0,-1,-3,-6,-7,-10,-12,-13,\ldots)$,
	so $\lambda$ is a partition of $2$-weight $32$,
	 with $4$ normal beads on runner $1$, at positions $-3$, $3$, $7$ and $9$. We have $k=2=l$, $\beta(\tlambda)=(9,7,5,4,2,0,-1,-4,-6,-7,-10,-12,-13,\ldots)$ and
	 $\tlambda=(10,9,8,8,7,6,6,4,3,3,1,0,\ldots)$.
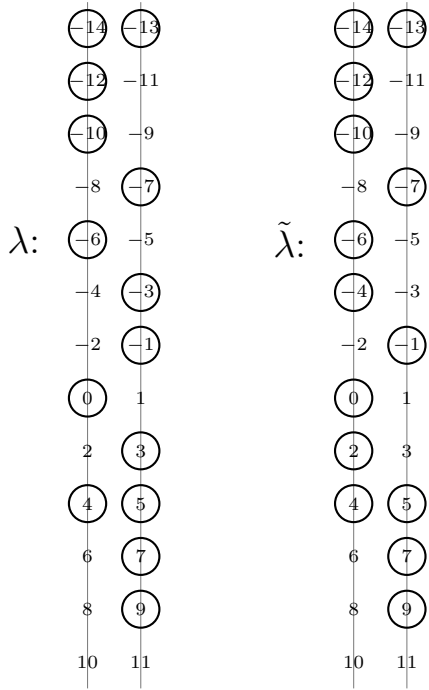
		 \begin{figure}[h]
	 \begin{tikzpicture}[y=-1cm, scale=0.7, every node/.style={font=\scalefont{0.6}}]

	 	\def\separation{5}
	 	\def\beadsize{.35}
	 	
	 	\foreach \x in {0,1}
	 	{\draw [thin, gray] (\x,-7.5) --(\x,5.5);
	 		\draw [thin, gray] (\x+\separation,-7.5) --(\x+\separation,5.5);
	 	};

	 	\foreach \b in {-14,...,11}
	 	\pgfmathsetmacro\z{Mod(\b,2)}
	 	\pgfmathsetmacro\y{(\b-\z)/2}
	 	{\node at (\z,\y) {$\b$};
	 		\node at (\z+\separation,\y) {$\b$};};
	 	
	 	\foreach \b in {9,7,5,4,3,0,-1,-3,-6,-7,-10,-12,-13,-14}
	 	\pgfmathsetmacro\z{Mod(\b,2)}
	 	\pgfmathsetmacro\y{(\b-\z)/2}
	 	\draw [thick] (\z,\y) circle [radius=\beadsize];
	 	
	 	\foreach \b in {9,7,5,4,2,0,-1,-4,-6,-7,-10,-12,-13,-14}
	 	\pgfmathsetmacro\z{Mod(\b,2)}
	 	\pgfmathsetmacro\y{(\b-\z)/2}
	 	\draw [thick] (\z+\separation,\y) circle [radius=\beadsize];
	 	
	 	\node at (-1.2,-3) {\large $\lambda$:};
	 	\node at (-1.2+\separation,-3) {\large $\tilde\lambda$:};
	 	
	 	\end{tikzpicture}
	 \end{figure}

We have $B_{\lambda} = \{ -3,3 \}$.  We consider first $B_{>-3} = \{-1,5,7,9\}$, $B_{>-4} = \{0,4\}$.  Hence
\begin{align*}
\phi(-3;-7)=(-1;-7),\quad \phi(-3;-5)&=(5;-5),\quad \phi(-3;-3)=(7;-3), \\
\phi(0;-4) &= (-4;-4),\\
\phi(4;-2) &= (9;-1) .
\end{align*}

	Next, $B_{>3} = \{5,7,9\}$ and $B_{>2} = \{4\}$.  Thus
	\begin{gather*}
\phi(3;-3) = (5;-3), \quad \phi(3;-1) = (7;-1),\quad \phi(3;1) = (9;1), \quad \phi(3;3) = (2;2) \\
\phi(4;0) = (2;0).
\end{gather*}

Note that $(5;-1) \in \bm(\lambda) \cap \bm(\tlambda)$.  Thus
$$ \phi(5;-1) = (5;-1) < (9;-1) = \phi(4;-2),$$
showing that $\phi$ is not order-preserving, since $(4;-2) < (5;-1)$.
	\end{example}

Using the bijection $\phi$, we can deduce that the order-preserving bijection between $\bm(\lambda)$ and $\bm(\tlambda)$ has some desirable properties.

\begin{lemma}\label{beadmovementbijection}
Let $\lambda$ be a partition, and let $\tlambda = s_a(\lambda)$.  The order-preserving bijection
	$\bm(\lambda) \iso \bm(\tlambda) : (b;q)\mapsto(b';q')$
	satisfies $(b';q') = (b;q)$ whenever $(b;q) \in \bm^{(j)}(\lambda)$ with $j \ne a-1,a$ and
	$$q'=
	\begin{cases}
	q \text{ or } q+1, & \text{if } (b;q) \in \bm^{(a-1)}(\lambda); \\
	q \text{ or } q-1, & \text{if } (b;q) \in \bm^{(a)}(\lambda). \\
	\end{cases}
	$$
Furthermore, if $q' \ne q$, then there exists a $\hat{b} \in \beta(\lambda) \setminus \beta(\tlambda)$ with $\hat{b} \equiv_e a$ such that either $\frac{\hat{b}-q}{e} \in [0,\, \wt_{\lambda}(\hat{b}) -N)$ or $\frac{q'-\hat{b}}{e} \in (0,\, N - \wt_{\lambda}(\hat{b})]$, where $N = |\{ x \in \beta(\tlambda) \mid x >_e \hat{b}\}|$.
\end{lemma}

\begin{proof}
Clearly, $\bm^{(j)}(\lambda) = \bm^{(j)}(\tlambda)$ for $j \ne a, a-1$.  Furthermore, from the bijection $\phi$ defined above, we see that, for each $q\in \mathbb{Z}$ with $q \equiv_e a$, the total number of bead movements of $\tlambda$ starting at $q$ and $q-1$ equals that of $\lambda$.
Since the order on the bead movements takes into account the starting positions first, the first assertion thus follows.

For the second assertion, if $q' \ne q$, then the number of bead movements of $\lambda$ starting at $q$ is not equal to that of $\tlambda$, so that necessarily $\phi(\bar{b};q) = (\tilde{b};q')$  or $\phi(\tilde{b};q') = (\bar{b};q)$ for some $\bar{b}, \tilde{b}$.  Examining the definition of $\phi$ then proves the assertion.
\end{proof}

\subsection{Armlengths of bead movements}

Let $S$ be a subset of $\mathbb{Z}$.  Define the function $\z_S : \mathbb{Z} \to [0,\, e]$ by
$$
\z_S(x) = \left| (x-e,\, x]\setminus S\right|
$$
for all $x \in\BZ$.

Let $\lambda$ be a partition.  We write $\z_\lambda$ for $\z_{\beta(\lambda)}$, and for each $ (b;q)\in\bm(\lambda)$, we define
$\z(b;q) := \z_{\lambda}(q)$.
Thus $\z(b;q)$, in terms of the abacus display of $\lambda$, is the number of unoccupied positions `passed' by the bead movement from $q$ to $q-e$, where the starting position, if unoccupied, is included in the count
(but the ending position, occupied or not, is never included in the count).  If $(b;q)$ corresponds to the rimhook $H$ of size divisible by $e$, then $\z(b;q)$ is the width of $H^{\operatorname{last}}$, or one less the width of $H^{\operatorname{last}}$ if the northeast-most node of $H^{\operatorname{last}}$ is the last node in its row in $[\lambda]$ (see Section \ref{S:intro}).  This explains why $\z(b;q)$ may be considered as the `armlength' of $(b;q)$.  Note that
$\z(b;q)$ depends only on $q$ and not on $b$.

Let $\lambda$ be a partition of $e$-weight $w$.  List the elements of $\bm(\lambda)$ in  ascending order with respect to the total order on $\bm(\lambda)$: $(b_1;q_1)< (b_2;q_2)<\dotsb< (b_w;q_w)$.
Define
$$
\z(\lambda) := (\z(b_1;q_1),\dotsc, \z(b_w;q_w))  \in [0,e]^{w}.
$$
This definition extends that used in Section \ref{S:intro}.  Note that, necessarily, $b_w = q_w$, so that $\z(b_w;q_w) \neq e$.

\begin{example}[Rouquier blocks]\label{example:Rouquier}
Let us consider a partition $\lambda$ lying in a Rouquier block of $e$-weight $w$ (see the end of Section~\ref{subsection:partitions} for the definition).
Using \cite[Lemma 4]{CK}, we deduce that for each $(b;q)\in\bm(\lambda)$ with $q\equiv_e a$, we have $\z(b;q)=a$ if $q\in\beta(\lambda)$ and
$\z(b;q)=a+1$ if $q\notin\beta$. Furthermore the bead movements in runner $a$ always precede those in runner $a+1$ in the total order on $\bm(\lambda)$.
It follows that the components of $\z(\lambda)$ are weakly increasing if and only if the $p$-quotient of
 $\lambda$ has the form $((1^{w_0}),\ldots,(1^{w_{e-1}}))$, in which case $\z(\lambda)=(0^{w_0},\dotsc,(e-1)^{w_{e-1}})$.
\end{example}

Let $\lambda$ be a partition and let $s_a(\lambda) = \tilde{\lambda}$, where $s_a$ is a simple reflection of the affine Weyl group $\W$.  We shall prove a sufficient condition for $\z(\lambda) = \z(\tlambda)$.  We begin with the following lemma:

\begin{lemma} \label{beadconfiginherit}
Let $\lambda$ be a partition whose abacus display has more removable beads than addable beads on runner $a$, and let $s_a(\lambda) = \tlambda$.  If there exists $q \in \beta(\lambda)$ with $q \equiv_e a$ such that $q-e \notin \beta(\lambda)$ and $(b;q-1) \in \bm(\lambda)$ for some $b$, then the same holds for $\tlambda$ (possibly for a different $q$).
\end{lemma}

In pictorial terms, the above lemma asserts that if the abacus display of $\lambda$ has the following bead configuration on runners $a$ and $a-1$, then so does that of $\tlambda$.
\begin{figure}[ht]\caption{Bead configuration}
	 \begin{tikzpicture}[y=-1cm, scale=1.5, every node/.style={font=\scalefont{1}}
	 ]

	 	\def\beadsize{.25}
	 	
	 	\foreach \x in {0,1}
	 	{\draw [thin, gray] (\x,-0.5) --(\x,1.5);
	 	};

	 	\draw (0,0) [dashed] circle [radius=\beadsize];
	 	\draw (0,1) [dashed] circle [radius=\beadsize];
	 	
	 	\node at (1,0) {$q-e$};
	 	\draw (1,1) [thick] circle [radius=\beadsize];
	 	\node at (1,1) {$q$};
	 	
	 	\draw [->, >=latex, line width = 2] (0,0.7) -- (0,0.3);

	 	\end{tikzpicture}

	\label{figure:disallowed-configuration}
\end{figure}
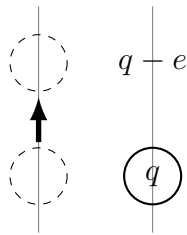
(Here, and hereafter, a dash-outlined circle indicates a position that may or may not be occupied by a bead.)

\begin{proof}
Suppose that there exists $q \in \beta(\lambda)$ with $q \equiv_e a$ such that $q-e \notin \beta(\lambda)$ and $(b;q-1) \in \bm(\lambda)$ for some $b \in \beta(\lambda)$ with $b \geq_e q-1$.  We claim that the sets
\begin{align*}
B &= \{ q \leq_e \bar{b} \leq_e b+1 \mid \bar{b} \in  \beta(\tlambda) \}, \\
C & = \{ c <_e q-1 \mid c \notin \beta(\tlambda) \}
\end{align*}
are non-empty.
If $q \in \beta(\tlambda)$, then $q \in B$.  On the other hand, if $q \notin \beta(\tlambda)$, then the bead of $\beta(\lambda)$ at $q$ is normal and $b >_e q-1$.  The existence of a bead of $\beta(\lambda)$ at $b$ with $b >_e q-1$ shows that there is some bead of $\beta(\lambda)$ at $q'$ with $q<_e q' \leq_e b+1$ which is not normal, so that $q' \in B$.
For $C$, clearly $q-1-e \in C$ if $q-1-e \notin \beta(\lambda)$.   If $q-1-e \in \beta(\lambda)$, then the bead of $\beta(\lambda)$ at $q-e-1$ is addable, so that either $\lambda$ has no removable bead at $\tilde{q}$ for all $\tilde{q} <_e q$, or $\lambda$ has a removable but not normal bead at $\tilde{q}$ for some $\tilde{q} <_e q$. In the former case, since $(b;q-1) \in \bm(\lambda)$, there exists $c <_e q-1$ such that $c \notin \beta(\lambda)$, and hence $c \notin \beta(\tlambda)$ so that $c \in C$.  In the latter case, $\tilde{q}-1 \in C$.

Let $\tilde{q} = \min(B)$.  Then $\tilde{q} \in \beta(\tlambda)$ and $\tilde{q} - e \notin \beta(\tlambda)$. Let $\tilde{b} = \min\{ \hat{b} \in \beta(\tlambda) \mid \hat{b} \geq _e\tilde{q} - 1\}$ (this set is nonempty since it contains $b$).  As $\tilde{q} -1 \geq_e q-1 >_e c \notin \beta(\tlambda)$ for any $c \in C$, we have $(\tilde{b};\tilde{q}-1) \in \bm(\tlambda)$.  \end{proof}

\begin{prop}\label{prop:zunchanged}
Let $\lambda$ be a partition, and let $s_a(\lambda) = \tlambda$.  Suppose that the abacus display of $\lambda$ does not have the bead configuration of Figure \ref{figure:disallowed-configuration} on its runners $a$ and $a-1$.
Then $\z(\lambda) = \z(\tlambda)$.
\end{prop}

\begin{proof}
If the abacus display of $\lambda$ has the same number of removable beads as addable beads on runner $a$, then $s_a(\lambda) = \lambda$, and so the statement is trivially true.
If $\tlambda$ has more removable beads than addable beads on runner $a$, then its abacus display also does not have the bead configuration of Figure \ref{figure:disallowed-configuration} on its runners $a$ and $a-1$, by Lemma \ref{beadconfiginherit}.  As such, we may assume, by interchanging $\lambda$ and $\tlambda$ if necessary, that $\lambda$ has more removable beads, say $k$ more, than addable beads on runner $a$. Let $X$ and $Y$ be the top unoccupied and bottom occupied positions on runner $a-1$ of the abacus of $\lambda$ respectively. Then all the bead movements of $\lambda$ on this
runner have starting positions amongst $X+e, X+2e, \ldots, Y$, and each of these positions is a starting position of some bead movement of $\lambda$.
We note that it is possible for $X >_e Y$, in which case there is no bead movement on this runner at all.
Let $Z = \min \{ b \geq_e X+1 \mid b \notin \beta(\lambda) \}$.  Then to avoid the undesired bead configuration, any position $V$ with $Z\leq_eV \leq_e Y+1$ is unoccupied; see Figure~\ref{figure:leqconfiguration}. Note that it is possible for $Z>_eY+1$.
\begin{figure}[h]
	\caption{An example of an abacus configuration avoiding Figure~\ref{figure:disallowed-configuration}}
	 \begin{tikzpicture}[y=-1cm, scale=0.7, every node/.style={font=\scalefont{0.6}}
	 ]

	 	\def\beadsize{.35}
	 		\foreach \x in {0,1}
	 	{\draw [thin, gray] (\x,-0.5) --(\x,11.5);
	 	};

	 	\draw (0,0) [thick] circle [radius=\beadsize];
	 	\draw (0,1) [thick] circle [radius=\beadsize];
	 	\draw (0,2) [thick] circle [radius=\beadsize];
	 	\node at (0,3) {$X$};
	 	\draw (1,1) [dashed] circle [radius=\beadsize];
	 	\draw (1,2) [dashed] circle [radius=\beadsize];
	 	\draw (0,4) [dashed] circle [radius=\beadsize];
	 	\draw (0,5) [dashed] circle [radius=\beadsize];
	 	\draw (0,6) [dashed] circle [radius=\beadsize];
	 	\draw (0,7) [dashed] circle [radius=\beadsize];
	 	\draw (0,8) [dashed] circle [radius=\beadsize];
	 	 \draw (1,3) [thick] circle [radius=\beadsize];
	 	\draw (1,4) [thick] circle [radius=\beadsize];
	 	 \draw (1,5) [thick] circle [radius=\beadsize];
	 	 	 \node at (1,6) {$Z$};
	 	 	  \node at (1,7) {$Z+e$};
	 		 \draw (0,9) [thick] circle [radius=\beadsize];
	 		 \node at (0,9) {$Y$};
	 		 \node at (1,9) {$Y+1$};
	 		 \draw (1,10) [dashed] circle [radius=\beadsize];
	 	 \draw (1,11) [dashed] circle [radius=\beadsize];

	 	\draw (1,0) [dashed] circle [radius=\beadsize];
	 	\draw (1,1) [dashed] circle [radius=\beadsize];


	 	\end{tikzpicture}
	\label{figure:leqconfiguration}
\end{figure}
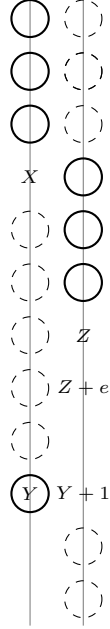

Recall that
 $\tlambda$ is obtained from $\lambda$ by moving the top $k$ of the $k+l$
 normal removable beads across from runner $a$ into runner $a-1$. For convenience's sake call these $k$ beads supernormal.

We claim first that if $\lambda$ has a supernormal bead at position $W\leq Y+1$, then $\wt_{\tlambda}(W-1)=0$.  Indeed, we must have $X< W < Z$.  Since there are no addable beads on runner $a-1$ between $X$ and $Z-e$, all removable beads on runner $a$ between $X$ and $W$ must be normal, and hence supernormal, since the bead at $W$ is.  Thus
  $\wt_{\tlambda}(W-1)=0$.

Note also that the number of unoccupied positions $<Z$ on runner $a$ is at most $l$, for otherwise there would be more than $k+l$ normal removable beads on runner $a$.

Each bead movement
$(b;q)\in\bm(\lambda)$ corresponds to $(b';q')\in\bm(\tlambda)$ via the order-preserving bijection.  We need to show that $z:=\z(b;q)$ and $z':=\z(b';q')$ are equal.  This is clear if $q \not\equiv_e a,a-1$ (in which case, $q = q'$ by Lemma \ref{beadmovementbijection}).  We have four other cases to consider:

\begin{description}
\item[Case 1. $q\equiv_e q'\equiv_e a$]  Then $q=q'$, and  since the total number of beads
in positions $q-1$ and $q$ of $\lambda$ and of $\tlambda$ is always the same, we have $z=z'$.

\item[Case 2. $q\equiv_e a$, $q'\equiv_e a-1$] By Lemma~\ref{beadmovementbijection},  there exists a supernormal bead at $\hat{b}$ such that $\frac{\hat{b}-q}{e} \in [0,\, \wt_{\lambda}(\hat{b})- N)$, where $N = |\{ x \in \beta(\tlambda) \mid x >_e \hat{b} \}|$. Since
 \begin{align*}
 \wt_{\lambda}(\hat{b}) &= |\{ x <_e \hat{b} \mid x \notin \beta(\lambda) \}| = |\{ q \leq_e x <_e \hat{b} \mid x \notin \beta(\lambda) \}|  + \wt_{\lambda}(q) \leq  \frac{\hat{b} - q}{e} + \wt_{\lambda}(q),
 \end{align*}
we see that $\wt_\lambda(q) \geq \wt_{\lambda}(\hat{b}) - \frac{\hat{b}-q}{e} > N$. 
But $\wt_{\lambda}(Z) \leq l$, and $l \leq N$ (as the $l$ normal but not supernormal beads on runner $a$ of $\beta(\lambda)$ lie below the supernormal ones (which include the bead at $\hat{b}$)).  Thus $q >_e Z$.
  Now if $Z <_e q \leq_e Y+1$, then both $q$ and $q-e$ are
 unoccupied, and so $z=z'$. Otherwise $q>_e \max(Y+1,Z)$, in which case $q-1$ is unoccupied, and any bead of $\lambda$ at position $q$ or $q-e$ must be supernormal (since the bead at position $\hat{b}$ is supernormal and $\hat{b} \geq_e q$). Here again we check easily that $z=z'$.

\item[Case 3. $q,q'\equiv_e a-1$]  Then $q=q'$.  Note that $\lambda$ cannot have any supernormal bead at any position $W$ with $q+1-e \leq_e W \leq_e b+1$, for then $\wt_\tlambda(W-1)=0$, contradicting $(b';q')=(b';q)\in\bm(\tlambda)$.  It follows that $\lambda$ and $\tlambda$ have the same bead configuration at positions $q$ and $q+1-e$, and therefore that $z=z'$.

\item[Case 4. $q\equiv_e a-1$, $q'\equiv_e a$]  By Lemma~\ref{beadmovementbijection}, $q'=q+1$,  and $\lambda$ has a supernormal bead at $\hat{b}$ with $\frac{q'-\hat{b}}{e} \in (0,\, N-\wt_{\lambda}(\hat{b})]$, where $N = |\{x \in \beta(\tlambda) \mid x >_e \hat{b} \}|$.  Thus $Y \geq_e q = q'-1 >_e \hat{b} -1 \geq_e X$, so that $\wt_{\tlambda}(\hat{b} -1) = 0$ by our claim above, and hence $\wt_{\lambda}(\hat{b}) = l$.  This implies that $N - l = N - \wt_{\lambda}(\hat{b}) > 0$. We deduce that there are $N-l$ beads of $\lambda$ on runner $a$ below $\hat{b}$ which are not normal.  Since these are not normal, they are at positions less than $Y$, and thus strictly less than $Z$. This shows that
    $$
    \frac{Z-\hat{b}}{e} > N-l = N-\wt_{\lambda}(\hat{b}) \geq \frac{q'-\hat{b}}{e}.
    $$
    Thus $Z >_e q'$.  Since $q'-e \geq \hat{b} \geq X$,
     it follows that the positions at $q'-e$ and $q'$ are occupied by beads of $\lambda$, and thus, regardless of whether these beads are supernormal, we have $z=z'$.
\end{description}
\end{proof}

The following example illustrates the necessity of the hypothesis in Proposition \ref{prop:zunchanged} on avoiding the configuration in Figure \ref{figure:disallowed-configuration}.

\begin{example}
Let $e=3$, $\lambda = (5,3,3)$ and $\tlambda=s_1(\lambda)=(4,3,3)$. Then
  $\z(\lambda) = (2,1,2) \ne (2,1,1) = \z(\tlambda)$.
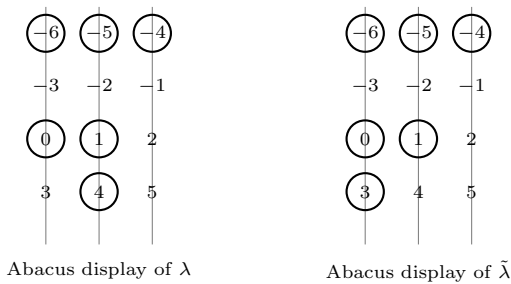
\medskip

\begin{center}
	 \begin{tikzpicture}[y=-1cm, scale=0.7, every node/.style={font=\scalefont{0.6}}
	 ]

	 	\def\beadsize{.35}
	 		\foreach \x in {0,1,2,6,7,8}
	 	{\draw [thin, gray] (\x,-0.5) --(\x,4);
	 	};

	 	\draw (0,0) [thick] circle [radius=\beadsize];
	 	\draw (1,0) [thick] circle [radius=\beadsize];
	 	\draw (2,0) [thick] circle [radius=\beadsize];
	 	\draw (0,2) [thick] circle [radius=\beadsize];
	 	\draw (1,2) [thick] circle [radius=\beadsize];
	 	\draw (1,3) [thick] circle [radius=\beadsize];
	 	\node at (0,0) {$-6$};
        \node at (1,0) {$-5$};
        \node at (2,0) {$-4$};
	 	\node at (0,1) {$-3$};
        \node at (1,1) {$-2$};
        \node at (2,1) {$-1$};
	 	\node at (0,2) {$0$};
        \node at (1,2) {$1$};
        \node at (2,2) {$2$};
	 	\node at (0,3) {$3$};
        \node at (1,3) {$4$};
        \node at (2,3) {$5$};

	 	\node at (1,4.5) {Abacus display of $\lambda$};

	 	\draw (6,0) [thick] circle [radius=\beadsize];
	 	\draw (7,0) [thick] circle [radius=\beadsize];
	 	\draw (8,0) [thick] circle [radius=\beadsize];
	 	\draw (6,2) [thick] circle [radius=\beadsize];
	 	\draw (7,2) [thick] circle [radius=\beadsize];
	 	\draw (6,3) [thick] circle [radius=\beadsize];
	 	\node at (6,0) {$-6$};
        \node at (7,0) {$-5$};
        \node at (8,0) {$-4$};
	 	\node at (6,1) {$-3$};
        \node at (7,1) {$-2$};
        \node at (8,1) {$-1$};
	 	\node at (6,2) {$0$};
        \node at (7,2) {$1$};
        \node at (8,2) {$2$};
	 	\node at (6,3) {$3$};
        \node at (7,3) {$4$};
        \node at (8,3) {$5$};
 	    \node at (7,4.5) {Abacus display of $\tilde{\lambda}$};


	 	\end{tikzpicture}
\end{center}
\end{example}

\subsection{The $0$-increasing $\W$-orbits}

We now present the main result of this section, showing that certain orbits of the action of the affine Weyl group $\W$ on the set of partitions of $e$-weight $w$ are parametrised by $w$-tuples of integers, through the function $\z$.
In preparation
 we make the following definitions, using some of the notational conventions introduced in \S~\ref{section:notations}.
\begin{defi} Let $m$ be a non-negative integer.
\begin{enumerate}
	\item A partition $\lambda$ (of $e$-weight $w$) is said to be {\em $m$-increasing} if $\z(\lambda) = (z_1,\dotsc,z_w)$ is $m$-increasing, i.e.\ if $z_{i+1}-z_{i}\geq m$ for all $i\in[1,w)$.
	\item		Set
$$\good:=[0,e-1]^w \subset \mathbb{Z}^w,$$
so that
$$\inc{\good}{m}=\left\{(z_1,\dotsc,z_w) \in \mathbb{Z}^w\mid 0\leq z_i \leq e-1 \text{ and } z_{i+1}-z_{i}\geq m \text{ for all } i\in[1,w)\right\},$$
and in particular
 $$\inc{\good}{0}=\left\{(z_1,\dotsc,z_w) \in \mathbb{Z}^w\mid 0\leq z_1 \leq \dotsb \leq
 z_w \leq e-1\right\}.$$

\end{enumerate}
\end{defi}

\begin{lemma}\label{L:0-increasing}
Let $\lambda$ be a $0$-increasing partition of $e$-weight $w$.
	Then the bead configuration of Figure~\ref{figure:disallowed-configuration} cannot occur anywhere on the abacus display of $\lambda$.
\end{lemma}

\begin{proof}
Suppose that $q\in\beta(\lambda)$ and
	$q-e\notin\beta(\lambda)$, so that $(q;q)\in\bm(\lambda)$.
	Then $(b;q-1)\notin\bm(\lambda)$ for any $b$, for then
	$\z(b;q-1)=\z(q;q)+1$, contradicting $\z(\lambda)\in\inc{\good}{0}$.
\end{proof}

\begin{thm}\label{thm:goodlabels} \hfill
\begin{enumerate}
		\item
		Let $B$ be a block of $e$-weight $w$. We have a bijection
		\begin{align*}
			\{0 \text{-increasing partitions in } B\} & \quad \iso \quad \inc{\good}{0} \\
			\lambda & \quad \mapsto \quad \z(\lambda) \\
\end{align*}	
	\item Let $\lambda$ and $\mu$ be partitions of $e$-weight $w$, and suppose that $\lambda$ is $0$-increasing. Then $\lambda$ and $\mu$ are in the same $\W$-orbit if and only if $\z(\lambda)=\z(\mu)$.
		\end{enumerate}
	\end{thm}

\begin{proof}
	Let $B$ be a block of $e$-weight $w$, and let $\mathfrak{w} \in \W$.  Let  $\tB=\mathfrak{w}(B)$.
	 Let the subsets of $0$-increasing partitions in $B$ and $\tB$ be denoted by $B_{\geq 0}$ and $\tB_{\geq 0}$ respectively.  By Proposition \ref{prop:zunchanged} and Lemma \ref{L:0-increasing}, we have $\z(\lambda) = \z(\mathfrak{w}(\lambda))$ for all $0$-increasing $\lambda$ in $B$.
	 In particular $\mathfrak{w}(B_{\geq 0 }) \subseteq \tB_{\geq 0}$, and therefore $\mathfrak{w}(B_{\geq 0 }) =\tB_{\geq 0}$, as the same argument applies to $B=\mathfrak{w^{-1}}(\tB)$.

The first assertion of the theorem is true for a Rouquier block $B$ of $e$-weight $w$; see Exam\-ple~\ref{example:Rouquier}. The above considerations then imply that it holds also for any block of $e$-weight $w$, by the transitivity of the $\W$-action. The second assertion also follows immediately.

\end{proof}

\begin{example}\label{example:e2}
Let $e=2$. The set of all $0$-increasing partitions consists of the following	
`nearly triangular partitions with tail': let
$$\lambda(m,s,t,j)=(m,m-1,\dotsc,\hat{s},\dotsc,\hat{t},\dotsc,2,1,1^{2j}),$$
a partition of $\frac{m(m+1)}{2} - s - t + 2j$. Here
 $m,s,t,j$ range over all integers such that $m \geq s > t \geq 1$ and $j\geq 0$, and
  $\hat{s}$ and $\hat{t}$  indicate the omission of $s$ and $t$ as parts.

The partition $\lambda(m,s,t,j)$ has $2$-core
$(r,r-1,\dotsc,1)$,
where
$$r=\left|m-2s+2t+\tfrac{1}{2}\right|-\tfrac{1}{2},$$
$2$-weight
$$w  = (m-s)(s-t)+t(s-t-1)+j$$
and $2$-quotient given by
$$\lambda(m,s,t,j)^{(i)}=
\begin{cases}
\left((s-t)^{m-s}\right) & \text{if } m+i \text{ is odd}, \\
\left(t^{s-t-1}, 1^j\right) & \text{if } m+i \text{ is even},
\end{cases}
$$
for $i\in\{0,1\}$. It is easily verified that $\z(\lambda(m,s,t,j))=(0^j,1^{w-j})\in\inc{\good}{0}$.

Moreover, all $0$-increasing partitions are obtained this way: given integers $r,w,j$ with
$r\geq 0$ and $w\geq j \geq 0$, there is a unique representation
$$w-j=n(n+r)+cn+d,\qquad (n\geq0, \quad c\in\{0,1\}, \quad d\in [0,n-1 + c(r+1)]),$$
and we take
  $$m=2n+r+1+c, \quad s=2n+r+1+c-d \quad\text{and}\quad t=n+c(r+1)-d.$$
  \end{example}

\begin{rem} (Rouquier blocks)\label{remark:Rouquier}
We do not expect a complete description of the $0$-increasing partitions for $e>2$ analogous to that given in Example~\ref{example:e2}. However the situation is completely manageable if we restrict to partitions in Rouquier blocks. Indeed, in Example~\ref{example:Rouquier}, we saw that given $\mathbf{z}=(0^{w_0},\dotsc, (e-1)^{w_{e-1}})\in\inc{\good}{0}$, it is easy to describe the partition $\lambda$ in a Rouquier block (of $e$-weight $w=w_0+\dotsb+w_{e-1}$) such that $\z(\lambda)=\mathbf{z}$, as the partition with $e$-quotient $((1^{w_0}),\dotsc,(1^{w_{e-1}}))$. So the $0$-increasing partitions in Rouquier blocks are precisely those whose $e$-quotients consist entirely of $1$-column partitions. Moreover, for $m\geq 1$, a partition $\lambda$ in a Rouquier block is $m$-increasing if and only if each component $\lambda^{(i)}$ of the $e$-quotient of $\lambda$ is either $\emptyset$ or $(1)$, with $\lambda^{(i)}=\lambda^{(j)}=(1)$ only if $i=j$ or $|i-j|\geq m$.
\end{rem}

\begin{example}\label{E:principalblocks} (Principal blocks)
Let us consider the blocks at the opposite extreme to Rouquier blocks, namely the blocks with empty $e$-core.
Given $w\geq 1$ and $0\leq a_1<\dotsb<a_w\leq e-1$, denote by $\langle a_1,\dotsc,a_w \rangle$ the partition (of $e$-weight $w$) and empty $e$-core whose $e$-quotient is $(1)$ in its $a_i$-th component, for each $i\in[1,\,w]$, and $\emptyset$ in all other components; these partitions include the `general vertices' and `$p$-general vertices' considered in \cite{MR} and \cite{FM}.
Let $\mathbf{z}=(z_1,\dotsc,z_{w})\in\inc{\good}{0}$. If
$w-1\leq z_1$ and $z_w\leq e-w$, then we have
$$\z(\langle z_1-w+1, z_2 -w +3, \dotsc, z_w+w-1\rangle) =\mathbf{z}.$$

It seems difficult to describe explicitly the ($0$-increasing) partition $\lambda$ with empty $e$-core such that
$\z(\lambda)=\mathbf{z}$ in general. In preparation for one further example in which it is possible,
given $w\geq 2$ and $0\leq a_1<\dotsb<a_{w-1}\leq e-1$ and $1\leq j \leq w-1$, let $\langle a_1,\dotsc,a_{j-1}, a_j^{(2)}, a_{j+1},\dotsc, a_w \rangle$ be the partition (of $e$-weight $w$ and empty $e$-core) whose $e$-quotient is $(2)$ in the $a_j$-th coordinate, $(1)$ in the $a_i$-th component, for each $i\in[1,\,w]\setminus\{j\}$, and $\emptyset$ in all other components; amongst these are the `semi-general vertices' considered in \cite{FM}.
Then if $w-1\leq z_1$ and $z_{w-1}\leq e-w+1 \leq z_w$, we have
 $$\z(\langle z_1-w+1, z_2 -w +3,\dotsc, z_{j-1}- w+2j-1,
   (z_j - w+ 2j)^{(2)}, z_{j+1} -w+2j+2,
 \dotsc, z_{w-1} + w-2\rangle) =  \mathbf{z},$$
 where $j=e-z_w\in[1,w-1]$.

We conclude with one final example in principal blocks.
 Given $w\in \mathbb{Z}^+$, there exist unique $n \in \mathbb{Z}^+$ and $r \in [0,\, n)$ such that $w=\frac{n(n+1)}{2}-r$. Let $e \geq n$, and
$$\nu(w) = (ne-r, (n-1)e-r, \dotsc,(r+1)e-r, (r-1)e+n-r, (r-2)e+n-r, \dotsc, n-r),$$
the partition of $we$ with empty $e$-core and $e$-quotient
$$ 
((0),\dotsc, (0), (r+1), (r+2),\dotsc, (n), (0), (1), \dotsc, (r-1)).$$ Then it is easy to check that $\z(\nu(w))=(e-1,\ldots, e-1)$. The appearance of the triangular numbers $\frac{n(n+1)}{2}$
suggests the difficulty of inverting the map $\lambda\mapsto \z(\lambda)$ in general.
	\end{example}

\subsection{The $1$-increasing partitions}
We end this section with an analysis of $1$-increasing partitions in Lemma~\ref{lemma:1unram}, preceded by an easy and useful lemma on which it depends.
\begin{lemma} \label{L:easy}
	Let $S \subseteq \mathbb{Z}$, and let $x,y \in \mathbb{Z}$ with $x<y$.
	\begin{enumerate}
		\item If $\z_S(x) < \z_S(y)$, then there exists $t\in(x,\,y]$ such that $t\notin S$ and $t-e\in S$.
		\item If $\z_S(x) > \z_S(y)$, then there exists $t\in(x,\,y]$ such that $t\in S$ and $t-e\notin S$.
		\item If $\z_S(x) = \z_S(y)$, then either $t\in S \Leftrightarrow t-e\in S$ for all $t\in (x,\,y]$, or there exist $t,t'\in(x,\,y]$ such that $t\notin S$ and $t-e\in S$, and $t'\in S$ and $t'-e\notin S$.
	\end{enumerate}
\end{lemma}

\begin{proof}
	We have $\z_S(y)-\z_S(x)=\sum_{t=x+1}^{y}(\z_S(t)-\z_S(t-1))$ and
	$$\z_S(t)-\z_S(t-1)=\begin{cases}
	1, & \text{if } t\notin S \text{ and } t-e\in S; \\
	-1, & \text{if } t\in S \text{ and } t-e\notin S; \\
	0, & \text{otherwise.}
	\end{cases}
	$$
The lemma follows immediately.
	\end{proof}

  Recall that a partition $\lambda$ is called a {\em hook} if and only if $\lambda_2\leq 1$. Thus, the empty partition is a hook, and any nonempty hook is uniquely expressible as $(x,1^y)$ with $x \in \BZ^+$ and $y \in \BZ_{\geq 0}$.

\begin{lemma} \label{lemma:1unram}
 Let $\lambda$ be a $1$-increasing partition, with $e$-quotient $(\lambda^{(0)},\dotsc,\lambda^{(e-1)})$.
\begin{enumerate}
\item Each $\lambda^{(j)}$ is a hook. 
\item Suppose that the abacus display of $\lambda$ has more removable beads than addable beads on runner $j$.
Then there is at most one addable bead on runner $j-1$.  Furthermore, if $\lambda^{(j-1)}, \lambda^{(j)} \ne \emptyset$, then there is such an addable bead only when the bottom bead on runner $j-1$ is at the preceding position of the least unoccupied position on runner $j$.
\item Suppose that the abacus display of $\lambda$ has more addable beads than removable beads on runner $j$.
Then there is at most one removable bead on runner $j$.  Furthermore, if $\lambda^{(j-1)}, \lambda^{(j)} \ne \emptyset$, then there is such a removable bead only when the bottom bead on runner $j$ is at the succeeding position of the least unoccupied position on runner $j-1$.
\end{enumerate}
\end{lemma}

\begin{proof} \hfill
\begin{enumerate}
\item Let $x = \max\{b \in \beta(\lambda) \mid b \equiv_e j\}$, and $y = \max\{b \in \beta(\lambda) \mid b <_e x\}$. Then
    $ \wt_{\lambda}(x) = \frac{x-y}{e} - 1 + \wt_{\lambda}(y)$.  If $\lambda^{(j)}_2 \geq 2$, then $\wt_{\lambda}(y) \geq 2$, so that $\frac{x-y}{e} \leq \wt_{\lambda}(x) - 1$.  Consequently, $(x;y), (y;y) \in \bm(\lambda)$, and clearly $\z(x;y) = \z(y;y)$.  But $\lambda$ is $1$-increasing and hence the components of $\z(\lambda)$ should be distinct, a contradiction.
\item By part (1), $\lambda^{(j)}$ and $\lambda^{(j-1)}$ are hooks (or empty).  It is then easy to see that $\lambda$ has at most one addable bead on runner $j-1$ if $\lambda^{(j-1)}$ or $\lambda^{(j)} = \emptyset$.  To deal with the remaining case, we first observe the following fact about $1$-increasing partitions:
    If $(b;q), (b';q-1) \in \bm(\lambda)$ for some $q \equiv_e j$ and $b,b' \in \beta(\lambda)$, then $\z(b';q-1) < \z(b;q)$ since $\lambda$ is $1$-increasing, so that $q \notin \beta(\lambda)$ and $q-e \in \beta(\lambda)$ by Lemma \ref{L:easy}(1).  The following summarises this in pictorial terms:
    \medskip

    
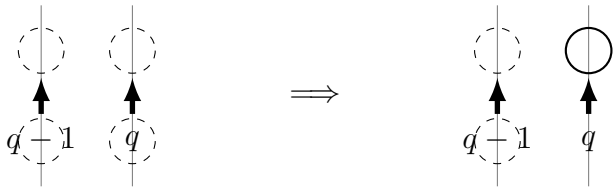
\begin{center}
   	 \begin{tikzpicture}[y=-1cm, scale=1.2, every node/.style={font=\scalefont{1}}
	 ]
	
	 \def\beadsize{.25}
	
	 \foreach \x in {0,1,5,6}
	 {\draw [thin, gray] (\x,-0.5) --(\x,1.5);
	 };

	 \draw (0,0) [dashed] circle [radius=\beadsize];
	 \draw (0,1) [dashed] circle [radius=\beadsize];
	 	 \draw (1,0) [dashed] circle [radius=\beadsize];
	 	 \draw (1,1) [dashed] circle [radius=\beadsize];

	 	 	 \draw (5,0) [dashed] circle [radius=\beadsize];
	 	 	 \draw (5,1) [dashed] circle [radius=\beadsize];
	 	 	 \draw (6,0) [thick] circle [radius=\beadsize];
	 	
	 \node at (3,0.5) {$\implies$};	 	
	 \node at (0,1) {$q-1$};
	 \node at (1,1) {$q$};
	  \node at (5,1) {$q-1$};
	  \node at (6,1) {$q$};
	 \draw [->, >=latex, line width = 2] (0,0.7) -- (0,0.3);
	  \draw [->, >=latex, line width = 2] (1,0.7) -- (1,0.3);
	   \draw [->, >=latex, line width = 2] (5,0.7) -- (5,0.3);
	    \draw [->, >=latex, line width = 2] (6,0.7) -- (6,0.3);
	
	 \end{tikzpicture}

    \end{center}
    \medskip

    Now suppose that $\lambda^{(j-1)}$ and $\lambda^{(j)}$ are both non-empty.  Let
    \begin{align*}
    M_j &:= \max\{ x\equiv_e j \mid x \in \beta(\lambda)\}, \\
    m_j &:= \min\{ y\equiv_e j \mid y \notin \beta(\lambda) \}, \\
    q_j &:= M_j - \wt_{\lambda}(M_j)e.
    \end{align*}
    and similarly define $M_{j-1}, m_{j-1}, q_{j-1}$.  Then $ q_j >_e q_{j-1}  + 1$
    as $\lambda$ has more removable beads than addable beads on runner $j$.  Thus
    \begin{equation} \label{E:M}
    \frac{M_{j-1} - (q_j-1)}{e} < \frac{M_{j-1} - q_{j-1}}{e} = \wt_{\lambda}(M_{j-1}).
    \end{equation}
    Since $\lambda^{(j)}$ is non-empty, we have $q_j \geq_e m_j$.  It suffices to show that $M_{j-1} \leq_e m_j-1$.
    This can be easily seen if $q_j = m_j$: otherwise $M_{j-1} >_e m_j-1 = q_j-1$, and thus $M_{j-1} \geq_e q_j+e -1$, so that, from \eqref{E:M} above, $(M_{j-1};q_j+e-1) \in \bm(\lambda)$; but we also have $(M_j;q_j+e) \in \bm(\lambda)$, contradicting the bead configuration above (as $q_j = m_j \notin \beta(\lambda)$).  If $q_j >_e m_j$, then, for any $m_j <_e q \leq_e q_j$, we have $q \in \beta(\lambda)$ since $\lambda^{(j)}$ is a hook, and hence $(q;q) \in \bm(\lambda)$, so that $(M_{j-1};q-1) \notin \bm(\lambda)$ by the bead configuration above; in particular $M_{j-1} \ne q-1$.  In addition, we must have $M_{j-1} <_e q_j-1$, as otherwise $(M_{j-1}; q_j-1) \in \bm(\lambda)$ by \eqref{E:M}.  Thus $M_{j-1} \leq_e m_j-1$, and our proof is complete.

\item This is entirely analogous to part (2).
\end{enumerate}
\end{proof}

\section{Parallelotopes and $q$-decomposition numbers} \label{S:main}

In this section we state our main result, which gives a complete description of $q$-decomposition numbers $d_{\lambda\mu}(q)$ when
$\mu$ is $4$-increasing.

\begin{defi}\label{def:hookquotient}
Let $\lambda$ be a partition with $e$-quotient $(\lambda^{(0)},\dotsc, \lambda^{(e-1)})$.  We say that $\lambda$ is a {\em hook-quotient partition} if $\lambda^{(j)}$ is a hook for all $j \in [0,\,e)$.
\end{defi}
Note that if $\lambda^{(j)}$ is a nonempty hook $(x,1^y)$, then
$$\bm^{(j)}(\lambda) = \{ (b;b-ie) \mid i \in [0,\,x) \} \cup \{ (b-ie;b-ie) \mid i \in [x,x+y) \},$$ where $b = \max\{ c \in \beta(\lambda) \mid c \equiv_e j\}$.  In particular, every hook-quotient partition has the desirable property that its bead movements have distinct starting positions.

By Lemma \ref{lemma:1unram}(1), all $1$-increasing partitions are hook-quotient partitions.

Let $\{ \ssb_i \mid 1 \leq i \leq w \}$ be the standard basis of $\mathbb{Z}^w$, so that any $\mathbf{z} = (z_1,\dotsc,z_w) \in \mathbb{Z}^w$ may be written as $\mathbf{z} = \sum_{i=1}^w z_i \ssb_i$.

\begin{defi}\label{def:modified}
Let $\lambda$ be a hook-quotient partition of $e$-weight $w$.  We define the {\em{$\lambda$-modified basis vectors}}
$\sb_i^{\lambda} \in \mathbb{Z}^w$ as follows.  Let $\bm(\lambda) = \{(b_1;q_1) < \dotsb < (b_w;q_w)\}$.  Let $j$ be a residue class modulo $e$, and suppose that $\{ i \in [1,\, w] \mid q_i \equiv_e j \} = \{ i_1 < \dotsb < i_r \}$.  Let $l=\min\{s\mid b_{i_s}=b_{i_r}\}$, that is, $(b_{i_l}; q_{i_l})$ is the final bead movement
associated to the bottom bead on runner $j$ of the abacus display of $\lambda$.
Define
$$
\sb_{i_\gamma}^{\lambda}
:=
\begin{cases}
\ssb_{i_\gamma}-\ssb_{i_{\gamma+1}}, & \text{ if $\gamma<l$}; \\
\ssb_{i_\gamma}, & \text{ if $\gamma=l$}; \\
\ssb_{i_\gamma}-\ssb_{i_{\gamma-1}}, & \text{ if $\gamma>l$}. \\
\end{cases}
$$
We do this for all residue classes $j$ modulo $e$.
\end{defi}
It is clear that $\sb_1^\lambda,\dotsc,\sb_w^{\lambda}$ form another basis of $\BZ^w$.

\begin{example}
	Let $\lambda$ be a hook-quotient partition and suppose that
	$\lambda^{(j)}=(4,1,1)$. The bead movements on runner $j$ and the corresponding modified basis vectors are depicted as arrows in Figure~\ref{figure:modifiedbasisvectors}.
\begin{figure}[ht]\caption{} 

\def\globalscaling{.7}
\def\arrowshortening{0.25}
\def\shadebuffer{0.3}
\def\beadmovementwidth{2}
\def\abacusmovementwidth{2}
\def\beadoperationscaling{.85}
\def\labelsize{1}
\def\beadsize{.13}
\def\circledbeadmovementsize{.3}
\def\runnershiftforRHSdiagram{13}

$$
\begin{tikzpicture}
[y=-1cm, scale=\globalscaling], every node/.style={scale=.8*\globalscaling}]

\draw (0,-0.5) -- (0,10.5);

\foreach \x in
{(0,0), (0,1), (0,3), (0,4), (0,8)}
{\draw
[fill] \x circle (\beadsize)
;};

\foreach \x in
{(0,2), (0,5), (0,6),(0,7),(0,9),(0,10)}
{\node [scale=\labelsize*\globalscaling] at \x {\----};};

\foreach \x in
{(0,-1),(0,11) }
{\node [scale=\labelsize*\globalscaling] at \x {$\vdots$};};

\draw [red, line width=\beadmovementwidth*\globalscaling,->, >=latex] (0,5-\arrowshortening) -- (0,4+\arrowshortening) node [midway, right, black, scale=\labelsize*\globalscaling]{$\; \sb_{i_3}^\lambda =\ssb_{i_3}$};
\draw [green, line width=\beadmovementwidth*\globalscaling,->, >=latex] (0,6-\arrowshortening) -- (0,5+\arrowshortening) node [midway,right,black, scale=\labelsize*\globalscaling]{$\;\sb_{i_4}^\lambda = \ssb_{i_4}-\ssb_{i_3}$};
\draw [green, line width=\beadmovementwidth*\globalscaling,->, >=latex] (0,7-\arrowshortening) -- (0,6+\arrowshortening) node [midway,right,black, scale=\labelsize*\globalscaling]{$\;\sb_{i_5}^\lambda = \ssb_{i_5}-\ssb_{i_4}$};
\draw [green, line width=\beadmovementwidth*\globalscaling,->, >=latex] (0,8-\arrowshortening) -- (0,7+\arrowshortening) node [midway,right,black, scale=\labelsize*\globalscaling]{$\;\sb_{i_6}^\lambda = \ssb_{i_6}-\ssb_{i_5}$};

\draw [blue, line width=\beadmovementwidth*\globalscaling,->, >=latex] (0,4-\arrowshortening) -- (0,3+\arrowshortening) node [midway,right,black, scale=\labelsize*\globalscaling]{$\;\sb_{i_2}^\lambda = \ssb_{i_2}-\ssb_{i_3}$};
\draw [blue, line width=\beadmovementwidth*\globalscaling,->, >=latex] (0,3-\arrowshortening) -- (0,2+\arrowshortening) node [midway,right,black, scale=\labelsize*\globalscaling]{$\;\sb_{i_1}^\lambda = \ssb_{i_1}-\ssb_{i_2}$};
\end{tikzpicture}
$$

		\label{figure:modifiedbasisvectors}
\end{figure}

For each runner with nonzero $e$-weight, only one basis vector is unmodified; it corresponds to the final bead movement of the bottom bead of the runner. In the figure this bead movement is represented by a red arrow.
\end{example}

\begin{defi}
Let $\lambda$ be a hook-quotient partition of $e$-weight $w$.
\begin{enumerate}
\item For a subset $\Gamma$ of $[1,\, w]$, define
$$
\sb_{\Gamma}^{\lambda} = \sum_{a \in \Gamma} \sb_a^{\lambda} \in \BZ^w
$$
(where $\sb_{\emptyset}^{\lambda}$ is of course to be read as $0$.)
\item Define the {\em{parallelotopes}} $\para_0(\lambda)$ and $\para(\lambda)$ of $\lambda$, anchored at $0$ and $\z(\lambda)$ respectively, by
\begin{align*}
\para_0(\lambda) & =\left\{ \sb_\Gamma^\lambda \mid \Gamma \subseteq [1,\,w] \right\}, \\
\para(\lambda) &=\z(\lambda) +  \para_0(\lambda) = \left\{ \z(\lambda) + \sb_\Gamma^{\lambda} \mid \Gamma \subseteq [1,\,w] \right\};
\end{align*}
they are subsets of $\BZ^w$, each of cardinality $2^w$.
\item If $\z(\mu) = \z(\lambda) + \sb_{\Gamma}^{\lambda} \in \para(\lambda)$, define $\dist{\lambda}{\mu} = |\Gamma|$.
\end{enumerate}
\end{defi}

The following important observation is an easy consequence of the definition of the modified basis vectors $\sb_i^\lambda$.

\begin{lemma} \label{preshiftby3}
Let $\lambda$ be a hook-quotient partition of $e$-weight $w$, and let $(z_1,\dotsc, z_w) \in \para_0(\lambda)$.  Then $-2 \leq z_i \leq 1$ for all $i \in[1,\,w]$.
\end{lemma}

\begin{lemma}\label{lemma:shiftby3}
Let $\lambda$ be a hook-quotient partition of $e$-weight $w$, and let $\mathbf{z}_1,\mathbf{z}_2 \in \mathbb{Z}^w$.  Suppose that $\mathbf{z}_1 - \mathbf{z}_2 \in \Pi_0(\lambda)$.
Then for any $m\in\mathbb{Z}$,
\begin{enumerate}
\item if $\mathbf{z}_1$ is $m$-increasing, then $\mathbf{z}_2$ is $(m-3)$-increasing.
\item if $\mathbf{z}_2$ is $m$-increasing, then $\mathbf{z}_1$ is $(m-3)$-increasing.
\end{enumerate}
In particular, given a partition $\mu$ such that $\z(\mu)\in\para(\lambda)$, if one of $\lambda$ and $\mu$ is $m$-increasing, then the other is $(m-3)$-increasing.
\end{lemma}

\begin{proof}
Let $\mathbf{z}_1 = (a_1,\dotsc, a_w)$ and $\mathbf{z}_2 = (b_1,\dotsc, b_w)$.  Then $-2 \leq a_i - b_i \leq 1$ for all $i \in [1,\,w]$, by Lemma~\ref{preshiftby3}.  Thus $b_i -2 \leq a_i \leq b_i + 1$ for any $i \in[1,\,w]$, so that
$(b_i - b_{i-1}) -3 \leq a_i - a_{i-1} \leq (b_i - b_{i-1}) + 3$.  The lemma thus follows.
\end{proof}

The following is the main theorem of this paper.
\begin{thm}\label{thm:main}
Let $\lambda$ and $\mu$ be partitions of $e$-weight $w$ with the
same $e$-core. Suppose that $\mu$ is
$4$-increasing.  Then
$$
d_{\lambda\mu}(q)=
\begin{cases}
q^{\dist{\lambda}{\mu}},  
& \text{if $\lambda$ is a hook-quotient partition and $\z(\mu) \in \para(\lambda)$}; \\
0,  & \text{otherwise.}
\end{cases}
$$
\end{thm}

\begin{rem}\label{rem:proportion}
	Theorem \ref{thm:main} determines `most' of the columns of the $q$-decomposition matrix of blocks of $e$-weight $w$ as $e$ tends to infinity, in the sense that the proportion of partitions in blocks of $e$-weight $w$ that are $4$-increasing tends to $1$. In fact, the latter is true of
	$m$-increasing partitions, for any $m$,
	 as we now demonstrate.

	Let $B$ be a block of $e$-weight $w$.
	Given a sequence $c_1,\dotsc,c_l$ of positive integers summing to $w$, a sequence $j_1,\dotsc,j_l$ of integers such that $0\leq j_1<\dotsb<j_l\leq e-1$ and a sequence
	of partitions $\nu^1,\dotsc,\nu^l$ such that $|\nu^\gamma|=|c_\gamma|$, there is a unique partition $\lambda\in B$ whose $e$-quotient $(\lambda^{(0)},\dotsc, \lambda^{(e-1)})$ satisfies
	$$\lambda^{(j)}=
	\begin{cases}
	\nu^{\gamma}, & \text{if } j=j_\gamma \text{ for some } \gamma\in[1,l]; \\
	\emptyset, & \text{otherwise.}
	\end{cases}$$
	Moreover all partitions in $B$ arise uniquely in this way. Hence
		$$\left|B\right|=\sum_{(c_1,\dotsc,c_l)\vDash w} {e\choose l}\prod_{i=1}^{l} p(c_i),$$
		where the sum is taken over the set of sequences $(c_1,\dots,c_l)$ described above and $p(c_i)$ denotes the number of partitions of $c_i$.
    Each term in the sum satisfies $l\leq w$, with equality if and only if $(c_1,\dotsc,c_w)=(1,\dotsc,1)$. It follows that $|B|$ is a polynomial in $e$ of degree $w$ with leading coefficient $1/w!$. On the other hand, by Theorem~\ref{thm:goodlabels}(1), for any nonnegative integer $m$, the number of $m$-increasing partitions in $B$ is 
    $$\left|\inc{\good}{m}\right|={{e-(m-1)(w-1)}\choose w},$$
    a polynomial in $e$ with the same properties. Our claim follows. We remark that the claim is easily seen to be true also for
    the set of generic partitions in $B$ as defined in Section \ref{S:intro}.     	
\end{rem}

\begin{rem}
Suppose that $\mu$ is a $4$-increasing partition, $\lambda$ is a hook-quotient partition and $\z(\mu)\in\para(\lambda)$. Then $\lambda$ is $1$-increasing, by Lemma~\ref{lemma:shiftby3}. On the other hand, any $1$-increasing partition is a hook-quotient partition by Lemma~\ref{lemma:1unram}(1) as mentioned earlier.  Thus the phrase `$\lambda$ is a hook-quotient partition'
in Theorem~\ref{thm:main} may be replaced by `$\lambda$ is a $1$-increasing partition'.
\end{rem}

\begin{example}
Let $e=10$, and let
$\lambda=(16,8,1^{13})$ and $\mu=(17,7,2^4,1^5)$, two partitions in the same block of $e$-weight $w=3$.  Then $\z(\lambda)=(0,7,8)$ and $\z(\mu)=(1,5,9)$, so that $\lambda$ is $1$-increasing and $\mu$ is $4$-increasing.
Then the modified basis vectors associated to (the hook-quotient partition) $\lambda$ are $\sb_1^\lambda=\ssb_1-\ssb_2$,
$\sb_2^\lambda=\ssb_2$ and
$\sb_3^\lambda=\ssb_3-\ssb_2$. So $\z(\mu)=\z(\lambda)+
\sb_1^\lambda + \sb_3^\lambda \in \Pi(\lambda)$, and so
$d_{\lambda\mu}(q)=q^2$,  by Theorem~\ref{thm:main}.
Note that $\mu$ is $4$-increasing but $\lambda$ is not $2$-increasing.
\end{example}

\begin{example}[Rouquier blocks]\label{example:mainRouquier}
	Fix a Rouquier block $B$ of $e$-weight $w$. We now prove Theorem~\ref{thm:main} for partitions $\lambda$ and $\mu$ in $B$; this serves as the base case of our inductive argument to prove the result in general.  In fact we show that the formula for $d_{\lambda\mu}(q)$ holds in $B$ under the assumption, weaker than that in the statement of the theorem, that $\mu$ is a $0$-increasing partition.
	
	Let us recall the closed formulas for $q$-decomposition numbers in $B$ obtained by Leclerc and the second author in \cite{LM}. They prove that, given $\lambda,\mu\in B$,
	\begin{equation}\label{LM}
		d_{\lambda \mu}(q) =q^{\delta(\lambda,\mu)}
		\sum_{\substack{\alpha^0,\dotsc,\alpha^e \\
				\beta^0,\dotsc, \beta^{e-1}}}
		\prod_{0\leq j \leq e-1}c^{\mu^{(j)}}_{\alpha^j \beta^j}
		c^{\lambda^{(j)}}_{\beta^j (\alpha^{j+1})'},
	\end{equation}
	where $\lambda$ and $\mu$ have $e$-quotients $(\lambda^{(0)},\lambda^{(1)},\dotsc, \lambda^{(e-1)})$ and $(\mu^{(0)},\mu^{(1)},\dotsc,\mu^{(e-1)})$ respectively, the sum runs over
	all partitions $\alpha^0,\dotsc,\alpha^e, \beta^0,\dotsc, \beta^{e-1}$ satisfying
	\begin{equation}\label{alphabetasizes}
		|\alpha^i|=\sum_{j=0}^{i-1} (|\lambda^{(j)}|-|\mu^{(j)}|), \qquad
		|\beta^i|=|\mu^{(i)}|+\sum_{j=0}^{i-1} (|\mu^{(j)}|-|\lambda^{(j)}|)
	\end{equation}	
	and
	\begin{equation*}
		\delta(\lambda,\mu)=\sum_{j =0}^{e-2} (e-1-j)(|\lambda^{(j)}|-|\mu^{(j)}|).
	\end{equation*}
	Here $c^{\rho}_{\sigma \tau}$ is the Littlewood-Richardson coefficient associated to partitions $\rho$, $\sigma$ and $\tau$ satisfying $|\rho|=|\sigma|+|\tau|$.		
	
	Let $\mu$ be a $0$-increasing partition in $B$.
	We have seen in Example~\ref{example:Rouquier} that
	the $e$-quotient of
	$\mu$ is $
    ((1^{w_0}),\dotsc,(1^{w_{e-1}}))$, where $w_i$ is the number of times $i$ appears as a component of $\z(\mu)$.
	Let $\lambda$ be an arbitrary partition in $B$, with $e$-quotient $(\lambda^{(0)},\dotsc, \lambda^{(e-1)})$.  Let $\alpha^0,\dotsc,\alpha^e, \beta^0,\dotsc, \beta^{e-1}$
	be partitions satisfying (\ref{alphabetasizes}).  Thus $|\alpha^i| = a_i$ and $|\beta^i| = b_i$, where
$a_i = \sum_{j=0}^{i-1} (|\lambda^{(j)}| - w_j)$ and $b_i = w_i + \sum_{j=0}^{i-1} (w_j - |\lambda^{(j)}|)$.
	By the Littlewood-Richardson rule, we have
	$$c^{(1^{w_i})}_{\alpha^i \beta^i}=\begin{cases}
	1, & \text{if } \alpha^i = (1^{a_i})
	\text{ and } \beta^i = (1^{b_i})	\\
	0, & \text{otherwise},
	\end{cases}$$
	and
	$$c^{\lambda^{(i)}}_{(1^{b_i}),({a_{i+1}})}=\begin{cases}
	1, & \text{if } \lambda^{(i)} \in \{ (a_{i+1},1^{b_i}),
    (a_{i+1}+1,1^{b_i-1})	\}, \\
	0, & \text{otherwise.}
	\end{cases}$$
	So by \eqref{LM}, we see that $d_{\lambda\mu}(q) = 0$ unless $\lambda$ is a hook-quotient partition.
	Now assume that the $e$-quotient $(\lambda^{(0)},\dotsc, \lambda^{(e-1)})$ of $\lambda$ has the form $\lambda^{(i)}=(x_i,1^{y_i})$, where $x_i =0$ only if $y_i = 0$, for all $i\in [0,\,e)$.
	For each $i \in [0,\,e)$, put $c_i = x_i - a_{i+1}$.  We note for later use that
$c_i = x_i - \sum_{j=1}^i (|\lambda^{(j)}| - w_j)$, and hence
\begin{equation}\label{equation:rouquierparameters}
w_i = c_i - c_{i-1} + y_i +x_{i-1}
\end{equation}
for all $i\in [0,\,e)$, where $c_{-1} = x_{-1} = 0$.
	By \eqref{LM} we have
	\begin{equation} \label{E:rouquierdecomp}
		d_{\lambda\mu}(q)=\begin{cases}
			q^{\sum_{j=0}^{e-1} (x_j-c_j)} & \text{if } 0\leq c_j \leq \min(1,x_j) \text{ for all }j\in[0,\,e); \\
			0 & \text{otherwise.}
		\end{cases}
	\end{equation}

		We now turn to the description of the parallelotopes associated to hook-quotient partitions in $B$. %
	Let $x \in \mathbb{Z}^+$ and $y\in\mathbb{Z}_{\geq 0}$.
Define $z^{x,y}:=(0^y,1^{x-1},0)\in\mathbb{Z}^{x+y}$, and
	$$
	\sb_i^{x,y}
	:=
	\begin{cases}
	\ssb_{i}-\ssb_{i+1}, & \text{ if $i \in [1,\, y]$}, \\
	\ssb_{i}, & \text{ if $i=y+1$}, \\
	\ssb_{i}-\ssb_{i-1}, & \text{ if $i \in [y+2,\, x+y]$},
	\end{cases}
	$$
for each $i \in [1,\,x+y]$.
	Write $\sb_\Gamma^{x,y}=\sum_{i\in\Gamma}\sb_i^{x,y}$ for any subset $\Gamma\subseteq[1,\,x+y]$,
	and define
	$$\Pi_0^{x,y}:=\{\sb_\Gamma^{x,y}\mid \Gamma\subseteq[1,x+y]\} \subseteq \mathbb{Z}^{x+y}.$$
	It is easy to see that $\para^{x,y} := z^{x,y} + \Pi_0^{x,y}$ contains precisely two $0$-increasing elements,
	 namely
	$z^{x,y}+\sb_{[y+1,\,x+y]}^{x,y}=(0^y, 1^x)$ and $z^{x,y} +\sb_{[y+2,\,x+y]}^{x,y}=(0^{y+1}, 1^{x-1})$,
	corresponding to subsets of size $x$ and $x-1$ respectively. 		

	Consider again the hook-quotient partition $\lambda$ in $B$, with $e$-quotient given by $\lambda^{(i)}=(x_i, 1^{y_i})$ (where $x_i =0$ only if $y_i = 0$) for all $i \in [0,\,e)$. Then
	we have
\begin{align*}
\z(\lambda)&=(0^{x_0+y_0},\dotsc,(e-1)^{x_{e-1}+y_{e-1}}) + (z^{x_0,y_0},\dotsc,z^{x_{e-1},y_{e-1}}), \\
\Pi_0(\lambda)&= \Pi_0^{x_0,y_0}\times\dotsb\times\Pi_0^{x_{e-1},y_{e-1}},
\end{align*}
where $z^{x_i,y_i}$ and $\Pi_0^{x_i,y_i}$ are to be left out in the respective expressions of $\z(\lambda)$ and $\para_0(\lambda)$ if $x_i = 0$. Thus
	the $0$-increasing elements of $\Pi(\lambda)=\z(\lambda)+\Pi_0(\lambda)$ have the following form:
\begin{align*}
&(0^{x_0+y_0},\dotsc,(e-1)^{x_{e-1}+y_{e-1}}) + (\sb^{x_i,y_i}_{[y_0+1+c'_0,x_0+y_0]}, \dotsc, \sb^{x_{e-1},y_{e-1}}_{[y_{e-1}+1+c'_{e-1},x_{e+1}+y_{e-1}]}) \\
&= (0^{x_0+y_0},\dotsc,(e-1)^{x_{e-1}+y_{e-1}}) + ((0^{y_0+c'_0},1^{x_0-c'_0}), \dotsc, (0^{y_{e-1}+c'_{e-1}},1^{x_{e-1}-c'_{e-1}}))\\
&=(0^{y_0+c'_0}, 1^{x_0-c'_0+y_1+c'_1}, \ldots, (e-1)^{x_{e-2}-c'_{e-2}+y_{e-1}+c'_{e-1}}, e^{x_{e-1} - c'_{e-1}})
\end{align*}
where $c'_i \in [0, \min(x_i,1)]$ for all $i \in [0,\,e)$.
Hence, given a $0$-increasing partition $\mu$ in $B$ with $\z(\mu)=(0^{w_0},\dotsc,(e-1)^{w_{e-1}})$, we have
	$\z(\mu)\in\Pi(\lambda)$ if and only if there exist $c'_i \in [0, \min(x_i,1)]$ for all $i \in [0,\,e)$ such that $w_j = x_{j-1} - c'_{j-1} + y_j+c'_j$ for all $j \in [0,\,e)$ (where $x_{-1} = c'_{-1} = 0$, which will force $x_{e-1} = c'_{e-1}$), in which case $\dist{\lambda}{\mu}=\sum_{0\leq j\leq i}x_j-c'_j$.  We conclude from this, and \eqref{equation:rouquierparameters} and \eqref{E:rouquierdecomp}, that
$$
d_{\lambda\mu}(q) =
\begin{cases}
q^{\dist{\lambda}{\mu}}, &\text{if } \z(\mu) \in \para(\lambda); \\
0, &\text{otherwise.}
\end{cases}
$$
Hence Theorem \ref{thm:main} holds for Rouquier blocks, even under the less stringent hypothesis that $\mu$ is $0$-increasing.
\end{example}

\begin{rem} \label{R:genvertex} (Principal blocks)
	Martin and Russell \cite{MR} and Fayers and Martin \cite{FM} described part of the $\Ext^1$-quiver for principal blocks of symmetric groups.   Their {\em general vertices} \cite{MR}, {\em $p$-general vertices} \cite{FM} and {\em semi-general vertices} \cite{FM} are examples of $1$-increasing partitions. To be more precise, let $\lambda$ be a partition with empty $p$-core and $p$-weight $w$, and let $\z(\lambda)=(\z_1,\dotsc,\z_w)$. Then $\lambda$ is a general vertex if and only if $w\leq \z_1<\dotsb<\z_w\leq p-w-1$. It is
	a $p$-general vertex if and only if $w\leq \z_1<\dotsb<\z_w= p-w$. It is a semi-general vertex if and only if
	$w-1\leq \z_1 < \dotsb <\z_{w-1} \leq p-w$ and $\z_w=p-w+h$ for some $h\in[1,w-1]$, such that $\z_h\neq \z_{h-1}+1$ when $h>1$ while $\z_1\neq w-1$ when $h=1$.
	
	The main theorems of \cite{MR} and \cite{FM} take the following form. Let $\lambda$ be a general vertex, a $p$-general vertex or a semi-general vertex; in particular $\lambda$ is a $p$-regular partition with empty $p$-core. Then the set of ($p$-regular) partitions $\mu$ for which $\Ext^1(D^\lambda,D^\mu)\neq 0$ is described explicitly, and it is shown that for any such $\mu$, $\Ext^1(D^\lambda,D^\mu)$ is $1$-dimensional, and either $[S^\lambda:D^\mu]\neq 0$ or $[S^\mu:D^\lambda] \ne 0$; here $S^\lambda$ is the Specht module associated to $\lambda$ and $D^\lambda$ is its unique simple quotient. Any such $\mu$ is one of the partitions appearing in Example~\ref{E:principalblocks}, hence both a hook-quotient partition and $0$-increasing, and it is easy to check that either $\z(\mu) = \z(\lambda) + \sb^{\lambda}_i$ for some $i \in [1,\,w]$ or $\z(\lambda) = \z(\mu) + \sb^{\mu}_i$ for some $i \in [1,\,w]$ (cf.\ Theorem \ref{thm:extquiver}).
\end{rem}

\begin{rem}
Example \ref{example:mainRouquier} and Remark \ref{R:genvertex} (as well as Remark \ref{lastremark:Mullineux}, as we shall see) describe situations in which the formula for $d_{\lambda\mu}(q)$ in Theorem \ref{thm:main} holds even though $\mu$ is not $4$-increasing. In fact, we do not know of any example of a hook-quotient partition $\lambda$ and a $0$-increasing partition $\mu$ in the same block for which the formula is incorrect.
Indeed, for $e=2$, it is not difficult to compute the canonical basis vector $G(\mu)$ for $0$-increasing partitions $\mu$ from the description of such partitions given in Example \ref{example:e2} and verify directly that Theorem \ref{thm:main} holds for any hook-quotient partition $\lambda$ and $0$-increasing partition $\mu$. 
On the other hand, despite what Example \ref{example:mainRouquier} may suggest, it is possible to have $d_{\lambda\mu}(q)\neq 0$ for a non-hook-quotient partition $\lambda$ and a $0$-increasing partition $\mu$. For example, for $e=3$, the partition $\lambda=(5,3,2,1,1)$ has $e$-quotient $(\emptyset,\emptyset, (2,2),\emptyset)$ and the partition $\mu=(6,3,2,1)$ is $0$-increasing, with $\z(\mu)=(1,1,2,2)$, but $d_{\lambda\mu}(q)=q+q^3$. The formula can also fail for a hook-quotient partition $\lambda$ and a non-$0$-increasing partition $\mu$.  For example, in a Rouquier block of $2$-weight $2$ (i.e.\ $e = 2$), when $\mu = (5)$ (so that $\z(\mu) = (2,1)$),
we have $d_{(3,2),\mu}(q) = 0$ even though $\z(\mu) \in \Pi((3,2)) = \{ \z((3,2)) = (1,1), (1,2),(2,0),(2,1) \}$, while $d_{(3,1,1),\mu}(q) = q$ even though $\z(\mu) \notin \para((3,1,1)) = \{ \z((3,1,1)) = (1,2), (2,2), (1,3), (2,3)\}$.
\end{rem}

\begin{rem} \label{remark:addrunner} Let $\lambda$ be a hook-quotient partition in a block $B$ of $e$-weight $w$.
It is possible for $\para(\lambda)$ to contain a $4$-increasing element that is not of the form $\z(\mu)$ for any partition $\mu\in B$. It is easy to see that this is true
of $\mathbf{z}\in\para(\lambda)$ if and only if the last coordinate of $\mathbf{z}$ is $e$, if and only if the last coordinate of $\z(\lambda)$ is $e-1$ and
$\mathbf{z}=\z(\lambda)+\sb_\Gamma^\lambda$ for some $\Gamma\subseteq [1,w]$ containing $w$. It is nevertheless still possible to understand such $\mathbf{z}\in\para(\lambda)$ using the idea of runner addition/removal \cite{JM, F2}. To an arbitrary partition $\lambda$ in $B$, let $\lambda^+$ be the partition obtained by adding a `full' runner; the partitions thus obtained lie in a common block $B^+$ of $(e+1)$-weight $w$. More precisely, $\lambda^+$ can be defined as follows: for all $r\in\mathbb{Z}$ and $s\in[0,e]$, we have $r(e+1)+s \in \beta(\lambda^+)$ if and only if $s\neq e$ and $(r+\left|\lambda\right|)e+s\in\beta(\lambda)$, or $s=e$ and $r< \left|\lambda\right|e$.
 Then $\z(\lambda) = \z(\lambda^+)$, $\lambda$ is a hook-quotient partition if and only if $\lambda^+$ is,
 and $\sb^{\lambda}_i = \sb^{\lambda^+}_i$ for all $i \in [1,w]$, so that $\para(\lambda) = \para(\lambda^+)$.
Thus Theorem~\ref{thm:main} is consistent with Theorem 3.1 of \cite{F2}, which states that $d^e_{\lambda\mu}(q)=d^{e+1}_{\lambda^+ \mu^+}(q)$ for all $\lambda,\mu\in B$, and additionally, {\em every} $4$-increasing element of $\para(\lambda)$ may be interpreted as $\z(\nu)$ for a partition $\nu$ in $B^+$, giving a $q$-decomposition number $d_{\lambda^+\nu}(q)$ for $B^+$.
 \end{rem}

\section{Moving along a parallelotope} \label{S:movealong}

Before we tackle the proof of Theorem \ref{thm:main}, we explain in this section how to explicitly obtain partitions appearing in a parallelotope $\para(\lambda)$ in terms of the hook-quotient partition $\lambda$, side-stepping the difficulty of inverting the map $\z$.  Our results culminate in a description, in terms of certain `bead operations', which will be
essential for the proof of Theorem \ref{thm:main}.

 The aforementioned bead operations may be of independent interest. For example, they may be used to compute the image of the Mullineux map of $e$-regular partitions in some cases, and should be more efficient than the one generalising the original algorithm proposed by Mullineux \cite{M} when $w$ is small compared to $e$; this application is detailed in Section~\ref{section:Mullineux}.

We begin by defining the bead operations that we shall use.

\begin{defi}
Let $S \subseteq \mathbb{Z}$ and $x \in \mathbb{Z}$.  Define $b_S(x)$ and $\beadoperation{x}(S)$
as follows:
$$b_S(x)  := \min \{ a \mid a > x,\ a \notin S,\ a-e \in S\},\quad
\beadoperation{x}(S) := S \cup \{b_S(x)\} \setminus \{b_S(x)-e\}.$$
\end{defi}

On the abacus, $\beadoperation{x}(S)$ is obtained from $S$ by sliding one bead one position down its runner into an unoccupied position, where the unoccupied position is chosen to be as near to $x$ as possible while being greater than $x$.  We note that $b_S(x)$ as well as $\beadoperation{x}(S)$ is defined if and only if $x-e < \max(S)$.

We have two basic lemmas arising from the above definition.

	\begin{lemma}\label{L:bothbeadgapvariant}
		Let $S\subseteq\mathbb{Z}$, and let $x\in \mathbb{Z}$ with $x-e<\max(S)$, so that $a:=b_S(x)$ is defined.
		Then one of the following holds:
		\begin{enumerate}
			\item $t\in S \Leftrightarrow t-e\in S$ for all $t\in(x,\, a)$;
			\item there exists $s\in (x,\, a)$ such that
			$s\in S$, $s-e\notin S$ and $\z_S(x) > \z_S(s)$.
		\end{enumerate}
	In particular, if $S=\beta(\sigma)$ for some partition $\sigma$, and $\z_{\sigma}(x) \leq \z(b;q)$ for all $(b;q) \in \bm(\sigma)$ with $q \in (x,\, a)$, then $t \in S \Leftrightarrow  t-e \in S$ for all $t\in (x,\, a)$.
	\end{lemma}
	
	\begin{proof}
		If the first statement does not hold, then by the definition of $b_S(x) = a$, there exists $s\in (x,\, a)$ such that
		$s\in S$ and $s-e\notin S$. Choose $s$ to be the least such. Then $t\in S \iff t-e\in S$ for all $t\in(x,\, s)$, and it follows from Lemma~\ref{L:easy} that $\z_S(x)>\z_S(s)$.

For the last assertion, if statement (2) holds, then $s\in (x,\, a)$ satisfies $(s;s) \in \bm(\sigma)$ and $\z_{\sigma} (x) > \z_{\sigma}(s) = \z(s;s)$, contradicting the hypothesis.  Thus statement (1) must hold.
	\end{proof}

\begin{lemma} \label{L:+-}
Let $S := \beta(\sigma)$ for some partition $\sigma$, and let $x \in \mathbb{Z}$.  Suppose that $a:= b_S(x) \leq x+e$ and let $T := \beadoperation{x}(S) = S \cup\{a\} \setminus \{ a-e\}$.
\begin{enumerate}
\item If $\z_S(x+e) \leq \z(b;q)$ for all $(b;q) \in \bm(\sigma)$ with $q \in (x+e,\, a+e]$, then $b_T(x+e) \leq a +e $.
\item If $\z_S(x-e) \leq \z(b;q)$ for all $(b;q) \in \bm(\sigma)$ with $q \in (x-e,\, a-e]$, then $b_T(x-e) \leq a - e$.
\end{enumerate}
\end{lemma}

\begin{proof} \hfill
\begin{enumerate}
\item Since $a\in T$, if $a + e \notin T$, then clearly by definition $b_T(x+e) \leq a+e$.  On the other hand, if $a+ e \in T$, then $a + e \in S$, and $a  \notin S$, so that $(a+e;a+e) \in \bm(\sigma)$.  Thus $\z_S(x+e) \leq \z(a+e;a+e) = \z_S(a + e)$, so that there exists $s+e \in (x+e,\, a+e)$ such that $s \in S$ and $s+e \notin S$ by Lemma \ref{L:easy}(1).  As $s \in T$ and $s+e \notin T$, we have $b_T(x+e) \leq s+e < a+e$.

\item Since $a-e \notin T$, if $a-2e \in T$ then $b_T(x-e) \leq a-e$ by definition.  On the other hand, if $a -2e \notin T$, then $a -2e \notin S$, while $a-e \in S$, so that $(a-e;a-e) \in \bm(\sigma)$.  Thus $\z_S(x-e) \leq \z(a-e;a-e) = \z_S(a-e)$, so that there exists $s-e \in (x-e,\, a-e)$ such that $s-2e \in S$ and $s-e \notin S$ by Lemma \ref{L:easy}(1).  As $s-2e \in T$ and $s-e \notin T$, we have $b_T(x-e) \leq s-e < a-e$.
\end{enumerate}
\end{proof}

The next technical lemma is an important one, laying the foundation for us to describe $\mu$ when $\z(\mu)$ is a $\lambda$-modified basis vector away from $\z(\lambda)$.

\begin{lemma} \label{L:omnibus}
Let $S = \beta(\sigma)$ for some partition $\sigma$ and let $x\in \BZ$, $k\in \BZ^+$, $l \in \mathbb{Z}_{\geq0}$.
Suppose that:
\begin{itemize}
	\item[(I)]
	\begin{enumerate}
		\item[(A)] $\z_S(x) < \z_S(x+e)$, or
		\item[(B)] $l=0$, $x \notin S$, $x-e \in S$ and $x-e < \max(S)$; 
	\end{enumerate}
	\item[(II)] $\z_S(x) \leq \z(b;q)$ whenever $(b;q) \in \bm(\sigma)$ and $q>x$, with the inequality being strict if $q \in (x,\,x+e]$;
	\item[(III)] $\z_S(x+ie) \leq \z(b;q)$ for all $(b;q) \in \bm(\sigma)$ with $q \in (x+ie,\, x+(i+1)e]$ and $i \in (-k,\, l]$.
\end{itemize}
Define $S_i$, for $i \in [ 0,\, k+l]$, recursively as follows: $S_0 := S$ and
$S_{i} := \beadoperation{x+f(i)e}(S_{i-1})$, where
$f: [1,\, k+l] \to (-k,\, l]$ is the bijection defined by
$$
f(i) :=
\begin{cases}
1-i, &\text{if } i \in [1,\, k]; \\
i-k, &\text{if } i \in (k,\, k+l].
\end{cases}
$$
For each $i \in [1,\, k+l]$, let $d_i := b_{S_{i-1}}(x+f(i)e)$, so that $S_{i} = S_{i-1} \cup \{ d_i \} \setminus \{d_i-e\}$.
\begin{enumerate}
\item Then $d_i \in (x+f(i)e,\, x+(f(i)+1)e]$ for all $i$, except possibly when $i=1$ and condition (IB) holds but not condition (IA).  Also,
\begin{alignat*}{2}
d_{i+1} &\leq d_i-e \qquad &&\text{for all } i \in [1,\, k); \\
d_{k+1} &\leq d_1+e; &&\\
d_{i+1} &\leq d_i +e \qquad &&\text{for all } i \in (k,\, k+l).
\end{alignat*}

\item If $s \in (x+f(i)e,\, d_i)$, then $s \in S$ if and only if $s-e \in S$.

\item Let $(b;q) \in \bm(\sigma)$.  Then
$$
q \notin \bigcup_{i=1}^k (x+(f(i)-1)e,\, d_i-e) \cup \bigcup_{i = k+1}^{k+l} (x+(f(i)+1)e,\, d_i +e)
\cup (x,\, d_1) \cup (x+e,\, d_1+e)
$$

\item For $i \in \{1\} \cup (k,\,k+l]$, write $d_{i_{\max}} = \max\{d_j \mid d_j \geq_e d_i \}$.  For $y \in \mathbb{Z}$, define
$$
g(y) := \begin{cases}
d_i, &\text{if } y = d_i-e \text{ and } i \in [2,k]; \\
d_{i_{\max}}, &\text{if } y = d_i -e \text{ and } i \notin [2,k];\\
y, &\text{otherwise.}
\end{cases}
$$
Write $\nu$ for the partition such that $\beta(\nu) = S_{k+l}$.  Then the map $\tilde{g} : \bm(\sigma) \to \bm(\nu)$ defined by $(b;q) \mapsto (g(b);q)$ is well-defined, injective and order-preserving, and
$$
\bm(\nu) \setminus \tilde{g}(\bm(\sigma)) = \{ (g(d_i-e);d_i) \mid i\in [1,\, k+l] \}.$$

\item Keep the notations in (4).  Let $(b;q) \in \bm(\sigma)$.
  Then
$$
\z_{\nu}(\tilde{g}(b;q)) =
\begin{cases}
\z_{\sigma}(b;q) +1, &\text{if } q \in 
(x-ke,\, x-(k-1)e]; \\
\z_{\sigma}(b;q) -1, &\text{if } q \in 
(x+le,\, x+(l+1)e]; \\
\z_{\sigma}(b;q), &\text{otherwise.}
\end{cases}
$$
In addition,
$$ \z_{\nu}( (g(d_i-e);d_i)) =
\begin{cases}
\z_{S}(x+f(i)e) - 1, &\text{if $k < i = k+l$}, \\
\z_{S}(x+f(i)e) + 1, &\text{if $i = 1 \leq l$}, \\
\z_{S}(x +f(i)e), & \text{otherwise},
\end{cases}
$$
for all $(g(d_i-e);d_i) \in \bm(\nu) \setminus \tilde{g}(\bm(\sigma))$.
\end{enumerate}
\end{lemma}

In Figure~\ref{fig:omnibuslemma}, Lemma \ref{L:omnibus} is illustrated for $k=4$ and $l=5$, with bead operations pictured on the abacus of $\sigma$, on the left, and the `extra' bead movements $(g(d_i-e);d_i)$ of $\nu$ shown on its abacus, on the right.

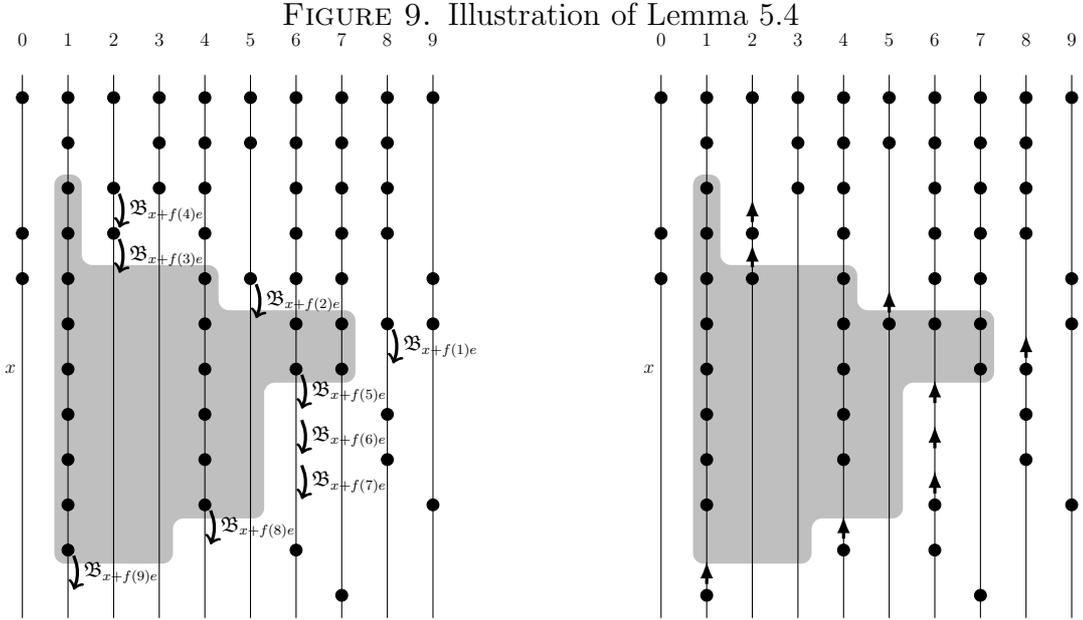
\begin{figure}[h]\caption{Illustration of Lemma~\ref{L:omnibus}}

\def\globalscaling{0.6}
\def\arrowshortening{0.25}
\def\shadebuffer{0.3}
\def\beadmovementwidth{2}
\def\abacusmovementwidth{2}
\def\beadoperationscaling{1.1}
\def\labelsize{1}
\def\beadsize{.13}
\def\circledbeadmovementsize{.3}
\def\runnershiftforRHSdiagram{14}


\begin{tikzpicture}
[y=-1cm, scale=\globalscaling], every node/.style={scale=.8*\globalscaling}]

\fill
[lightgray, rounded corners]
(1-\shadebuffer,-4-\shadebuffer)
-- (1+\shadebuffer,-4-\shadebuffer)
-- (1+\shadebuffer,-2-\shadebuffer)
-- (4+\shadebuffer,-2-\shadebuffer)
-- (4+\shadebuffer,-1-\shadebuffer)
-- (7+\shadebuffer,-1-\shadebuffer)
-- (7+\shadebuffer,0+\shadebuffer)
-- (5+\shadebuffer,0+\shadebuffer)
-- (5+\shadebuffer,3+\shadebuffer)
-- (3+\shadebuffer,3+\shadebuffer)
-- (3+\shadebuffer,4+\shadebuffer)
-- (1-\shadebuffer,4+\shadebuffer)
-- cycle;

\foreach \x in {0,...,9}
{
\node [scale=\labelsize*\globalscaling, above] at (\x, -7) {$\x$};
\draw (\x,-6.5) -- (\x,5.5);};
\node  [scale=\labelsize*\globalscaling,left] at (0,0) {$x$};
\foreach \x in
{(0,-6), (0,-3), (0,-2),
	(1,-6), (1,-5), (1,-4), (1,-3), (1,-2), (1,-1), (1,0), (1,1), (1,2), (1,3), (1,4),
 (2,-6),         (2,-4), (2,-3),	
 (3,-6), (3,-5), (3,-4),
 (4,-6), (4,-5), (4,-4), (4,-3), (4,-2), (4,-1), (4,0), (4,1), (4,2), (4,3),
 (5,-6), (5,-5),                 (5,-2),
 (6,-6), (6,-5), (6,-4), (6,-3), (6,-2), (6,-1), (6,0),                       (6,4),
 (7,-6), (7,-5), (7,-4), (7,-3), (7,-2), (7,-1), (7,0),                (7,5),
 (8,-6), (8,-5), (8,-4), (8,-3),         (8,-1),        (8,1), (8,2),
 (9,-6), (9,-1), (9,-2), (9,3)
 }
{\draw
[fill] \x circle (\beadsize)
;};


\foreach \x/\y/\f in
{2/-3/4, 2/-2/3, 5/-1/2, 8/0/1, 6/1/5, 6/2/6, 6/3/7, 4/4/8, 1/5/9}
{
	\draw [->, line width = \abacusmovementwidth*\globalscaling, rounded corners] (\x+\arrowshortening/2, \y-1+\arrowshortening/2) -- (\x+.25,\y-1+.5)
	node
	 [scale=\beadoperationscaling*\globalscaling, right,xshift=-2] {$\beadoperation{x+f(\f)e}$}
	-- (\x+\arrowshortening/2, \y-\arrowshortening/2);
};

\fill
[lightgray, rounded corners]
(1-\shadebuffer+\runnershiftforRHSdiagram,-4-\shadebuffer)
-- (1+\shadebuffer+\runnershiftforRHSdiagram,-4-\shadebuffer)
-- (1+\shadebuffer+\runnershiftforRHSdiagram,-2-\shadebuffer)
-- (4+\shadebuffer+\runnershiftforRHSdiagram,-2-\shadebuffer)
-- (4+\shadebuffer+\runnershiftforRHSdiagram,-1-\shadebuffer)
-- (7+\shadebuffer+\runnershiftforRHSdiagram,-1-\shadebuffer)
-- (7+\shadebuffer+\runnershiftforRHSdiagram,0+\shadebuffer)
-- (5+\shadebuffer+\runnershiftforRHSdiagram,0+\shadebuffer)
-- (5+\shadebuffer+\runnershiftforRHSdiagram,3+\shadebuffer)
-- (3+\shadebuffer+\runnershiftforRHSdiagram,3+\shadebuffer)
-- (3+\shadebuffer+\runnershiftforRHSdiagram,4+\shadebuffer)
-- (1-\shadebuffer+\runnershiftforRHSdiagram,4+\shadebuffer)
-- cycle;

\foreach \x in {0,...,9}
{
	\node [scale=\labelsize*\globalscaling, above] at (\x+\runnershiftforRHSdiagram, -7) {$\x$};
	\draw (\x+\runnershiftforRHSdiagram,-6.5) -- (\x+\runnershiftforRHSdiagram,5.5);};
\node  [scale=\labelsize*\globalscaling,left] at (0+\runnershiftforRHSdiagram,0) {$x$};

\foreach \x in
{(0,-6), (0,-3), (0,-2),
	(1,-6), (1,-5), (1,-4), (1,-3), (1,-2), (1,-1), (1,0), (1,1), (1,2), (1,3),       (1,5),
	(2,-6),                 (2,-3), (2,-2),	
	(3,-6), (3,-5), (3,-4),
	(4,-6), (4,-5), (4,-4), (4,-3), (4,-2), (4,-1), (4,0), (4,1), (4,2),        (4,4),
	(5,-6), (5,-5),                         (5,-1),
	(6,-6), (6,-5), (6,-4), (6,-3), (6,-2), (6,-1),                      (6,3), (6,4),
	(7,-6), (7,-5), (7,-4), (7,-3), (7,-2), (7,-1), (7,0),                (7,5),
	(8,-6), (8,-5), (8,-4), (8,-3),                 (8,0), (8,1), (8,2),
	(9,-6), (9,-1), (9,-2), (9,3)
}
{\draw
	[fill, xshift=\runnershiftforRHSdiagram cm] \x circle (\beadsize)
	;};

\foreach \x/\y in
{2/-3, 2/-2, 5/-1, 8/0, 6/1, 6/2, 6/3, 4/4, 1/5}
{
\draw [line width=\beadmovementwidth*\globalscaling,->, >=latex] (\x+\runnershiftforRHSdiagram,\y-\arrowshortening) -- (\x+\runnershiftforRHSdiagram,\y-1+\arrowshortening);
};

\end{tikzpicture}

\label{fig:omnibuslemma}
	\end{figure}

\begin{proof} \hfill
\begin{enumerate}
\item Note that condition (IA) ensures that $d_1 \leq x+e$ by Lemma \ref{L:easy}, and the remaining assertions then follow from Condition (III) and Lemma \ref{L:+-}.  If condition (IA) does not hold, then condition (IB) holds and guarantees that $d_1$ is defined and $d_2 \leq x$ when $k \geq 2$; the remaining assertions again follow from Condition (III) and Lemma \ref{L:+-}(2).

\item This follows from part (1), conditions (II) and (III), and Lemma \ref{L:bothbeadgapvariant}.

\item Suppose first $\sigma$ has a bead movement starting at a position $q$ with $q \in (x,\,d_1)$.  Let $q' = \min\{ y \mid  x< y \leq_e q \}$.  Then $q' \in (x,x+e]$.  Furthermore, by part (2), $q \in S$ if and only if $q' \in S$, so that $(b';q') \in \bm(\sigma)$ for some $b' \in S$.  By condition (II), we have $\z_S(x) < \z(b';q')$, so that there exists $s \in (x,\, q']$ such that $s \notin S$ and $s-e \in S$ by Lemma \ref{L:easy}(1).  Thus, $x < s \leq q' \leq q_1 < d_1$, contradicting the definition of $d_1$.

    We now prove by induction that $\sigma$ has no bead movement whose starting position lies in the open interval $(x+f(i)e,d_i)$ for all $i \in [1,\,k+l]$, for which we've just seen the base case of $i=1$.  For $i \in [2,\,k]$, we have $f(i) = 1-i$.  By part (2), for any $b' \in (x+(2-i)e,d_{i-1})$, we have $b' \in S$ if and only if $b' -e \in S$. This implies that $\sigma$ has a bead movement starting at $b'-e$ only if it has a bead movement starting at $b'$.  By induction hypothesis, $\sigma$ has no bead movement whose starting position lies in $(x+(2-i)e,d_{i-1})$, and hence $\sigma$ has no bead movement whose starting position lies in $(x+(1-i)e,d_{i-1}-e)$.  Since $d_{i-1}-e \geq d_i$ by part (1), we conclude that $\sigma$ has no bead movement whose starting position lies in $(x+f(i)e,d_{i})$.  An analogous argument also applies for $i \in (k,\,k+l]$.

    Finally we show that $\sigma$ has no bead movement whose starting position lies in the open interval $(x+(f(i)\pm 1)e,d_i \pm e)$ for all $i \in [1,\,k+l]$.  Once again, this holds because by part (2), any bead movement in this interval will imply a bead movement in $(x+f(i)e,\, d_i)$, contradicting what we have shown above.

    Thus, for all $(b;q) \in \bm(\sigma)$, we have
    $$
    q \notin \bigcup_{i=1}^{k+l} \bigcup_{\varepsilon \in \{0, \pm 1\}} (x+(f(i)+\varepsilon)e,\, d_i + \varepsilon e).
    $$
    Part (3) now follows, using part (1).

\item This is straightforward, and may be proved by, say, induction on $k+l$ and using parts (1) and (3).

\item By part (1), we have $$S_{k+l} = ((S \setminus \{ d_i-e \mid i \in [1,\, k] \}) \cup \{ d_i \mid i \in [1,\, k +l]\}) \setminus \{ d_i-e \mid i \in (k,\, k+l]\}.$$
Thus $\z_S(q) = \z_{S_{k+l}}(q)$ if $q < \min_i\{d_i-e\}$ ($= d_k-e$) or $q-e \geq \max_i\{d_i\}$ $\left(= \begin{cases} d_{k+l} &\text{if } l>0\\ d_1 &\text{if } l=0 \end{cases}\right)$. 

For $q \in [d_k-e,\, x-(k-1)e]$, we have $S_{k+l} \cap (q-e,\, q] = (S \cap (q-e,\, q]) \setminus \{d_k-e\}$ by part (1).  Thus $\z_{S_{k+l}}(q) = \z_S(q) + 1$.

For $q \in (x-ie,\, x-(i-1)e]$ for some $i \in [1,\, k)$, we have, since $f(i) = 1-i$, $q \in [d_i-e,\, x-(i-1)e]$ by part (3), so that $d_i \leq  x+ (f(i)+1)e$ even when $i=1$ and Condition (IA) does not hold.  Consequently, $S_{k+l} \cap (q-e,\, q] = ((S \cap (q-e,\, q])\setminus \{ d_i -e \}) \cup \{ d_{i+1} \} $ by part (1), and hence $\z_{S_{k+l}}(q) = \z_S(q)$.

For $q \in (x,\, x+e]$, we have $q \in [d_1,\, x+e]$ by part (3), and hence $d_1 \leq x+e$.  Thus $S_{k+l} \cap (q-e,\, q] = (S \cap (q-e,\, q]) \cup \{d_1\} \setminus \{d_{k+1} -e \}$ by part (1). So, $\z_{S_{k+l}}(q) = \z_S(q) - \delta_{l0}$.

For $q \in (x+e,\, x+2e]$, we have $q \in[d_1+e,\, x+2e]$ by part (3), and hence $d_1 \leq x+e$.  Thus $S_{k+l} \cap (q-e,\, q] = (S \cap (q-e,q]) \cup \{d_{k+1}\} \setminus \{d_{k+2}-e \}$ by part (1), and so, $\z_{S_{k+l}}(q) = \z_S(q) - \delta_{l1}$.

For $q \in (x + je,\, x+(j+1)e]$ for some $j \in [2,\, l]$, we have, since $f(k+j-1) = j-1$, $q \in [d_{k+j-1}+e,\, x+(j+1)e]$ by part (3), so that $S_{k+l} \cap (q-e,\, q] = (S  \cap (q-e,\, q]) \cup \{ d_{k+j} \} \setminus \{ d_{k+j+1}-e \}$ by part (1).  Thus $\z_{S_{k+l}}(q) = \z_S(q) - \delta_{jl}$.

If $q > x+(l+1)e$, then $q > d_1 + e$ if $l=0$ and $q> d_{k+l} + e$ if $l>0$ by part (3).  Thus $q > \max_i\{d_i\} + e$.  This proves the first assertion completely.
\medskip

For the second assertion,
let $h: (d_i-e,\, d_i] \to (x+(f(i)-1)e,\, x+f(i)e]$ be the bijection satisfying $h(y) \equiv_e y$ for all $y \in (d_i-e,\, d_i]$.  From part (2), we see that $y \in S \cap (d_i-e,\, d_i)$ if and only if $h(y) \in S$.  Thus,
\begin{align*}
\z_S(d_i) = |(d_i-e,\, d_i] \setminus S| &= |\{ h(y) \mid y \in (d_i-e,\,d_i) \setminus S \}| + \I_{d_i \notin S} \\
&= \z_S(x+f(i)e) - \I_{h(d_i) \notin S} + \I_{d_i \notin S}.
\end{align*}

We consider three different cases:
\begin{description}
\item[{$i \in [2,\, k]$}]  In this case, $h(d_i) = d_i-e \in S$ and $d_i \in S_{k+l}$, so that
$$(d_i-e,\, d_i] \setminus S = ((d_i-e,\, d_i] \setminus S_{k+l}) \cup (\{ d_i\} \setminus S) $$
by part (1).  Thus
$$
\z_S(x+f(i)e) + \I_{d_i \notin S} = \z_S(d_i) = \z_{S_{k+l}}(d_i) + \I_{d_i \notin S},
$$
giving $\z_S(x+f(i)e) = \z_{S_{k+l}}(d_i).$

\item[{$i \in (k,\, k+l]$}]  In this case, $h(d_i) = d_i-e \notin S_{k+l}$ and $d_i \notin S$.  Also
$$(d_i-e,\, d_i] \setminus S_{k+l} = ((d_i-e,\, d_i] \setminus (S \cup \{ d_{i-1}, d_i \} \setminus \{ d_{i+1}-e\} )$$
by part (1), where $d_{i-1}$ is to be read as $d_1$ when $i= k+1$.
Thus,
$$
\z_S(x+f(i)e)- \I_{d_i-e \notin S} + 1 = \z_S(d_i) = \z_{S_{k+l}}(d_i) + \I_{d_{i-1} > d_i-e} + \delta_{i,k+l}.
$$
Since $d_{i-1} > d_i-e$ if and only if $d_i-e \in S$, we get
$$
\z_S(x+f(i)e)= \z_{S_{k+l}}(d_i) + \delta_{i,k+l}.
$$
\item[$i=1$] Clearly, $d_1 \notin S$.  Also since $d_1-e \in S$, we have $h(d_1) \in S$ by part (2).  If $d_1-e > x$ (only if $l=0$), then $(d_1-e,\,d_1] \setminus S_{k+l} = (d_1-e,\,d_1] \setminus (S \cup \{d_1\})$.  Thus,
$$
\z_S(x) +1 = \z_S(d_1) = \z_{S_{k+l}}(d_1) +1.$$
If $d_1 -e \leq x$, then $ (d_1-e,\,d_1] \setminus S_{k+l} = (d_1-e,\,d_1] \setminus (S \cup \{d_1\} \setminus \{d_{k+1}-e\})$. Thus,
$$
\z_S(x) +1 = \z_S(d_1) = \z_{S_{k+l}}(d_1) + \delta_{l0}.$$
This yields $\z_{S_{k+l}}(d_1) = \z_S(x) + 1 - \delta_{l0}$ always, regardless of the value of $d_1$.
\end{description}
\end{enumerate}
\end{proof}

\begin{defi}
For $x,k,l \in \mathbb{Z}$ with $k >0$ and $l\geq0$, define
$$
\beadoperation{x}^{k,l} := (\beadoperation{x+le} \circ \beadoperation{x+(l-1)e} \circ \dotsb \circ \beadoperation{x+e}) \circ (\beadoperation{x-(k-1)e} \circ \beadoperation{x-(k-2)e} \circ \dotsb \circ\beadoperation{x}).$$
\end{defi}

Using the notation just introduced, $S_{k+l}$ in Lemma \ref{L:omnibus} can be written as $\beadoperation{x}^{k,l}(S)$.

\begin{defi}
Let $\lambda$ be a hook-quotient partition with $e$-weight $w$, with $\bm(\lambda) = \{ (b_1;q_1) < (b_2;q_2) <\dotsb < (b_w;q_w)\}$.  We define a partial order $\succeq_{\lambda}$ on $[1,\, w]$ as follows:  $i \succeq_{\lambda} j$ if and only if $b_i \equiv_e b_j$, and either $i\geq j \geq m$ or $i \leq j \leq m$, where $(b_m;q_m)$ is the final bead movement of the bottom bead on the runner.

In addition, we write $i \succ_{\lambda} j$ for $i \succeq_{\lambda} j$ and $i \ne j$.
\end{defi}

\begin{rem}
Let $\lambda$ be a hook-quotient partition.  If $\bm(\lambda) = \{ (b_1;q_1) < (b_2;q_2) <\dotsb < (b_w;q_w)\}$, and $\{ i \in [1,\, w] \mid q_i \equiv_e a \} = \{i_1 < i_2< \dotsb < i_v\}$, with $(b_{i_m};q_{i_m})$ being the final bead movement of the bottom bead on runner $a$, then
\begin{gather*}
i_1 \succeq_{\lambda}i_2 \succeq_{\lambda} \dotsb \succeq_{\lambda}i_m, \\
i_v \succeq_{\lambda}i_{v-1} \succeq_{\lambda} \dotsb \succeq_{\lambda}i_m.
\end{gather*}
In addition, $i_j \not\succeq_{\lambda}r$ and $r \not\succeq_{\lambda}i_j$ for all other $r$ and any $j$.
\end{rem}

Using Lemma \ref{L:omnibus}, we can now describe how to obtain the partition $\mu$ from the hook-quotient partition $\lambda$ when $\z(\mu)$ is $0$-increasing and one $\lambda$-modified basis vector away from $\z(\lambda)$.

\begin{prop} \label{P:movealongdescription}
Let $\lambda$ be a hook-quotient partition with $e$-weight $w$.  Let $r \in [1,\,w]$ and suppose that $\z(\lambda) + \sb^{\lambda}_{r} = \z(\mu)$ for a $0$-increasing partition $\mu$ lying in the same block as $\lambda$.
For each $i \in [1,\,w]$, let $(b_i;q_i)$ (resp.\ $(b'_i;q'_i)$) be the $i$-th bead movement of $\lambda$ (resp.\ $\mu$).  Let $g_r = \max\{ x <_e q_r \mid x \notin \beta(\lambda)\}$,
and let $\sigma$ be the partition obtained from $\lambda$ by moving the bead at $b_r$ to $g_r$, i.e. $$\beta(\sigma) = \beta(\lambda) \cup \{ g_r \} \setminus \{b_r\}.$$
Then $\beta(\mu) = \beadoperation{q_r}^{k,l}(\beta(\sigma))$, where $k = \frac{q_r-g_r}{e}$ and $l = \frac{b_r-q_r}{e}$.

In particular,
\begin{enumerate}
\item $q_t = q'_t$ if and only if $t \not\succeq_{\lambda} r$, in which case $\sb^{\lambda}_t = \sb^{\mu}_t$ if $\mu$ is a hook-quotient partition;
\item if $s \succeq_{\lambda} r$, then
\begin{enumerate}
\item $q'_s \in (q_s,\, q_s + e]$, except possibly when $r=s$ and $(b_r;q_r)=(q_r;q_r)$ is the initial bead movement of the bottom bead of that runner of $\lambda$;
 \item $\alpha \in \beta(\sigma) \iff \alpha-e \in \beta(\sigma)$ for all $\alpha \in (q_s,\, q_s')$;
\end{enumerate}
\item if $t \succeq_{\lambda} s \succeq_{\lambda} r$, then $q'_t - q_t \leq q'_s - q_s$;
\item if $q'_s = q_s +e$, then $(b_r;q_r) = (q_r;q_r)$ is the only bead movement of the bottom bead of that runner of $\lambda$, and $q_r \not\equiv_e q'_r > q_r +e $;
\item writing $\mathbf{G}_\lambda^\mu = \bigcup_{s \succeq_{\lambda} r} \left( (q_s,\, q'_s) \cup (q_s-e,\, q'_s-e) \right)$, we have
    \begin{enumerate}
      \item $x \in \beta(\sigma) \cap \mathbf{G}_\lambda^\mu$ if and only if $y \in \beta(\sigma)$ for all $y \in \mathbf{G}_\lambda^\mu$ with $y \equiv_e x$;
      \item $q_t,q_t-e \notin \mathbf{G}_\lambda^\mu$ for all $t \not\succeq_{\lambda} r$.
    \end{enumerate}
\end{enumerate}
\end{prop}

\begin{proof}
Suppose that $(b_r;q_r)$ lies in runner $a$.
Note first that
\begin{align*}
\{ q_s \mid s \succeq_{\lambda} r,\ q_s \leq q_r \} &= \{ q_r - i e \mid i \in [0,\,k) \}, \\
\{ q_s \mid s \succeq_{\lambda} r,\ q_s > q_r \} &= \{ q_r + i e \mid i \in [1,\, l] \},
\end{align*}
so that $\{q_r + ie \mid i \in (-k,\,l] \} = \{ q_s \mid s \succeq_\lambda r\}$.
We consider the following cases separately:
\begin{description}
\item[Case 1] $(b_r;q_r)$ is a non-final bead movement of the bottom bead of runner $a$.
\item[Case 2] $(b_r;q_r)$ is a (and the unique) bead movement of a non-bottom bead of runner $a$.
\item[Case 3] $(b_r;q_r)$ is the final bead movement of the bottom bead of runner $a$.
\end{description}
All cases are proved similarly by applying Lemma \ref{L:omnibus}.
We shall only show the details for Case 3.

In this case,
$\bm(\lambda) = \bm(\sigma) \cup \bm^{(a)}(\lambda)$ (disjoint union).  In addition, for $x \in \mathbb{Z}$, we have
$$\z_{\sigma}(x) = \z_{\lambda}(x) + \I_{x \in [b_r,\,b_r+e)} - \I_{x \in [g_r,\, g_r + e)}.$$
We show that the conditions in Lemma \ref{L:omnibus} hold with $x = q_r$ (and $k =\frac{q_r - g_r}{e}$ and $l = \frac{b_r - q_r}{e}$).

Since $\z(\mu) = \z(\lambda) + \sb^{\lambda}_{r} = \z(\lambda) + \ssb_{r}$ and $\mu$ is $0$-increasing, we have
$$
\z_{\lambda}(q_i) + \delta_{ir} \leq \z_{\lambda}(q_j) + \delta_{jr}
$$
whenever $i < j$ (equivalently, $q_i < q_j$).  Thus, when in addition $i \succeq_{\lambda} r$, we have
\begin{align*}
\z_{\sigma} (q_i) = \z_{\lambda}(q_i) + \delta_{q_i,b_r}
&\leq \z_{\lambda}(q_j) + \delta_{jr} - \delta_{ir} + \delta_{q_i,b_r}\\
&= \z_{\sigma} (q_j) + \delta_{jr} - \delta_{ir} + \delta_{q_i,b_r} - \I_{q_j \in [b_r,\, b_r+e)}.
\end{align*}
In particular, $\z_{\sigma}(q_i) > \z_{\sigma}(q_j)$ only if $j=r$ or $q_i = b_r$.  Since $\bm^{(a)}(\sigma) = \emptyset$, we see that condition (III) holds.
Putting $i = r$, we see that condition (II) holds, and condition (IA) also holds when $l >0$ (so that $q_r+e = q_s$ for some $s$).  On the other hand, if $ l = 0$, then $q_r = b_r \notin \beta(\sigma)$ while $q_r-e \in \beta(\sigma)$.  Furthermore, $\z_{\lambda}(q_r) = m_r - 1 \leq e-2$, where $\z(\mu) = (m_1,\dotsc, m_w)$, so that there exists $y \in (q_r-e,\, q_r) \cap \beta(\lambda)$. Hence $y \in \beta(\sigma)$ and so $q_r-e < \max(\beta(\sigma))$. Thus, condition (IB) holds.

Let $\beta(\nu) = \beadoperation{q_{r}}^{k, l} (\beta(\sigma))$.  Then by Lemma \ref{L:omnibus}(5) and using the notations there, we have
\begin{align*}
\z_{\nu}(\tilde{g}(b;q)) 
&= \z_{\sigma}(b;q) + \I_{q \in (q_{r} - ke,\, q_{r} - (k-1)e]} - \I_{q \in (q_{r}+le,\, q_{r} + (l+1)e]} \\
&= \z_{\sigma}(b;q) + \I_{q \in (g_r,\, g_{r} +e]} - \I_{q \in (b_{r},\, b_r+e]} \\
&= \z_{\lambda}(b;q)
\end{align*}
for all $(b;q) \in \bm(\sigma) = \bm(\lambda) \setminus \bm^{(a)}(\lambda)$, and, for each $i \in [1,\, k+l]$,
\begin{align*}
\z_{\nu}(g(d_i-e);d_i) &= \z_{\sigma}(q_{r}+ f(i)e) - \I_{i=k+l>k} + \I_{i=1\leq l} \\
&= \z_{\lambda}(q_{s_i}) + \I_{q_{s_i} \in [b_r,\,b_r+e)} - \I_{i=k+l>k} + \I_{i=1\leq l} \\
&= \z_{\lambda}(q_{s_i}) + \delta_{i1}
\end{align*}
where $s_i \succeq_{\lambda} r$ such that $q_{s_i} = q_r + f(i)e$.  By Lemma \ref{L:omnibus}(1,3), for any $(b_t;q_t) \in  \bm(\lambda) \setminus \bm^{(a)}(\lambda) = \bm(\sigma)$,
we have $q_t \notin (q_{s_i},\, d_i)$ for all $i \in [1,\, k+l]$, so that $(b_t;q_t)$ is the $t$-th bead movement of $\nu$, and that $(g(d_i-e);d_i)$ is the $s_i$-th bead movement of $\nu$ for each $i \in [1,\, k+l]$. Thus,
$\z(\nu) = \z(\lambda) + \ssb_{s_1} = \z(\lambda) + \ssb_{r} = \z(\mu)$, and hence $\nu = \mu$ by Theorem~\ref{thm:goodlabels}.

Now we briefly describe how to obtain the remaining assertions.  We have already shown the first assertion of part (1) above, and the second assertion follows from the fact that $\sb^{\mu}_s$ depends only on whether $b'_s$ is the bottom bead on that runner, and if so, whether $q'_s$ is its final bead movement, and the answers to these questions are affirmative if and only if the respective answers to the same questions about $b_s$ and $q_s$ are affirmative.

Part (2) follows from Lemma \ref{L:omnibus}(1,2,3) and the fact that condition (IA) in Lemma \ref{L:omnibus} holds whenever $l>0$.

Part (3) follows from Lemma \ref{L:omnibus}(1).

For part (4), if $q_s' = q_s +e$, then we must have $d_1 \geq q_r + e$ by Lemma \ref{L:omnibus}(1).  This forces $l=0$, as otherwise condition (IA)---$\z_\sigma(q_r) < \z_{\sigma}(q_r+e)$---holds, so that $d_1 = b_{\beta(\sigma)}(q_r) <q_r+e$ by Lemma \ref{L:easy}(1), since $q_r \notin \beta(\sigma)$.
Let $c_r = \max\{ x \geq_e q_r \mid x < d_1\}$.
By part (2b), we see that $u \in \beta(\sigma)$ if and only if $u-e \in \beta(\sigma)$ for all $u \in (q_r,\, d_1)$.  Since $q_r\notin \beta(\sigma)$, we have $q_r+e, q_r+2e,\dotsc, c_r \notin \beta(\sigma)$. Hence $d_1 -e \ne c_r$.  Consequently, $q'_r = d_1 > q_r +e$ and $r \ne s$, so that, by part (2a), $(b_r;q_r) = (q_r;q_r)$ is the initial bead movement of the bottom bead of that runner.  Since $s \succ_{\lambda} r$, this forces $(q_r;q_r)$ to be the only bead movement of the bottom bead of that runner.

Part (5) follows from Lemma \ref{L:omnibus}(1,2,3).
\end{proof}

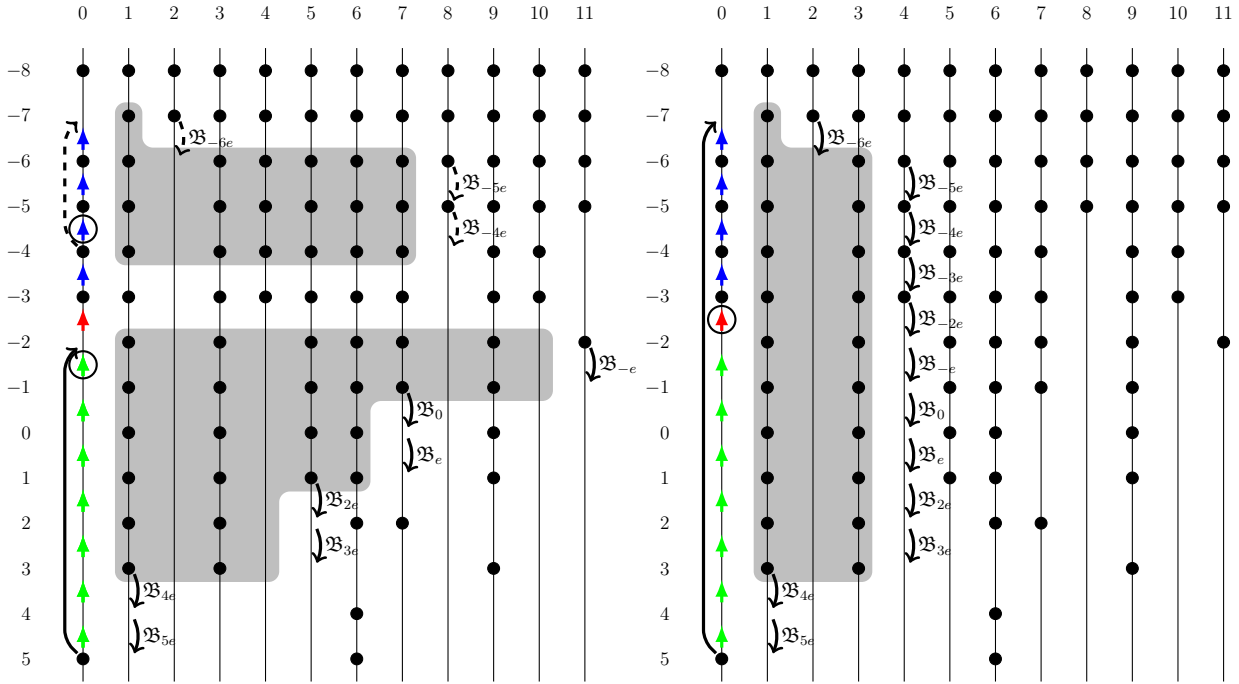
\begin{figure}[h]\caption{Illustration of Proposition~\ref{P:movealongdescription}}

\def\globalscaling{0.6}
\def\arrowshortening{0.25}
\def\shadebuffer{0.3}
\def\beadmovementwidth{2}
\def\abacusmovementwidth{2}
\def\beadoperationscaling{1.1}
\def\labelsize{1}
\def\beadsize{.13}
\def\circledbeadmovementsize{.3}
\def\runnershiftforRHSdiagram{14}


\begin{tikzpicture}
[y=-1cm, scale=\globalscaling], every node/.style={scale=.8*\globalscaling}]

\fill
[lightgray, rounded corners]
(1-\shadebuffer,3-\shadebuffer)
-- (10+\shadebuffer,3-\shadebuffer)
-- (10+\shadebuffer,4+\shadebuffer)
-- (6+\shadebuffer,4+\shadebuffer)
-- (6+\shadebuffer,6+\shadebuffer)
-- (4+\shadebuffer,6+\shadebuffer)
-- (4+\shadebuffer,8+\shadebuffer)
-- (1-\shadebuffer,8+\shadebuffer)
-- cycle;

\fill
[lightgray, rounded corners]
(1-\shadebuffer,1+\shadebuffer)
-- (7+\shadebuffer,1+\shadebuffer)
-- (7+\shadebuffer,-1-\shadebuffer)
-- (1+\shadebuffer,-1-\shadebuffer)
-- (1+\shadebuffer,-2-\shadebuffer)
-- (1-\shadebuffer,-2-\shadebuffer)
-- cycle;

\fill
[lightgray, rounded corners]
(\runnershiftforRHSdiagram+1-\shadebuffer,8+\shadebuffer)
-- (\runnershiftforRHSdiagram+3+\shadebuffer,8+\shadebuffer)
-- (\runnershiftforRHSdiagram+3+\shadebuffer,-1-\shadebuffer)
-- (\runnershiftforRHSdiagram+1+\shadebuffer,-1-\shadebuffer)
-- (\runnershiftforRHSdiagram+1+\shadebuffer,-2-\shadebuffer)
-- (\runnershiftforRHSdiagram+1-\shadebuffer,-2-\shadebuffer)
-- cycle;

\foreach \x in {0,...,11}
{
\node [scale=\labelsize*\globalscaling, above] at (\x, -4) {$\x$};
\draw (\x,-3.5) -- (\x,10.5);};
row labels
\foreach \y in {-8,...,5}
{\node [scale=\labelsize*\globalscaling, left] at (-1, \y+5) {$\y$};};

\foreach \x in
{
 (0,-3),   (1,-3), (2,-3), (3,-3), (4,-3), (5,-3), (6,-3), (7,-3), (8,-3),(9,-3), (10,-3),            (11,-3),
             (1,-2),  (2,-2),          (3,-2), (4,-2), (5,-2), (6,-2), (7,-2), (8,-2),(9,-2), (10,-2),             (11,-2),
 (0,-1), (1,-1),           (3,-1), (4,-1), (5,-1), (6,-1), (7,-1), (8,-1), (9,-1), (10,-1),                  (11,-1),
  (0,0), (1,0),              (3,0), (4,0), (5,0),    (6,0), (7,0),  (8,0),        (9,0), (10,0),           (11,0),
   (0,1), (1,1),            (3,1), (4,1), (5,1), (6,1), (7,1),              (9,1), (10,1),
    (0,2),   (1,2),          (3,2), (4,2), (5,2), (6,2), (7,2),            (9,2),                   (10,2),
             (1,3),           (3,3),           (5,3), (6,3), (7,3),          (9,3), (11,3),
             (1,4),          (3,4),          (5,4), (6,4), (7,4),           (9,4),
             (1,5),          (3,5),          (5,5), (6,5),                    (9,5),
             (1,6),          (3,6),          (5,6), (6,6),                   (9,6),
             (1,7),          (3,7),                   (6,7),     (7,7),
             (1,8),          (3,8),                                            (9,8),
                                                                     (6,9),    (6,10),
              (0,10)
}
{\draw
[fill] \x circle (\beadsize)
;};

\draw [red, line width=\beadmovementwidth*\globalscaling,->, >=latex] (0,3-\arrowshortening) -- (0,2+\arrowshortening);
\foreach \x in {4,...,10}
{
\draw [green, line width=\beadmovementwidth*\globalscaling,->, >=latex] (0,\x-\arrowshortening) -- (0,\x-1+\arrowshortening);
};
\foreach \x in {-1,...,2}
{
\draw [blue, line width=\beadmovementwidth*\globalscaling,->, >=latex] (0,\x-\arrowshortening) -- (0,\x-1+\arrowshortening);
};
\foreach \x in {4,...,10}
{
\draw [green, line width=\beadmovementwidth*\globalscaling,->, >=latex] (0,\x-\arrowshortening) -- (0,\x-1+\arrowshortening);
};

\draw [thick] (0,3.5) circle (\circledbeadmovementsize);

\draw [->, line width = \abacusmovementwidth*\globalscaling, rounded corners](-\arrowshortening/2,10-\arrowshortening/2) -- (-0.4,9.6) -- (-0.4,3.4) -- (-\arrowshortening/2, 3+\arrowshortening/2);
\draw [->, line width = \abacusmovementwidth*\globalscaling, rounded corners] (11+\arrowshortening/2, 3+\arrowshortening/2) -- (11.25,3.5) node [scale=\beadoperationscaling*\globalscaling, right,xshift=-3] {$\beadoperation{-e}$} -- (11+\arrowshortening/2, 4-\arrowshortening/2);
\draw [->, line width = \abacusmovementwidth*\globalscaling, rounded corners] (7+\arrowshortening/2, 4+\arrowshortening/2) -- (7.25,4.5) node [scale=\beadoperationscaling*\globalscaling, right,xshift=-3] {$\beadoperation{0}$} -- (7+\arrowshortening/2, 5-\arrowshortening/2);
\draw [->, line width = \abacusmovementwidth*\globalscaling, rounded corners] (7+\arrowshortening/2, 5+\arrowshortening/2) -- (7.25,5.5) node [scale=\beadoperationscaling*\globalscaling, right,xshift=-3] {$\beadoperation{e}$} -- (7+\arrowshortening/2, 6-\arrowshortening/2);
\draw [->, line width = \abacusmovementwidth*\globalscaling, rounded corners] (5+\arrowshortening/2, 6+\arrowshortening/2) -- (5.25,6.5) node [scale=\beadoperationscaling*\globalscaling, right,xshift=-3] {$\beadoperation{2e}$} -- (5+\arrowshortening/2, 7-\arrowshortening/2);
\draw [->, line width = \abacusmovementwidth*\globalscaling, rounded corners] (5+\arrowshortening/2, 7+\arrowshortening/2) -- (5.25,7.5) node [scale=\beadoperationscaling*\globalscaling, right,xshift=-3] {$\beadoperation{3e}$} -- (5+\arrowshortening/2, 8-\arrowshortening/2);
\draw [->, line width = \abacusmovementwidth*\globalscaling, rounded corners] (1+\arrowshortening/2, 8+\arrowshortening/2) -- (1.25,8.5) node [scale=\beadoperationscaling*\globalscaling, right,xshift=-3] {$\beadoperation{4e}$} -- (1+\arrowshortening/2, 9-\arrowshortening/2);
\draw [->, line width = \abacusmovementwidth*\globalscaling, rounded corners] (1+\arrowshortening/2, 9+\arrowshortening/2) -- (1.25,9.5) node [scale=\beadoperationscaling*\globalscaling, right,xshift=-3] {$\beadoperation{5e}$} -- (1+\arrowshortening/2, 10-\arrowshortening/2);

\draw [thick] (0,0.5) circle (\circledbeadmovementsize);

\draw [->, dashed, line width = \abacusmovementwidth*\globalscaling, rounded corners](-\arrowshortening/2,1-\arrowshortening/2) -- (-0.4,0.6) -- (-0.4,-1.6) -- (-\arrowshortening/2, -2+\arrowshortening/2);
\draw [->, dashed, line width = \abacusmovementwidth*\globalscaling, rounded corners] (8+\arrowshortening/2, 0+\arrowshortening/2) -- (8.25,0.5) node [scale=\beadoperationscaling*\globalscaling, right,xshift=-3] {$\beadoperation{-4e}$} -- (8+\arrowshortening/2, 1-\arrowshortening/2);
\draw [->, dashed, line width = \abacusmovementwidth*\globalscaling, rounded corners] (2+\arrowshortening/2, -2+\arrowshortening/2) -- (2.25,-1.5) node [scale=\beadoperationscaling*\globalscaling, right,xshift=-3] {$\beadoperation{-6e}$} -- (2+\arrowshortening/2, -1-\arrowshortening/2);
\draw [->, dashed, line width = \abacusmovementwidth*\globalscaling, rounded corners] (8+\arrowshortening/2, -1+\arrowshortening/2) -- (8.25,-0.5) node [scale=\beadoperationscaling*\globalscaling, right,xshift=-3] {$\beadoperation{-5e}$} -- (8+\arrowshortening/2, 0-\arrowshortening/2);


\foreach \x in {0,...,11}
{
\node [scale=\labelsize*\globalscaling, above] at (\x +\runnershiftforRHSdiagram, -4) {$\x$};
\draw (\x+\runnershiftforRHSdiagram,-3.5) -- (\x+\runnershiftforRHSdiagram,10.5);};
\foreach \y in {-8,...,5}
{\node [scale=\labelsize*\globalscaling, left] at (-1+\runnershiftforRHSdiagram, \y+5) {$\y$};};

\foreach \x in
{
 (0,-3),   (1,-3), (2,-3), (3,-3), (4,-3), (5,-3), (6,-3), (7,-3), (8,-3),(9,-3), (10,-3),             (11,-3),
             (1,-2),  (2,-2),          (3,-2), (4,-2), (5,-2), (6,-2), (7,-2), (8,-2),(9,-2), (10,-2),            (11,-2),
 (0,-1), (1,-1),           (3,-1), (4,-1), (5,-1), (6,-1), (7,-1), (8,-1), (9,-1), (10,-1),           (11,-1),
  (0,0), (1,0),              (3,0), (4,0), (5,0),    (6,0), (7,0),  (8,0),        (9,0), (10,0),            (11,0),
   (0,1), (1,1),            (3,1), (4,1), (5,1), (6,1), (7,1),              (9,1), (10,1),
    (0,2),   (1,2),          (3,2), (4,2), (5,2), (6,2), (7,2),            (9,2),                        (10,2),
             (1,3),           (3,3),           (5,3), (6,3), (7,3),          (9,3), (11,3),
             (1,4),          (3,4),          (5,4), (6,4), (7,4),           (9,4),
             (1,5),          (3,5),          (5,5), (6,5),                    (9,5),
             (1,6),          (3,6),          (5,6), (6,6),                   (9,6),
             (1,7),          (3,7),                   (6,7),  (7,7),
             (1,8),          (3,8),                                                  (9,8),
                                                                            (6,9),    (6,10),
              (0,10)
}
{
\draw
[fill,xshift=\runnershiftforRHSdiagram cm] \x circle (\beadsize)
;};

\draw [red, line width=\beadmovementwidth*\globalscaling,-> , >=latex] (\runnershiftforRHSdiagram,3-\arrowshortening) -- (\runnershiftforRHSdiagram,2+\arrowshortening);
\foreach \x in {4,...,10}
{
\draw [green, line width=\beadmovementwidth*\globalscaling,->, >=latex] (\runnershiftforRHSdiagram,\x-\arrowshortening) -- (\runnershiftforRHSdiagram,\x-1+\arrowshortening);
};
\foreach \x in {-1,...,2}
{
\draw [blue, line width=\beadmovementwidth*\globalscaling,->, >=latex ] (\runnershiftforRHSdiagram,\x-\arrowshortening) -- (\runnershiftforRHSdiagram,\x-1+\arrowshortening);
};

\draw [thick] (\runnershiftforRHSdiagram, 2.5) circle (\circledbeadmovementsize);

\draw [->, line width = \abacusmovementwidth*\globalscaling, rounded corners](\runnershiftforRHSdiagram-\arrowshortening/2,10-\arrowshortening/2) -- (\runnershiftforRHSdiagram-0.4,9.6) -- (\runnershiftforRHSdiagram-0.4,-1.6) -- (\runnershiftforRHSdiagram-\arrowshortening/2, -2+\arrowshortening/2);

\draw [->, line width = \abacusmovementwidth*\globalscaling, rounded corners] (\runnershiftforRHSdiagram+4+\arrowshortening/2, 1+\arrowshortening/2) -- (\runnershiftforRHSdiagram+4.25,1.5) node [scale=\beadoperationscaling*\globalscaling, right,xshift=-3] {$\beadoperation{-3e}$} -- (\runnershiftforRHSdiagram+4+\arrowshortening/2, 2-\arrowshortening/2);
\draw [->, line width = \abacusmovementwidth*\globalscaling, rounded corners] (\runnershiftforRHSdiagram+4+\arrowshortening/2, 0+\arrowshortening/2) -- (\runnershiftforRHSdiagram+4.25,0.5) node [scale=\beadoperationscaling*\globalscaling, right,xshift=-3] {$\beadoperation{-4e}$} -- (\runnershiftforRHSdiagram+4+\arrowshortening/2, 1-\arrowshortening/2);
\draw [->, line width = \abacusmovementwidth*\globalscaling, rounded corners] (\runnershiftforRHSdiagram+2+\arrowshortening/2, -2+\arrowshortening/2) -- (\runnershiftforRHSdiagram+2.25,-1.5) node [scale=\beadoperationscaling*\globalscaling, right,xshift=-3] {$\beadoperation{-6e}$} -- (\runnershiftforRHSdiagram+2+\arrowshortening/2, -1-\arrowshortening/2);
\draw [->, line width = \abacusmovementwidth*\globalscaling, rounded corners] (\runnershiftforRHSdiagram+4+\arrowshortening/2, -1+\arrowshortening/2) -- (\runnershiftforRHSdiagram+4.25,-0.5) node [scale=\beadoperationscaling*\globalscaling, right,xshift=-3] {$\beadoperation{-5e}$} -- (\runnershiftforRHSdiagram+4+\arrowshortening/2, 0-\arrowshortening/2);

\draw [->, line width = \abacusmovementwidth*\globalscaling, rounded corners] (\runnershiftforRHSdiagram+4+\arrowshortening/2, 2+\arrowshortening/2) -- (\runnershiftforRHSdiagram+4.25,2.5) node [scale=\beadoperationscaling*\globalscaling, right,xshift=-3] {$\beadoperation{-2e}$} -- (\runnershiftforRHSdiagram+4+\arrowshortening/2, 3-\arrowshortening/2);
\draw [->, line width = \abacusmovementwidth*\globalscaling, rounded corners] (\runnershiftforRHSdiagram+4+\arrowshortening/2, 3+\arrowshortening/2) -- (\runnershiftforRHSdiagram+4.25,3.5) node [scale=\beadoperationscaling*\globalscaling, right,xshift=-3] {$\beadoperation{-e}$} -- (\runnershiftforRHSdiagram+4+\arrowshortening/2, 4-\arrowshortening/2);
\draw [->, line width = \abacusmovementwidth*\globalscaling, rounded corners] (\runnershiftforRHSdiagram+4+\arrowshortening/2, 4+\arrowshortening/2) -- (\runnershiftforRHSdiagram+4.25,4.5) node [scale=\beadoperationscaling*\globalscaling, right,xshift=-3] {$\beadoperation{0}$} -- (\runnershiftforRHSdiagram+4+\arrowshortening/2, 5-\arrowshortening/2);
\draw [->, line width = \abacusmovementwidth*\globalscaling, rounded corners] (\runnershiftforRHSdiagram+4+\arrowshortening/2, 5+\arrowshortening/2) -- (\runnershiftforRHSdiagram+4.25,5.5) node [scale=\beadoperationscaling*\globalscaling, right,xshift=-3] {$\beadoperation{e}$} -- (\runnershiftforRHSdiagram+4+\arrowshortening/2, 6-\arrowshortening/2);
\draw [->, line width = \abacusmovementwidth*\globalscaling, rounded corners] (\runnershiftforRHSdiagram+4+\arrowshortening/2, 6+\arrowshortening/2) -- (\runnershiftforRHSdiagram+4.25,6.5) node [scale=\beadoperationscaling*\globalscaling, right,xshift=-3] {$\beadoperation{2e}$} -- (\runnershiftforRHSdiagram+4+\arrowshortening/2, 7-\arrowshortening/2);
\draw [->, line width = \abacusmovementwidth*\globalscaling, rounded corners] (\runnershiftforRHSdiagram+4+\arrowshortening/2, 7+\arrowshortening/2) -- (\runnershiftforRHSdiagram+4.25,7.5) node [scale=\beadoperationscaling*\globalscaling, right,xshift=-3] {$\beadoperation{3e}$} -- (\runnershiftforRHSdiagram+4+\arrowshortening/2, 8-\arrowshortening/2);
\draw [->, line width = \abacusmovementwidth*\globalscaling, rounded corners] (\runnershiftforRHSdiagram+1+\arrowshortening/2, 8+\arrowshortening/2) -- (\runnershiftforRHSdiagram+1.25,8.5) node [scale=\beadoperationscaling*\globalscaling, right,xshift=-3] {$\beadoperation{4e}$} -- (\runnershiftforRHSdiagram+1+\arrowshortening/2, 9-\arrowshortening/2);
\draw [->, line width = \abacusmovementwidth*\globalscaling, rounded corners] (\runnershiftforRHSdiagram+1+\arrowshortening/2, 9+\arrowshortening/2) -- (\runnershiftforRHSdiagram+1.25,9.5) node [scale=\beadoperationscaling*\globalscaling, right,xshift=-3] {$\beadoperation{5e}$} -- (\runnershiftforRHSdiagram+1+\arrowshortening/2, 10-\arrowshortening/2);

\end{tikzpicture}

	\label{figure:movealongbasisvector}
\end{figure}

\begin{example}
Let $e=12$ and $\lambda=(67,62,57,49,44,43,38^2,36,35,32,30^2,29,28,25,23^2,22,\allowbreak 21,\allowbreak 18, \allowbreak 17^3,16,15,14,13,12^3,11,10,8^2,7^5,6^2,5^2,4^5,3^{11},2^{11},1^{13})$. Then $\wt(\lambda)=19$ and
$\z(\lambda)=(0,1,1,3,3,4,5,5,6,7,7,7,7,8,9,9,10,10,10)$. The $12$ bead movements on runner $0$ of the abacus of $\lambda$ are shown on the abaci in Figure~\ref{figure:movealongbasisvector} as straight upward-facing arrows. Proposition~\ref{P:movealongdescription} is illustrated there in three cases, each involving a modified basis vector corresponding to a circled bead movement. In the notation of the proposition, for $r=8$ we have $(b_8;q_8)=(60;-12)$, $g_8=-24$, $k=1$ and $l=6$; the adjustment needed to obtain $\sigma$ from $\lambda$ is indicated by the long curved upward-pointed arrow in the bottom half of the lefthand abacus, and the ensuing bead operations by curved downward-pointing arrows.
 When $r=3$, we have $(b_3;q_3)=(-48;-48)$, $g_3=-84$, $k=3$ and $l=0$; see the dashed arrows in the top half of the lefthand abacus. Finally, for $r=6$, we have $(b_6;q_6)=(60;-24)$, $g_6=-84$, $k=5$ and $l=7$; this case is depicted on the abacus on the right.
The set $\mathbf{G}_{\lambda}^\mu$ corresponds to the shaded regions in the figure.
\end{example}

\begin{rem}
Note that our description of $\mu$ in Proposition \ref{P:movealongdescription} when $\z(\mu) = \z(\lambda) + \sb_r^{\lambda}$ shows that $\mu = \lambda_{H_r}$, where $H_r$ is the rimhook of $\lambda$ of size divisible by $e$ that is associated to the $r$-th bead movement $(b_r;q_r)$ of $\lambda$ (see Section \ref{S:intro}).  Thus, $\sb_r^{\lambda} = \z(\lambda_H) - \z(\lambda)$, and so Theorem \ref{intro-prop} and the first part of Theorem \ref{intro-thm} are special cases of Theorem \ref{thm:main}.
\end{rem}

We next address how to obtain the partition $\mu$ from a hook-quotient partition $\lambda$ when $\z(\mu)$ is more than one $\lambda$-modified basis vector away from $\z(\lambda)$.
The following lemma effectively allows us to obtain $\mu$ from $\lambda$ by adding one modified basis vector at a time when $\mu$ is $4$-increasing.

\begin{lemma} \label{L:getcloser}
Let $\lambda$ be a hook-quotient partition of $e$-weight $w$.  Let $\mu$ be a $4$-increasing partition lying in the same block as $\lambda$.  Suppose that $\z(\mu) = \z(\lambda) + \sb_{\Gamma}^{\lambda}$, for some non-empty $\Gamma\subseteq [1,\, w]$.  Let $r$ be maximal in $\Gamma$ with respect to $\succeq_{\lambda}$.
Then there exists a $1$-increasing partition $\nu$ lying in the same block as $\lambda$ such that
\begin{enumerate}
\item $\z(\lambda) + \sb^{\lambda}_r = \z(\nu)$,
\item $\sb^{\nu}_s = \sb^{\lambda}_s$ for all $s \in \Gamma \setminus \{r\}$,
\item $\z(\mu) = \z(\nu) + \sb_{\Gamma \setminus \{r\}}^{\nu}$.
\end{enumerate}
\end{lemma}

\begin{proof}
Let $\z(\lambda) + \sb^{\lambda}_r = z = (\z_1,\dotsc, \z_w)$.  Then $\z(\mu)-z = \sb^{\lambda}_{\Gamma\setminus \{r\}} \in \para_0(\lambda)$, so that $z$ is $1$-increasing by Lemma \ref{lemma:shiftby3}(1).  Note that if $\z(\mu) = (a_1,\dotsc, a_w)$ and $\z(\lambda) = (b_1,\dotsc,b_w)$, then $\z_w \leq a_w \leq e-1$ and $\z_1 \geq b_1 \geq 0$, so that by Theorem~\ref{thm:goodlabels}, there exists a partition $\nu$ lying in the same block as $\lambda$ such that $\z(\nu) = z$.  Thus $\nu$ is $1$-increasing and hence a hook-quotient partition by Lemma \ref{lemma:1unram}(1).  By Proposition \ref{P:movealongdescription}(1) and the maximality of $r$ in $\Gamma$, we have $\sb^\lambda_s = \sb^{\nu}_s$ for all $s \in \Gamma \setminus \{r\}$.  Thus
$$
\z(\mu) = \z(\lambda) + \sb^{\lambda}_{\Gamma} = \z(\nu) + \sb^{\lambda}_{\Gamma\setminus \{r\}} = \z(\nu) + \sb^{\nu}_{\Gamma\setminus \{r\}}.$$
\end{proof}

\begin{rem}\label{remark:movingalgorithm}
Given a hook-quotient partition $\lambda$ and a $4$-increasing partition $\mu$ with $\z(\mu) = \z(\lambda) + \sb^{\lambda}_{\Gamma}$, by iterating Lemma \ref{L:getcloser}, we can recursively find  a sequence $\nu^0, \nu^1,\dotsc, \nu^{|\Gamma|}$ of partitions with $\nu^0 = \lambda$, $\nu^{|\Gamma|} = \mu$, and $\z(\nu^i) = \z(\nu^{i-1}) + \sb^{\nu^{i-1}}_{\gamma_i}$ for all $i \in [1,\,|\Gamma|]$, where $\gamma_i \in \Gamma \setminus \{ \gamma_1,\dotsc, \gamma_{i-1} \}$ is maximal with respect to $\succeq_{\nu^{i-1}}$.  Using Proposition \ref{P:movealongdescription}, we may compute each $\nu^i$ in turn, and obtain $\mu$ in $|\Gamma|$ steps.

To extend this algorithm to a more general setting, allowing $\mu$ to be merely $0$-increasing, the main obstruction is that the inter\-mediate partitions $\nu^1,\dotsc, \nu^{|\Gamma|-1}$ may not be $0$-increasing or hook-quotient.  When the inter\-mediate partitions have these desired properties, then the algorithm will indeed produce the correct $\mu$.

As an example, take $e=4$ and $\lambda=(7,3,3,2,2,1)$, the partition with $4$-core $(2)$ and $4$-quotient $((1),\emptyset,(2,1),\emptyset)$. Then
$z(\lambda)=(1,1,2,3)$ and the modified basis vectors corresponding to $\lambda$ are
$\sb_1^\lambda=\ssb_1-\ssb_3$, $\sb_2^\lambda=\ssb_2$,  $\sb_3^\lambda=\ssb_3$ and $\sb_4^\lambda=\ssb_4-\ssb_3$.
The $0$-increasing partition $\mu$ in the same block for which $\z(\mu)=(1,2,3,3)=\z(\lambda)+\sb^\lambda_2+\sb^\lambda_3$ can be computed using the algorithm described above, via the intermediate partition $\nu$ in the same block for which $\z(\nu)=(1,1,3,3)=\z(\lambda)+\sb^\lambda_3$; two applications of
Proposition \ref{P:movealongdescription} yield
$\nu=(9,3,2,2,2)$, with $4$-quotient $((3),\emptyset, \emptyset, (1))$, and
$\mu=(10,4,2,1,1)$, with $4$-quotient $(\emptyset,(2),(1),(1))$. On the other hand one cannot arrive at $\mu$ using the algorithm via the partition $\rho$ in the same block for which
$\z(\rho)=(1,2,2,3)=\z(\lambda)+\sb^\lambda_2$, since one finds that $\rho=(7,4,4,1,1,1)$, with $4$-quotient $(\emptyset,\emptyset,(2,2),\emptyset)$, is not
a hook-quotient partition.
Incidentally, it is indeed true that $d_{\lambda\mu}(q)=q^2$, even though $\mu$ is not $4$-increasing and hence does not satisfy the hypotheses of Theorem~\ref{thm:main}.


See also Remark~\ref{remark:calculateMullineux} for a general situation in which this algorithm is useful.

\end{rem}
We end this section with the following technical result, required to prove Corollary~\ref{C:noremovable}: when the abacus of a hook-quotient partition has no removable bead on runner $a$, this property is inherited by the partitions in its sub-parallelotope generated by its modified basis vectors corresponding to bead movements starting in runner $a$ and $a-1$.

\begin{lemma} \label{L:noremovable}
Let $\lambda$ be a $1$-increasing partition and $\mu$ a $0$-increasing partition lying in the same block.  Suppose that
\begin{itemize}
\item $\z(\mu) = \z(\lambda) + \sb^\lambda_r$;
\item $\lambda$ has no removable bead on runner $a$;
\item the $r$-th bead movement of $\lambda$ lies in runner $a$ or $a-1$.
\end{itemize}
Then $\mu$ has no removable bead on runner $a$.
\end{lemma}

\begin{proof}
	Let $y=\max\{b\in\beta(\lambda)\mid b\equiv_e a\}$. We claim that $s\in\beta(\lambda)$ for all $s\leq_e y-1$. If $t\in\beta(\lambda)$ for all $t\leq_e y$, this is obvious, since $\lambda$ has no removable bead on runner $a$. Otherwise $(y;y)\in\bm(\lambda)$, which implies that $\lambda$, being $1$-increasing, does not have a bead movement starting at $y-1$, but $y-1 \in \beta(\lambda)$ since $\lambda$ has no removable bead on runner $a$, and hence $\wt_{\lambda}(y-1) = 0$, and so the claim holds in this other case as well.
	
	Let $(b_r;q_r)$ be the $r$-th bead movement of $\lambda$, and adopt
	the notations of Proposition~\ref{P:movealongdescription}
	and Lemma~\ref{L:omnibus}.
	In addition, let $y'=\max\{b\in\beta(\mu)\mid b\equiv_e a\}$. 			

If $q_r\equiv_e a-1$, then $g_r>y$ by the claim above.  Thus by Proposition~\ref{P:movealongdescription},   $y'=y$ and  $s\in\beta(\mu)$ for all $s\leq_e y-1$. So in this case $\mu$ has no removable bead on runner $a$.

		Suppose instead that $q_r\equiv_e a$.
	By Proposition~\ref{P:movealongdescription},
	we have $y'\leq y$, with equality only if $b_r=q_r$.
		If $(b_r;q_r)$ is the only bead movement of $\lambda$ on runner $a$, then $l=0$, $k=1$, $y'=y-e$ and $d_1>y$, so that $s\in\beta(\mu)$ for all $s\leq_e y-1-e$. Hence $\mu$ has no removable bead on runner $a$ in this case.
		So we may assume that $\lambda$ has a least two bead movements on runner $a$, so that $(y;y), (\overline{y};y-e)\in\bm(\lambda)$, where $\overline{y}$ is either $y$ or $y-e$. Since $\lambda$ is $1$-increasing, we have $\z_\lambda(y)>\z_\lambda(y-e)$, and therefore, by Lemma~\ref{L:easy}, that there exists $s\in(y-e,\,y] \setminus \beta(\lambda) = (y-e,\,y-2] \setminus \beta(\lambda)$ such that $s-e\in \beta(\lambda)$. It follows from Lemma~\ref{L:omnibus} that $d_i\neq y-1$ for any $i$, and therefore that $s\in\beta(\mu)$ for all $s\leq_e y-1-e$.	
	
	Hence we are reduced to the case where $y'=y$ and $d_j=y-1+e$ for some $j$, for otherwise $\mu$ has no removable bead on runner $a$.	If $q_r<_e y$, then $d_i<q_r+e\leq_e y$ for all $i$ by Proposition \ref{P:movealongdescription}(2a), a contradiction. So the only situation left to consider is $b_r=q_r=y$, with $l=0$ and $k\geq 2$. We must have $d_1=y-1+e$ and therefore $y-1-e\in S_1$ and $y-1\notin S_1$. It follows that $d_i\leq y-1$ for all $i>1$, and therefore that $y'<y$, a contradiction.
\end{proof}

\begin{cor} \label{C:noremovable}
Let $\Gamma \subseteq [1,\, w]$.  Let $\lambda$ be a partition of $e$-weight $w$ whose abacus display has no removable bead on runner $a$, and suppose that for all $s \in \Gamma$, the $s$-th bead movement of $\lambda$ lies in runner $a$ or $a-1$.  If $\z(\lambda) + \sb^{\lambda}_{\Gamma} = \z(\mu)$ for some $4$-increasing partition $\mu$ lying in the same block as $\lambda$, then the abacus display of $\mu$ has no removable bead on runner $a$.
\end{cor}

\begin{proof}
Note first that $\lambda$ is $1$-increasing by Lemma \ref{lemma:shiftby3}.
We prove by induction on $|\Gamma|$, with $|\Gamma|=0$ being trivial.  For $|\Gamma| >0$, let $r \in \Gamma$ be maximal with respect to $\succeq_{\lambda}$.  By Lemma \ref{L:getcloser}, $\z(\lambda) + \sb_r^{\lambda} = \z(\nu)$ for some $1$-increasing partition $\nu$ lying in the same block as $\lambda$ (and $\mu$), with $\z(\mu) = \z(\nu) + \sb^{\nu}_{\Gamma\setminus \{r\}}$.  Furthermore, by Proposition \ref{P:movealongdescription}(1),  the $s$-th bead movement of $\nu$ lies in runner $a$ or $a-1$ for all $s \in \Gamma \setminus \{r\}$.  By Lemma \ref{L:noremovable} we see that the abacus display of $\nu$ has no removable bead on runner $a$, and hence by induction the abacus display of $\mu$ has no removable bead on runner $a$.
\end{proof}

\section{Proof of the Main Theorem}\label{section:mainproof}

This section is dedicated to proving the following proposition, from which the our main theorem follows easily.

\begin{prop}\label{prop:main}
	Let $B$ be a block of $e$-weight $w$ and suppose that the abacus display of the corresponding $e$-core has at least one removable bead on runner $a$. If Theorem~\ref{thm:main} holds for $B$ and for all blocks of $e$-weight strictly less than $w$, then it also holds for $s_a(B)$.
	\end{prop}

\begin{proof}[Proof of Theorem \ref{thm:main} using Proposition \ref{prop:main}]
We induct on the $e$-weight $w$ of the block $B$. If $w=0$, then the theorem clearly holds for $B$. Otherwise, assume the theorem is true for all blocks of $e$-weight strictly less than $w$. By \cite[Lemma 3.1]{F} (see also \cite[Lemma 2.16]{ST}), there exists a sequence $B_0, B_1, \dotsc, B_n = B$ of blocks of $e$-weight $w$, such that $B_0$ is a Rouquier block and $B_{i+1}=s_{a_i}(B_{i})$ for some $a_i$, for all $i\in[0,n-1]$, and that the abacus display of the $e$-core of $B_i$ has at least one removable bead on runner $a_i$. Since Theorem~\ref{thm:main} is true for all Rouquier blocks (see Example~\ref{example:mainRouquier}), it follows by Proposition~\ref{prop:main} that it holds for $B$.
\end{proof}

\subsection{Setup}
Let $\b$ be as in the statement of Proposition~\ref{prop:main}. We write $\b=B_{\kappa,w}$ for some $e$-core $\kappa$ whose abacus has $k>0$ removable beads on runner $a$. Let $\tb=s_a(\b)$, so that $\tb=B_{\tkappa,w}$, where $\tkappa=s_a(\kappa)$;
 the blocks $\b$ and $\tb$ form a Scopes $[w:k]$-pair.

Let $\cb=B_{\ckappa,\w}$ be the block of $e$-weight $\w := w-k-1$ whose $e$-core $\ckappa$ is obtained from $\tkappa$ by moving the bottom bead on runner $a$ to the top unoccupied position on runner $a-1$. Similarly let $\hb=B_{\hat\kappa,\w}$ be the block of $e$-weight $\w$ whose $e$-core is obtained from $\kappa$ by moving the bottom bead on runner $a-1$ to top occupied position on runner $a$. Then $\cb=s_a(\hb)$, and $\hb$ and $\cb$ form a Scopes $[\w:k+2]$-pair.

The following easy lemma singles out a class of well-behaved partitions in $B$ and $\tB$.  From now on, we write $E$ and $F$ for $E_a$ and $F_a$.
\begin{lemma}\label{L:nonexceptional}
The following statements about a partition $\lambda$ in $B$ and $\tilde{\lambda}= s_a(\lambda)$ in $\tilde{B}$ are equivalent:
\begin{enumerate}
\item $F(\lambda) = 0$;
\item $E^{(k)}(\lambda) = \tilde{\lambda}$;
\item $F^{(k)}(\tilde{\lambda}) = \lambda$;
\item $E(\tilde{\lambda}) = 0$.
\end{enumerate}
\end{lemma}

\begin{defi}
	A partition $\lambda \in B$ or a partition $\tlambda \in \tB$ is called {\em non-exceptional} if it satisfies the equivalent conditions in Lemma~\ref{L:nonexceptional}. Otherwise it is called {\em exceptional}.
\end{defi}

\begin{lemma}\label{L:nonexceptional2}
	The rule $\lambda\mapsto \tlambda=s_a(\lambda)$ gives a bijection between the non-exceptional partitions in $\b$ and those in $\tb$.
	It restricts to a bijection between the non-exceptional $0$-increasing partitions of $B$ and $\tB$, and for these partitions $\z(\lambda)=\z(\tlambda)$. The map also restricts to a bijection between the non-exceptional hook-quotient partitions of $B$ and $\tB$.  Furthermore, for a $0$-increasing hook-quotient non-exceptional $\lambda$ in $\b$, we have 	$\sb_i^\lambda=\sb_i^\tlambda$ for all $i\in [1,w]$ and $\para(\lambda)=\para(\tlambda)$.
	\end{lemma}

\begin{proof} The first statement is an immediate consequence of Lemma~\ref{L:nonexceptional}. The statement about non-exceptional $0$-increasing partitions follows from
	Theorem~\ref{thm:goodlabels}. For the statement about non-exceptional hook-quotient partitions, note that the $e$-quotient of $\tlambda$ is obtained from that of
	$\lambda$ by swapping its $a$-th and $(a-1)$-st components.  When $\lambda$ in $\b$ is $0$-increasing, hook-quotient and non-exceptional, then using Lemma \ref{L:0-increasing} and Figure \ref{figure:leqconfiguration}, we deduce that there does not exist $q \equiv_e a$ such that $(b;q),(b';q-1) \in \bm(\lambda)$, and consequently, $\sb_i^\lambda=\sb_i^\tlambda$ for all $i\in [1,w]$ and $\para(\lambda)=\para(\tlambda)$.
	\end{proof}
	
We now turn attention to exceptional partitions and begin with the following key observation.
\begin{lemma}\label{lemma:nottooexceptional} \hfill
\begin{enumerate}
\item
Suppose that $\lambda\in\b$ is exceptional and $1$-increasing. Then
\begin{enumerate}
\item $\langle \chevf (\lambda), \hat{\lambda} \rangle \ne 0$ for a unique partition $\hat{\lambda}$ in $\hat{B}$, and $\chevf (\hat{\lambda}) = 0$;
\item $\langle \cheve^{(k+1)}(\lambda), \clambda \rangle \ne 0$ for a unique partition $\clambda$ in $\check{B}$, and $\cheve(\clambda)=0$.
\end{enumerate}
\item
Let $\mu$ be a $4$-increasing partition in $\b$.  If Theorem \ref{thm:main} holds for $\mu$, then $\cheve^{(k+2)}(G(\mu))=0$ and $\chevf^{(2)}(G(\mu))=0$.
\end{enumerate}
\end{lemma}

\begin{proof}
When $\lambda$ is exceptional, we have $F(\lambda) \ne 0$. Thus $\lambda$ has at least one addable bead on runner $a-1$.  Hence by Lemma~\ref{lemma:1unram}(2), it has exactly one addable bead on runner $a-1$, and hence exactly $k+1$ removable beads on runner $a$, since its $e$-core $\kappa$ has $k$ removable beads on runner $a$.  Part (1) thus follows.

Let $\mu$ be as in part (2), and suppose that $\lambda$ is a partition such that $d_{\lambda\mu}(q)\neq 0$.
By Lemma~\ref{lemma:shiftby3} and the assumption that Theorem~\ref{thm:main} holds for $\mu$ we know that $\lambda$ is $1$-increasing. Whether or not $\lambda$ is exceptional we deduce by the first part that $\cheve^{(k+2)}(\lambda) = 0$ and $\chevf^{(2)}(\lambda) = 0$. So part (2) is proved.
\end{proof}

Fix $\clambda\in\cb$ such that $\cheve(\clambda)=0$. Then $\clambda$ has exactly $k+2$ addable beads on runner $a-1$ and no removable bead on runner $a$.  Let $C=\{c_0,\dotsc,c_{k+1}\}\subseteq \beta(\clambda)$ be the set of addable beads of $\clambda$ on runner $a-1$, with $c_0<\dotsb<c_{k+1}$. Then we have
$$\chevf^{(k+2)}(\clambda)=\hat{\lambda},\qquad \chevf^{(k+1)}(\clambda)=\sum_{j=0}^{k+1} q^j \epl{j}
\qquad\text{ and }\qquad
\chevf(\clambda)=\sum_{j=0}^{k+1} q^j \eptl{j},$$
where $\hat{\lambda}\in\hb$, $\epl{j}\in \b$ and $\eptl{j}\in \tb$ are the partitions
such that
\begin{align*}
\beta(\hat{\lambda}) & = \beta(\clambda)\cup \{ c+1 \mid c\in C\} \setminus C, \\
\beta(\epl{j}) &=\beta(\hat{\lambda})\cup\{c_j\}\setminus\{c_j+1\}, \\
\beta(\eptl{j})&=\beta(\clambda)\cup \{c_{k+1-j}+1\}\setminus\{c_{k+1-j}\}.
\end{align*}
It is easy to see that for any $j \in [0,k+1]$,  $\epl{j}$ is exceptional.  In fact even though it may not necessarily be $1$-increasing, it satisfies the consequence of Lemma \ref{lemma:nottooexceptional}(1a).
We say that the set $\flambda = \{\epl{j} \mid j \in [0,\,k+1] \}\cup \{\eptl{j} \mid j \in [0,\,k+1] \}$ is an {\em exceptional family}, {\em generated by} $\clambda$. Each such family has the {\em{leading}} exceptional  partitions $\epl{0}\in\b$ and $\eptl{0}\in\tb$.
We have
\begin{align*}
s_a(\epl{0})&=\kashe^k(\epl{0})=\eptl{0},\\
s_a(\epl{j})&=\kashe^k(\epl{j})=\eptl{k+2-j} \qquad (j\in [1,\,k+1]).
\end{align*}
Also, for any $j \in [0,\,k+1]$, we have
\begin{equation}\label{Eonexceptionals}
\cheve^{(k)}(\epl{j})=
\sum_{i=0}^{k-j} q^{i+j-k}\eptl{i}
+ \sum_{i= k-j+2}^{k+1} q^{i+j-k-2}\eptl{i}.
\end{equation}

\subsection{Hook-quotient exceptional families}

\begin{lemma}\label{L:hookquotientfamilycriteria}
Let $\flambda$ be an exceptional family, generated by $\clambda\in\cb$.  The following statements are equivalent:
\begin{enumerate}
\item The $(a-1)$-th and $a$-th components of the $e$-quotient of $\clambda$ is of the form $((x),(1^y))$.
\item The $(a-1)$-th and $a$-th components of the $e$-quotients of $\epl{j}$ and $\eptl{j}$ are hooks for all $j$.
\end{enumerate}
\end{lemma}

\begin{proof}
	Suppose that (1) holds. Let $\xi = \{ b \in \beta(\clambda) \mid b \equiv_e a-1\}$.  Then $\xi - ie \notin \beta(\clambda)$ for all $i \in [1,\, x]$ while $\xi - \gamma e \in \beta(\clambda)$ for all $\gamma >x$.
Since $\clambda$ has no removable bead on runner $a$ and $k+2$ addable beads on runner $a-1$, the latter addable beads
	are at $\xi - (x+y+k+1)e$, $\xi - (x+i) e$ ($i \in [1,\,k]$) and $\xi$.
	 Thus
the $(a-1)$-th and $a$-th components of the $e$-quotient of $\epl{j}$ is
$$
\begin{cases}
(\emptyset, (x+1,1^{y+k})), &\text{if } j=0; \\
((x+k+1, 1^y),\emptyset),   &\text{if } j=k+1; \\
((j,1^y), (x+1, 1^{k-j})),  &\text{if } j \in [1,\,k].
\end{cases}
$$
Similarly, the $(a-1)$-th and $a$-th components of the $e$-quotient of
	$\eptl{j}$ is
$$
\begin{cases}
(\emptyset, (x+k+1,1^y)), &\text{if } j=0; \\
((x+1, 1^{y+k}),\emptyset), &\text{if } j=k+1; \\
((x+1, 1^{j-1}),(k+1-j,1^y)), &\text{if } j \in [1,\, k].
\end{cases}
$$
See Example~\ref{E:exceptionalfamily} below.
	
	Conversely suppose that (2) holds. First note that the partition $\eptl{0}\in\b$ is obtained from $\clambda$ by moving the bottom of the $k+2$ ($\geq 3$) addable beads on runner $a-1$ from its position $b$ to $b+1$. It follows that $\wt_{\eptl{0}}(b+1)\geq 2$. Since the $a$-th component of the $e$-quotient of $\eptl{0}$ is a hook, we deduce that the bottom bead on runner $a$ of $\eptl{0}$ is at $b+1$, and that the $a$-component of the $e$-quotient of $\clambda$ has the form $(1^y)$. Next, we consider $\eptl{k+1}\in\b$, obtained from $\clambda$ by moving the top of the $k+2$ ($\geq 3$) addable beads on runner $a-1$ from its position $c$ to $c+1$. Since the $(a-1)$-th component of the $e$-quotient of $\eptl{k+1}$ is a hook, we deduce that
	$\wt_{\clambda}(c)=0$ and that the $(a-1)$-th component of the $e$-quotient of $\clambda$ has the form $(x)$.
\end{proof}

\begin{defi}\label{D:hqef}
An exceptional family in which every member is a hook-quotient partition is called a \emph{hook-quotient exceptional family}.
\end{defi}

By Lemma \ref{L:hookquotientfamilycriteria}, an exceptional family generated by $\clambda \in \cb$ is a hook-quotient exceptional family if and only if $\clambda^{(i)}$ is a hook for all $i \in [0,e)$, and $\clambda^{(a-1)} = (x)$ and $\clambda^{(a)} = (1^y)$ for some $x,y \in \mathbb{Z}_{\geq 0}$, where $(\clambda^{(0)},\dotsc, \clambda^{(e-1)})$ is the $e$-quotient of $\clambda$.

\begin{example} \label{E:exceptionalfamily}
A hook-quotient exceptional family with $k=3$ is depicted in Figure~\ref{figure:exceptionalfamily}.  We show the abacus configuration of runners $a-1$ and $a$ of all its members, together with the bead movements in those runners.
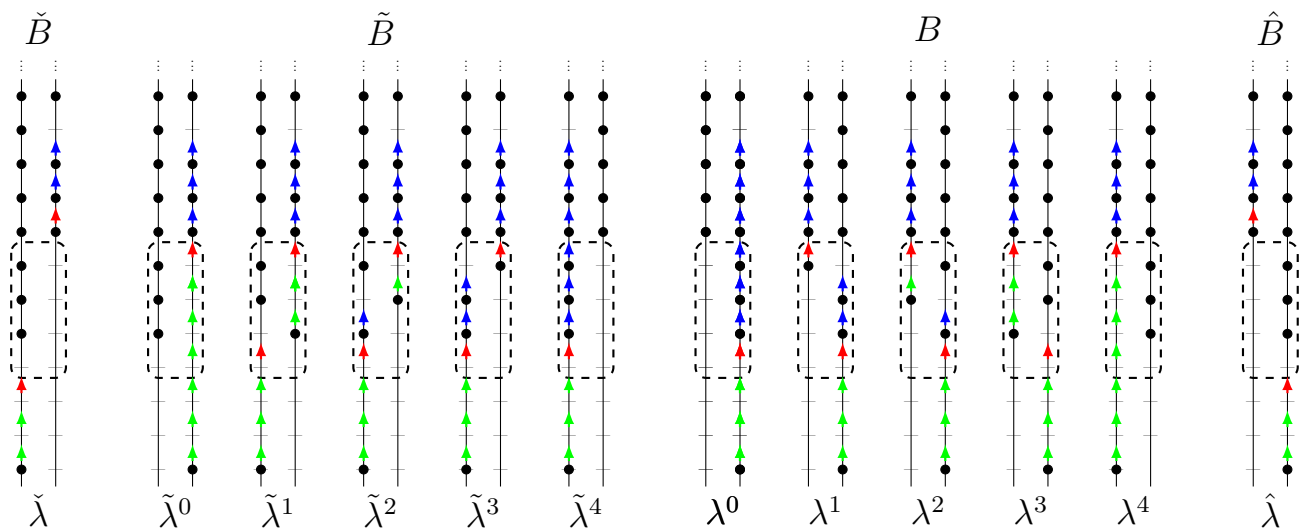
\begin{figure}[h]\caption{A hook-quotient exceptional family.}

\def\globalscaling{.45}
\def\arrowshortening{0.25}
\def\shadebuffer{0.3}
\def\beadmovementwidth{1.8}
\def\abacusmovementwidth{2}
\def\beadoperationscaling{.85}
\def\labelsize{1}
\def\partitionlabelsize{2.5    }
\def\beadsize{.13}
\def\circledbeadmovementsize{.3}
\def\runnershiftforRHSdiagram{13}
\def\runnerstart{-0.5}
\def\runnerend{11.5}
\def\gap{\----}
\def\pivot{red}
\def\belowpivot{green}
\def\abovepivot{blue}

\begin{tikzpicture}
[y=-1cm, scale=\globalscaling], every node/.style={scale=.8*\globalscaling}]

\draw (.5,\runnerstart-1.5) node[scale=\partitionlabelsize*\globalscaling] {$\check{B}$};
\draw (10.5,\runnerstart-1.5) node[scale=\partitionlabelsize*\globalscaling] {$\tilde{B}$};
\draw (26.5,\runnerstart-1.5) node[scale=\partitionlabelsize*\globalscaling] {$B$};
\draw (36.5,\runnerstart-1.5) node[scale=\partitionlabelsize*\globalscaling] {$\hat{B}$};

\foreach \x/\y in
{0/0,4/0,7/0,10/0,13/0,16/0, 20/0, 23/0, 26/0, 29/0, 32/0, 36/0
}
{
\draw [rounded corners, thick, dashed] (\x - \shadebuffer, \y+4+\shadebuffer)
rectangle ( \x + 1 +\shadebuffer, \y+8+\shadebuffer);
\draw (\x,\y+\runnerstart) -- (\x,\y+\runnerend);
\draw (\x+1,\y+\runnerstart) -- (\x+1,\y+\runnerend);
\node [scale=\labelsize*\globalscaling] at (\x, \y+\runnerstart-0.5) {$\vdots$};
\node [scale=\labelsize*\globalscaling] at (\x+1, \y+\runnerstart-0.5) {$\vdots$};
};

\def\xdiagramshift{0};
\def\ydiagramshift{0};
\draw (\xdiagramshift+.5, \ydiagramshift+\runnerend+.7)
node[scale=\partitionlabelsize*\globalscaling] {$\check{\lambda}$};
\foreach \x/\y in
{
0/0, 0/1, 0/2, 0/3, 0/4, 0/5, 0/6, 0/7, 0/11,
1/0,        1/2, 1/3, 1/4
}
{\draw
[fill] (\xdiagramshift+\x,\ydiagramshift+\y) circle (\beadsize)
;};
\foreach \x/\y in
{0/8, 0/9, 0/10, 1/1, 1/5, 1/6, 1/7,1/8,1/9,1/10,1/11}
{\node [scale=\labelsize*\globalscaling] at (\xdiagramshift+\x,\ydiagramshift+\y) {\gap};};

\foreach\x/\y/\z in
{1/2/\abovepivot,
1/3/\abovepivot,
1/4/\pivot,
0/9/\pivot,
0/10/\belowpivot,
0/11/\belowpivot
}
{\draw [\z, line width=\beadmovementwidth*\globalscaling,->, >=latex] (\xdiagramshift+\x,\ydiagramshift+\y-\arrowshortening)
-- (\xdiagramshift+\x,\ydiagramshift+\y-1+\arrowshortening);
};

\def\xdiagramshift{4};
\def\ydiagramshift{0};
\draw (\xdiagramshift+.5, \ydiagramshift+\runnerend+.7)
node[scale=\partitionlabelsize*\globalscaling] {$\tilde{\lambda}^{0}$};
\foreach \x/\y in
{
0/0, 0/1, 0/2, 0/3, 0/4, 0/5, 0/6, 0/7,
1/0,        1/2, 1/3, 1/4,                                 1/11
}
{\draw
[fill] (\xdiagramshift+\x,\ydiagramshift+\y) circle (\beadsize)
;};
\foreach \x/\y in
{0/8, 0/9, 0/10, 1/1, 1/5, 1/6, 1/7,1/8,1/9,1/10,0/11}
{\node [scale=\labelsize*\globalscaling] at (\xdiagramshift+\x,\ydiagramshift+\y) {\gap};};

\foreach\x/\y/\z in
{1/2/\abovepivot,
1/3/\abovepivot,
1/4/\abovepivot,
1/5/\pivot,
1/6/\belowpivot,
1/6/\belowpivot,
1/7/\belowpivot,
1/8/\belowpivot,
1/9/\belowpivot,
1/10/\belowpivot,
1/11/\belowpivot
}
{\draw [\z, line width=\beadmovementwidth*\globalscaling,->, >=latex] (\xdiagramshift+\x,\ydiagramshift+\y-\arrowshortening)
-- (\xdiagramshift+\x,\ydiagramshift+\y-1+\arrowshortening);
};

\def\xdiagramshift{7};
def\ydiagramshift{0};
\draw (\xdiagramshift+.5, \ydiagramshift+\runnerend+.7)
node[scale=\partitionlabelsize*\globalscaling] {$\tilde{\lambda}^{1}$};
\foreach \x/\y in
{
0/0, 0/1, 0/2, 0/3, 0/4, 0/5, 0/6,                   0/11,
1/0,        1/2, 1/3, 1/4,               1/7
}
{\draw
[fill] (\xdiagramshift+\x,\ydiagramshift+\y) circle (\beadsize)
;};
\foreach \x/\y in
{0/7, 0/8, 0/9, 0/10, 0/11, 1/1, 1/5, 1/6,1/8,1/9,1/10,1/11}
{\node [scale=\labelsize*\globalscaling] at (\xdiagramshift+\x,\ydiagramshift+\y) {\gap};};

\foreach\x/\y/\z in
{1/2/\abovepivot,
1/3/\abovepivot,
1/4/\abovepivot,
1/5/\pivot,
1/6/\belowpivot,
1/6/\belowpivot,
1/7/\belowpivot,
0/8/\pivot,
0/9/\belowpivot,
0/10/\belowpivot,
0/11/\belowpivot
}
{\draw [\z, line width=\beadmovementwidth*\globalscaling,->, >=latex] (\xdiagramshift+\x,\ydiagramshift+\y-\arrowshortening)
-- (\xdiagramshift+\x,\ydiagramshift+\y-1+\arrowshortening);
};

\def\xdiagramshift{10};
def\ydiagramshift{0};
\draw (\xdiagramshift+.5, \ydiagramshift+\runnerend+.7)
node[scale=\partitionlabelsize*\globalscaling] {$\tilde{\lambda}^{2}$};
\foreach \x/\y in
{
0/0, 0/1, 0/2, 0/3, 0/4, 0/5,       0/7,                   0/11,
1/0,        1/2, 1/3, 1/4,          1/6
}
{\draw
[fill] (\xdiagramshift+\x,\ydiagramshift+\y) circle (\beadsize)
;};
\foreach \x/\y in
{0/6, 0/8, 0/9, 0/10, 0/11, 1/1, 1/5, 1/7,1/8,1/9,1/10,1/11}
{\node [scale=\labelsize*\globalscaling] at (\xdiagramshift+\x,\ydiagramshift+\y) {\gap};};

\foreach\x/\y/\z in
{1/2/\abovepivot,
1/3/\abovepivot,
1/4/\abovepivot,
1/5/\pivot,
1/6/\belowpivot,
0/7/\abovepivot,
0/8/\pivot,
0/9/\belowpivot,
0/10/\belowpivot,
0/11/\belowpivot
}
{\draw [\z, line width=\beadmovementwidth*\globalscaling,->, >=latex] (\xdiagramshift+\x,\ydiagramshift+\y-\arrowshortening)
-- (\xdiagramshift+\x,\ydiagramshift+\y-1+\arrowshortening);
};

\def\xdiagramshift{13};
def\ydiagramshift{0};
\draw (\xdiagramshift+.5, \ydiagramshift+\runnerend+.7)
node[scale=\partitionlabelsize*\globalscaling] {$\tilde{\lambda}^{3}$};
\foreach \x/\y in
{
0/0, 0/1, 0/2, 0/3, 0/4,        0/6,  0/7,                   0/11,
1/0,        1/2, 1/3, 1/4, 1/5
}
{\draw
[fill] (\xdiagramshift+\x,\ydiagramshift+\y) circle (\beadsize)
;};
\foreach \x/\y in
{0/5, 0/8, 0/9, 0/10, 0/11, 1/1, 1/6, 1/7,1/8,1/9,1/10,1/11}
{\node [scale=\labelsize*\globalscaling] at (\xdiagramshift+\x,\ydiagramshift+\y) {\gap};};

\foreach\x/\y/\z in
{1/2/\abovepivot,
1/3/\abovepivot,
1/4/\abovepivot,
1/5/\pivot,
0/6/\abovepivot,
0/7/\abovepivot,
0/8/\pivot,
0/9/\belowpivot,
0/10/\belowpivot,
0/11/\belowpivot
}
{\draw [\z, line width=\beadmovementwidth*\globalscaling,->, >=latex] (\xdiagramshift+\x,\ydiagramshift+\y-\arrowshortening)
-- (\xdiagramshift+\x,\ydiagramshift+\y-1+\arrowshortening);
};

\def\xdiagramshift{16};
def\ydiagramshift{0};
\draw (\xdiagramshift+.5, \ydiagramshift+\runnerend+.7)
node[scale=\partitionlabelsize*\globalscaling] {$\tilde{\lambda}^{4}$};
\foreach \x/\y in
{
0/0,          0/2, 0/3, 0/4, 0/5,   0/6,  0/7,                   0/11,
1/0,  1/1, 1/2, 1/3, 1/4 }
{\draw
[fill] (\xdiagramshift+\x,\ydiagramshift+\y) circle (\beadsize)
;};
\foreach \x/\y in
{0/1, 0/8, 0/9, 0/10, 0/11, 1/5, 1/6, 1/7,1/8,1/9,1/10,1/11}
{\node [scale=\labelsize*\globalscaling] at (\xdiagramshift+\x,\ydiagramshift+\y) {\gap};};

\foreach\x/\y/\z in
{0/2/\abovepivot,
0/3/\abovepivot,
0/4/\abovepivot,
0/5/\abovepivot,
0/6/\abovepivot,
0/7/\abovepivot,
0/8/\pivot,
0/9/\belowpivot,
0/10/\belowpivot,
0/11/\belowpivot
}
{\draw [\z, line width=\beadmovementwidth*\globalscaling,->, >=latex] (\xdiagramshift+\x,\ydiagramshift+\y-\arrowshortening)
-- (\xdiagramshift+\x,\ydiagramshift+\y-1+\arrowshortening);
};

\def\xdiagramshift{36};
\def\ydiagramshift{0};
\draw (\xdiagramshift+.5, \ydiagramshift+\runnerend+.7)
node[scale=\partitionlabelsize*\globalscaling] {$\hat{\lambda}$};
\foreach \x/\y in
{
1/0, 1/1, 1/2, 1/3, 1/4, 1/5, 1/6, 1/7, 1/11,
0/0,        0/2, 0/3, 0/4
}
{\draw
[fill] (\xdiagramshift+\x,\ydiagramshift+\y) circle (\beadsize)
;};
\foreach \x/\y in
{1/8, 1/9, 1/10, 0/1, 0/5, 0/6, 0/7,0/8,0/9,0/10,0/11}
{\node [scale=\labelsize*\globalscaling] at (\xdiagramshift+\x,\ydiagramshift+\y) {\gap};};
\foreach\x/\y/\z in
{0/2/\abovepivot,
0/3/\abovepivot,
0/4/\pivot,
1/9/\pivot,
1/10/\belowpivot,
1/11/\belowpivot
}
{\draw [\z, line width=\beadmovementwidth*\globalscaling,->, >=latex] (\xdiagramshift+\x,\ydiagramshift+\y-\arrowshortening)
-- (\xdiagramshift+\x,\ydiagramshift+\y-1+\arrowshortening);
};

\def\xdiagramshift{20};
\def\ydiagramshift{0};
\draw (\xdiagramshift+.5, \ydiagramshift+\runnerend+.7)
node[scale=\partitionlabelsize*\globalscaling] {$\lambda^{0}$};
\foreach \x/\y in
{
1/0,       1/2, 1/3, 1/4, 1/5, 1/6, 1/7, 1/11,
0/0, 0/1,  0/2, 0/3, 0/4
}
{\draw
[fill] (\xdiagramshift+\x,\ydiagramshift+\y) circle (\beadsize)
;};
\foreach \x/\y in
{1/8, 1/9, 1/10, 1/1, 0/5, 0/6, 0/7,0/8,0/9,0/10,0/11}
{\node [scale=\labelsize*\globalscaling] at (\xdiagramshift+\x,\ydiagramshift+\y) {\gap};};
\foreach\x/\y/\z in
{1/2/\abovepivot,
1/3/\abovepivot,
1/4/\abovepivot,
1/5/\abovepivot,
1/6/\abovepivot,
1/7/\abovepivot,
1/8/\pivot,
1/9/\belowpivot,
1/10/\belowpivot,
1/11/\belowpivot
}
{\draw [\z, line width=\beadmovementwidth*\globalscaling,->, >=latex] (\xdiagramshift+\x,\ydiagramshift+\y-\arrowshortening)
-- (\xdiagramshift+\x,\ydiagramshift+\y-1+\arrowshortening);
};

\def\xdiagramshift{20};
\def\ydiagramshift{0};
\draw (\xdiagramshift+.5, \ydiagramshift+\runnerend+.7)
node[scale=\partitionlabelsize*\globalscaling] {$\lambda^{0}$};
\foreach \x/\y in
{
1/0,       1/2, 1/3, 1/4, 1/5, 1/6, 1/7, 1/11,
0/0, 0/1,  0/2, 0/3, 0/4
}
{\draw
[fill] (\xdiagramshift+\x,\ydiagramshift+\y) circle (\beadsize)
;};
\foreach \x/\y in
{1/8, 1/9, 1/10, 1/1, 0/5, 0/6, 0/7,0/8,0/9,0/10,0/11}
{\node [scale=\labelsize*\globalscaling] at (\xdiagramshift+\x,\ydiagramshift+\y) {\gap};};
\foreach\x/\y/\z in
{1/2/\abovepivot,
1/3/\abovepivot,
1/4/\abovepivot,
1/5/\abovepivot,
1/6/\abovepivot,
1/7/\abovepivot,
1/8/\pivot,
1/9/\belowpivot,
1/10/\belowpivot,
1/11/\belowpivot
}
{\draw [\z, line width=\beadmovementwidth*\globalscaling,->, >=latex] (\xdiagramshift+\x,\ydiagramshift+\y-\arrowshortening)
-- (\xdiagramshift+\x,\ydiagramshift+\y-1+\arrowshortening);
};

\def\xdiagramshift{23};
\def\ydiagramshift{0};
\draw (\xdiagramshift+.5, \ydiagramshift+\runnerend+.7)
node[scale=\partitionlabelsize*\globalscaling] {$\lambda^{1}$};
\foreach \x/\y in
{
1/0,  1/1,   1/2, 1/3, 1/4,         1/6, 1/7, 1/11,
0/0,              0/2, 0/3, 0/4, 0/5
}
{\draw
[fill] (\xdiagramshift+\x,\ydiagramshift+\y) circle (\beadsize)
;};
\foreach \x/\y in
{1/8, 1/9, 1/10, 0/1, 1/5, 0/6, 0/7,0/8,0/9,0/10,0/11}
{\node [scale=\labelsize*\globalscaling] at (\xdiagramshift+\x,\ydiagramshift+\y) {\gap};};
\foreach\x/\y/\z in
{0/2/\abovepivot,
0/3/\abovepivot,
0/4/\abovepivot,
0/5/\pivot,
1/6/\abovepivot,
1/7/\abovepivot,
1/8/\pivot,
1/9/\belowpivot,
1/10/\belowpivot,
1/11/\belowpivot
}
{\draw [\z, line width=\beadmovementwidth*\globalscaling,->, >=latex] (\xdiagramshift+\x,\ydiagramshift+\y-\arrowshortening)
-- (\xdiagramshift+\x,\ydiagramshift+\y-1+\arrowshortening);
};

\def\xdiagramshift{26};
\def\ydiagramshift{0};
\draw (\xdiagramshift+.5, \ydiagramshift+\runnerend+.7)
node[scale=\partitionlabelsize*\globalscaling] {$\lambda^{2}$};
\foreach \x/\y in
{
1/0,  1/1,     1/2, 1/3, 1/4, 1/5,        1/7, 1/11,
0/0,              0/2, 0/3, 0/4,        0/6
}
{\draw
[fill] (\xdiagramshift+\x,\ydiagramshift+\y) circle (\beadsize)
;};
\foreach \x/\y in
{1/8, 1/9, 1/10, 0/1, 0/5, 1/6, 0/7,0/8,0/9,0/10,0/11}
{\node [scale=\labelsize*\globalscaling] at (\xdiagramshift+\x,\ydiagramshift+\y) {\gap};};
\foreach\x/\y/\z in
{0/2/\abovepivot,
0/3/\abovepivot,
0/4/\abovepivot,
0/5/\pivot,
0/6/\belowpivot,
1/7/\abovepivot,
1/8/\pivot,
1/9/\belowpivot,
1/10/\belowpivot,
1/11/\belowpivot
}
{\draw [\z, line width=\beadmovementwidth*\globalscaling,->, >=latex] (\xdiagramshift+\x,\ydiagramshift+\y-\arrowshortening)
-- (\xdiagramshift+\x,\ydiagramshift+\y-1+\arrowshortening);
};

\def\xdiagramshift{29};
\def\ydiagramshift{0};
\draw (\xdiagramshift+.5, \ydiagramshift+\runnerend+.7)
node[scale=\partitionlabelsize*\globalscaling] {$\lambda^{3}$};
\foreach \x/\y in
{
1/0,  1/1,     1/2, 1/3, 1/4,  1/5,      1/6,  1/11,
0/0,              0/2, 0/3, 0/4, 0/7
}
{\draw
[fill] (\xdiagramshift+\x,\ydiagramshift+\y) circle (\beadsize)
;};
\foreach \x/\y in
{1/7, 1/8, 1/9, 1/10, 0/1, 0/5, 0/6,0/8,0/9,0/10,0/11}
{\node [scale=\labelsize*\globalscaling] at (\xdiagramshift+\x,\ydiagramshift+\y) {\gap};};
\foreach\x/\y/\z in
{0/2/\abovepivot,
0/3/\abovepivot,
0/4/\abovepivot,
0/5/\pivot,
0/6/\belowpivot,
0/7/\belowpivot,
1/8/\pivot,
1/9/\belowpivot,
1/10/\belowpivot,
1/11/\belowpivot
}
{\draw [\z, line width=\beadmovementwidth*\globalscaling,->, >=latex] (\xdiagramshift+\x,\ydiagramshift+\y-\arrowshortening)
-- (\xdiagramshift+\x,\ydiagramshift+\y-1+\arrowshortening);
};

\def\xdiagramshift{32};
\def\ydiagramshift{0};
\draw (\xdiagramshift+.5, \ydiagramshift+\runnerend+.7)
node[scale=\partitionlabelsize*\globalscaling] {$\lambda^{4}$};
\foreach \x/\y in
{
1/0,  1/1,     1/2, 1/3, 1/4,        1/6,  1/7, 0/11,
0/0,              0/2, 0/3, 0/4, 1/5
}
{\draw
[fill] (\xdiagramshift+\x,\ydiagramshift+\y) circle (\beadsize)
;};
\foreach \x/\y in
{1/8, 1/9, 1/10, 0/1, 0/5, 0/6, 0/7,0/8,0/9,0/10,0/11}
{\node [scale=\labelsize*\globalscaling] at (\xdiagramshift+\x,\ydiagramshift+\y) {\gap};};
\foreach\x/\y/\z in
{0/2/\abovepivot,
0/3/\abovepivot,
0/4/\abovepivot,
0/5/\pivot,
0/6/\belowpivot,
0/7/\belowpivot,
0/8/\belowpivot,
0/9/\belowpivot,
0/10/\belowpivot,
0/11/\belowpivot
}
{\draw [\z, line width=\beadmovementwidth*\globalscaling,->, >=latex] (\xdiagramshift+\x,\ydiagramshift+\y-\arrowshortening)
-- (\xdiagramshift+\x,\ydiagramshift+\y-1+\arrowshortening);
};
\end{tikzpicture}

\label{figure:exceptionalfamily}
\end{figure}
\end{example}

\begin{lemma} \label{L:exceptional-1-increasing}
Let $\tlambda \in \tb$ be an exceptional partition, and suppose that $\tlambda$ is $1$-increasing.  Then $\tlambda$ belongs to a hook-quotient exceptional family $\flambda$.
\end{lemma}

\begin{proof}
By Lemma~\ref{lemma:1unram}(1,3), $\tlambda$ is a hook-quotient partition and has exactly one removable bead on runner $a$ (and $k+1$ addable beads on runner $a-1$.  In addition, denote the partition obtained from $\tlambda$ by moving the unique removable bead on runner $a$ to its preceding unoccupied position as $\clambda$, the $(a-1)$-th and $a$-th components of the $e$-quotient of $\clambda$ has the form $((x),(1^y))$.  Clearly, $\clambda$ is hook-quotient, and $E(\clambda)=0$.  Thus it
follows from Lemma~\ref{L:hookquotientfamilycriteria} that $\clambda$ generates a hook-quotient exceptional family $\flambda$ and that $\tlambda$ belongs to $\flambda$.
\end{proof}

\begin{defi}
	Define $\y = \min\{ x \equiv_e a \mid x \notin \beta(\tilde{\kappa}) \}$ (where $\tilde{\kappa}$ is the $e$-core of $\tb$).  Define the subset $\mathcal{R}$ of $\mathbb{Z}$ as follows:
	$$
	\mathcal{R} = \{ \y + \gamma e \mid \gamma \in [0,k] \} \cup \{ \y-1 + \gamma e \mid \gamma \in [0,k] \}.
	$$
\end{defi}

The subset $\mathcal{R}$ is illustrated in Figure~\ref{figure:exceptionalfamily} as the region in the abacus enclosed by the dash-outlined rectangular boxes.

We record some straightforward
properties of hook-quotient exceptional families that
will be useful to us.

\begin{lemma} \label{L:region}
	Let $\flambda$ be a hook-quotient exceptional family, generated by $\clambda$.
	\begin{enumerate}
		\item For $\eptl{j}$ ($j \in [0,\,k+1]$): Let $x \in \mathcal{R}$ with $x \equiv_e a$, say $x = \y+\gamma e$, where $\gamma \in [0,\, k]$.
		\begin{enumerate}
            \item $\eptl{j}$ has a bead movement starting at $x$ or $x-1$, but not both.  More precisely, it has a bead movement at $x$ if $\gamma < k+1 -j$, and at $x-1$ if $\gamma \geq k+1-j$.  This is its $\tilde{i}_{\flambda,\gamma}$-th bead movement, where $\tilde{i}_{\flambda,\gamma}$ is dependent of $\flambda$ and $\gamma$ but independent of $j$.  Furthermore,
                \begin{gather*}
				\tilde{i}_{\flambda,0} \prec_{\eptl{j}} \tilde{i}_{\flambda,1} \prec_{\eptl{j}} \dotsb \prec_{\eptl{j}} \tilde{i}_{\flambda,k-j}, \\
				\tilde{i}_{\flambda, k} \prec_{\eptl{j}} \tilde{i}_{\flambda,k-1} \prec_{\eptl{j}} \dotsb \prec_{\eptl{j}} \tilde{i}_{\flambda,k+1-j}.
			\end{gather*}
            \item If $\gamma \ne k$, then $\eptl{j}$ has a bead at $x$ or $x-1$, but not both.  More precisely, $\eptl{j}$ has a bead at $x-1$ if $\gamma \ne k-j$, and at $x$ if $\gamma = k-j$.
            \item If $\gamma = k$, then $\eptl{j}$ has no bead at $x$ if $j>0$ and no bead at $x-1$ if $j=0$.
		\end{enumerate}
		\item For $\clambda$:
		\begin{enumerate}
			\item We have $x \notin \beta(\clambda)$ for all $x \geq_e \y$ and $|\{ x <_e \y \mid x \notin \beta(\clambda)\}| = 1$.
			\item We have $x \in \beta(\clambda)$ for all $x <_e \y-1 + ke$ and $|\{ x \geq_e \y -1 + ke \mid x \in \beta(\clambda) \}|= 1$.
			\item $\clambda$ has no bead movements starting at any $x \in \mathcal{R}$.
		\end{enumerate}
	\end{enumerate}
\end{lemma}

There are of course statements analogous to Lemma \ref{L:region} about $\lambda^j \in B$ and $\hat{\lambda} \in \hb$ as well.  In fact, the $i_{\flambda,\gamma}$ for $\lambda^j$ (in the analogue of Lemma \ref{L:region}(1a)) equals $\tilde{i}_{\flambda,\gamma}$.  This leads to the following definition:

\begin{defi}
Let $\flambda$ be a hook-quotient exceptional family.  Define the subsets $\internal{\flambda}$ and $\external{\flambda}$ of $[1,\,w]$ as follows:
\begin{align*}
\internal{\flambda} &= \{i_0,i_1,\dotsc, i_{k} \} ,\\
\external{\flambda} & = [1,\,w] \setminus \internal{\flambda},
\end{align*}
where $i_\gamma = \tilde{i}_{\flambda,\gamma}$ of Lemma \ref{L:region}(1a) for all $\gamma \in [0,\,k]$.

We call $\internal{\flambda}$ and $\external{\flambda}$ the sets of {\em internal} and {\em external coordinates} of the family $\flambda$ respectively.
\end{defi}

\begin{rem}
Whenever we write $\internal{\flambda} = \{i_0,\dotsc, i_k\}$, we always assume that $i_0 < i_1 < \dotsb < i_k$ in the natural order of integers, so that the $i_{\gamma}$-th bead movement of every $\eptl{j}$ and every $\epl{j}$ starts at either $\y + \gamma e$ or $\y-1 + \gamma e$.
\end{rem}

Let $\flambda$ be a hook-quotient exceptional family with $\Int(\flambda) = \{i_0,\dotsc, i_k\}$.
Define $\pb{\flambda}_0, \pb{\flambda}_1, \dotsc, \pb{\flambda}_{k+1} \in \BZ^w$ as follows:
$$\pb{\flambda}_0=-\ssb_{i_0},\quad \pb{\flambda}_{1}=\ssb_{i_0}-\ssb_{i_1},\quad\dotsc,\quad\pb{\flambda}_k=
\ssb_{i_{k-1}}-\ssb_{i_{k}},\quad\pb{\flambda}_{k+1}=\ssb_{i_{k}}.$$
The sum of these $k+2$ vectors is zero, and any $k+1$ of them form a basis of the $\BZ$-span of $\{ \ssb_{i_{j}} \mid j \in [0,\,k]\}$.  For any subset $J$ of $[0,\, k+1]$, write $\pb{\flambda}_J$ for $\sum_{j \in J} \pb{\flambda}_j$. Then
$\pb{\flambda}_J = \pb{\flambda}_{J'}$ for two distinct subsets $J$ and $J'$ of $[0,\,k+1]$ if and only if $\{J,J'\} = \{ \emptyset, [0,\, k+1] \}$.

The following lemma is straightforward, and can be easily verified.

\begin{lemma}\label{lemma:paraexceptional}
Let $\flambda$ be a hook-quotient exceptional family, with $\Int(\flambda) = \{i_0,\dotsc, i_{k}\}$.
\begin{enumerate}
\item We have
\begin{alignat*}{3}
\z(\epl{0})&=\z(\eptl{0}), \\
\z(\epl{j})&=\z(\epl{0})-\ssb_{i_{j-1}} &&= \z(\epl{0}) + \pb{\flambda}_{[0,j)}&\qquad &(j \in [1,\,k+1]), \\
\z(\eptl{j})&= \z(\eptl{0}) - \ssb_{i_{k+1-j}} &&= \z(\eptl{0}) - \pb{\flambda}_{(k+1 -j,k+1]} &&(j \in [1,\, k+1]).
\end{alignat*}
In particular, $\z(\epl{j}) = \z(\eptl{k+2-j})$ for all $j \in [1,\, k+1]$.
\item For any $x\in \external{\flambda}$,
 $\sb_x^{\lambda}$ is constant for all $\lambda\in\flambda$; denote the common value by $\sb_x^\flambda$, and further write $\sb_{X}^{\flambda}$ for $\sum_{x\in X} \sb_x^{\flambda}$ whenever $X \subseteq \external{\flambda}$.
\item We have
$$\sb_{i_\gamma}^{\epl{j}}=
\begin{cases}
-\pb{\flambda}_{\gamma} & \text{ if } \gamma < j, \\
\pb{\flambda}_{\gamma+1} & \text{ if } \gamma \geq j,
\end{cases}
\qquad \text{and} \qquad
\sb_{i_\gamma}^{\eptl{j}}=
\begin{cases}
-\pb{\flambda}_{\gamma} & \text{ if } \gamma < k+1-j, \\
\pb{\flambda}_{\gamma+1} & \text{ if } \gamma \geq k+1-j.
\end{cases}
$$
In particular, $\para_0(\epl{j}) = \para_0(\eptl{k+1-j})$.
\item Let $z \in \BZ^w$.  Then $z \in \para_0(\epl{j})$ ($= \para_0(\eptl{k+1-j})$ by part (3))  if and only if $$z = -\pb{\flambda}_{I_{<j}} + \pb{\flambda}_{I_{>j}} + \sb_X^{\flambda}$$ for some $I_{<j} \subseteq [0,\, j)$, $I_{>j} \subseteq (j,\, k+1]$ and $X \subseteq \Ext(\flambda)$.
\item Let
$$\para(\flambda) = \left\{ \z(\epl{0}) + \pb{\flambda}_J + \sb_X^{\flambda} \mid J \subseteq [0,\, k+1],\ X \subseteq \external{\flambda} \right\}.$$
Then the cardinality of $\para(\flambda)$ is
$2^{w-k-1}\left(2^{k+2}-1\right)$ and
$$\bigcup_{j=0}^{k+1}\para(\epl{j}) \quad = \quad \para(\flambda)
\quad = \quad
\bigcup_{j=0}^{k+1}\para(\eptl{j}).
$$
\end{enumerate}
\end{lemma}

\begin{defi}
Let $\flambda$ be a hook-quotient exceptional family.  Denote by $\xi_{\flambda} : [1,\, \w] \to \mathbb{Z}$ the unique order-preserving injection whose image is $\external{\flambda}$, and write $\zeta_{\flambda} : \external{\flambda} \to [1,\,\w]$ for its left inverse.

In addition, let $\pi_{\flambda} : \mathbb{Z}^w \to \mathbb{Z}^{\w}$ be the $\mathbb{Z}$-linear map defined by
 $$
 \pi_{\flambda}(\ssb_r) = \begin{cases}
 \ssb_{\zeta_{\flambda}(r)}, &\text{if } r \in \external{\flambda};\\
 0, &\text{otherwise}.
 \end{cases}
 $$
\end{defi}

Thus $\pi_{\flambda}(a_1,\dotsc, a_w) = (a_{\xi_{\flambda}(1)}, \dotsc, a_{\xi_{\flambda}(\w)})$.

Clearly, the following statements, about two hook-quotient exceptional families $\flambda$ and $\fsigma$, are equivalent:
\begin{enumerate}
\item[(i)] $\Int(\flambda) = \Int(\fsigma)$.
\item[(ii)] $\Ext(\flambda) = \Ext(\fsigma)$.
\item[(iii)] $\xi_{\flambda} = \xi_{\fsigma}$.
\item[(iv)] $\pi_{\flambda} = \pi_{\fsigma}$.
\end{enumerate}

\begin{lemma} \label{L:projection}
Let $\flambda$ be a hook-quotient exceptional family, and let $j \in [0,k+1]$.
\begin{enumerate}
\item If $i\in\internal{\flambda} $ and
  $x\in\external{\flambda}$, then
 $ i\not\succeq_{\eptl{j}} x$.
\item We have 
$\pi_{\flambda}(\z(\eptl{j})) = \z(\clambda)$;  in particular, 
$\check{\lambda}$ is at least as increasing as 
$\eptl{j}$.
\item We have $$
\pi_{\flambda}(\epsilon_r^{\eptl{j}}) =
\begin{cases}
\epsilon_{\zeta_{\flambda}(r)}^{\check{\lambda}}, &\text{if } r \in \external{\flambda},\\
 0, &\text{otherwise};
 \end{cases}
 $$
 \item For all $s,t \in \Ext(\flambda)$, we have $\zeta_{\flambda}(s) \succeq_{\check{\lambda}} \zeta_{\flambda}(t)$ if and only if $s \succeq_{\eptl{j}} t$.
\end{enumerate}
\end{lemma}

\begin{proof}
Part (1) is clear.  The other assertions can be easily checked when $j \ne 0,k+1$, as in this case $\bm(\clambda) = \{ (b;q) \in \bm(\eptl{j}) \mid q \notin \mathcal{R} \}$.  The only non-trivial part is the image under $\pi_{\flambda}$ of $\sb^{\eptl{j}}_r$ when $\sb^{\eptl{j}}_r = \ssb_r - \ssb_s$ with $r \in \external{\flambda}$ and $s \in \internal{\flambda}$.  In this instance, the $r$-th bead movement of $\eptl{j}$ starts at either $\y -e$ or $\y+(k+1)e-1$, and the $s$-th bead movement of $\eptl{j}$ is the final bead movement of the bottom bead on runner $a$ or $a-1$, starting at $\y$ or $\y-1+ke$.  In either case, this $r$-th bead movement becomes the final bead movement of the bottom bead on that runner when the unique removable bead of $\eptl{j}$ on runner $a$ is moved to its preceding empty position on runner $a-1$ to obtain $\clambda$.  Thus $\pi_{\flambda}(\sb_r^{\eptl{j}}) = \ssb_r = \sb_{\zeta_{\flambda}(r)}^{\clambda}$.  Subsequently, we can easily obtain the assertions for $j = 0, k+1$ by comparing them to those for, say, $j=1$.
\end{proof}

The next lemma shows that, under certain conditions, the change of bead movements when getting from a partition in $\cb$ that generates a hook-quotient exceptional family to a partition one modified basis vector away, do not `cross' the region $\CR$.

\begin{lemma} \label{L:emptyintersectionwithregionlambda}
Let $\clambda,\cmu \in \cb$, and $r \in [1,\w]$.  Suppose that:
\begin{itemize}
\item $\clambda$ is a hook-quotient partition,
\item $\z(\cmu) = \z(\clambda) + \sb_r^{\clambda}$,
\item $\clambda$ generates a hook-quotient exceptional family $\flambda$,
\item $\z(\eptl{c}) + \sb_{\xi_{\flambda}(r)}^{\flambda}$ is $1$-increasing for some $c \in [0,k+1]$.
\end{itemize}
For each $s \in [1,\w]$, let $(b_s;q_s)$ (resp.\ $(b'_s;q_s')$) be the $s$-th bead movement of $\clambda$ (resp.\ $\cmu$).  Then $(\bigcup_{s=1}^{\w} [q_s,\,q'_s]) \cap \mathcal{R} = \emptyset$.
\end{lemma}

\begin{proof}
Let $\z(\eptl{c}) + \sb_{\xi_{\flambda}(r)}^{\eptl{c}} = (m_1,\dotsc, m_w)$.  Then
$$
\z(\cmu) = \z(\clambda) + \sb_r^{\clambda} = \pi_{\flambda} (\z(\eptl{c}) + \sb_{\xi_{\flambda}(r)}^{\eptl{c}} ) = \pi_{\flambda}(m_1,\dotsc, m_w) = (m_{\xi_{\flambda}(1)},\dotsc, m_{\xi_{\flambda}(\w)})$$
by Lemma \ref{L:projection}(2,3), so that $\z_{\cmu}(q_i') = m_{\xi_{\flambda}(i)}$.

Let $X := \bigcup_{s=1}^{\w} [q_s,\,q'_s]$.  Assume for the sake of contradiction that there exists $x \in X \cap \mathcal{R}$, say $x \in [q_t,\,q_t'] \cap \mathcal{R}$.  Since $q_t \notin \mathcal{R}$ by Lemma \ref{L:region}(2c), we have $q_t < x \leq q_t'$, so that $t \succeq_{\clambda} r$ by Proposition \ref{P:movealongdescription}(1), and hence $\xi_{\flambda}(t) \succeq_{\eptl{c}} \xi_{\flambda}(r)$ by Lemma \ref{L:projection}(4). Also, $q_t \not\equiv_e a-1$ by Lemma \ref{L:region}(2b).  Thus, by replacing $x$ with $x-1$ if necessary, we may assume that $x \equiv_e a-1$.

By Lemmas \ref{L:region}(2(a,b)) and \ref{L:easy} and Proposition \ref{P:movealongdescription}(2b), we have
\begin{align*}
\z_{\clambda}(x) &= \z_{\clambda}(q_t) + \I_{x = \y-1+ke \notin \beta(\clambda)}, \\
\z_{\clambda}(x+1) &= \z_{\clambda}(x) + \I_{x - e+1 = \y-e \in \beta(\clambda)}.
\end{align*}
Let $\eptl{c}$ be obtained by moving the bead of $\clambda$ at $c_{\flambda}$ to $c_{\flambda}+1$, so that $\beta(\eptl{c}) = \beta(\clambda) \cup\{c_{\flambda}+1\} \setminus \{c_{\flambda}\}$.  Thus,
\begin{align*}
\z_{\eptl{c}}(x) &= \z_{\clambda}(x) + \delta_{x,c_{\flambda}} - \delta_{x,c_{\flambda}+e} \\
&= \z_{\clambda}(q_t)+ \delta_{x,c_{\flambda}} - \delta_{x,c_{\flambda}+e} + \I_{x = \y-1+ke \notin \beta(\clambda)} \\
&= \z_{\eptl{c}}(q_t) + \delta_{x,c_{\flambda}} - \delta_{x,c_{\flambda}+e} + \I_{x = \y-1+ke \notin \beta(\clambda)}. 
\end{align*}
Similarly, $\z_{\eptl{c}}(x+1) = \z_{\eptl{c}}(q_t) + \I_{x -e+1= \y-e \in \beta(\clambda)} + \I_{x = \y-1+ke \notin \beta(\clambda)}$.
By Lemma \ref{L:region}(1a), $\eptl{c}$ has a bead movement starting at $x$ or $x+1$, but not both; let this bead movement be the $j$-th one, starting at $\bar{x}$.  Then $j \in \internal{\flambda}$, and $j > \xi_{\flambda}(t)$ since the $\xi_{\flambda}(t)$-th bead movement of $\eptl{c}$ starts at $q_t$ which is less than $x$.  Thus, from the definition $m_j$, we get
\begin{align*}
m_{j} &\leq  \z_{\eptl{c}}(\bar{x}) + \delta_{{j},\xi_{\flambda}(r)} = \z_{\eptl{c}}(\bar{x}) \\
&= \z_{\eptl{c}}(q_t) + \delta_{\bar{x},c_{\flambda}} - \delta_{\bar{x},c_{\flambda}+e} + \I_{x = \y-1+ke \notin \beta(\clambda)} + \I_{\bar{x}-e=\y-e \in \beta(\clambda)}  \\
&\leq \z_{\eptl{c}}(q_t) + \I_{x = \y-1+ke \notin \beta(\clambda)} + \I_{\bar{x}-e=\y-e \in \beta(\clambda)} \\
&= m_{\xi_{\flambda}(t)} - \delta_{\xi_{\flambda}(t),\xi_{\flambda}(r)} + \I_{x = \y-1+ke \notin \beta(\clambda)} + \I_{\bar{x}-e= \y-e  \in \beta(\clambda)} \\
&= m_{\xi_{\flambda}(t)} - \delta_{tr} + \I_{x = \y-1+ke \notin \beta(\clambda)} + \I_{\bar{x}-e= \y-e  \in \beta(\clambda)},
\end{align*}
where the third line follows since $\bar{x} \ne c_{\flambda}$ by Lemma \ref{L:region}(1a), while the fourth line follows since $\xi_{\flambda}(t) \succeq_{\eptl{c}} \xi_{\flambda}(r)$ as noted earlier.  Since $(m_1,\dotsc, m_w)$ is $1$-increasing (and $j < \xi_{\flambda}(t)$), this implies that $t \ne r$, and one of the following two mutually exclusive events must occur: $x = \y-1+ke \notin \beta(\clambda)$ or $\bar{x} -e = \y-e \in \beta(\clambda)$.  From this, we also conclude that $\y-1+\gamma e \notin X \cap \mathcal{R}$ for all $\gamma \in [1,\,k)$.

Let $t^-$ be maximal with respect to $\succeq_{\clambda}$ subject to $t \succ_{\clambda} t^- \succeq_{\clambda} r$.  Then $q_t - q_{t^-} \in \{ \pm e\}$.  By Proposition \ref{P:movealongdescription}(3), $q'_{t^-} - q_{t^-} \geq q'_t - q_t$, and so $q_t < x \leq q'_{t}$ implies that
$$q_{t^-} = q_t - (q_t - q_{t^-}) < x- (q_t - q_{t^-})  \leq q'_t - (q_t - q_{t^-})  \leq q'_{t^-}.$$
Thus $x-(q_t-q_{t^-}) \in (q_{t^-},\, q'_{t^-}] \subseteq X$.
Since $\y-1+\gamma e \notin X \cap \mathcal{R}$ for all $\gamma \in [1,\,k)$, this shows that $q_{t^-} = q_{t} +e$ if $x = \y-1+ke$, and $q_{t^-} = q_{t} -e$ if $x = \y-1$.

Suppose first that $x = \y-1+ke \notin \beta(\clambda)$ (and hence $q_{t^-} = q_t +e$).  By Lemma \ref{L:region}(2b), $\clambda$ has a bead movement, say its $u$-th one, starting at $x+e$, which is the final bead movement of the bottom bead on runner $a-1$.  Since $q_r \equiv_e q_t \not\equiv_e a-1 \equiv_e x+e = q_u$, we have $u \not\succeq_{\clambda} r$, and so $q_u' = q_u$ by Proposition \ref{P:movealongdescription}(1).   Since $q_u = x+e = x-(q_t-q_{t^-}) \in (q_{t^-},\,q'_{t^-}]$, we also have $q_u = q'_{t^-}$ by Proposition \ref{P:movealongdescription}(5b).  Thus $m_{\xi_{\flambda}(u)} = \z_{\cmu}(q_u') = \z_{\cmu}(q_u) = \z_{\cmu}(q'_{t^-}) = m_{\xi_{\flambda}(t^-)}$, contradicting $(m_1,\dotsc, m_w)$ being $1$-increasing.

Now suppose that $\bar{x}-e = \y-e \in \beta(\clambda)$.  Then $x = \y-1$ and hence $q_{t^-} = q_t -e$ as noted earlier. By Lemma \ref{L:region}(2a), $\clambda$ has a bead movement $(\y-e;\y-e)$, which is the only bead movement of the bottom bead on runner $a$.  If this is its $v$-th bead movement, then $v$ is minimal with respect to $\succeq_{\clambda}$.  If $v \succeq_{\clambda} r$, then $r = v$, and so $q_{t^-} = q_t + e$ since $t \succ_{\clambda} t^- \succeq_{\clambda} v$, a contradiction.  Thus $v \not\succeq_{\clambda} r$ and so $q'_v = q_v$ by Proposition \ref{P:movealongdescription}(1).  Since $\y-e-1 = x - e = x - (q_t - q_{t^-}) \in (q_{t^-}, q'_{t^-}]$, and $q'_{t^-} \ne \y-e-1$ by Proposition \ref{P:movealongdescription} (as $\y-e-1 \in \beta(\clambda)$ by Lemma \ref{L:region}(2b) and $q_r \leq_e q_{t^-}$), we have $q_v = \y-e = x-e+1 \in (q_{t^-}, q'_{t^-}]$.  Thus $q_v = q'_{t^-}$ by Proposition \ref{P:movealongdescription}(5b), and hence $m_{\xi_{\flambda}(v)} = \z_{\cmu}(q_v') = \z_{\cmu}(q_v) = \z_{\cmu}(q'_{t^-}) = m_{\xi_{\flambda}(t^-)}$, contradicting $(m_1,\dotsc, m_w)$ being $1$-increasing.  This completes the proof.
\end{proof}

We also have some sort of a converse to Lemma \ref{L:emptyintersectionwithregionlambda}.

\begin{lemma} \label{L:emptyintersectionwithregionmu}
Let $\clambda,\cmu \in \cb$, and $r \in [1,\w]$.  Suppose that:
\begin{itemize}
\item $\clambda$ is a hook-quotient partition,
\item $\z(\cmu) = \z(\clambda) + \sb_r^{\clambda}$,
\item $\cmu$ generates a hook-quotient exceptional family $\fmu$,
\item $\z(\eptm{c})$ for some $c \in [0,k+1]$ is $1$-increasing.
\end{itemize}
For each $s \in [1,\w]$, let $(b_s;q_s)$ (resp.\ $(b'_s;q_s')$) be the $s$-th bead movement of $\clambda$ (resp.\ $\cmu$).  Then $(\bigcup_{s=1}^{\w} [q_s,\,q'_s]) \cap \mathcal{R} = \emptyset$.
%
\end{lemma}

\begin{proof}
Let $\z(\eptm{c}) = (m_1,\dotsc, m_w)$.  Then $\z(\cmu) = (m_{\xi_{\fmu}(1)},\dotsc, m_{\xi_{\fmu}(\w)})$ by Lemma \ref{L:projection}(2), so that $\z_{\cmu}(q'_s) = m_{\xi_{\fmu}(s)}$ for all $s \in [1,\,\w]$.

Let $T = \{ s \in [1,\,\w] \mid [q_s,\,q'_s] \cap \mathcal{R} \ne \emptyset\}$.  Suppose for the sake of contradiction that $T \ne \emptyset$.  Take $t \in T$, and let $x_t := \max([q_t,\,q'_t] \cap \mathcal{R})$.  Since $q'_t \notin \mathcal{R}$ by Lemma \ref{L:region}(2c), we have $x_t \equiv_e a$ and $q_t \leq x_t < q'_t$, so that $t \succeq_{\clambda} r$ by Proposition \ref{P:movealongdescription}(1).  Also, by Lemma \ref{L:region}(2a), $q_t' \not\equiv_e a$.

By Lemma \ref{L:region}(1a), $\eptm{c}$ has a bead movement starting at $\bar{x}_t$, where $\bar{x}_t \in \{x_t,x_t-1\}$, and let this bead movement be its $j_t$-th one.  Then $j_t < \xi_{\fmu}(t)$, since $\bar{x}_t < q'_t$.  Let $c_{\fmu} \in \beta(\cmu)$ be such that $\beta(\eptm{c}) = \beta(\cmu) \cup \{ c_{\fmu}+1 \} \setminus \{ c_{\fmu}\}$.  By Lemma \ref{L:region}(1(a,b)), we see that $\bar{x}_t \ne c_{\fmu}$.  Thus
\begin{align*}
m_{j_t} &= \z_{\eptm{c}}(\bar{x}_t) = \z_{\cmu}(\bar{x}_t) - \delta_{\bar{x}_t, c_{\fmu}+e} \\
&= \z_{\cmu}(x_t) - \I_{\bar{x}_t +1 -e = \y-e \in \beta(\cmu)} - \delta_{\bar{x}_t, c_{\fmu}+e} \\
&= \z_{\cmu}(q'_t) - (\sum_{a = x_t + 1}^{q'_t} J_{\cmu}(a)) - \I_{\bar{x}_t +1 -e = \y-e \in \beta(\cmu)} - \delta_{\bar{x}_t, c_{\fmu}+e} \\
&= m_{\xi_{\fmu}(t)} - (\sum_{a = x_t + 1}^{q'_t} J_{\cmu}(a))  - \I_{\bar{x}_t +1 -e = \y-e \in \beta(\cmu)} - \delta_{\bar{x}_t, c_{\fmu}+e}
\end{align*}
where the second line follows from Lemma \ref{L:region}(2a), and
$$J_{\cmu}(a) = \z_{\cmu}(a) - \z_{\cmu}(a-1) = \I_{a \notin \beta(\cmu),\, a-e\in \beta(\cmu)} - \I_{a \in \beta(\cmu),\, a-e\notin \beta(\cmu)}.$$

We claim that $\sum_{a = x_t+1}^{q'_t} J_{\cmu}(a) = |N_t| -1$, where
$$
N_t = \{ s \succ_{\clambda} t \mid q_s - q_t \in \{ \pm e\},\ q'_s + (q_t-q_s) \in (x_t,\, q'_t] \}.
$$
By Proposition \ref{P:movealongdescription}(2b), we see that $J_{\csigma}(a) = 0$ for all $a \in (x_t,\,q'_t)$, where $\beta(\csigma) = \beta(\clambda) \cup \{ g_r \} \setminus \{b_r\}$ and $g_r = \max\{ g <_e q_r \mid g \notin \beta(\clambda)\}$, so that $J_{\cmu}(a) \ne 0$ for $a \in (x_t,\, q'_t]$ only if $a \in \{q'_s -e, q'_s,q'_s+e \}$ for some $s \succeq_{\clambda} r$.  Using Proposition \ref{P:movealongdescription}(2a), we deduce that
$$\{ s \succeq_{\clambda} r \mid \{q'_s -e, q'_s,q'_s+e \} \cap (x_t,\,q'_t] \ne \emptyset \} \subseteq \{ s \succeq_{\clambda} r  \mid q_s \in \{ q_t-e,q_t,q_t+e \}\} = \{t_-, t, t_+ \},$$
where $q_{t_-} = q_t - e$ and $q_{t_+} = q_t + e$ (if $t_-$ or $t_+$ does not exist, then treat any expression involving it as vacuous), and so $J_{\cmu}(a) \ne 0$ for $ a \in (x_t,\,q'_t]$ only if $a \in \{ q'_{t_-} + e, q'_t, q'_t - e, q'_{t_+} - e \}$.
We have the following three cases:
\begin{enumerate}
\item[(i)] $t_+$ exists and $t \succ_{\clambda} t_+$;
\item[(ii)] $t_-$ exists and $t \succ_{\clambda} t_-$;
\item[(iii)] $t_{\pm} \succ_{\clambda} t$.
\end{enumerate}
We shall only verify the claim when (i) holds, the others being similar.  In this case, $q_t',q'_{t_-} \in \beta(\cmu)$, and $q'_t -e \in \beta(\cmu)$ if and only if ($t_-$ exists and) $q'_t -e = q'_{t_-}$ if and only if ($t_-$ exists and) $q'_{t_-} + e \in \beta(\cmu)$.  By Proposition \ref{P:movealongdescription}(2a,3), $q'_t \in (q_t,\,q_t+e]$ and $q_t'-q_t \leq q_{t_+}'-q_{t_+}$.  Thus $q'_t - e \leq q_t \leq x_t$, and $q_{t_+}' -e \geq q_t'$, so that $q'_t-e \notin (x_t,\,q'_t]$ while $q'_{t_+}-e \in (x_t,\,q'_t]$ if and only if $q_{t_+}' -e = q_t'$.
Consequently, 
$$
J_{\cmu}(q'_{t_-}+e) =
\begin{cases}
1, &\text{if } q_t'\ne q_{t_-}' +e; \\
0, &\text{otherwise,}
\end{cases} \quad \text{ and } \quad
J_{\cmu}(q'_t) =
\begin{cases}
0, &\text{if } q'_{t_-} = q_t' - e; \\
-1, &\text{otherwise.}
\end{cases}$$
Therefore $\sum_{a = x_t + 1} ^{q'_t} J(a,\cmu) = \I_{q'_{t_-} +e\in (x_t,q'_t]} - 1$, and the claim is verified in this case.

Assuming the claim, we have
\begin{equation} \label{E:mudifference}
m_{\xi_{\flambda}(t)} - m_{j_t} = |N_t| - 1 + \I_{\bar{x}_t +1 -e = \y-e \in \beta(\cmu)} + \delta_{\bar{x}_t, c_{\fmu}+e} > 0,
\end{equation}
since $(m_1,\dotsc, m_w)$ is $1$-increasing (and $j_t < \xi_{\flambda}(t)$).
Note that $\I_{\bar{x}_t+1-e= \y-e \in \beta(\cmu)} + \delta_{\bar{x}_t,c_{\fmu} + e} \leq 1$ (as $\bar{x}_t+1 = \y$ implies that $\bar{x}_t \ne c_{\fmu}+e$); thus $N_t \ne \emptyset$.

Now, we assume that $t$ is maximal with respect to $\succeq_{\lambda}$ in $T$, and let $s \in N_{t}$. Then $q'_s + (q_t-q_s)> x_t$, so that $q_s \leq x_{t} +(q_s- q_t) <  q'_s$.
Thus if $x_{t} + (q_s - q_t) \in \mathcal{R}$, then $s \in T$, contradicting the maximality of $t$.
Hence $x_{t} + (q_s - q_t) \notin \mathcal{R}$.
This implies, since $q_s -q_{t} \in \{ \pm e\}$, that either $x_{t} = \y$ and $x_{t}+(q_s-q_{t}) = \y-e$, or $x_{t} = \y+ke$ and $x_{t}+(q_s-q_{t}) = \y+ (k+1)e$.

In the former case, we have $\y \in [q_{t},\, q'_{t})$ and $\y-e \in [q_s,\,q'_s)$.
Since $\y \notin \beta(\cmu)$ while $\y-e$ or $\y-2e \in \beta(\cmu)$ by Lemma \ref{L:region}(2a), this forces $\y = q_{t}$ and $\y-e = q_s$ by Proposition \ref{P:movealongdescription}(2b).
But $|\{ x \in \beta(\clambda) \mid x \geq_e \y-e \}| \leq 1$ since $\clambda$ is a hook-quotient partition (and $x \notin \beta(\ckappa)$ for all $x \geq_e \y-e$).  Thus the two bead movements of $\clambda$ starting at $q_s = \y-e$ and $q_{t} = \y$ both belong to the bottom bead on runner $a$ (i.e.\ $b_s = b_t$), so that $t \succ_{\clambda} s$, a contradiction.

In the latter case, we have $\y+ke \in [q_{t},\, q'_{t})$ and $\y+(k+1)e = [q_s,\, q'_s)$.  Similar arguments as before but applied to runner $a-1$ will show that $q_t \geq \y+ke-1$ and $q_t \ne \y+ke-1$, forcing $q_{t} = \y+ke$ and $q_s = \y+(k+1)e$.  Thus, by considering the $e$-core $\ckappa$ of $\clambda$, the bottom bead of $\clambda$ on runner $a$ must have bead movements starting at $\y, \y+e, \dotsc, \y+(k+1)e$, which $\cmu$ does not have by Lemma \ref{L:region}(2a).  This shows that this bead is at $b_r$, and $q_r \leq_e \y$, by Proposition \ref{P:movealongdescription}(1).  Now, $N_{t} = \{s\}$, and so \eqref{E:mudifference} forces $c_{\fmu} = \bar{x}_{t} - e = \y+(k-1)e - 1$.  Let $t_1$ be such that $q_{t_1} = \y+(k-1)e$; then $t \succ_{\clambda} t_1 \succeq_{\clambda} r$.  Since $q_{t_1} \in [q_{t_1},\,q'_{t_1}]$, we have $t_1 \in T$, and $\bar{x}_{t_1} = \y+(k-1)e$ since $c_{\fmu} = \y+(k-1)e$.  Replacing $t$ by $t_1$ in \eqref{E:mudifference}, we get, since $N_{t_1} = \{t\}$, $m_{\xi_{\fmu}(t_1)} - m_{j_{t_1}} = 0$, contradicting $\eptm{c}$ is $1$-increasing.
\end{proof}

Under the hypotheses of Lemmas \ref{L:emptyintersectionwithregionlambda} or \ref{L:emptyintersectionwithregionmu}, we have that one of the partitions generates a hook-quotient exceptional family if and only if the other one does, with the same external coordinates:

\begin{cor} \label{C:externalpara}
Let $\clambda,\cmu \in \cb$ and $r \in [1,\,\w]$.  Suppose that one of the hypotheses in Lemmas \ref{L:emptyintersectionwithregionlambda} and \ref{L:emptyintersectionwithregionmu} hold.  Then both $\clambda$ and $\cmu$ generate exceptional families, denoted $\flambda$ and $\fmu$, with $\external{\flambda} = \external{\fmu}$.  Furthermore, $\z(\eptm{j}) = \z(\eptl{j}) + \sb^{\flambda}_{\xi_{\flambda}(r)}$ for all $j \in [0,\,k+1]$.
\end{cor}

\begin{proof}
By Lemmas \ref{L:emptyintersectionwithregionlambda} and \ref{L:emptyintersectionwithregionmu}, we have
\begin{equation} \label{E:emptyintersectionwithregion}
\left(\bigcup_{s=1}^{\w} [q_s,\,q_s']\right) \cap \CR = \emptyset
\end{equation}
where $(b_s;q_s)$ (resp.\ $(b_s';q_s')$) is the $s$-th bead movement of $\clambda$ (resp.\ $\cmu$).  As the hypotheses of Lemmas \ref{L:emptyintersectionwithregionlambda} and \ref{L:emptyintersectionwithregionmu} ensure $\cmu$ is $1$-increasing, it follows that if one of $\clambda$ and $\cmu$ generate a hook-quotient exceptional family, so will the other, with $\external{\flambda} = \external{\fmu}$, by Lemmas \ref{L:hookquotientfamilycriteria} and \ref{L:region}(2).

For the last assertion, it suffices to show that $\z(\eptm{0}) = \z(\eptl{0}) + \sb^{\flambda}_{\xi_{\flambda}(r)}$, since if this holds, then
$$
\z(\eptm{j}) = \z(\eptm{0}) - \mathbf{e}_{i_{k+1-j}} = \z(\eptl{0}) - \mathbf{e}_{i_{k+1-j}} + \sb^{\flambda}_{\xi_{\flambda}(r)} = \z(\eptl{j})+ \sb^{\flambda}_{\xi_{\flambda}(r)} $$ for all $j \in [1,\ k+1]$ by Lemma \ref{lemma:paraexceptional}(1), where $\internal{\flambda} = \{i_0,\dotsc, i_k\} = \internal{\fmu}$.

Let $(\tilde{b}_i;\tilde{q}_i)$ (resp.\ $(\tilde{b}'_i;\tilde{q}'_i)$) be the $i$-th bead movement of $\eptl{0}$ (resp.\ $\eptm{0}$).  We need to show that
\begin{equation} \label{E:eptlm}
\z_{\eptm{0}}(\tilde{q}'_i) = \z_{\eptl{0}}(\tilde{q}_i) + \delta_{i,\xi_{\flambda}(r)} - \delta_{i,(\xi_{\flambda}(r))^-},
\end{equation}
where $(\xi_{\flambda}(r))^- = \max_{\succeq_{\eptl{0}}} \{ n \in [1,\,w] \mid n \prec_{\eptl{0}} \xi_{\flambda}(r) \}$.

By Lemma \ref{L:projection},
$$\z_{\eptm{0}}(\tilde{q}'_{\xi_{\fmu}(s)}) = \z_{\cmu}({q}'_{s}) = \z_{\clambda}(q_s) + \delta_{sr} - \delta_{s,r^-} = \z_{\eptl{0}}(\tilde{q}_{\xi_{\flambda}(s)}) + \delta_{\xi_{\flambda}(s),\xi_{\flambda}(r)} - \delta_{\xi_{\flambda}(s),\xi_{\flambda}(r^-)}
$$
for all $s \in [1,\,\w]$, where $r^- = \max_{\succeq_{\clambda}} \{ t \in [1,\,\w] \mid t \prec_{\clambda} r \}$.  Thus, \eqref{E:eptlm} holds for all $i \in \external{\flambda}$, except possibly when $i = \xi_{\flambda}(r^-) \ne (\xi_{\flambda}(r))^-$.  But as $\xi_{\flambda}(r^-) \ne (\xi_{\flambda}(r))^-$ only when $r^-$ is undefined and $(\xi_{\flambda}(r))^- \in \internal{\flambda}$, \eqref{E:eptlm} indeed holds for all $i \in \external{\flambda}$.

\def\q{\mathfrak{q}}
It remains to consider when $i = i_\gamma \in \internal{\flambda} = \internal{\fmu}$.  Let $\q_{\gamma} = \tilde{q}_{i_{\gamma}}$.  Then $\q_{\gamma} = \tilde{q}'_{i_{\gamma}} = \y+ \gamma e$ by Lemma \ref{L:region}(1a).
We have
\begin{align*}
\z_{\eptm{0}}(\q_{\gamma})
= \z_{\cmu}(\q_{\gamma})
&= \z_{\clambda}(\q_{\gamma})
- \I_{g_r \in (\q_{\gamma} -e,\, \q_{\gamma}]} + \I_{b_r \in (\q_{\gamma} -e,\, \q_{\gamma}]} + n_{\gamma} \\
&= \z_{\eptl{0}}(\q_{\gamma}) 
- \I_{g_r \in (\q_{\gamma} -e,\, \q_{\gamma}]} + \I_{b_r \in (\q_{\gamma} -e,\, \q_{\gamma}]} + n_{\gamma},
\end{align*}
where $g_r = \max\{x <_e q_r \mid x \notin \beta(\clambda) \}$, $n_{\gamma} = |N^-_{\gamma}| - |N_{\gamma}|$ and
\begin{align*}
N_{\gamma} &= \{ s \succeq_{\clambda} r \mid q'_s \in (\q_{\gamma} -e,\, \q_{\gamma}] \}, \\
N^{-}_{\gamma} &= \{ s \succeq_{\clambda} r \mid q'_s -e \in (\q_{\gamma} -e,\, \q_{\gamma}] \}.
\end{align*}
We note the following:
\begin{enumerate}
\item[(i)] $s \in N^-_{\gamma}$ if and only if $s \succeq_{\clambda} r$ and $q_s \in (\q_{\gamma},\, \q_{\gamma}+e-2]$: this is clear by \eqref{E:emptyintersectionwithregion} for $\gamma < k$, and follows from the fact that $q_s' \leq \y+(k+1)e-1$ by (Lemma \ref{L:region}(2) and Proposition \ref{P:movealongdescription}(5)) for $\gamma = k$.
\item[(ii)] $s \in N_{\gamma}$ if and only if $s \succeq_{\clambda} r$ and $q_s \in (\q_{\gamma}-e - \delta_{\gamma,0} , \q_{\gamma})$: this is clear by \eqref{E:emptyintersectionwithregion} for $\gamma > 0$, and follows from the fact that $q_s \geq \y -e$ (by Proposition \ref{P:movealongdescription}(5) and Lemma \ref{L:region}(2)) when $\gamma = 0$.
\item[(iii)] $|N^j_{\gamma}|, |N^{j-}_{\gamma}| \leq 1$: this follows from (i) and (ii), since the difference of two distinct $q_s$'s is a non-zero multiple of $e$.
\end{enumerate}
Since $s \succeq_{\clambda} r$ if and only if $g_r <_e q_s \leq_e b_r$, we thus have
$$
n_{\gamma} = \begin{cases}
1, &\text{if } |N_{\gamma}| = 0, |N^{-}_{\gamma}| = 1 \\
-1, &\text{if } |N_{\gamma}| = 1, |N^{-}_{\gamma}| = 0 \\
0, &\text{otherwise}
\end{cases}
\quad =  \I_{g_r \in (\q_{\gamma} -e,\, \q_{\gamma}-2]} - \I_{b_r \in (\q_{\gamma} -e - \delta_{\gamma,0},\, \q_{\gamma})}.
$$
Hence,
\begin{align*}
\z_{\eptm{0}}(\q_{\gamma})
&= \z_{\eptl{0}}(\q_{\gamma})
- \I_{g_r \in (\q_{\gamma} -e,\, \q_{\gamma}]} + \I_{b_r \in (\q_{\gamma} -e,\, \q_{\gamma}]} + n_{\gamma} \\
&= \z_{\eptl{0}}(\q_{\gamma}) - \I_{\gamma=k, g_r = \y + ke} - \I_{\gamma = 0, b_r = \y-e}
\end{align*}
by Lemma \ref{L:region}(2).
But $i_{\gamma} = (\xi_{\flambda}(r))^-$ if and only if  either ($\gamma = 0$ and $b_r = \y-e$) or ($\gamma=k$ and $g_r = \y-1 + ke$).  Consequently, $\z_{\eptm{0}}(\q_{\gamma}) = \z_{\eptl{0}}(\q_{\gamma}) - \delta_{i_{\gamma},(\xi_{\flambda}(r))^-}$, as desired.
\end{proof}

The next proposition allows us to travel along the `external' modified basis vectors first before the `internal' ones to get to a partition lying in the parallelotope of a member of an exceptional family.

\begin{prop} \label{P:moveinparallelogramtilde}
Let $\tmu\in\tb$ be a $4$-increasing partition, and let $\tlambda \in \tb$ be an exceptional hook-quotient partition.  Suppose that
$\z(\tmu) = \z(\tlambda) + \sb_{\Gamma}^{\tlambda} \in \Pi(\tlambda)$ for some $\Gamma \subseteq [1,\, w]$.
\begin{enumerate}
\item Then $\tlambda$ belongs to a hook-quotient exceptional family $\flambda$, say $\tlambda = \eptl{j}$.
\item There exists a partition $\csigma \in \cb$ such that $\pi_{\flambda}(\z(\tmu))= \z(\csigma)$.
\item The partition $\csigma$ generates a hook-quotient exceptional family $\fsigma$, which satisfies
    \begin{itemize}
    \item $\external{\fsigma} = \external{\flambda}$;
    \item $\z(\tmu)  = \z(\epts{j}) + \sb_{I}^{\epts{j}}$ where $I = \Gamma \cap \internal{\flambda}$;
    \item $\z(\csigma) = \z(\clambda) + \sb_{\zeta_{\flambda}(X)}^{\clambda}$ where $X = \Gamma \cap \external{\flambda}$.
    \end{itemize}
    In particular, $\dist{\eptl{j}}{\tmu} = \dist{\epts{j}}{\tmu} + \dist{\clambda}{\csigma}$.
\end{enumerate}
\end{prop}

\begin{proof}
By Lemma \ref{lemma:shiftby3}, $\tlambda$ is $1$-increasing.  Thus $\tlambda$ belongs to a hook-quotient exceptional family $\flambda$ by Lemma \ref{L:exceptional-1-increasing}.  This proves part (1).

For part (2) and (3) we argue by induction on $n = |X|$ (where $X = \Gamma \cap \external{\flambda}$).
This is clear for $n=0$, with $\fsigma = \flambda$.  For $n>0$, choose $r \in X$ maximal with respect to $\succeq_{\eptl{j}}$.  Then, $r$ is maximal in $\Gamma$ with respect to $\succeq_{\eptl{j}}$ by Lemma~\ref{L:projection}(1).  Thus, by
 Lemma~\ref{L:getcloser},
   $\z(\eptl{j})+\sb_r^{\eptl{j}} = \z(\tnu)$ for some $1$-increasing partition $\tnu$ in $\tb$, and $\z(\tmu) = \z(\tnu) + \sb_{\Gamma \setminus \{r\}}^{\tnu} \in \para(\tnu)$.
By Theorem \ref{thm:goodlabels}, there exists a unique $\cnu \in \cb$ such that $\z(\cnu) = \pi_{\flambda}(\z(\tnu))$.  Then,
$$
\z(\cnu) = \pi_{\flambda}(\z(\tnu)) = \pi_{\flambda}(\z(\eptl{j}) + \sb_r^{\eptl{j}}) = \z(\clambda) + \sb_{\zeta_{\flambda}(r)}^{\clambda}$$ by Lemma \ref{L:projection}(2,3).  Consequently, by Lemma \ref{L:emptyintersectionwithregionlambda}, $\cnu$ generates an exceptional family $\fnu$ with $\external{\fnu} = \external{\flambda}$ 
and $\tnu = \eptn{j}$.
By induction, there exists a hook-quotient exceptional family $\fsigma$ with $\external{\fsigma} = \external{\fnu} = \external{\flambda}$ with $\z(\csigma) = \pi_{\fnu}(\z(\tmu)) = \pi_{\flambda}(\z(\tmu))$, satisfying
\begin{align*}
\z(\tmu) & = \z(\epts{j}) + \sb_{I'}^{\epts{j}}, \\
\z(\csigma) &= \z(\cnu) + \sb_{\zeta_{\fnu}(X')}^{\cnu},
\end{align*}
where
\begin{align*}
I' &= (\Gamma\setminus \{r\}) \cap \internal{\fnu} = \Gamma \cap \internal{\flambda} = I, \\
X' &= (\Gamma \setminus \{r\}) \cap \external{\fnu} = (\Gamma \cap \external{\flambda}) \setminus \{r\} = X \setminus \{r\}.
\end{align*}
Our choice of $r$ ensures that for any $x \in X'$, we have $x \not\succeq_{\eptl{j}} r$ so that $\zeta_{\flambda}(x) \not\succeq_{\clambda} \zeta_{\flambda}(r)$ by Lemma \ref{L:projection}(4), and hence, $\sb_{\zeta_{\fnu}(x)}^{\cnu} = \sb_{\zeta_{\fnu}(x)}^{\clambda} = \sb_{\zeta_{\flambda}(x)}^{\clambda}$ by Proposition \ref{P:movealongdescription}(1).  Thus
\begin{align*}
\z(\csigma) &= \z(\cnu) + \sb_{\zeta_{\fnu}(X')}^{\cnu} = (\z(\clambda) + \sb_{\zeta_{\flambda}(r)}^{\clambda}) + \sb_{\zeta_{\flambda}(X')}^{\clambda} = \z(\clambda) + \sb_{\zeta_{\flambda}(X)}^{\clambda}.
\end{align*}
\end{proof}

We have a converse to Proposition \ref{P:moveinparallelogramtilde}.

\begin{prop} \label{P:impt}
Let $\fsigma$ be a hook-quotient exceptional family, $\clambda \in \cb$ be a hook-quotient partition, $\tilde{\mu} \in \tb$ be $4$-increasing.  Suppose that
\begin{align*}
\z(\tmu) &= \z(\epts{j}) + \sb_I^{\epts{j}}, \\
\z(\csigma) &= \z(\clambda) + \sb_X^{\clambda}
\end{align*}
for some $j$, $I \subseteq \internal{\flambda}$ and $X \subseteq [1,\,\w]$.  Then $\clambda$ generates a hook-quotient exceptional family $\flambda$ with $\external{\flambda} = \external{\fsigma}$, and $$\z(\tmu) = \z(\eptl{j}) + \sb_{I}^{\eptl{j}} +\sb_{\xi_{\flambda}(X)}^{\flambda}.$$
In particular, $\dist{\eptl{j}}{\tmu} = \dist{\epts{j}}{\tmu} + \dist{\clambda}{\csigma}$.
\end{prop}

\begin{proof}
We note first that $\z(\csigma)=\pi_{\fsigma}(\z(\epts{j}))=\pi_{\fsigma}(\z(\tmu))$ by Lemma \ref{L:projection}(2,3).  Hence $\csigma$ is $4$-increasing.

We prove by induction on $n = |X|$.  This is clear for $n=0$.  For $n>0$, choose $r \in X$ maximal with respect to $\succeq_{\clambda}$.  By Lemma~\ref{L:getcloser}, since $\csigma$ is $4$-increasing, we have $\z(\clambda)+\sb_r^{\clambda} = \z(\cnu)$ for some $1$-increasing partition $\cnu$ in $\check{B}$, and $\z(\csigma) = \z(\cnu) + \sb_{X \setminus \{r\}}^{\cnu}$.  By induction, $\cnu$ generates a hook-quotient exceptional family $\fnu$ with $\external{\fnu} = \external{\fsigma}$, and
$$
\z(\tmu) = \z(\eptn{j}) + \sb_{I}^{\eptn{j}} + \sb_{\xi_{\fnu}(X \setminus \{r\})}^{\fnu}.
$$
Thus $\tilde{\nu}^{j}$ is $1$-increasing by Lemma \ref{lemma:shiftby3}, so that by Lemma \ref{L:emptyintersectionwithregionmu}, $\clambda$ generates a hook-quotient exceptional family
$\flambda$ with $\external{\flambda} = \external{\fnu} = \external{\fsigma}$,
and
$\z(\eptl{j}) + \epsilon_{\xi_{\flambda}(r)}^{\flambda} = \z({\tnu}^{j})$.
Our choice of $r$ ensures that, for any $x \in X \setminus \{r\}$, we have $x \not\succeq_{\clambda} r$, so that $\xi_{\flambda}(x) \not\succeq_{\eptl{j}} \xi_{\flambda}(r)$ by Lemma \ref{L:projection}(4), and hence $\sb^{\flambda}_{\xi_{\flambda}(x)} = \sb^{\fnu}_{\xi_{\flambda}(x)} = \sb^{\fnu}_{\xi_{\fnu}(x)}$ by Proposition \ref{P:movealongdescription}(1).  In addition, since $\internal{\flambda} = \internal{\fnu}$, we have $\sb^{\eptl{j}}_i = \sb^{\eptn{j}}_i$ for all $i \in \internal{\flambda}$.  Thus, \begin{align*}
\z(\tmu) &= \z(\eptn{j}) + \sb_{I}^{\eptn{j}} + \sb_{\xi_{\fnu}(X \setminus \{r\})}^{\fnu} \\
&= (\z(\eptl{j}) + \sb_{\xi_{\flambda}(r)}^{\flambda}) + \sb_I^{\eptl{j}} +
\sb_{\xi_{\flambda}(X \setminus \{r\})}^{\flambda} \\
&= \z(\eptl{j}) + \sb_I^{\eptl{j}} +
\sb_{\xi_{\flambda}(X)}^{\flambda}.
\end{align*}
\end{proof}

Next, we study the following question:  when does an exceptional partition lie in the `internal' parallelotope of a member of an hook-quotient exceptional family?

\begin{lemma} \label{L:movealonginternalbasis}
Let $\flambda$ be a hook-quotient exceptional family with $\internal{\flambda} = \{ i_0, i_1,\dotsc, i_k\}$.  Suppose that $\z(\tmu) = \z(\eptl{j}) + \sb_{i_{\gamma}}^{\eptl{j}}$ for some $1$-increasing partition $\tmu \in \tb$, $1$-increasing $\eptl{j} \in \flambda$ and $\gamma \in [0,\,k]$.  The following statements are equivalent:
\begin{enumerate}
\item $j \geq 1$ and $\gamma = k+1-j$.
\item $j \geq 1$ and $\tmu= \eptl{j-1}$.
\item $\tmu$ is exceptional.
\end{enumerate}
\end{lemma}

\begin{proof}
Since $\z(\eptl{j-1}) - \z(\eptl{j}) = \ssb_{i_{k+1- j}} - \ssb_{i_{k+2-j}} = \pb{\flambda}_{k+2-j} = \sb^{\eptl{j}}_{i_{k+1-j}}$ (where $\ssb_{i_{k+2-j}} = 0$ when $j=1$) by Lemma \ref{lemma:paraexceptional}(1,3), we see that (1) and (2) are equivalent.  Trivially, (2) implies (3).  So suppose that (3) holds.
Then $\tmu$ belongs to a hook-quotient exceptional family $\fmu$ by Lemma \ref{L:exceptional-1-increasing}.  
For each $i \in [1,\,w]$, let $(b_i;q_i)$ (resp.\ $(b'_i;q'_i)$) be the $i$-th bead movement of $\eptl{j}$ (resp.\ $\tmu$).
Then $\beta(\tmu) = \beadoperation{q_r}^{k,l}(\beta(\tsigma))$ where $k = \frac{b_r-q_r}{e}$, $l = \frac{q_r-g_r}{e}$, $\beta(\tsigma) = \beta(\eptl{j}) \cup \{ g_r\} \setminus \{ b_r\}$ and $g_r = \max\{ x <_e q_r \mid x \notin \beta(\eptl{j})\}$ by Proposition \ref{P:movealongdescription}.
By Lemma \ref{L:region}(1a), we have $q_{i_{\gamma}} = \y + \gamma e - \I_{\gamma \geq k+1-j}$ for all $\gamma \in [0,k]$.

Suppose first that $\gamma \leq k-j$.  Then $q_{i_\gamma} = \y+\gamma e \equiv_e a$.
If $\gamma \geq 1$, then $q_s \geq_e q_{i_{\gamma}}$, and hence $q'_s > q_r \geq_e q_{i_{\gamma}}$, for all $s \succeq_{\eptl{j}} i_{\gamma}$, so that $q'_r \notin \{ \y+\gamma e, \y-1+\gamma e \}$ for all $r \in [1,\,w]$, contradicting Lemma \ref{L:region}(1a) for $\tmu$.
On the other hand, if $\gamma = 0$, so that $q_{i_{0}} = \y$, then since $(b_{i_0};q_{i_0})$ is the final bead movement of the bottom bead on runner $a$ of $\eptl{j}$, we have $\sb^{\eptl{j}}_{i_0} = \ssb_{i_0}$, so that $\z(\tmu) = \z(\eptl{j}) + \ssb_{i_0}$ and hence $\z_{\tmu}(q'_r) = \z_{\eptl{j}}(q_r) + \delta_{r,i_0}$ for all $r \in [1,\,w]$.
Since $\tmu$ is $1$-increasing, we have $1 \leq \z_{\tmu}(q'_{i_1}) - \z_{\tmu}(q'_{i_0}) = \z_{\eptl{j}}(q_{i_1}) - \z_{\eptl{j}}(q_{i_0}) -1$, so that $\z_{\eptl{j}}(q_{i_1}) - \z_{\eptl{j}}(q_{i_0}) \geq 2$.
This implies that there exist $x_1,x_2\in (q_{i_0},\,q_{i_1}] \setminus \beta(\eptl{j})$, $x_1 < x_2$, such that $x_1-e,x_2-e \in \beta(\eptl{j})$.
Since $q_{i_1} \in \{ \y+e, \y+e -1\}$, and $|\{ \y,\y-1\} \cap \beta(\eptl{j})| = 1$ by Lemma \ref{L:region}(1b), this shows in fact $x_1 \in (q_{i_0},\,\y+e-2]$, so that $q'_{i_0} = b_{\beta(\tsigma)}(q_{i_0})\leq x_1$ and hence $q'_{i_0} - q_{i_0} \leq x_1 - q_{i_0} \leq e-2$.
Together with Proposition \ref{P:movealongdescription}(3), this shows that $q'_r \notin \{ \y,\y-1\}$ for all $r \in [1,\,w]$, contradicting Lemma \ref{L:region}(1a) for $\tmu$.

Now suppose that $\gamma \geq k+2-j$, so that $q_{i_{\gamma}}, q_{i_{\gamma-1}} \equiv_e a-1$ and $i_{\gamma-1} \succ_{\eptl{j}} i_{\gamma}$.  Then $q_{i_{\gamma}}+1 -e \notin \beta(\eptl{j})$ by Lemma \ref{L:region}(1b), so that $q'_{i_{\gamma}} > q_{i_{\gamma}} + 1$.  Since $\z_{\eptl{j}}(q_{i_{k-1}}) < \z_{\eptl{j}}(q_{i_{k}})$ as $\eptl{j}$ is $1$-increasing, there exists $x \in (q_{i_{\gamma-1}},\,q_{i_{\gamma}}] \setminus \beta(\eptl{j})$ with $x-e \in \beta(\eptl{j})$.  If $x < q_{i_{\gamma}}$, then $q'_{i_{\gamma -1}} \leq x < q_{i_\gamma}$.  On the other hand, if $x = q_{i_{\gamma}}$, then $q_{i_{\gamma}} \notin \beta(\eptl{j})$, and so $q'_{i_{\gamma-1}} < q_{i_{\gamma-1}} + e = q_{i_{\gamma}}$ by Proposition \ref{P:movealongdescription}(4,2a).  Either way, we get $q'_{i_{\gamma-1}} < q_{i_{\gamma}}$, and so $q'_r \notin \{q_{i_{\gamma}}, q_{i_{\gamma}}+1\}$ for all $r \in [1,\,w]$, contradicting Lemma \ref{L:region}(1a) for $\tmu$.

Thus, $\gamma = k+1-j$, and hence $j \geq 1$, so that (1) holds, and we are done.
\end{proof}

\begin{lemma} \label{L:internalparallelogram}
Let $\tmu \in \tb$ be a $4$-increasing partition, and suppose that $\z(\tmu) = \z(\eptl{j}) + \sb^{\eptl{j}}_{I}$ for some member $\eptl{j}$ of a hook-quotient exceptional family $\flambda$ and $I \subseteq \internal{\flambda} = \{ i_0,i_1,\dotsc,i_{k} \}$.  The following statements are equivalent:
\begin{enumerate}
\item $j \geq |I|$ and $I = \{ i_\gamma \mid \gamma \in [k+1-j,k-j+|I|] \}$.
\item $j \geq |I|$ and $\tmu = \eptl{j-|I|}$.
\item $\tmu$ is exceptional.
\end{enumerate}
\end{lemma}

\begin{proof}
Note first that $\z(\eptl{j}) + \sb^{\eptl{j}}_{i_{k+1-j}} = \z(\eptl{j-1})$ always (see proof of Lemma \ref{L:movealonginternalbasis}).
Therefore if $i_{k+1-j} \in I$, then $\z(\tmu) - \z(\eptl{j-1}) = \sb_{I\setminus \{i_{k+1-j}\}}^{\eptl{j}}$, so that $\eptl{j-1}$ is $1$-increasing by Lemma \ref{lemma:shiftby3}.
As $i_{k+1-j}$ is maximal in $\internal{\flambda}$ with respect to $\succeq_{\eptl{j}}$ by Lemma \ref{L:region}(1e), we have $\sb^{\eptl{j}}_{i_{\gamma}} = \sb^{\eptl{j-1}}_{i_{\gamma}}$ for all $\gamma \ne k+1-j$ by Proposition \ref{P:movealongdescription}(1).  Thus
\begin{equation} \label{E:inductioneqn}
\z(\tmu) = \z(\eptl{j}) + \sb_I^{\eptl{j}} = \z(\eptl{j-1}) + \sb_{I \setminus \{i_{k+1-j}\}}^{\eptl{j}} = \z(\eptl{j-1}) + \sb_{I \setminus \{i_{k+1-j}\}}^{\eptl{j-1}}.
\end{equation}

We prove by induction on $|I|$, with the base case of $|I| = 1$ following from Lemmas \ref{L:movealonginternalbasis} and \ref{lemma:shiftby3}.  Suppose then that $|I|>1$.

Suppose that (1) holds.  Then we deduce from \eqref{E:inductioneqn} that (2) holds, by induction.

Clearly (2) implies (3).

Suppose that (3) holds.  If $i_{k+1-j} \notin I$, then let $r \in I$ be maximal with respect to $\succeq_{\eptl{j}}$.  By Lemma \ref{L:getcloser}, there exists a $1$-increasing partition $\tnu$ such that $\z(\tnu) = \z(\eptl{j}) + \sb_r^{\eptl{j}}$ and $\z(\tmu) = \z(\tnu) + \sb^{\tnu}_{I \setminus \{r\}}$.  In addition, for each $i \in I \setminus \{r\}$, the $i$-th bead movement of $\eptl{j}$ starts at the same position as that of $\tnu$ by Proposition \ref{P:movealongdescription}(1).  By Lemma \ref{L:movealonginternalbasis}, $\tnu$ is not exceptional and hence has no removable bead on runner $a$.  By Corollary \ref{C:noremovable}, $\tmu$ has no removable bead on runner $a$, and hence is not exceptional, a contradiction.  Thus $i_{k+1-j} \in I$, and we deduce from \eqref{E:inductioneqn} that (1) holds, by induction.
\end{proof}

\subsection{Inductive construction of canonical basis}

With all the preliminary lemmas in place, we are now able to prove that Theorem \ref{thm:main} holds by explicitly constructing the canonical basis vector $G(\tmu)$ in the Fock space given $G(\mu)$ and the canonical basis in blocks of smaller $e$-weight.  We begin with the case of $\mu$ being
a leading member of a hook-quotient exceptional family in $\tb$.

\begin{prop} \label{P:leadingexceptional}
Let $\fmu$ be a hook-quotient exceptional family, and suppose that $\eptm{0}$ is $4$-increasing.  If Theorem \ref{thm:main} holds for $\cmu$, $\hat{\mu}$ and $\mu^{0}$, then it holds for $\eptm{0}$.
\end{prop}

\begin{proof}
We have
$$F(G(\check{\mu})) = F\left(\sum_{\clambda} d_{\clambda\cmu}(q) \clambda\right) = \sum_{\clambda : \z(\cmu) \in \Pi(\clambda)} q^{ \dist{\clambda}{\cmu}}\, F(\clambda) =
\sum_{\flambda : \z(\check{\mu}) \in \Pi(\check{\lambda})}  \sum_{j=0}^{k+1} q^{ \dist{\clambda}{\cmu} + j}\, \eptl{j},$$
where the final equality follows from Proposition \ref{P:impt}
(with $\tmu = \eptm{0}$ and $\fsigma = \fmu$).  Thus, $F(G(\cmu)) - \tmu \in \bigoplus_{\tlambda} q\mathbb{Z}[q]\, \tlambda$. As $\overline{F(G(\cmu))} = F(\overline{G(\cmu)}) = F(G(\cmu))$, we see that $F(G(\cmu)) = G(\tmu)$. Since
$$
\dist{\eptl{j}}{\eptl{0}} = j = \dist{\eptm{j}}{\eptm{0}}
$$
by Lemma \ref{lemma:paraexceptional}(1,3), we see from Proposition \ref{P:impt} that
$$
\dist{\clambda}{\cmu} + j = \dist{\eptl{j}}{\eptm{0}}.
$$
Thus we conclude that $d_{\tlambda\eptm{0}}(q) \ne 0$ if and only if $\tlambda$ belongs to some hook-quotient exceptional family $\flambda$ with $\z(\cmu) \in \Pi(\clambda)$, in which case $d_{\tlambda\eptm{0}}(q) = q^{\dist{\tlambda}{\eptm{0}}}$.  Hence, it remains to show that $\tlambda$ belongs to some hook-quotient exceptional family $\flambda$ with $\z(\cmu) \in \Pi(\clambda)$ if and only if $\z(\eptm{0}) \in \Pi(\tlambda)$.  The forward implication follows from Proposition \ref{P:impt} (with $\tmu = \eptm{0}$ and $\fsigma = \fmu$).  For the converse, we show first that $\tlambda$ is exceptional when $\z(\eptm{0}) \in \Pi(\tlambda)$.  Suppose on the contrary that $\tlambda$ is not exceptional and $\z(\eptm{0}) \in \Pi(\tlambda)$.  Then $\z(\mu^{0})=\z({\tmu}^{0}) \in \Pi({\tlambda})=\Pi(\lambda)$ by Lemmas \ref{L:nonexceptional2} and \ref{lemma:paraexceptional}(1).  Since the main theorem holds for $G(\mu^{0})$, we have $d_{\lambda\mu^{0}}(q) \ne 0$. But, arguing as above, we have $G(\mu^{0})=EG(\hat{\mu})$.  Thus $\langle G(\hat{\mu}), F(\lambda) \rangle \ne 0$, and hence $F(\lambda) \ne 0$, so that $\lambda$, and hence $\tlambda$, is exceptional, a contradiction. Thus, $\tlambda$ is exceptional, and applying Proposition \ref{P:moveinparallelogramtilde}, we see that $\tlambda$ is a member of a hook-quotient exceptional family $\flambda$ and that there exists another hook-quotient exceptional family $\fsigma$ with $\z(\csigma) \in \Pi(\clambda)$ such that $\z(\eptm{0}) = z(\epts{j}) + \sb^{\epts{j}}_I$ for some $j \in [0,\,k+1]$ and $I \subseteq \internal{\fsigma}$.  By Lemma \ref{L:internalparallelogram}, we have $\fmu = \fsigma$, so that $\z(\cmu) = \z(\csigma) \in \Pi(\clambda)$, completing our proof.
\end{proof}

We are thus left with the case when $\mu$ is not one of the leading members of exceptional hook-quotient families.  The next lemma tells us that it is also not in the `external' parallelotopes of the latter.

\begin{lemma} \label{L:leading}
Let $\mu \in B$ be a $4$-increasing partition, and let $\flambda$ be a hook-quotient exceptional family.  If
$\z(\mu) - \z(\epl{0}) = \sb_X^{\flambda}$ for some $X \subseteq \external{\flambda}$ , then $\mu = \mu^{0}$ for a hook-quotient exceptional family $\fmu$ with $\external{\flambda} = \external{\fmu}$.
\end{lemma}

\begin{proof}
Let $\tmu \in \tb$ be such that $\z(\tmu) = \z(\mu)$.  Then
$$
\z(\tmu) - \z(\eptl{0}) = \z(\mu) - \z(\epl{0})  = \sb_X^{\flambda} $$
by Lemma \ref{lemma:paraexceptional}(1).
Applying Proposition \ref{P:moveinparallelogramtilde}, we see that $\z(\mu) = \z(\tmu) = \z(\epts{0}) = \z(\sigma^{0})$, where $\fsigma$ is the hook-quotient exceptional family with $\z(\csigma) = \pi_{\flambda}(\z(\tmu))$ and $\external{\fsigma} = \external{\flambda}$.  Thus $\mu = \eps{0}$ by Theorem~\ref{thm:goodlabels}(1).
\end{proof}

Notwithstanding Lemma \ref{L:leading}, when $\mu$ is not one of the leading members of hook-quotient exceptional families, it may still be in the parallelotope of a member of some hook-quotient exceptional family.  With this in mind, the next lemma looks at $\para(\flambda)$ of a hook-quotient exceptional family, which by Lemma \ref{lemma:paraexceptional}(5) is the union of parallelotopes of its members.

\begin{lemma} \label{L:para(flambda)}
Let $\flambda$ be a hook-quotient exceptional family.
Suppose that $z \in \para(\flambda)$, say
$$
z = \z(\epl{0}) + \pb{\flambda}_{J} + \sb_{X}^{\flambda}.
$$
where $J \subseteq [0,\, k+1]$ and $X \subseteq \external{\flambda}$
\begin{enumerate}
\item For $a \in [0,\, k+1]$, write $J_{<a}^c = \{ j \notin J \mid j<a \}$ and $J_{>a} = \{ j \in J \mid j>a \}$.  Let $j \in [0,\, k+1]$.
\begin{enumerate}
\item If $j \notin J$, then
$$z = \z(\epl{j}) - \pb{\flambda}_{J_{<j}^c} + \pb{\flambda}_{J_{>j}} + \sb_X^{\flambda} \in \para(\epl{j}).$$ Conversely, if $J \ne [0,\, k+1]$ and $j \in J$, then $z \notin \para(\epl{j})$.
\item If $k+1-j \in J$, then $$z = \z(\eptl{j}) - \pb{\flambda}_{J_{<k+1-j}^c} + \pb{\flambda}_{J_{>k+1-j}} + \sb_X^{\flambda} \in \para(\eptl{j}).$$  Conversely, if $J \ne \emptyset$ and $k+1-j \notin J$, then $z\notin \para(\eptl{j})$.
\end{enumerate}
\item If $J$ is non-empty and proper (as a subset of $[0,\,k+1]$), then the following three sets determine each other:
\begin{itemize}
\item $J$
\item $\{ j \in [0,\,k+1] \mid z \in \para(\epl{j})\}$
\item $\{ j \in [0,\,k+1] \mid z \in \para(\eptl{j}) \}$
\end{itemize}
Furthermore, the cardinalities of the first and third sets are equal, while the sum of the cardinalities of the last two sets is $k+2$.
\end{enumerate}
\end{lemma}

\begin{proof}
For part (1a), since $\z(\epl{0}) = \z(\epl{j}) + \ssb_{i_j} = \z(\epl{j})-\pb{\flambda}_{[0,\,j)}$ by Lemma \ref{lemma:paraexceptional}(1,3), we have
$$
z = \z(\epl{j}) - \pb{\flambda}_{[0,\,j) \setminus J} + \pb{\flambda}_{J \setminus [0,\,j)} + \sb^{\flambda}_X.
$$
Clearly, $[0,\,j) \setminus J = J^c_{<j}$.  If $j \notin J$, then $J \setminus [0,\,j) = J_{>j}$, so that $z \in \para(\epl{j})$ by Lemma \ref{lemma:paraexceptional}(4).
For the converse, if $j \in J$, then since $\pb{\flambda}_j = - \sum_{r \in [0,\, k+1] \setminus \{j\}} \pb{\flambda}_{r}$, we see that, the coefficient of $\pb{\flambda}_{j'}$, for any $j' \notin J$, when $z - \z(\epl{j})$ is expressed in terms of the basis $\{ \pb{\flambda}_{\gamma}, \sb_x^{\flambda} \mid \gamma \ne j,\ x \in \external{\flambda} \}$ is $-1$ if $j' > j$ and $-2$ if $j' < j$.  Hence $z \notin\para(\epl{j})$ if such a $j'$ exists by Lemma \ref{lemma:paraexceptional}(4).

Part (1b) is similar,
and part (2) follows immediately from part (1).
\end{proof}

\begin{defi}
Let $\flambda$ be a hook-quotient exceptional family.  For $\mu \in B$, define
$$n_{\flambda,\mu} = |\{ j \in [0,\,k+1] \mid \z(\mu) \in \Pi(\epl{j}) \}|.$$
By Lemma \ref{lemma:paraexceptional}(5), $n_{\flambda,\mu} > 0$ if and only if there exist a unique $J \subsetneq [0,\,k+1]$ and a unique $X \subseteq \external{\flambda}$ such that
$$
\z(\mu) = \z(\epl{0}) + \pb{\flambda}_J + \sb_X^{\flambda}.
$$
We write $s_{\flambda,\mu}$ for $|X|$ when this happens, and say that $\mu$ is {\em $s_{\flambda,\mu}$-separated} from $\flambda$.
\end{defi}

\begin{rem}
When $n_{\flambda,\mu} >0$, with $\z(\mu) = \z(\epl{0}) + \pb{\flambda}_J + \sb_X^{\flambda}$, we have $n_{\flambda,\mu} = k+2 -|J|$ by Lemma \ref{L:para(flambda)}(1a).
\end{rem}

\begin{prop}\label{prop:parexceptional}
Suppose that $\mu\in B$ is $4$-increasing and $\mu \ne \sigma^{0}$ for any hook-quotient exceptional family $\fsigma$.  Let $\flambda$ be a hook-quotient exceptional family such that $n_{\flambda,\mu} > 0$.  If Theorem \ref{thm:main} holds for $\mu$, then for any
$j \in [0,\,k+1]$ we have
$$\la \cheve^{(k)}(G(\mu)) - q^{s_{\flambda,\mu}}[n_{\flambda,\mu}-1]_q \chevf (\clambda),\,
\eptl{j} \ra =
\begin{cases}
q^{\dist{\eptl{j}}{\tmu}}, & \text{if } \z(\tmu)
\in \Pi(\eptl{j});\\
0, & \text{otherwise.}
\end{cases}$$
\end{prop}

\begin{proof}
There exist $J \subseteq [0,\, k+1]$ and (unique) $X \subseteq \external{\flambda}$ such that
$$\z(\mu) = \z(\epl{0}) + \pb{\flambda}_J + \sb_X^{\flambda}.$$  Since $\mu \ne \sigma^{0}$ for any hook-quotient exceptional family $\fsigma$, we see that $\pb{\flambda}_J \ne 0$ by Lemma \ref{L:leading}, so that $\emptyset \ne J \subsetneq [0,\,k+1]$ (and hence $J$ is unique too).
Working modulo partitions not in the exceptional
family $\flambda$, we have, since Theorem \ref{thm:main} holds for $\mu$,
$$
G(\mu)  \equiv \sum_
{\substack{
i \in [0,\,k+1]: \\
\z(\mu)\in\Pi(\epl{i})
}}
q^{\dist{\epl{i}}{\mu}} \epl{i} \\
=
\sum_
{i \in [0,\,k+1] \setminus J}
q^{|J^c_{<i}| + |J_{>i}| +s_{\flambda,\mu}}
\epl{i}
$$
by Lemma \ref{L:para(flambda)}(1a) (and using the notations there).

Using formula~(\ref{Eonexceptionals}) we deduce that
$$
\langle E^{(k)}(G(\mu)),\eptl{j} \rangle
=
\sum_{i \in [0,\,k+1-j) \setminus J}
q^{s_{\flambda,\mu}+c(i)+j-k}
+
\sum_{i \in (k+1-j,\, k+1] \setminus J}
q^{s_{\flambda,\mu}+c(i)+j-k-2},
$$
where $c(i)=i + |J^c_{<i}|+|J_{>i}|$.
Observe that 
\begin{alignat*}{2}
c(0) &= |J_{>0}| &&= k+1 - n_{\flambda,\mu} + \I_{0 \notin J}, \\
c(k+1) &= k+1 + |J^c_{<k+1}| &&= k+ n_{\flambda, \mu}  + \I_{k+1 \in J} .
\end{alignat*}
Moreover, for $i \in [1,\, k+1]$, we have $c(i+1)-c(i) = \I_{i\notin J} + \I_{i+1 \notin J}$.  Thus if $[0,\,k+1] \setminus J = \{ a_1 < a_2 < \dotsb < a_{k+2-|J|} \}$, then
$c(a_1) = k+2 - n_{\flambda,\mu}$, $c(a_{s+1}) = c(a_s) + 2$ for all $s \in [1,k+1-|J|]$, and $c(a_{k+2-|J|}) = k + n_{\flambda,\mu}$.
Consequently,
$$\langle E^{(k)}(G(\mu)),\eptl{j} \rangle
=
\begin{cases}
q^{s_{\flambda,\mu}+j}[n_{\flambda,\mu}-1]_q, & \text{if } k+1-j \notin J; \\
q^{s_{\flambda,\mu}+j}[n_{\flambda,\mu}-1]_q + q^{s_{\flambda,\mu} + c(k+1-j) + j - k - 1},
& \text{if } k+1-j \in J.
\end{cases}
$$
The desired result now follows, since
$F(\clambda) = \sum_{j=0}^{k+1} q^j \eptl{j}$, and
\begin{align*}
s_{\flambda,\mu} + c(k+1-j) + j-k-1 &= s_{\flambda,\mu} + |J^c_{<k+1-j}| + |J_{>k+1-j}| \\
&= \dist{\eptl{j}}{\mu}
\end{align*}
by Lemma \ref{L:para(flambda)}(1b).
\end{proof}

\begin{lemma} \label{L:unique}
Suppose that $\mu \in B$ is $4$-increasing and is $0$-separated from a hook-quotient exceptional family $\fnu$.  Then $\z(\cnu) = \pi_{\fnu}(\z(\mu))$.  In particular, $\cnu$ (and $\hat{\nu}$) are $4$-increasing.

If in addition $n_{\flambda,\mu} > 0$ for another hook-quotient exceptional family $\flambda$, then either $\xi_\fnu = \xi_\flambda$ or $\z(\cnu) \notin \Pi(\check{\lambda})$.
\end{lemma}

\begin{proof}
The first part follows from Lemma \ref{L:projection}(2,3).

When $n_{\flambda,\mu} >0$, then using Lemma \ref{L:projection}(2,3), we get $\pi_{\flambda}(\z(\mu)) \in \Pi(\check{\lambda})$.  Let
\begin{align*}
\z(\mu) = (m_1,\dotsc, m_w), \qquad \text{and} \qquad
\z(\clambda) = (l_1,\dotsc, l_{\w}).
\end{align*}
Then
\begin{align*}
\pi_{\flambda}(\z(\mu)) &= (m_{\xi_{\flambda}(1)},\dotsc, m_{\xi_{\flambda}(\w)}), \\
\z(\cnu) = \pi_{\fnu}(\z(\mu)) &= (m_{\xi_{\fnu}(1)},\dotsc, m_{\xi_{\fnu}(\w)}).
\end{align*}
If $\xi_\flambda \ne \xi_\fnu$, say $\xi_{\flambda}(a) \ne \xi_{\fnu}(a)$, then $|m_{\xi_{\flambda}(a)} - m_{\xi_{\fnu}(a)}| \geq 4$ since $\mu$ is $4$-increasing.  Thus,
$$
m_{\xi_{\fnu}(a)} - l_a = (m_{\xi_{\fnu}(a)} - m_{\xi_{\flambda}(a)}) + (m_{\xi_{\flambda}(a)} - l_a)
\begin{cases}
\geq 3, &\text{if } m_{\xi_{\fnu}(a)} > m_{\xi_{\flambda}(a)}; \\
\leq -2, &\text{if } m_{\xi_{\fnu}(a)} < m_{\xi_{\flambda}(a)}
\end{cases}
$$
by Lemma \ref{preshiftby3} (since $(m_{\xi_{\flambda}(1)},\dotsc, m_{\xi_{\flambda}(\w)}) = \pi_{\flambda}(\z(\mu)) \in \Pi(\clambda)$).  Consequently, by Lemma \ref{preshiftby3} again, $\z(\cnu) = \pi_{\fnu}(\z(\mu)) = (m_{\xi_{\fnu}(1)},\dotsc,m_{\xi_{\fnu}(\w)}) \notin \Pi(\clambda)$.
\end{proof}

\begin{cor} \label{C:unique}
Suppose that $\mu \in B$ is $4$-increasing, and $n_{\flambda,\mu} > 0$ for a hook-quotient exceptional family $\flambda$.  Then there is a unique hook-quotient exceptional family $\fsigma$ from which $\mu$ is $0$-separated and which satisfies $\z(\csigma) \in \Pi(\clambda)$, namely that with $\z(\csigma) = \pi_{\flambda}(\z(\mu))$.
\end{cor}

\begin{proof}
Since $n_{\flambda,\mu} > 0$, we have $\z(\mu) \in \Pi(\epl{j})$ for some $j$.  Let $\tmu \in \tb$ be such that $\z(\tmu) = \z(\mu)$.  Then $\z(\tmu) \in \Pi(\eptl{j'})$ for some $j' \in [0,\,k+1]$ by Lemma \ref{lemma:paraexceptional}(5).  By Proposition \ref{P:moveinparallelogramtilde}, the partition $\csigma \in \cb$ satisfying $\z(\csigma) = \pi_{\flambda}(\z(\tmu))$ generates a hook-quotient exceptional family $\fsigma$ from which $\mu$ is $0$-separated, and $\z(\csigma) \in \Pi(\clambda)$.

Let $\fnu$ be an(other) hook-quotient exceptional family from which $\mu$ is $0$-separated.  Then $\pi_{\fnu}(\z(\mu)) = \z(\cnu)$ by Lemma \ref{L:unique}.  In addition, if $\z(\cnu) \in \Pi(\clambda)$, then $\xi_{\fnu} = \xi_{\flambda}$ by Lemma \ref{L:unique}, so that
$$
\z(\cnu) = \pi_{\fnu}(\z(\mu)) = \pi_{\flambda}(\z(\mu)) = \z(\csigma).$$
Thus $\csigma= \cnu$ by Theorem~\ref{thm:goodlabels}(1) and hence $\fsigma = \fnu$.
\end{proof}

\begin{prop} \label{P:}
Let $\mu \in B$ be a $4$-increasing partition.
  Assume that Theorem \ref{thm:main} holds for $\mu$ and all $4$-increasing $\clambda \in \cb$. Then
$$
\sum_{\substack{{\flambda}: \\ n_{\flambda,\mu} > 0}} q^{s_{\flambda,\mu}} [n_{\flambda,\mu}-1]_q\, \clambda = \sum_{\substack{{\fsigma}: \\ s_{\fsigma,\mu}=0}} [n_{\fsigma,\mu}-1]_q\, G(\csigma).
$$
\end{prop}

\begin{proof}
By Lemma \ref{L:unique}, whenever $s_{\fsigma,\mu}=0$, we have $\csigma$ to be $4$-increasing, so that we may apply Theorem \ref{thm:main} to $\csigma$ to get
$$
G(\csigma) = \sum_{\substack{\clambda \in \cb :\\ \z(\csigma) \in \Pi(\clambda)}} q^{\dist{\clambda}{\csigma}}\, \clambda.
$$
Thus,
$$
\sum_{\substack{{\fsigma}: \\ s_{\fsigma,\mu}=0}} [n_{\fsigma,\mu}-1]_q\, G(\csigma) = \sum_{\substack{{\fsigma}: \\ s_{\fsigma,\mu}=0}} [n_{\fsigma,\mu}-1]_q\sum_{\substack{\clambda \in \cb: \\ \z(\csigma) \in \Pi(\clambda)}} q^{\dist{\clambda}{\csigma}}\, \clambda.
$$
Let $\tmu \in \tb$ be such that $\z(\tmu) = \z(\mu)$.
By Proposition \ref{P:impt}, we see that each $\clambda$ in the sum generates a hook-quotient exceptional family $\flambda$ with $\external{\flambda}= \external{\fsigma}$ and $s_{\flambda,\mu} = \dist{\clambda}{\csigma}$, and $\{ j \mid \z(\tmu) \in \para(\epts{j}) \} \subseteq \{ j \mid \z(\tmu) \in \para(\eptl{j}) \}$.  By Proposition \ref{P:moveinparallelogramtilde}, since $\pi_{\flambda}(\z(\tmu)) = \pi_{\fsigma}(\z(\tmu)) = \z(\csigma)$, we also have $\{ j \mid \z(\mu) \in \para(\eptl{j}) \} \subseteq \{ j \mid \z(\mu) \in \para(\epts{j}) \}$.  Thus $n_{\fsigma,\mu} = n_{\flambda,\mu}$ by Lemma \ref{L:para(flambda)}(2), so that
$$
\sum_{\substack{{\fsigma}: \\ s_{\fsigma,\mu}=0}} [n_{\fsigma,\mu}-1]_q\, G(\csigma) = \sum_{\substack{{\fsigma}: \\ s_{\fsigma,\mu}=0}} \, \sum_{\substack{\flambda: \\ \z(\csigma) \in \Pi(\clambda)\\ n_{\flambda,\mu} >0}} [n_{\flambda,\mu}-1]_q\, q^{s_{\flambda,\mu}}\,  \clambda
= \sum_{\substack{\flambda: \\ n_{\flambda,\mu} > 0}} \sum_{\substack{\fsigma:\\ s_{\fsigma,\mu} = 0 \\ \z(\csigma) \in \Pi(\clambda)}} [n_{\flambda,\mu}-1]_q\, q^{s_{\flambda,\mu}}\, \clambda$$
To complete the proof, it remains to show that
for each hook-quotient exceptional family $\flambda$ with $n_{\flambda,\mu} > 0$, there exists a unique hook-quotient exceptional family $\fsigma$ from which $\mu$ is $0$-separated and which satisfies $\z(\csigma) \in \Pi(\clambda)$, but this is precisely the content of Corollary \ref{C:unique}.
\end{proof}

\begin{cor} \label{C:others}
Let $\tmu \in \tB$ be $4$-increasing, and suppose that $\tmu \ne \eptl{0}$ for any hook-quotient exceptional family $\flambda$.  Let $\mu \in B$ with $\z(\mu) = \z(\tmu)$.  If Theorem \ref{thm:main} holds for $\mu$ and $\clambda$ for any hook-quotient exceptional family $\flambda$, then
$$
G(\tmu) = E^{(k)}(G(\mu)) - \sum_{\fsigma :\, s_{\fsigma,\mu}=0} [n_{\fsigma,\mu}-1]_q\, F(G(\csigma))
$$
and Theorem \ref{thm:main} holds for $\tmu$.
\end{cor}

\begin{proof}
Since $E^{(k)}(G(\mu)) - \sum_{\fsigma :\, s_{\fsigma,\mu}=0} [n_{\fsigma,\mu}-1]_q\, F(G(\csigma))$ is bar-invariant, while the asserted formula $H(\tmu)$ for $G(\tmu)$ by Theorem \ref{thm:main} satisfies $H(\tmu) - \tmu \in \bigoplus_{\tlambda} q\mathbb{Z}[q] \tlambda$, it suffices to show that $E^{(k)}(G(\mu)) - \sum_{\fsigma :\, s_{\fsigma,\mu}=0} [n_{\fsigma,\mu}-1]_q\, F(G(\csigma)) = H(\tmu)$ i.e.
$$
\langle E^{(k)}(G(\mu)) - \sum_{\fsigma :\, s_{\fsigma,\mu}=0} [n_{\fsigma,\mu}-1]_q\, F(G(\csigma)), \tlambda \rangle =
\begin{cases}
q^{\dist{\tlambda}{\tmu}}, &\text{if } \z(\tmu) \in \Pi(\tlambda); \\
0, &\text{otherwise}.
\end{cases}
$$
When $\tlambda$ is a member of a hook-quotient exceptional family $\flambda$, the above equality follows from Propositions \ref{prop:parexceptional} and \ref{P:}.

When $\tlambda$ is exceptional but not a member of any hook-quotient exceptional family, then $\z(\tmu) \notin \Pi(\tlambda)$ by Proposition \ref{P:moveinparallelogramtilde}(1).
On the other hand, since Theorem \ref{thm:main} holds for $G(\mu)$, we see that $G(\mu)$ is a $\Zq$-linear combination of partitions which are either non-exceptional or members of hook-quotient exceptional families by Lemma \ref{lemma:nottooexceptional}(2), so that the same is true for $E^{(k)}(G(\mu))$.
In addition, since Theorem \ref{thm:main} holds for $G(\csigma)$, we see that $G(\csigma)$, with $s_{\fsigma,\mu}=0$, is a $\Zq$-linear combination of partitions which generates hook-quotient exceptional families by Proposition \ref{P:impt}, so that $F(G(\csigma))$ is a $\Zq$-linear combination of partitions belonging to hook-quotient exceptional families.
Thus $E^{(k)}(G(\mu))  - \sum_{\fsigma :\, s_{\fsigma,\mu}=0} [n_{\fsigma,\mu}-1]_q\, F(G(\csigma))$ is a $\Zq$-linear combination of partitions which are either non-exceptional or members of hook-quotient exceptional families.
This shows that the above equality also holds for exceptional $\tlambda$ which is not a member of any hook-quotient exceptional family.

Finally when $\tlambda$ is non-exceptional, let $\lambda \in B$ such that $\z(\lambda) = \z(\tlambda)$.
Then $E^{(k)}(\lambda) = \tlambda$ and $\chevf^{(k)}(\tlambda) = \lambda$, so that $\langle E^{(k)}(\mu), \tlambda \rangle = 0$ for all $\mu \in B$ with $\mu \ne \lambda$.  Thus,
\begin{align*}
\langle E^{(k)}(G(\mu)) - \sum_{\fsigma :\, s_{\fsigma,\mu}=0} [n_{\fsigma,\mu}-1]_q\, F(G(\csigma)), \tlambda \rangle = d_{\lambda\mu}(q) &=
\begin{cases}
q^{\dist{\lambda}{\mu}}, &\text{if } \z(\mu) \in \Pi(\lambda); \\
0, &\text{otherwise}.
\end{cases}
\end{align*}
Now $z(\mu) \in \para(\lambda)$ only if $\lambda$ is $1$-increasing by Lemma \ref{lemma:shiftby3}, and hence $\sb^{\lambda}_i = \sb^{\tlambda}_i$ for all $i \in [1,\, w]$ by Lemma \ref{L:nonexceptional2}, so that $\para(\lambda) = \para(\tlambda)$ and $\dist{\lambda}{\mu} = \dist{\tlambda}{\mu}$.  The same conclusion is also true when $z(\tmu) \in \para(\tlambda)$.  Thus $z(\mu) \in \para(\lambda)$ if and only if $z(\tmu) \in \para(\tlambda)$, in which case $\dist{\lambda}{\mu} = \dist{\tlambda}{\tmu}$.  Hence
\begin{align*}
\langle E^{(k)}(G(\mu)) - \sum_{\fsigma :\, s_{\fsigma,\mu}=0} [n_{\fsigma,\mu}-1]_q\, F(G(\csigma)), \tlambda \rangle = \begin{cases}
q^{\dist{\tlambda}{\tmu}}, &\text{if } \z(\tmu) \in \Pi(\tlambda); \\
0, &\text{otherwise}.
\end{cases}
\end{align*}
\end{proof}

We can now prove Proposition \ref{prop:main}, thereby completing the proof of Theorem \ref{thm:main}.

\begin{proof}[Proof of Proposition \ref{prop:main}]
This follows immediately from Proposition \ref{P:leadingexceptional} and Corollary \ref{C:others}.
\end{proof}

\section{Hypercubes} \label{S:cubes}

In this section we reformulate Theorem~\ref{thm:main}, replacing $w$-parallelotopes in $\mathbb{Z}^w$ by $w$-hypercubes in $\mathbb{Z}^{w(w+1)/2}$. One advantage of this point of view is that the powers of $q$ appearing in the $q$-decomposition numbers can be understood in terms of the natural box metric in the larger lattice, which is independent of the partitions in question.
We will also use this reinterpretation in terms of hypercubes in Section~\ref{section:tilings} to show that the parallelotopes have nice intersections.

Recall that $\{ \ssb_i \mid 1 \leq i \leq \mathbb{Z}^w\}$ is the standard basis for $\mathbb{Z}^w$.  Formally define additional linearly independent vectors $\ssb_{ij}$ for $1 \leq i < j \leq w$, so that
$$\widehat{\mathbb{Z}^w}:=\bigoplus_{1\leq i\leq w}
\mathbb{Z}\ssb_{i} \oplus
\bigoplus_{1\leq i<j\leq w} \mathbb{Z}\ssb_{ij}$$ is
a free $\mathbb{Z}$-lattice of rank $w(w+1)/2$.  Let $\proj:\widehat{\mathbb{Z}^w}\to\mathbb{Z}^w$ be the
linear map given by $\proj(\ssb_i)=\ssb_i$ and $\proj(\ssb_{ij})=\ssb_i-\ssb_j$.

Let $\lambda$ be a partition of $e$-weight $w$, and write $\bm(\lambda)=\{(b_1;q_1),\dotsc,(b_w;q_w)\}$.
Firstly we define a lift of $z(\lambda)$ to $\widehat{\mathbb{Z}^w}$.
For $1\leq i<j\leq w$, let
$$\crossing^\lambda_{ij}=
\begin{cases}
1, & \text{if } q_i > q_j-e, \text{ or both } q_{i}=q_{j}-e\text{ and } b_{i}=b_{j}; \\
0,& \text{otherwise.}
\end{cases}
$$
We define
$$\hat\z(\lambda)=\z(\lambda)+\sum_{1\leq i<j\leq w}\crossing^\lambda_{ij}(-\ssb_i + \ssb_j + \ssb_{ij})\in\widehat{\mathbb{Z}^w}.$$
Clearly $\proj(\hat\z(\lambda))=\z(\lambda)$.

Suppose now that $\lambda$ is a hook-quotient partition. We lift $\para(\lambda)$ to a $w$-hypercube in $\widehat{\mathbb{Z}^w}$.
We first define the {\em lifted $\lambda$-modified basis vectors}
$$\hat\sb^\lambda_i =
\begin{cases}
\ssb_i & \text{ if } \sb^\lambda_i=\ssb_i; \\
\ssb_{ij}& \text{ if } \sb^\lambda_i=\ssb_i - \ssb_j \text{ for some } i<j; \\
-\ssb_{ji}& \text{ if } \sb^\lambda_i=\ssb_i - \ssb_j \text{ for some } j<i. \\
\end{cases}
$$
For a subset $\Gamma$ of $[1,\, w]$, let
$$
\hat\sb_{\Gamma}^{\lambda} = \sum_{a \in \Gamma} \hat\sb_a^{\lambda} \in \widehat{\BZ^w},$$
and define the
{\em{hypercubes}} $\cube_0(\lambda)$ and $\cube(\lambda)$ of $\lambda$, anchored at $0$ and $\hat\z(\lambda)$ respectively, by
\begin{align*}
	\cube_0(\lambda) & =\left\{ \hat\sb_\Gamma^\lambda \mid \Gamma \subseteq [1,\,w] \right\}, \\
	\cube(\lambda) &=\hat\z(\lambda) +  \cube_0(\lambda) = \left\{ \hat\z(\lambda) + \hat\sb_\Gamma^{\lambda} \mid \Gamma \subseteq [1,\,w] \right\};
\end{align*}
they are subsets of $\widehat{\BZ^w}$, each of cardinality $2^w$.
It is clear from the definitions that $\proj(\hat\sb_{\Gamma}^{\lambda})=\sb_{\Gamma}^{\lambda}$, and therefore that $\proj$ sends $\cube(\lambda)$ bijectively onto $\para(\lambda)$.

Define the box norm $\| \cdot \|$ in $\widehat{\BZ^w}$ by
$\left\Vert\sum c_i\ssb_i +  \sum c_{ij} \ssb_{ij}\right\Vert =
\sum |c_i| +  \sum |c_{ij}|$, so that $\| \hat\sb_\Gamma^{\lambda} \|=\left|\Gamma\right|$.

\begin{prop}\label{P:parallelepipedsversuscubes}
Let $\lambda$ and $\mu$ be partitions of $e$-weight $w$ with the same $e$-core, and suppose that $\mu$ is $4$-increasing. For any $\Gamma \in [1,w]$, we have $\z(\mu) = \z(\lambda) + \sb^{\lambda}_{\Gamma}$ if and only if
$\hat\z(\mu) = \hat\z(\lambda) + \hat\sb^{\lambda}_{\Gamma}$.   In particular, if $\mu \in \para(\lambda)$, then
 $\dist{\lambda}{\mu}=\left\Vert\hat{\z}(\mu)-\hat{\z}(\lambda)\right\Vert$.
\end{prop}

\begin{proof}
We show only that $\z(\mu) = \z(\lambda) + \sb^{\lambda}_{\Gamma}$ implies $\hat\z(\mu) = \hat\z(\lambda) + \hat\sb^{\lambda}_{\Gamma}$; the converse is trivial by applying the projection map $p$.  By induction on $|\Gamma|$, and using Lemma~\ref{L:getcloser}, we are reduced to proving the following statement: if $\lambda$ is a hook-quotient partition and $\mu$ is a $1$-increasing partition in the same block such that $\z(\mu)=\z(\lambda)+\sb^\lambda_i$ for some $i\in[1,w]$, then $\hat\z(\mu)=\hat\z(\lambda)+\hat\sb^\lambda_i$.  Since
$$
\proj(\hat\z(\lambda) + \hat\sb^{\lambda}_i) = \z(\lambda) + \sb^{\lambda}_i = \z(\mu),$$
and $\ker(\proj)$ is spanned by $\{ \ssb_{rs} -\ssb_r + \ssb_s \mid 1 \leq r < s \leq w\}$, it suffices to show that the coefficient of $\ssb_{rs}$ in $\hat\z(\mu)$ and in $\hat\z(\lambda)+\hat\sb^\lambda_i$ are equal for all $1 \leq r < s \leq w$, or equivalently,
$$
\crossing^{\mu}_{rs} - \crossing^{\lambda}_{rs} =
\begin{cases}
1, &\text{if } \sb^{\lambda}_i = \ssb_r - \ssb_s; \\
-1, &\text{if } \sb^{\lambda}_i = \ssb_s - \ssb_r; \\
0, &\text{otherwise.}
\end{cases}
$$
This is a straightforward, albeit tedious, verification using Proposition \ref{P:movealongdescription}(1,2,3,5), and our proposition follows.
	\end{proof}

In the light of Proposition \ref{P:parallelepipedsversuscubes}, Theorem~\ref{thm:main} may be restated as follows.

\begin{thm}\label{thm:maincube}
	Suppose $\lambda$ and $\mu$ are partitions of $e$-weight $w$ with the
	same $e$-core. Then if $\mu$ is
	$4$-increasing we have
	$$
	d_{\lambda\mu}(q)=
	\begin{cases}
	q^{\left\Vert\hat{\z}(\mu)-\hat{\z}(\lambda)\right\Vert},  & \text{if } \hat\z(\mu) \in \cube(\lambda); \\
	0,  & \text{otherwise.}
	\end{cases}
	$$
\end{thm}

We conclude this section with an analysis of the hypercubes of partitions in
hook-quotient exceptional families, required in the proof of Lemma~\ref{C:cubesbijectparas}.
As before, let $B$ and $\tb$ be two blocks of $e$-weight $w$ forming a $[w:k]$-pair, and let $\flambda$ be a hook-quotient exceptional family with respect to this pair, with $\internal{\flambda} = \{ i_0,i_1,\dotsc, i_k \}$.

Define
$$\hpb{\flambda}_0=-\ssb_{i_0},\quad \hpb{\flambda}_{1}
=\ssb_{i_0,i_1},\quad\dotsc,\quad\hpb{\flambda}_k=
\ssb_{i_{k-1},i_k},\quad\hpb{\flambda}_{k+1}=\ssb_{i_{k}}.$$
For each $J \subseteq [0,k+1]$, write $\hpb{\flambda}_J$ for $\sum_{j \in J} \hpb{\flambda}_j$.

The following is a cubic analogue of Lemma \ref{lemma:paraexceptional}.

\begin{lemma}\label{L:cubeexceptional} \hfill
\begin{enumerate}
\item We have
\begin{alignat*}{2}
\hat\z(\epl{0}) &= \hat\z(\eptl{0}) - \hpb{\flambda}_{[0,k+1]}, \\
\hat\z(\epl{j}) &= \hat\z(\epl{0}) + \hpb{\flambda}_{[0,j)} &\qquad &(j \in[1,k+1]), \\
\hat\z(\eptl{j}) &= \hat\z(\eptl{0}) - \hpb{\flambda}_{(k+1-j,k+1]} = \hat\z(\epl{0}) + \hpb{\flambda}_{[0,k+1-j]} &\qquad & (j \in[1,k+1]).
\end{alignat*}

\item For any $x \in \external{\flambda}$, $\hat\sb^{\lambda}_x$ is constant for all $\lambda \in \flambda$; denote the common vector by $\hat\sb^{\flambda}_x$.

\item We have
$$\hat\sb_{i_\gamma}^{\epl{j}}= \hat\sb_{i_\gamma}^{\eptl{k+1-j}}=
\begin{cases}
-\hpb{\flambda}_{\gamma} & \text{ if } \gamma < j; \\
\hpb{\flambda}_{\gamma+1} & \text{ if } \gamma \geq j.
\end{cases}$$
In particular, $\cube_0(\epl{j}) = \cube_0(\eptl{k+1-j})$.

\item Let $z \in \widehat{\BZ^w}$.  Then $z \in \cube_0(\epl{j})$ ($= \cube_0(\eptl{k+1-j})$ by part (3))  if and only if $$z = -\hpb{\flambda}_{I_{<j}} + \hpb{\flambda}_{I_{>j}} + \hat\sb_X^{\flambda}$$ for some $I_{<j} \subseteq [0,\, j)$, $I_{>j} \subseteq (j,\, k+1]$ and $X \subseteq \Ext(\flambda)$.

\item We have
\begin{align*}
\cube(\epl{j}) &= \{
\hat\z(\epl{0}) + \hpb{\flambda}_J + \hat\sb^{\flambda}_X \mid j \notin J,\ X \subseteq \external{\flambda} \}; \\
\cube(\eptl{k+1-j}) &= \{
\hat\z(\epl{0}) + \hpb{\flambda}_J + \hat\sb^{\flambda}_X \mid j \in J,\ X \subseteq \external{\flambda} \}.
\end{align*}
\end{enumerate}
\end{lemma}

\section{Tilings}\label{section:tilings}

In this section we show that the parallelotopes $\para(\lambda)$ associated to hook-quotient partitions in a block $B$ of weight $w$ assemble to form `tilings' of a subset of $\mathbb{Z}^w$ with nice properties. We also prove a geometric analogue: the real convex hulls $\hull{\para(\lambda)}$ tile a region of $\mathbb{R}^w$.

We continue to denote by $\inc{\Omega}{m}$ the set of $m$-increasing elements in a subset $\Omega\subset\mathbb{Z}^w$ or, more generally, in a subset $\Omega\subset\mathbb{R}^w$. We extend this convention to subsets $\Omega\subset\widehat{\mathbb{R}^w}$, defining $\inc{\Omega}{m}$ as the set of elements in $\Omega$ that map to $m$-increasing elements of $\mathbb{R}^w$ under the linear map $\proj:\widehat{\mathbb{R}^w}\to\mathbb{R}^w$ defined by
$p(\ssb_i)=\ssb_i$ and $p(\ssb_{ij})=\ssb_i-\ssb_j$.

We start with the discrete parallelotopes. Set
$$\goodplus:=[0,e]^w \subset \mathbb{Z}^w,$$
so that
$$\plusinc{\good}{m}=\{(z_1,\dotsc,z_w) \in \mathbb{Z}^w\mid 0\leq z_i \leq e \text{ and } z_{i+1}-z_{i}\geq m \text{ for all admissible } i \}$$
for any nonnegative integer $m$.
It was explained in Remark~\ref{remark:addrunner} that $\plusinc{\good}{m}$ may be understood as the set of $\z$-labels taken by $m$-increasing partitions in a block of $(e+1)$-weight $w$. 
It is also, at least when $m\geq 4$, the subset of $\mathbb{Z}^w$ tiled by (appropriate truncations of) the parallelotopes associated to any block of $e$-weight $w$, as the following result shows.
	
\begin{prop}\label{prop:discreteunion}
	Let $B$ be a block of $e$-weight $w$. Then
	$$\bigcup_{\lambda\in B}\inc{\para(\lambda)}{4} = \plusinc{\good}{4},$$
	where the union is taken over all hook-quotient partitions in $B$.
\end{prop}
When $e$ is large in comparison to $w$, most of the parallelotopes appearing in the union in Proposition~\ref{prop:discreteunion} are untruncated, since  $\inc{\para(\lambda)}{4}=\para(\lambda)$ for all $7$-increasing partitions $\lambda$, by Lemma~\ref{lemma:shiftby3}. The same lemma implies that
$\inc{\para(\lambda)}{4}$ is nonempty only if $\lambda$ is $1$-increasing, so the union may just as well be taken over only the $1$-increasing partitions in $B$. 	

\begin{proof}
	We proceed using the inductive setup described at the beginning of Section~\ref{section:mainproof}.
	Suppose that $B$ is a Rouquier block. Then for any $1$-increasing partition $\lambda$ in $B$, we have $\sb_i^\lambda=\ssb_i$ for all $i\in[1,w]$, and the result is immediate.
		Next, suppose that the statement of the proposition holds for a block $\b$ of weight $w$, and let $\tb=s_a(\b)$. By Lemma~\ref{L:exceptional-1-increasing}, any $1$-increasing partition in $\tb$ is either non-exceptional or belongs to a hook-quotient exceptional family. The statement thus follows directly from Lemma~\ref{L:nonexceptional2} and Lemma~\ref{lemma:paraexceptional}(5).
	\end{proof}

By a {\em face} of a parallelotope $\para(\lambda)$ we mean a subset of the form
$$F=F_{\Gamma_0,\Gamma_1}=\left\{\z(\lambda)+\sb_\Gamma^\lambda\mid \Gamma_1 \subseteq \Gamma \subseteq [1,w]\setminus\Gamma_0\right\},$$
where $\Gamma_0$ and $\Gamma_1$ are disjoint subsets of $[1,w]$, or the empty set.

\begin{prop}\label{prop:discreteintersection}
	Let $\lambda$ and $\mu$ be hook-quotient partitions in $B$. Then $\inc{\para(\lambda)}{4}\cap\inc{\para(\mu)}{4}=\inc{F}{4}$, where $F$ is a common face of $\para(\lambda)$ and $\para(\mu)$. In particular, if either $\lambda$ or $\mu$ is $7$-increasing, then $\para(\lambda)$ and $\para(\mu)$ intersect in a common face.
	\end{prop}	
	We can define a {\em face} of a hypercube $\cube(\lambda)$ analogously to that for $\para(\lambda)$, or, more simply, as the intersection of $\cube(\lambda)$ with (a translate of) a coordinate hyperplane of $\widehat{\mathbb{Z}^w}$.
In contrast to the situation for parallelotopes, it is clear that any two hypercubes $\cube(\lambda)$ and $\cube(\mu)$ intersect in a common face. Thus Proposition~\ref{prop:discreteintersection} is a consequence of the following lemma.
\begin{lemma} \label{C:discretecubesbijectparas}
		Let $B$ be a block of $e$-weight $w$.
		The projection map $\proj: \widehat{\BZ^w} \to \BZ^w$ maps $\bigcup_{\lambda} \inc{\cube(\lambda)}{4}$ bijectively onto
		$\plusinc{\good}{4}$, where $\lambda$ runs over all hook-quotient partitions in $B$. 
	\end{lemma}
	
	\begin{proof}
		Clearly, $p$ maps $\bigcup_{\lambda} \inc{\cube(\lambda)}{4}$ onto $\plusinc{\good}{4}$, as $\proj$ maps each $\inc{\cube(\lambda)}{4}$ onto $\inc{\para(\lambda)}{4}$.  Let $\lambda,\mu$ be hook-quotient partitions in $B$ and suppose that $\proj(\mathbf{a}) = \proj(\mathbf{b})$ for some $\mathbf{a} \in \inc{\cube(\lambda)}{4}$ and $\mathbf{b} \in \inc{\cube(\mu)}{4}$.  
		
		Suppose that $p(\mathbf{a})\in\inc{\good}{4}$, so that
		$p(\mathbf{a})=\z(\nu)$ for a partition $\nu$ in $B$.  We have that $\nu$ is $4$-increasing, and $\z(\nu) \in \proj(\cube(\lambda)) \cap \proj(\cube(\mu)) = \para(\lambda) \cap \para(\mu)$.  By Proposition \ref{P:parallelepipedsversuscubes}, we thus have $\hat\z(\nu) \in \cube(\lambda) \cap \cube(\mu)$.  Since $\proj$ maps $\cube(\lambda)$ bijectively onto $\para(\lambda)$, and $\proj(\mathbf{a}) = \z(\nu) = \proj(\hat\z(\nu))$, we must have $\mathbf{a} = \hat\z(\nu)$.  Similarly, $\mathbf{b} =  \hat\z(\nu)$, and hence $\mathbf{a} = \mathbf{b}$.
		
		To handle the general case, consider the embedding $B\hookrightarrow B^+:\lambda\mapsto\lambda^+$ of $B$ into a block $B^+$ of $(e+1)$-weight $w$, as described in Remark~\ref{remark:addrunner}. It is easy to check that
		$\hat\z(\lambda) = \hat\z(\lambda^+)$, and $\sb^{\lambda}_i = \sb^{\lambda^+}_i$ for all $i \in [1,w]$, so that $\cube(\lambda) = \cube(\lambda^+)$.
		Applying the special case considered in the previous paragraph to the block $B^+$, we obtain the desired conclusion that $\proj$ maps  $\bigcup_{\lambda\in B} \inc{\cube(\lambda)}{4}=\bigcup_{\lambda\in B} \inc{\cube(\lambda^+)}{4}$ bijectively onto $\plusinc{\good}{4}$, since the latter is the set of values taken by $\z(\nu)$ as $\nu$ ranges over $4$-increasing partitions in $B^+$.
	\end{proof}

We next formulate and prove real analogues of Propositions \ref{prop:discreteunion} and \ref{prop:discreteintersection}. To that end, we define various tiles and regions.
For any hook-quotient partition $\lambda$ of $e$-weight $w$, define
the solid parallelotope $\hull{\para(\lambda)}$ as the convex hull of $\para(\lambda)$ in $\mathbb{R}^w$, and the solid hypercube
$\hull{\cube(\lambda)}$ as the convex hull of $\cube(\lambda)$ in $\widehat{\mathbb{R}^w}$, so that
 $$\hull{\para(\lambda)}=\left\{\z(\lambda)+\sum_{i=1}^{w} a_i \sb^\lambda_i
 \,\bigg|\,
  0\leq a_i \leq 1 \text{ for all } i\in[1,w]\right\}\subset \mathbb{R}^w$$
 and
 $$\hull{\cube(\lambda)}=\left\{\hat\z(\lambda)+\sum_{i=1}^{w} a_i \hat\sb^\lambda_i
 \,\bigg|\,
  0\leq a_i \leq 1 \text{ for all } i\in[1,w]\right\}\subset \widehat{\mathbb{R}^w}.$$
We additionally define the half-open solid parallelotope
 $$\hhull{\para(\lambda)}=\left\{\z(\lambda)+\sum_{i=1}^{w} a_i \sb^\lambda_i
 \,\bigg|\,
  0\leq a_i < 1 \text{ for all } i\in[1,w]\right\}\subset \mathbb{R}^w.$$

Turning to regions, we define $\hull{\goodplus}=\left\{x\in\mathbb{R}\mid 0\leq x \leq e\right\}^w$ and
$\hhull{\goodplus}=\left\{x\in\mathbb{R}\mid 0\leq x < e\right\}^w$, so that
\begin{align*}
\hullinc{\goodplus}{m}&=\{(z_1,\dotsc,z_w) \in \mathbb{R}^w\mid 0\leq z_i \leq e \text{ and } z_{i+1}-z_{i}\geq m \text{ for all admissible } i\}, \\
\hhullinc{\goodplus}{m}&=\{(z_1,\dotsc,z_w) \in \mathbb{R}^w\mid 0\leq z_i < e \text{ and } z_{i+1}-z_{i}\geq m \text{ for all admissible } i \}
\end{align*}
for any nonnegative integer $m$.

\begin{prop}\label{disjointtiling}
	Let $B$ be a block of $e$-weight $w$. Then
\begin{align*}
\coprod_{\lambda\in B} \hhullinc{\para(\lambda)}{4} &= \hhullinc{\goodplus}{4}, \\
\bigcup_{\lambda\in B} \hullinc{\para(\lambda)}{4} &= \hullinc{\goodplus}{4}.
\end{align*}
	Both unions, the first of which is disjoint, are taken over all hook-quotient partitions in $B$.
		\end{prop}
As is the case with Proposition~\ref{prop:discreteunion}, the unions here may be taken over just the $1$-increasing partitions in $B$, by the obvious real analogue of Lemma~\ref{lemma:shiftby3}.  Furthermore
 $\hullinc{\para(\lambda)}{4}=\hull{\para(\lambda)}$ and
 $\hhullinc{\para(\lambda)}{4}=\hhull{\para(\lambda)}$ as long as $\lambda$ is $7$-increasing.

\begin{proof} We prove the first equality; the second then follows by taking closures in $\mathbb{R}^w$.
The argument mirrors the proof of Proposition~\ref{prop:discreteunion}.	
Let $B$ be a Rouquier block of weight $w$. We have seen that for any $1$-increasing partition $\lambda$ in $B$, we have
$\sb^\lambda_i=\ssb_i$ for all $i\in[1,w]$. In particular $\hhull{\para(\lambda)}$ is contained in $\hhull{\goodplus}$. Conversely,
let $\mathbf{x}=(x_1,\dotsc,x_w)$ be a $4$-increasing element of $\hhull{\goodplus}$. Then there is a unique $1$-increasing partition
 $\lambda$ in $B$ such that $\mathbf{x}\in\hhull{\para(\lambda)}$, namely the unique partition $\lambda$ in $B$ such that
 $z(\lambda)=(\lfloor x_1 \rfloor,\dotsc,\lfloor x_w \rfloor)$.

 Now suppose the theorem holds for a block $\b$ of weight $w$, and let $\tb=s_a(\b)$. Any $1$-increasing partition $\tlambda\in\tb$ is either
 nonexceptional or belongs to a hook-quotient exceptional family. If the former is true, we have $\hhull{\para(\lambda)}=\hhull{\para(\tlambda)}$,
 where $\lambda=s_a(\tlambda)$. To complete the proof we establish the following analogue of Lemma~\ref{lemma:paraexceptional}(5):
 For any hook-quotient exceptional family $\flambda$, we have
 \begin{equation}\label{twounions}
 \coprod_{i=0}^{k+1} \hhull{\para(\lambda^i)} = \coprod_{i=0}^{k+1} \hhull{\para(\tlambda^i)},
 \end{equation}
 where both unions are disjoint.

Let
$$\hhull{\para(\flambda)}:=\z(\lambda^0)+\left\{\sum_{j=0}^{k+1}a_j\pb{\flambda}_j+\sum_{i\in\external{\flambda}}b_i\sb_i^{\flambda}
\,\bigg|\,
 0\leq a_j, b_i \leq 1, \text{ with }
(a_j,a_{j'})\neq (0,1) \text{ if } j<j'\right\}.$$
The closure of $\hhull{\para(\flambda)}$ in $\mathbb{R}^w$ is the convex hull of the set $\Pi(\flambda)$ defined in Lemma~\ref{lemma:paraexceptional}.
The expression of an element of $\hhull{\para(\flambda)}$ is unique up to replacing
$a_j$ by $a_j+c$ for all $j\in[0,k+1]$, for some $c\in\mathbb{R}$. In particular for each element there is a unique {\em minimal expression} with $a_j=0$ for some $j$ and a unique {\em maximal expression} with $a_j=1$ for some $j$.
It is straightforward to deduce from Lemma~\ref{lemma:paraexceptional} the following: for each $j\in[1,\,w]$, we have that $\hhull{\para(\lambda^j)}$ is the subset of $\hhull{\para(\flambda)}$ consisting of elements whose minimal expression satisfies
$\min\{j'\mid a_{j'}=0\}=j$,
and likewise
$\hhull{\para(\tlambda^i)}$ is the subset  consisting of elements whose maximal expression satisfies
$\max\{j'\mid a_{j'}=1\} = k+1-j$. It follows that both of the unions in (\ref{twounions}) are disjoint and equal to $\hhull{\para(\flambda)}$.
\end{proof}

\begin{prop}\label{prop:realintersection}
	Let $\lambda$ and $\mu$ be hook-quotient partitions in $B$. Then $\hullinc{\para(\lambda)}{7}\cap\hullinc{\para(\mu)}{7}=\inc{F}{7}$, where $F$ is a common face of $\hull{\para(\lambda)}$ and $\hull{\para(\mu)}$. In particular, if either $\lambda$ or $\mu$ is $10$-increasing, then $\para(\lambda)$ and $\para(\mu)$ intersect in a common face.
\end{prop}

It is clear what we mean by a face of $\hull{\para(\lambda)}$ or of $\hull{\cube(\lambda)}$,
as both are real polytopes in an obvious way.
Before we prove Proposition \ref{prop:realintersection}, we note the following corollary. See, e.g.\, \cite[8.1]{Z}, for background on polytopal complexes.

\begin{cor} \label{C:polytopalcomplex}
	The solid parallelotopes $\hull{\para(\lambda)}$, as $\lambda$ ranges over all $10$-increasing partitions in $B$, are the $w$-cells of a pure polytopal complex in $\mathbb{R}^w$.
	\end{cor}
	
As in the discrete case, any two hypercubes $\hull{\cube(\lambda)}$ and $\hull{\cube(\mu)}$ in $\widehat{\mathbb{R}^w}$ clearly intersect in a common face. Thus Proposition \ref{prop:realintersection} is a consequence of the following lemma. The proof depends on Lemma~\ref{C:discretecubesbijectparas}, but due to the nature of the argument, works only for $7$-increasing points.

\begin{lemma} \label{C:cubesbijectparas}
	Let $B$ be a block of $e$-weight $w$.
	The projection map $\proj: \widehat{\mathbb{R}^w} \to \mathbb{R}^w$ maps $\bigcup_{\lambda} \hullinc{\cube(\lambda)}{7}$ bijectively onto
	$\hullinc{\goodplus}{7}$, where $\lambda$ runs over all hook-quotient partitions in $B$. 
\end{lemma}

\begin{proof}
	By Example \ref{example:Rouquier}, for any $1$-increasing partition $\lambda$ in a Rouquier block (of $e$-weight $w$), we have $\z(\lambda) = \hat\z(\lambda)$, and $\sb^{\lambda}_i = \ssb_i$, and hence $\hat\sb^{\lambda}_i = \ssb_i$, for all $i \in [1,w]$.  Thus $\proj$, when restricted to $\cube(\lambda)$, is the identity map.  This shows that the lemma holds when $B$ is the Rouquier block.
	
	To complete the proof, it suffices to show that if $B$ and $\tb$ are blocks of $e$-weight $w$ forming a $[w:k]$-pair, and the lemma holds for $B$, then it holds for $\tb$.  To this end, consider first a ($0$-increasing) non-exceptional hook-quotient partition $\lambda$ in $B$, associated with the non-exceptional hook-quotient partition $\tlambda$ in $\tb$.  Then $\sb^{\lambda}_i = \sb^{\tlambda}_i$ for all $i \in [1,w]$ by Lemma \ref{L:nonexceptional2}.  It is also easy to check that $\hat\z(\lambda) = \hat\z(\tlambda)$, so that $\cube(\lambda) = \cube(\tlambda)$, and hence $\hull{\cube(\lambda)} = \hull{\cube(\tlambda)}$.  Next let $\flambda$ be a hook-quotient exceptional family.  By Lemma \ref{L:cubeexceptional}(5),
	\begin{align*}
		\cube(\flambda)^{\mathbb{R}}_B : = \bigcup_{j=0}^{k+1} \hull{\cube(\epl{j})} &= \left\{ \hat\z(\epl{0}) + \sum_{\gamma=0}^{k+1} a_{\gamma} \hpb{\flambda}_\gamma + \sum_{x \in \external{\flambda}} b_x \sb^{\flambda}_x \,\bigg|\, 0 \leq a_{\gamma},b_x \leq 1\ \forall \gamma,x,\ a_{\gamma} =0 \text{ for some }\gamma \right\}, \\
		\cube(\flambda)^{\mathbb{R}}_{\tb} : = \bigcup_{j=0}^{k+1} \hull{\cube(\eptl{j})} &= \left\{ \hat\z(\epl{0}) + \sum_{\gamma=0}^{k+1} a_{\gamma} \hpb{\flambda}_\gamma + \sum_{x \in \external{\flambda}} b_x \sb^{\flambda}_x \,\bigg|\, 0 \leq a_{\gamma},b_x \leq 1\ \forall \gamma,x,\ a_{\gamma} =1 \text{ for some }\gamma \right\}.
	\end{align*}
	For each $\mathbf{x} = \hat\z(\epl{0}) + \sum_{\gamma=0}^{k+1} a_{\gamma} \hpb{\flambda}_\gamma + \sum_{x \in \external{\flambda}} b_x \sb^{\flambda}_x \in \cube(\flambda)^{\mathbb{R}}_B$, define $f_{\flambda}(\mathbf{x})$ to be $\mathbf{x} + (1-\mathbf{x}_{\max})\sum_{\gamma=0}^{k+1} \hpb{\flambda}_{\gamma}$, where $\mathbf{x}_{\max} = \max\{ a_{\gamma} \mid \gamma \in [0,k+1]\}$.  Then $f_{\flambda}$ is a bijection from $\cube(\flambda)^{\mathbb{R}}_B$ to $\cube(\flambda)^{\mathbb{R}}_{\tb}$, satisfying $\proj (f_{\flambda}(\mathbf{x})) = \proj(\mathbf{x})$ for all $\mathbf{x} \in \cube(\flambda)^{\mathbb{R}}_B$.  In particular, for any $m \in \mathbb{Z}$, we have $\mathbf{x} \in \inc{(\cube(\flambda)^{\mathbb{R}}_B)}{m}$ if and only if $f_{\flambda}(\mathbf{x}) \in \inc{(\cube(\flambda)^{\mathbb{R}}_{\tb})}{m}$.
	
	Let $f : \bigcup_{\lambda \in B} \hullinc{\cube(\lambda)}{7} \to \bigcup_{\tlambda \in \tb} \hullinc{\cube(\tlambda)}{7}$ be defined as follows:
	$$
	f(\mathbf{x}) =
	\begin{cases}
	\mathbf{x}, &\text{if } \mathbf{x} \in \hull{\cube(\flambda)} \text{ for some non-exceptional hook-quotient partition $\lambda$ in $B$}; \\
	f_{\flambda}(\mathbf{x}), &\text{if } \mathbf{x} \in \cube(\flambda)^{\mathbb{R}}_B \text{ for some hook-quotient exceptional family $\flambda$}.
	\end{cases}
	$$
	We claim that $f$ is a well-defined function.
	For this, we only need to consider the case where $\mathbf{x} \in \cube(\flambda)^{\mathbb{R}}_B$ for some hook-quotient exceptional family $\flambda$, with $\internal{\flambda} = \{ i_0,\dotsc,i_k \}$ say, and $f_{\flambda}(\mathbf{x}) \ne \mathbf{x}$.
	In this case, $\mathbf{x}_{\max} \ne 1$, so that
	$$\mathbf{x} = \hat\z(\epl{0}) + \sum_{\gamma=0}^{k+1} a_{\gamma} \hpb{\flambda}_{\gamma} + \sum_{x \in \external{\flambda}} b_x \hat\sb^{\flambda}_x$$
	with $a_{\gamma} < 1$ for all $\gamma$.  Let
	$$\mathbf{y} = \hat\z(\epl{0}) + \sum_{x \in \external{\flambda}} \lfloor b_x \rfloor \hat\sb^{\flambda}_x.$$
	Then $\proj(\mathbf{y}) = \z(\epl{0}) + \sum_{x \in \external{\flambda}} \lfloor b_x \rfloor \sb_x^{\flambda} \in \para(\epl{0})$, and is $4$-increasing by (the real analogues of) Lemmas \ref{preshiftby3} and \ref{lemma:shiftby3}.
	By considering the block $B^+$ of $(e+1)$-weight $w$ introduced in Remark~\ref{remark:addrunner} if necessary and arguing as in the proof of Lemma \ref{C:discretecubesbijectparas}, we may assume that $\proj(\mathbf{y}) = \z(\nu)$ for a partition $\nu$ in $B$; in fact, by Lemma \ref{L:leading}, $\nu = \nu^0$, the leading member of a hook-quotient exceptional family $\fnu$ with $\internal{\flambda} = \internal{\fnu}$.
	By Proposition \ref{P:parallelepipedsversuscubes}, we have $\hat\z(\epn{0}) = \mathbf{y}$.
	Note also that $\hat\z(\epn{0}) = \mathbf{y} \in \cube(\mu)$ for any hook-quotient partition $\mu$ in $B$ satisfying $\hull{\cube(\mu)} \ni \mathbf{x}$.
	
	Suppose for the sake of contradiction that $\mathbf{x} \in \hull{\cube(\mu)}$ for some non-exceptional partition $\mu$ in $B$.  Then $\mathbf{y} \in \cube(\mu) = \cube(\tmu)$, where $\tmu$ is the partition in $\tb$ associated to $\mu$, we thus have $\hat\z(\eptn{0}), \mathbf{y} \in \bigcup_{\tilde{\rho} \in \tb} \inc{\cube(\tilde{\rho})}{4}$.  But $\hat\z(\eptn{0}) \ne \hat\z(\epn{0}) = \mathbf{y}$ by Lemma \ref{L:cubeexceptional}(1) while $\proj(\hat\z(\eptn{0})) = \z(\eptn{0}) = \z(\epn{0}) = \proj(\mathbf{y})$ by Lemma \ref{lemma:paraexceptional}(1), contradicting Lemma \ref{C:discretecubesbijectparas}.
	
	Now, if $\mathbf{x} \in \cube(\fmu)^{\mathbb{R}}_B$ for another hook-quotient exceptional family $\fmu$, say
	$$\mathbf{x} = \hat\z(\epm{0}) + \sum_{\gamma = 0}^{k+1} a'_{\gamma} \hpb{\fmu}_{\gamma} + \sum_{x\in \external{\fmu}} b'_x \hat\sb^{\fmu}_x \in \hull{\cube(\epm{j})},$$
	then $\hat\z(\epn{0}) \in \cube(\epm{j})$ and hence $\z(\epn{0}) \in \para(\epm{j})$, say $\z(\epn{0}) = \z(\epm{j}) + \sb^{\epm{j}}_{\Gamma}$.  By (the $B$-analogue of) Proposition \ref{P:moveinparallelogramtilde}, $\z(\epn{0}) = \z(\eps{j}) + \sb^{\fsigma}_{\Gamma \cap \internal{\fsigma}}$ for some hook-quotient exceptional family $\fsigma$ with $\internal{\fsigma} = \internal{\fmu}$ and $\z(\csigma) = \z(\cmu) + \sb^{\cmu}_{\zeta_{\fmu}(\Gamma \cap \external{\fmu})}$.  But by (the $B$-analogue of) Lemma \ref{L:internalparallelogram}, we must then have $\nu^0 = \sigma^0$; in particular
	$$\internal{\flambda} = \internal{\fnu} = \internal{\fsigma} = \internal{\fmu} = \{i_0,\dotsc, i_k\}.$$
	In addition $\Gamma \cap \internal{\fmu} = \Gamma \cap \internal{\fsigma} = \{ i_0,\dotsc, i_{j-1} \}$.  Consequently,
	$$\z(\epn{0}) = \z(\epm{j}) + \sb^{\epm{j}}_{\Gamma} = \z(\epm{j}) + \sb^{\epm{j}}_{\{i_0,\dotsc, i_{j-1} \}} + \sb^{\fmu}_{\Gamma\cap \external{\fmu}} = \z(\epm{0}) + \sb^{\fmu}_{\Gamma\cap \external{\fmu}}, $$ and so
	\begin{equation}
	\hat\z(\epl{0}) + \sum_{x \in \external{\flambda}} \lfloor b_x \rfloor \hat\sb^{\flambda}_x = \mathbf{y} = \hat\z(\epn{0}) = \hat\z(\epm{0}) + \hat\sb^{\fmu}_{\Gamma\cap \external{\fmu}} \label{3}
	\end{equation}
	by Proposition \ref{P:parallelepipedsversuscubes}.
	
	Let $\chi$ be the projection from $\bigoplus_{i=1}^w \mathbb{R}\ssb_i \oplus \bigoplus_{1\leq i < j \leq w} \mathbb{R}\ssb_{ij}$ onto $\mathbb{R}\ssb_{i_0} \oplus \mathbb{R}\ssb_{i_{k}} \oplus \bigoplus_{\gamma =1}^k \mathbb{R} \ssb_{i_{\gamma-1},i_{\gamma}}$, along the span of the other standard basis vectors.  Then
	$$
	\chi(\hat\z(\epl{0})) + \sum_{\gamma = 0}^{k+1} a_{\gamma}\pb{\flambda}_{\gamma} = \chi(\mathbf{x}) = \chi(\hat\z(\epm{0})) + \sum_{\gamma = 0}^{k+1} a'_{\gamma}\pb{\fmu}_{\gamma}.
	$$
	By applying $\chi$ to \eqref{3}, we get
	$\chi(\hat\z(\epl{0})) = \chi(\hat\z(\epn{0})) = \chi(\hat\z(\epm{0})$.  Combining the two equations, we get $a_\gamma = a'_{\gamma}$ for all $\gamma$.  Hence $f_{\flambda}(\mathbf{x}) = f_{\fmu}(\mathbf{x})$.  This proves our claim that $f$ is well-defined.
	
	To complete our proof, note first that $f$ is clearly surjective, and suppose that $\proj(\tilde{\mathbf{a}}) = \proj(\tilde{\mathbf{b}})$ for some $\tilde{\mathbf{a}},\tilde{\mathbf{b}}\in \bigcup_{\tlambda \in \tb} \hullinc{\cube(\tlambda)}{7}$.  Then $\tilde{\mathbf{a}} = f(\mathbf{a})$ and $\tilde{\mathbf{b}} = f(\mathbf{b})$ for some $\mathbf{a},\mathbf{b} \in \bigcup_{\lambda \in B} \hullinc{\cube(\lambda)}{7}$.  Then
	$$
	\proj(\mathbf{a}) = (\proj \circ f)(\mathbf{a}) = \proj(f(\mathbf{a})) = \proj(\tilde{\mathbf{a}}) = \proj(\tilde{\mathbf{b}}) = \proj(f(\mathbf{b})) = (\proj \circ f)(\mathbf{b}) = \proj(\mathbf{b}),$$
	so that $\mathbf{a} = \mathbf{b}$ since the lemma holds for $B$, and hence $\tilde{\mathbf{a}} = f(\mathbf{a}) = f(\mathbf{b}) = \tilde{\mathbf{b}}$, and our proof is complete.
\end{proof}

\begin{lemma}\label{lem:goodedges}
Let $\lambda$ and $\mu$ be $4$-increasing partitions with $e$-weight $w$ lying in the same block, and suppose that $\|\hat\z(\lambda) - \hat\z(\mu)\| =1$.  Then there exists $i \in [1,w]$ such that $\z(\mu) =\z(\lambda) + \sb^{\lambda}_i$ or $\z(\lambda) = \z(\mu) + \sb^{\mu}_i$.
\end{lemma}

\begin{proof}
By interchanging $\lambda$ and $\mu$ if necessary, we may assume that $\hat\z(\mu) - \hat\z(\lambda) = \ssb_i$ or $\ssb_{ij}$. Then
$\z(\mu) - \z(\lambda) = \proj(\hat\z(\mu) - \hat\z(\lambda)) = \ssb_i$ or $\ssb_i - \ssb_j$.
Let $\mathbf{z} = \frac{1}{2} (\z(\lambda) + \z(\mu))$.  Then $\mathbf{z} \in \hhullinc{\goodplus}{4}$, so that, by Proposition \ref{disjointtiling}, $\mathbf{z} = \z(\nu) + \sum_{i=1}^w a_i \sb^{\nu}_i$ for some $1$-increasing partition $\nu$, with $0 \leq a_i < 1$ for all $i \in [1,w]$.
Note that both $\ssb_i$ and $\ssb_i - \ssb_j$ are of the form $\sb^{\nu}_{\Gamma} - \sb^{\nu}_{\Delta}$ where $\Gamma$ and $\Delta$ are disjoint subsets of $[1,w]$.
Thus
$$
\z(\nu) + \sum_{i=1}^w a_i\sb^{\nu}_i = \mathbf{z} = \z(\lambda) + \frac{1}{2}(\z(\mu) - \z(\lambda)) = \z(\lambda) + \frac{1}{2}(\sb^{\nu}_{\Gamma} - \sb^{\nu}_{\Delta}),$$
so that
$$
\frac{1}{2}(\sb^{\nu}_{\Gamma} - \sb^{\nu}_{\Delta}) - \sum_{i=1}^w a_i\sb^{\nu}_i = \z(\nu) - \z(\lambda) \in \BZ^w = \bigoplus_{i=1}^w \BZ\sb^{\nu}_i,$$
and hence
$$
a_i =
\begin{cases}
\frac{1}{2}, &\text{if } i \in \Gamma \cup \Delta; \\
0, &\text{otherwise}.
\end{cases}
$$
In other words, $\sum_{i=1}^w a_i\sb^{\nu}_i = \frac{1}{2}(\sb^{\nu}_{\Gamma} + \sb^{\nu}_{\Delta})$.  Thus,
\begin{align}
\z(\lambda) &= \z(\nu) + \sb^{\nu}_{\Delta}, \label{1} \\
\z(\mu) &= \z(\nu) + \sb^{\nu}_{\Gamma}, \label{2}
\end{align}
and hence
\begin{align*}
\hat\z(\lambda) &= \hat\z(\nu) + \hat\sb^{\nu}_{\Delta}, \\
\hat\z(\mu) &= \hat\z(\nu) + \hat\sb^{\nu}_{\Gamma}
\end{align*}
by Proposition \ref{P:parallelepipedsversuscubes}.  This yields
$$
1 = \| \hat\z(\lambda) - \hat\z(\mu) \| = \| \hat\sb^{\nu}_{\Delta} - \hat\sb^{\nu}_{\Gamma} \| = |\Gamma| + |\Delta|.
$$
Hence one of $\Gamma$ and $\Delta$ is empty, while the other is a singleton set.  Substituting this into \eqref{1} and \eqref{2} and using Theorem \ref{thm:goodlabels}(1), the lemma follows.
\end{proof}

We deduce that the full subgraph of the $\Ext^1$-quiver of a block of a $v$-Schur algebra in characteristic $0$ generated by the simple modules labelled by $4$-increasing partitions is determined by the map $\hat\z$, in a particularly easy way.

\begin{thm}\label{thm:extquiver}
	Let $\lambda$ and $\mu$ be $4$-increasing partitions with $e$-weight $w$ lying in the same block. Let $L(\lambda)$ and $L(\mu)$ be the corresponding simple modules of (a block of) a $v$-Schur algebra over a field of characteristic zero in which $v$ is a primitive $e$-th root of unity \cite{DJ}. The following statements are equivalent:
\begin{enumerate}
\item $\Ext^1(L(\lambda),L(\mu)) \ne 0$.
\item $\Ext^1(L(\lambda), L(\mu))$ is one-dimensional.
\item $d_{\lambda\mu}(q) = q$ or $d_{\mu\lambda}(q)=q$.
\item There exists $i \in [1,w]$ such that $\z(\mu) = \z(\lambda) + \sb^{\lambda}_i$ or $\z(\lambda) = \z(\mu) + \sb^{\mu}_i$.
\item $\| \hat\z(\lambda) - \hat\z(\mu)\| =1$.
\end{enumerate}
\end{thm}

\begin{proof}
	The $q$-decomposition numbers are the graded decomposition numbers of graded versions of the $v$-Schur algebras at primitive $e$-th roots of unity in fields of characteristic zero \cite{A2,SW}. Moreover, these graded $v$-Schur algebras are
	standard Koszul algebras \cite{Webster}; it follows from the general theory of this class of algebras that (1)---(3) are equivalent \cite{ADL}. Statements (3) and (4) are equivalent by Theorem \ref{thm:main}, while (4) and (5) are equivalent by Proposition \ref{P:parallelepipedsversuscubes} and Lemma \ref{lem:goodedges}.
	\end{proof}

\section{The Mullineux-Kleshchev involution}\label{section:Mullineux}
In this section we prove that the Mullineux-Kleshchev involution on $e$-regular partitions translates via $\lambda\mapsto\z(\lambda)$ to very simply defined involutions on $\mathbb{Z}^w$, $w\geq0$, as long as $\lambda$ is $0$-increasing.

			 	Recall that a  partition
			 	$\lambda=(\lambda_1,\lambda_2,\dotsc)$ is called {\em $e$-regular} if there does not exist $i$ such that $\lambda_i=\lambda_{i+1}=\dotsb=\lambda_{i+e-1}\neq 0$. Put
			 	$$\goodreg:=[1,e-1]^w \subset \mathbb{Z}^w,$$
			 	so that
			 		$$\reginc{\good}{0}=\{(z_1,\dotsc,z_w) \in \mathbb{Z}^w\mid
			 		1\leq z_1\leq \dotsb\leq z_w\leq e-1
			 		)\}.$$

			\begin{prop} \label{P:e-regular} \hfill
				\begin{enumerate}
					\item Let $\lambda$ be a partition of $e$-weight $w$, and let $\z(\lambda)=(\z_1,\dotsc,\z_w)$. Then  $\lambda$ is $e$-regular if and only if $\z_i>0$ for all $i\in [1,\,w]$.
					\item  The set of $e$-regular $0$-increasing partitions in any block of $e$-weight $w$ is in bijection with
					 $\reginc{\good}{0}$ via $\lambda\mapsto\z(\lambda)$.
					\end{enumerate}
									\end{prop}
			
			\begin{proof}
				Let $\bm(\lambda)=\{(b_1,q_1),\dotsc,(b_w,q_w)\}$, ordered as usual.
				Suppose that $\z_i=0$. Then $\{q_i+e-1,\ldots, q_i-1, q_i\}\subset\beta(\lambda)$ and $r\notin\beta(\lambda)$ for some $r<_e q_i$, and
				hence $\lambda$ is $e$-singular. Conversely, suppose that $\lambda$ is $e$-singular. There exists $r<s$ such that $r\notin\beta(\lambda)$ and
				$s,s+1,\dotsc, s+e-1\in\beta(\lambda)$.	We can in fact assume that $s=r+1$, and then
				we find $b_i=q_i=r+e$ and $\z_i=0$ for some $i\in[1,\,w]$. This completes the proof of the first assertion. The second then follows, c.f.\ Theorem~\ref{thm:goodlabels}.
			\end{proof}

				Recall the Mullineux-Kleshchev involution $\lambda\mapsto \lambda^*$ on the set of $e$-regular partitions (see, for example, \cite[\S7]{LLT}); it preserves $e$-weights and conjugates $e$-cores. There are at least two ways of calculating $\lambda^*$ from $\lambda$, the original description of Mullineux \cite{M} and the later reinterpretation of Kleshchev \cite{Kl}, both of which may be formulated as being given by a series of `bead moves' on the abacus. The maximum number of moves required to carry out either algorithm grows linearly with $|\lambda|$, insensitive to the $e$-weight of $\lambda$.

				Consider the involution
				$$*:\mathbb{Z}^w \to \mathbb{Z}^w;\ (z_1,\dotsc,z_w)\mapsto(e-z_w,\dotsc, e-z_1).$$
				It restricts to an involution of
				 $\reginc{\good}{0}$, and more generally to an involution of $\reginc{\good}{m}$ and of $\plusinc{\good}{m}$ for all $m\in\mathbb{Z}$.
				
			\begin{prop} \label{prop:Mullineux1}
					Let $\lambda$ be a $0$-increasing $e$-regular partition of $e$-weight $w$.  Then $\z(\lambda^*) =\z(\lambda)^*$.
				\end{prop}
				
				\begin{proof}
				Suppose first that $\lambda$ lies in a Rouquier block.
				We may write $\z(\lambda)$ as $(1^{w_1},\dotsc, (e-1)^{w_{e-1}})$, by
				Proposition \ref{P:e-regular}(1).
				  Then $\lambda$ has $e$-quotient $(\emptyset, (1^{w_1}),\dotsc, (1^{w_{e-1}}))$ by
				Example \ref{example:Rouquier}.
								Thus $\lambda^*$ has $e$-quotient $(\emptyset,(1^{w_{e-1}}), (1^{w_{e-2}}),\dotsc, (1^{w_1}))$ by \cite[Theorem 2.1]{P} (see also \cite[Proposition 3.7]{T}).  Using Example \ref{example:Rouquier} again, we see that $\z(\lambda^*) = (1^{w_{e-1}},2^{w_{e-2}},\dotsc,(e-1)^{w_{1}})$, and thus the claim is verified.
				For the general case, we use the fact that the Mullineux-Kleshchev involution intertwines
				the action of Kashiwara's operators $\kashe_a$ and $\kashf_a$ on the set of $e$-regular partitions with that of
				$\kashe_{e-1-a}$ and $\kashf_{e-1-a}$ \cite{Kl}, and hence the induced action
				of the simple reflection $s_a$ of the affine Weyl group $\W$ with that of $s_{e-1-a}$.  If the proposition holds for a $0$-increasing $e$-regular partition $\lambda$, so that in particular $\lambda^*$ is $0$-increasing, then
				$\z(s_a(\lambda)^*) = \z(s_{e-1-a}(\lambda^*)) = \z(\lambda^*) = \z(\lambda)^*= \z(s_a(\lambda))^* $,
				by Theorem~\ref{thm:goodlabels}, so it holds for $s_a(\lambda)$ as well. Since every $\W$-orbit contains a partition in a Rouquier block, we are done.
			\end{proof}

Next we turn our attention to parallelotopes.

\begin{prop}\label{prop:Mullineux2}
	Let $\lambda$ be a hook-quotient partition of $e$-weight $w$. Then the conjugate partition $\lambda'$ is a hook-quotient partition, and
	$\para(\lambda')=\para(\lambda)^* $. More precisely, for any $\Gamma\subseteq [1,\,w]$, we have $ \z(\lambda')+\sb_{\Gamma'}^{\lambda'}= (\z(\lambda)+\sb_\Gamma^\lambda)^*$,
	where $\Gamma'=\{i\in[1,\,w]\mid w+1-i\notin\Gamma\}$.
	\end{prop}
	\begin{proof} We first write the $\beta$-numbers, $e$-quotient, bead movements, $\z$ and modified basis vectors associated to the conjugate partition $\lambda'$ in terms of those of $\lambda$. We have $\beta(\lambda')=\{x\mid 1-x\notin \beta(\lambda)\}$. Thus runner $a$ of the abacus of $\lambda'$ is obtained from  runner $e-1-a$ of $\lambda$ by turning it upside down and interchanging occupied and unoccupied positions. In particular
	the components of $e$-quotients are related by the formula $\lambda'^{(a)} = (\lambda^{(e-1-a)})'$, for $a\in[0,e-1]$, so that $\lambda'$ is a hook-quotient partition if and only if $\lambda$ is.

Let $\lambda$ be a hook-quotient partition, and let $q_1 < \dotsb < q_w$ and $q'_1 < \dotsb < q'_w$ be the starting positions of the bead movements of $\lambda$ and $\lambda'$, respectively.  Note that for any partition $\mu$, $x \in \mathbb{Z}$ is a starting position of some bead movement of $\mu$ if and only if there exist $x^+,x^- \in \mathbb{Z}$ such that $x \leq_e x^+ \in \beta(\mu)$ and $x >_e x^- \notin \beta(\mu)$, if and only if there exist $y^+,y^- \in \mathbb{R}$ ($y^\pm = 1 - x^{\mp}$) such that $1-x \geq_e y^- \notin \beta(\mu')$ and $1-x <_e y^+ \in \beta(\mu')$, if and only if $1-x+e$ is a starting position of some bead movement of $\mu'$. It follows that $q'_i = 1-q_{w+1-i}+e)$. Next, writing
	$\z(\lambda)=(z_1,\dotsc,z_w)$ and $\z(\lambda')=(z'_1,\dotsc,z'_w)$, we have, for all $i\in[1,w]$,
$z_i=\left|\left\{x\in (q_i-e,q_i]\mid x\notin\beta(\lambda)\right\}\right|$ and	
$z'_{w+1-i}=| |\{ x \in (q'_{w+1-i}-e, q'_{w+1-i}]  \mid x \notin \beta(\lambda) \} =\left|\left\{x\in [q_i-e,q_i)\mid x\in\beta(\lambda)\right\}\right|$.
Hence $z_i + z'_{w+1-i} = e - \I_{q_i\in\beta(\lambda)} +\I_{q_i-e\in\beta(\lambda)}$.
Now
\begin{align*}
	\sb_{[1,w]}^{\lambda} &= \sum_{\substack{i\in[1,w]: \; q_i\in\beta(\lambda)\\ \text{and } q_i-e\notin\beta(\lambda)}} \ssb_i
	-\sum_{\substack{i\in [1,w]: \; q_i\notin\beta(\lambda)\\ \text{and } q_i-e\in\beta(\lambda)}} \ssb_i
	\\
	&= \sum_{i=1}^{w}\left(\I_{q_i\in\beta(\lambda)} - \I_{q_i-e\in\beta(\lambda)}\right)\ssb_i \\
	&= \sum_{i=1}^{w}\left( e-z_i - z'_{w+1-i}\right)\ssb_i \\
	&= \z(\lambda')^* - \z(\lambda).
		\end{align*}
This completes the proof when $\Gamma= [1,w]$. The general case follows from the following observation: for all $i,j\in[1,w]$, we have $\sb^{\lambda}_i=\ssb_i$ if and only if $\sb^{\lambda'}_{w+1-i} = \ssb_{w+1-i}$ and
$\sb^{\lambda}_i=\ssb_i-\ssb_j$ if and only if $\sb^{\lambda'}_{w+1-i} = \ssb_{w+1-i}-\ssb_{w+1-j}$.
		\end{proof}

\begin{cor}\label{cor:Mullineux}
	Let $\lambda$ be a hook-quotient partition and $\mu$ an $e$-regular $0$-increasing partition in the same block of $e$-weight $w$. Then $\z(\mu)\in\para(\lambda)$ if and only if $\z(\mu^*)\in\para(\lambda')$, and if so, then
	$\dist{\lambda}{\mu}+\dist{\lambda'}{\mu^*}=w$. 	
	\end{cor}

\begin{rem}\label{lastremark:Mullineux}	
Theorem 9.4 of \cite{LLT} states that for any pair of partitions $\lambda,\mu$ in a block of weight $w$, with $\mu$ $e$-regular,
\begin{equation}\label{LLT:Mullineux}
	d_{\lambda' \mu^*}(q) = q^{w} d_{\lambda\mu}(q^{-1}).
\end{equation}	 	
Through Corollary~\ref{cor:Mullineux}, we see that the statement of Theorem~\ref{thm:main} is compatible with Equation~(\ref{LLT:Mullineux}) for any hook-quotient partition $\lambda$ and $0$-increasing
$e$-regular partition $\mu$, even though the theorem is established only when $\mu$ is $4$-increasing.	
On the other hand, since for any hook-quotient partition $\lambda$ it is always the case that $d_{\lambda\lambda}(q)=1$ and
$\z(\lambda)\in\Pi(\lambda)$ with $\dist{\lambda}{\lambda}=0$, the formula in the statement of Theorem~\ref{thm:main} correctly computes $d_{\lambda',\lambda^*}(q)=q^w$ for any $0$-increasing hook-quotient partition $\lambda$.		
\end{rem}

\begin{rem}\label{remark:calculateMullineux}
	As an important special case of Propositions~\ref{prop:Mullineux1} and~\ref{prop:Mullineux2}, we have
		$\z(\lambda^*)=\z(\lambda')+\sb_{[1,w]}^{\lambda'}$ for any $0$-increasing hook-quotient partition of weight $w$.
		If $\lambda$ (and hence $\lambda^*$) is $4$-increasing, we can thus efficiently compute its Mullineux-Kleshchev image $\lambda^*$ by first conjugating $\lambda$ and then applying a sequence of bead operations, as described in Remark~\ref{remark:movingalgorithm}. In fact, as mentioned in that remark, even when $\lambda$ is merely $0$-increasing, it may still be possible to find $\lambda^*$ in this way. The maximum number of bead operations required to carry out this procedure grows quadratically in $w$, independently of $|\lambda|$ and of $e$, and is thus more efficient than either Mullineux or Kleshchev's algorithms when $w$ is small compared to $|\lambda|$.
\end{rem}

\begin{rem}
	The obvious analogues of the results of this section with $\z$, $\sb$ and $\para$ replaced by $\hat\z$, $\hat\sb$ and $\cube$ hold true with respect to the involution
	$$*:\widehat{\mathbb{Z}^w} \to \widehat{\mathbb{Z}^w};\ \sum z_i \ssb_i + \sum z_{ij}\ssb_{ij} \mapsto \sum (e-z_{w+1-i}) \ssb_i + \sum z_{w+1-j, w+1-i}\ssb_{ij}
	$$
	extending that on $\mathbb{Z}^w$ defined above.
\end{rem}

\end{document}